\documentclass[a4paper]{amsart}

\usepackage{amsfonts}
\usepackage{amssymb}
\usepackage{amsmath}
\usepackage{hyperref}
\usepackage{mathrsfs}
\usepackage{centernot}
\usepackage{mathdots}
\usepackage{stmaryrd}
\usepackage[all]{xy}

\usepackage{mathtools}

\usepackage{color}

\newtheorem{thm}{Theorem}[section]
\newtheorem{lem}[thm]{Lemma}
\newtheorem{prp}[thm]{Proposition}
\newtheorem{cor}[thm]{Corollary}

\newtheoremstyle{roman} 
    {8.0pt plus 2.0pt minus 4.0pt}                    
    {8.0pt plus 2.0pt minus 4.0pt}                    
    {\normalfont}                
    {}                           
    {\bfseries}                  
    {.}                          
    {5pt plus 1pt minus 1pt}     
    {}  

\theoremstyle{roman}

\newtheorem{notation}[thm]{Notation}

\newtheorem{example}[thm]{Example}
\newtheorem{remark}[thm]{Remark}
\newtheorem{reduction}[thm]{Reduction}

\theoremstyle{plain}

\newcommand{\Step}[1]{\noindent {\it Step #1.}} 

\newcommand{\rem}[1]{}

\newcommand{\gap}[1]{{\color{blue} #1}}

\newcommand{\N}{\mathbb{N}}

\newcommand{\Z}{\mathbb{Z}}

\newcommand{\HH}{{\mathrm{H}}}

\newcommand{\frakm}{{\mathfrak{m}}}

\newcommand{\frakp}{{\mathfrak{p}}}
\newcommand{\frakq}{{\mathfrak{q}}}

\newcommand{\calP}{{\mathcal{P}}}

\newcommand{\calR}{{\mathcal{R}}}
\newcommand{\calS}{{\mathcal{S}}}

\newcommand{\bfB}{{\mathbf{B}}}

\newcommand{\bfG}{{\mathbf{G}}}

\newcommand{\bfK}{{\mathbf{K}}}

\newcommand{\bfU}{{\mathbf{U}}}
\newcommand{\bfV}{{\mathbf{V}}}

\newcommand{\bfX}{{\mathbf{X}}}

\newcommand{\veps}{\varepsilon}
\newcommand{\vphi}{\varphi}

\newcommand{\idealof}{\unlhd} 

\newcommand{\suchthat}{\,:\,}
\newcommand{\where}{\,|\,}

\newcommand{\quo}[1]{\overline{#1}}

\newcommand{\Trings}[1]{\left< #1 \right>}

\newcommand{\trings}[1]{\langle {#1} \rangle}

\DeclareMathOperator{\ann}{ann}

\DeclareMathOperator{\Aut}{Aut} %
\DeclareMathOperator{\Br}{Br} %
\DeclareMathOperator{\Cent}{Cent} %
\DeclareMathOperator{\disc}{disc} %
\DeclareMathOperator{\End}{End} %
\DeclareMathOperator{\Hom}{Hom} %
\DeclareMathOperator{\id}{id} %
\DeclareMathOperator{\im}{im} %
\DeclareMathOperator{\ind}{ind} %
\DeclareMathOperator{\Int}{Int} %
\DeclareMathOperator{\Jac}{Jac} %
\DeclareMathOperator{\lcm}{lcm}

\DeclareMathOperator{\Max}{Max}
\DeclareMathOperator{\Nr}{Nr} %
\DeclareMathOperator{\Nrd}{Nrd} %
\newcommand{\op}{\mathrm{op}} %
\DeclareMathOperator{\rank}{rank}
\DeclareMathOperator{\Spec}{Spec} %
\DeclareMathOperator{\Sym}{Sym} %
\DeclareMathOperator{\Tr}{Tr} %
\DeclareMathOperator{\Trd}{Trd} %

\newcommand{\nGL}[2]{\mathrm{GL}_{#2}({#1})}

\newcommand{\nMat}[2]{\mathrm{M}_{#2}(#1)}

\newcommand{\trans}{{\mathrm{t}}}

\newcommand{\uGL}{{\mathbf{GL}}}
\newcommand{\uPGL}{{\mathbf{PGL}}}
\newcommand{\uSL}{{\mathbf{SL}}}

\newcommand{\uO}{{\mathbf{O}}}
\newcommand{\uSO}{{\mathbf{SO}}}

\newcommand{\uSp}{{\mathbf{Sp}}}

\newcommand{\uU}{{\mathbf{U}}}

\newcommand{\umu}{{\boldsymbol{\mu}}}

\newcommand{\Zar}{\mathrm{Zar}}
\newcommand{\et}{\mathrm{\acute{e}t}}
\newcommand{\fppf}{\mathrm{fppf}}
\newcommand{\fpqc}{\mathrm{fpqc}}

\newcommand{\units}[1]{{#1^\times}}

\usepackage[foot]{amsaddr}  

\usepackage{enumitem} 

\usepackage{dsfont} 

\numberwithin{equation}{section} 

\renewcommand{\rank}{\operatorname{rk}}

\newcommand{\Herm}[2][]{\mathcal{H}^{#1}(#2)}     

\newcommand{\rrk}{\operatorname{rrk}}         

\newcommand{\Hyp}[2][]{\mathds{h}^{#1}_{#2}}      

\newcommand{\Lag}{{\mathrm{Lag}}}
\newcommand{\uLag}{{\mathbf{Lag}}}

\newcommand{\rproj}[1]{{\calP({#1})}}              

\newcommand{\SMatII}[4]{\left[\begin{array}{cc} {#1} & {#2} \\ {#3} &
{#4} \end{array}\right]}
\newcommand{\smallSMatII}[4]{\left[\begin{smallmatrix} {#1} & {#2} \\ {#3} &
{#4} \end{smallmatrix}\right]}

\newcommand{\Aff}{{\mathcal{A}f\!f}}  

\renewcommand{\Sym}{{\mathcal{S}}}    
\renewcommand{\Cent}{{\mathrm{Z}}}    

\title[An Exact Sequence of Witt Groups]{An $8$-Periodic
Exact Sequence of Witt Groups of Azumaya Algebras with Involution}

\author{Uriya A.\ First$^*$}
\date{\today}
\address{$^*$Department of Mathematics, University of Haifa}
\email{uriya.first@gmail.com}


\makeatletter
\@namedef{subjclassname@2020}{%
  \textup{2020} Mathematics Subject Classification}
\makeatother


\begin{document}

\maketitle

\begin{abstract}
Given an Azumaya algebra with involution $(A,\sigma)$ over a commutative ring
$R$ and some auxiliary data, we 
construct an $8$-periodic chain complex
involving the Witt groups of $(A,\sigma)$ and other algebras with involution, and prove 
it is exact when $R$ is semilocal.
When $R$ is a field, this recovers an  $8$-periodic exact sequence
of Witt groups of  Grenier-Boley and Mahmoudi, which in turn generalizes exact
sequences of Parimala--Sridharan--Suresh
and Lewis.
We  apply this result in several ways: We establish the Grothendieck--Serre
conjecture on principal homogeneous bundles and the local purity
conjecture
for certain outer forms of $\uGL_n$ and $\uSp_{2n}$, provided some assumptions on $R$.
We show that a $1$-hermitian form over a quadratic \'etale or quaternion Azumaya
algebra over a semilocal ring $R$ is isotropic if and only if its trace
(a quadratic form over
$R$)
is isotropic, generalizing a result of Jacobson. 	
We also apply it to characterize the kernel of the restriction
map $W(R)\to W(S)$ when $R$ is a (non-semilocal) $2$-dimensional regular domain
and $S$ is a quadratic \'etale $R$-algebra, generalizing a theorem of Pfister.
In the process, we establish  many fundamental results concerning   Azumaya algebras with involution
and hermitian forms over them.
\end{abstract}

\section*{Introduction}

	Central simple algebras with involution over fields,
	in the sense of \cite[\S2]{Knus_1998_book_of_involutions},
	play a major role in the study of classical algebraic groups.
	Indeed, all forms of $\uGL_n$, $\uO_n$  and $\uSp_{2n}$ arise as
	the algebraic groups of unitary elements in a central simple algebra with involution.
	
	When the base field is replaced with a (commutative) ring
	$R$ (always with $2\in\units{R}$), the role of central simple algebras with involution
	is played by   Azumaya  algebras with involution. These
	are   the locally free $R$-algebras with $R$-involution $(A,\sigma)$
	which specialize to a central simple algebra with involution at the residue 
	field of every prime $\frakp\in\Spec R$. 
	
	The Witt group of $\veps$-hermitian forms over $(A,\sigma)$, 
	denoted $W_\veps(A,\sigma)$, is an important invariant of $(A,\sigma)$,
	capturing fine arithmetic properties. For example,
	when $A$ is a field $F$ 
	and $\sigma=\id_F$, the affirmation of the
	quadratic form version of Milnor's conjecture by 
	Orlov, Vishik and Voevodsky \cite{Orlov_2007_Minlor_conj_quad_forms} shows
	that the cohomology groups $\HH^i_{\et}(F,\umu_{2,F})$ can be recovered
	from $W(F):=W_1(F,\id_F)$; this was recently generalized to the case
	where $A $ is a semilocal commutative ring by Jacobson \cite{Jacobson_2018_cohomological_invariants}.
	
	In this paper, we introduce an $8$-periodic chain
	complex involving the Witt group of $(A,\sigma)$ ---
	an Azumaya algebra with involution over $R$ ---
	and prove it is exact when $R$ is semilocal.
	Several applications of the exactness are then presented.

\subsection*{The Main Result}

	Let $R$ be a   commutative  ring with $2\in\units{R}$,
	let $(A,\sigma)$ be an Azumaya $R$-algebra with involution
	($\sigma$ is applied exponentially)
	and    let $\veps\in A$
	be a central element such that $\veps^\sigma\veps=1$.
	Let $\lambda,\mu\in \units{A}$ be elements
	satisfying  $\lambda^\sigma=-\lambda$, $\mu^\sigma=-\mu$, $\lambda\mu=-\mu\lambda$
	and $\lambda^2\in R$.
	Then  the centralizer of $\lambda$ in $A$, denoted
	$B$, is Azumaya     over its center
	and $\tau_1:=\sigma|_B$ and $\tau_2:=\Int(\mu^{-1})\circ \tau_1$
	are involutions of $B$.
    We construct an  octagon, i.e.\ an $8$-periodic  chain complex,   
	of Witt groups:
	\begin{equation}\label{EQ:into-oct}
    \xymatrix{
    W_\veps(A,\sigma) \ar[r]^{\pi_1^{(\veps)}} & 
    W_\veps(B,\tau_1) \ar[r]^{\rho_1^{(\veps)}} & 
    W_{-\veps}(A,\sigma) \ar[r]^{\pi_2^{(-\veps)}} & 
    W_{\veps}(B,\tau_2) \ar[d]^{\rho_2^{(\veps)}}\\
    W_{-\veps}(B,\tau_2) \ar[u]^{\rho_2^{(-\veps)}} & 
    W_{\veps}(A,\sigma)  \ar[l]^{\pi_2^{(\veps)}} & 
    W_{-\veps}(B,\tau_1) \ar[l]^{\rho_1^{(-\veps)}} & 
    W_{-\veps}(A,\sigma) \ar[l]^{\pi_1^{(-\veps)}} 
    }
    \end{equation}
    Its maps  are induced by
    functors between the relevant categories of hermitian
    forms; their definition, which depends on $\lambda$ and $\mu$, 
    is given in \ref{subsec:octagon} below.
    
	The main result of this paper (Theorem~\ref{TH:exactness}) 
	asserts that the octagon \eqref{EQ:into-oct}
	is exact when $R$ is semilocal.    
    
\medskip

	When $R$ is a field,   \eqref{EQ:into-oct} is isomorphic
    to an   octagon of Witt groups introduced by Grenier-Boley and 
    Mahmoudi \cite[\S6]{Grenier_2005_octagon_of_Witt_grps}, who also proved
    it is exact.\footnote{
		The term $W_{-\veps}(B,\tau_1)$
		on the bottom row of \eqref{EQ:into-oct}
		and the maps adjacent to it
		differ from their counterparts in \cite[p.~980]{Grenier_2005_octagon_of_Witt_grps}. 
		However, the octagons become  
		the same once identifying the term $W_{-\veps}(B,\tau_1)$ on the bottom
		row of \eqref{EQ:into-oct}
		with the corresponding term $W_{\veps}(B,\tau_1)$ in 
		{\it op.\ cit.}
		via \emph{$\lambda$-conjugation} (``scaling by $\lambda$'') 
		in the sense of \ref{subsec:conjugation} below. 
    } 
    Many special cases of the latter result were known previously.
    For example, 
    Parimala, Sridharan and Suresh \cite[Appendix]{Bayer_1995_Serre_conj_II}
    established the exactness of the top row
    of \eqref{EQ:into-oct} when $R$ is a field.
    Furthermore, the $5$-term and $7$-term exact sequences of Witt groups
    that Lewis \cite{Lewis_1982_improved_exact_sequences} associates
    to quadratic field extensions and quaternion division algebras, respectively,
    can be recovered from \eqref{EQ:into-oct} when $R$ is a field.
    Predating Lewis, 
    Baeza  \cite[Korollar~2.9]{Baeza_1974_torsion_of_Witt_group_semilocal}
    and
    Mandelberg \cite[Proposition~2.1]{Mandelberg_1975_classification_of_quad_forms_semilocal_ring} 
    established the exactness of Lewis' $5$-term   sequence at two  places
    when $R$ is semilocal; 
	we extend these works in \ref{subsec:applications:low-degree}, showing that both
	of  Lewis' sequences  
	remain exact when base ring is semilocal.

	When $R$ is a general, the octagon  \eqref{EQ:into-oct} seem related to the 
	octagon of $L$-groups
	considered by Ranicki in \cite[Remark~22.22]{Ranicki_1992_algebraic_L_theory}.
	Other octagons involving Witt groups  of central simple algebras with involution
	appear in \cite{Lewis_1983_Witt_groups_equivariant_forms} and
	\cite{Lewis_1985_periodicity_of_Clifford_algebras}.
	
\medskip

	In the process of  proving the exactness of   \eqref{EQ:into-oct} when 
	$R$ is semilocal,
	we give necessary and sufficient conditions for a hermitian space
	to be in the image 
	of the functors $\pi_1^{(\pm \veps)},\pi_2^{(\pm \veps)},\rho_1^{(\pm \veps)}, \rho_2^{(\pm \veps)}$ 
	(the exactness of the octagon answers this only up to Witt equivalence),
	see Theorem~\ref{TH:finer-exactness}.
	These conditions, which seem novel even when $R$ is a field,
	involve the Brauer classes of $A$ and $B$  and the discriminant of the hermitian space at hand; 
	they are needed  for some of the applications.
	For example, 
	given a unimodular $(-\veps)$-hermitian space $(P,f)$ over $(A,\sigma)$,
	we show that
	there exists a unimodular $ \veps $-hermitian space $(Q,g)$ over $(B,\tau_1)$
	such that $\rho_1^{(\veps)}(Q,g)\cong (P,f)$ if and only if
	$\pi_2(P,f)$ is hyperbolic and at least one of the following hold: 
	(1)~$(\tau_2, \veps)$ is not orthogonal (see~\ref{subsec:Az-alg-inv}), 
	(2)~the
	Brauer class of $B$ is nontrivial, 
	(3)~$(\tau_2, \veps)$ is   orthogonal, $n:=\frac{\rank_R P}{\deg A}$ is even
	and the discriminant of $f$ (see~\ref{subsec:disc}) equals $ \lambda^n (\units{R})^2$.
	We further we show that any \emph{anistropic} hermitian space whose Witt class
	lives in the kernel of some map in \eqref{EQ:into-oct} is the image of a hermitian space
	under the functor corresponding to the preceding map in
	\eqref{EQ:into-oct}, see Corollary~\ref{CR:exactness-anisotropic}.

\medskip
	
	While proving that the octagon \eqref{EQ:into-oct} is exact
	when $R$ is a field takes only several pages, showing
	the exactness when $R$ is semilocal is significantly more involved;
	the proof occupies most of this paper, and is surveyed in \ref{subsec:proof-overview}.
	One reason why the field case is   simpler
	is the fact that when $R$ is a field, 
	every Witt class contains a representative with no isotropic vectors,
	which allows for a short clean proof; see
	Remark~\ref{RM:Es-for-field}. 
	In contrast,
	the proof when $R$ is semilocal  relies on
	two ingredients: careful  analysis of the image of the functors $\pi_*^{(\pm \veps)}$
	and $\rho_*^{(\pm\veps)}$ when $R$ is a field,
	and lifting of  information from the residue fields of $R$ to $R$ itself, usually using results
	from \cite{First_2020_orthogonal_group}. 
	The complexity of the former ingredient manifests
	in the length of Theorem~\ref{TH:finer-exactness}, which describes the images
	of 	$\pi_*^{(\pm \veps)}$ and $\rho_*^{(\pm\veps)}$ when $R$ is semilocal.
	
\medskip

	We do not know if \eqref{EQ:octagon}
	remains exact when $R$ is not assumed to be semilocal.
	However, Theorem~\ref{TH:kernel-of-rest-quads} below 
	(see also the proof of Corollary~\ref{CR:exactness-for-quad-etale})	
	can be regarded
	as a positive partial result when $R$ is a regular $2$-dimensional domain.
	We further note that
	if the Witt group $W_\veps(A,\sigma)$
	is replaced by the Zariski sheaf
	associated to the presheaf $U\mapsto W_\veps(A_{U},\sigma_{U})$
	on $\Spec R$, then \eqref{EQ:octagon} becomes an exact sequence of sheaves.
	Indeed, it is exact   at the  stalks.

	The octagon \eqref{EQ:into-oct} also seems to be related with Bott periodicity.
	Clarifying  this connection seems an interesting problem,
	which may lead to new insights.

\subsection*{Applications}

	A celebrated
	application of
	the well-known exactness
	of 	\eqref{EQ:into-oct} when $R$ is a field is
	Bayer-Fluckiger and Parimala's proof of Serre's Conjecture II for classical groups
	\cite{Bayer_1995_Serre_conj_II}.

	Knowing that \eqref{EQ:into-oct} is exact when $R$ is semilocal
	allows for a new set of applications.
	Here, we use it to establish some open cases of the Grothendieck--Serre
	conjecture  and the  local purity
	conjecture, 
	show that the trace of a $1$-hermitian form over a quadratic \'etale or quaternion
	Azumaya algebra 
	is isotropic if and only if the original form is isotropic,	
	and 
	characterize the kernel of the restriction map $W (R )\to W (S )$
	when $R$ is an arbitrary $2$-dimensional  regular domain and $S$ is a quadratic \'etale 
	$R$-algebra.	
	Further applications appear in \cite{Bayer_2019_Gersten_Witt_complex_prerprint}.

\medskip

	In more detail,  given a  regular local ring $R$   with fraction field $K$,
	Grothendieck 
	\cite[Remark~3, pp.~26-27]{Grothendieck_1958_homological_torsion},
	\cite[Remark~1.11.a]{Grothendieck_68_groupe_de_Brauer_II}
	and Serre 
	\cite[p.~31]{Serre_1958_espaces_fibres_algebriques}
	conjectured that for every reductive (connected) group
	$R$-scheme $\bfG$, the kernel of the restriction map 
	\[\HH^1_{\et}(R,\bfG)\to \HH^1_{\et}(K,\bfG) \]
	is trivial.
	Under the same assumptions,
	the  local purity conjecture  predicts that
	\[
	\im \left(\HH^1_{\et}(R,\bfG)\to \HH^1_{\et}(K,\bfG)\right)=
	{\textstyle\bigcap_{\frakp\in R^{(1)}}} \left(\HH^1_{\et}(R_\frakp,\bfG)\to \HH^1_{\et}(K,\bfG)\right),
	\]
	where $R^{(1)}$ is the set of height-$1$ primes  of   $R$;
	we then say that purity holds for $\bfG$.
	In fact, both conjectures are  believed to hold under
	the milder assumption that $R$ is a regular \emph{semilocal} domain,
	which we assume through the following paragraphs.
	

	The Grothendieck--Serre   conjecture  was addressed by numerous authors
	and is known to hold in many cases.
	Most notably, Nisnevich \cite{Nisnevich_1984_rationally_triv_PHSs}
	proved
	the conjecture when $R$ is a discrete
	valuation ring, Guo 
	\cite{Guo_2019_Groth_Serre_conj_semilocal_Dedekind}
	established the case where
	$R$ is a semilocal Dedekind domain,
	and
	Fedorov--Panin \cite{Fedorov_2015_Grothendieck_Serre_conj}
	and Panin \cite{Panin_2020_Grothendieck_Serre_R_contains_finite_fld}
	proved the conjecture when
	$R$ contains a field.
	Many positive results for particular groups $\bfG$ 
	are known as well, see  \cite[\S5]{Panin_2018_Grothendieck_Serre_conj_survey} for
	a survey.

	The local purity conjecture is also known to hold in many cases:
	Colloit-Th\'el\`ene and Sansuc showed
	that it holds for all reductive group schemes
	when $\dim R\leq 2$, even without assuming that $R$ is semilocal,
	see \cite[Corollary~6.14]{Colliot_1979_quadratic_fiberations}.
	When $R$ is a regular  local ring containing a field $k$ of characteristic $0$, purity
	was established for $\uO_n$, $\uSO_n$, $\uPGL_n$, $\uSL_{1}(A)$ ($A$ is a central simple
	$k$-algebra), $\uSL_n/\umu_d$ ($d\mid n$) and $\mathbf{Spin}_n$
	in \cite{Panin_2010_purity_conj_reductive_grps}, and
	for groups of type $\mathrm{G}_2$ in \cite{Chernousov_2007_purity_for_G2}.
	In fact, for $\uO_n$, it is enough to assume that $k$ is any
	field of characteristic not $2$,
	see Scully \cite[p.~12]{Scully_Artin_Springer_over_semilocal}
	and also
	Panin--Pimenov \cite[Corollary~3.1]{Panin_2010_rational_isotropy_implies_isotropy}.

\medskip
	
	To relate  the octagon \eqref{EQ:into-oct} to the Grothendieck--Serre conjecture, let $(A,\sigma)$
	be a degree-$n$ Azumaya $R$-algebra with involution  and
	let $\uU(A,\sigma)$ denote the group $R$-scheme of unitary
	elements in $(A,\sigma)$. Then $\uU(A,\sigma)$ is a form of $\uGL_n$, $\uO_n$ or $\uSp_n$,
	depending on whether $\sigma$ is unitary, orthogonal or symplectic, respectively.
	We show (Proposition~\ref{PR:GS-equiv-conds}) that if the restriction map 
	\begin{align}\label{EQ:into-rest-W}
	W_1(A,\sigma)\to W_1(A_K,\sigma_K) 
	\end{align}
	is injective, then the Grothendieck--Serre conjecture holds for the neutral component of $\uU(A,\sigma)$
	(and, more generally, for the neutral component of the isometry
	group  scheme of any unimodular $1$-hermitian form over $(A,\sigma)$);
	this is well-known when $A=R$ \cite[Proposition~1.2]{Colliot_1979_quadratic_forms_over_semilocal_rings}.

	In accordance with the Grothendieck--Serre conjecture,
	it is conjectured that \eqref{EQ:into-rest-W} is  
	injective when $R$ is regular semilocal.
	Provided $2\in\units{R}$, 
	this has
	been established by Balmer--Walter \cite[Corollary~10.4]{Balmer_2002_Gersten_Witt_complex}
	(see also Pardon \cite{Pardon_1982_Gersten_conjecture})
	and Balmer--Preeti \cite[p.~3]{Balmer_2005_shifted_Witt_groups_semilocal}
	when $\dim R\leq 4$ and $A=R$,
	and when $R$ is local and contains a field by Gille  \cite[Theorem~7.7]{Gille_2013_coherent_herm_Witt_grps}.
	We use the former result and the exactness
	of \eqref{EQ:into-oct} to establish
	the injectivity of $W_1(A,\sigma)\to W_1(A_K,\sigma_K)$ in when
	$\dim R\leq 4$ and one of the following hold:
	\begin{enumerate}[label=(\arabic*)]
		\item $\sigma$ is unitary and $\ind A=1$;  
		\item $\sigma$ is symplectic and $\ind A\leq 2$;
	\end{enumerate}
	see Theorem~\ref{TH:GS-new}.
	As a result, the Grothendieck--Serre conjectures holds for $\bfU(A,\sigma)$ if (1) or (2)
	holds and $\dim R\leq 4$
	(Corollary~\ref{CR:GS-new}).

	By similar means, we use   \eqref{EQ:into-oct} together  
	with results
	of Gille \cite[Theorem~7.7]{Gille_2013_coherent_herm_Witt_grps} 
	and Scully \cite[Theorem~5.1]{Scully_Artin_Springer_over_semilocal}
	to show that purity holds for $\uU(A,\sigma)$ in cases (1) and (2), provided
	that $R$ is  regular  local and contains a field  of characteristic not $2$
	(Theorem~\ref{TH:purity-new}).
	Here, the exactness of \eqref{EQ:into-oct} is not sufficient, and
	we have to use the finer information provided by  Theorem~\ref{TH:finer-exactness} about 
	the image  of $\pi_*^{(\pm\veps)}$, $\rho_*^{(\pm \veps)}$.
	
\medskip

	Suppose next that $R$ is any semilocal ring and let $(A,\sigma)$
	be a quadratic \'etale $R$-algebra with its standard involution,
	or a degree-$2$ Azumaya $R$-algebra with its (unique) symplectic involution.
	Write $\Tr$ for the trace map from $A$ to $R$.
	If $(P,f)$ is a unimodular $1$-hermitian space over $(A,\sigma)$,
	then $(P,\Tr\circ f)$ is a unimodular $1$-hermitian space over $(R,\id_R)$.
	We show that $\Tr\circ f$ is isotropic if and only if $f$ is isotropic
	(Theorem~\ref{TH:Jacobson-semilocal}).
	When $R$ is a field, this  goes back to Jacobson \cite{Jacobson_1940_note_on_hermitian_forms}
	(see also \cite[Theorems~10.1.1, 10.1.7]{Scharlau_1985_quadratic_and_hermitian_forms}).
	The quick proof in the case $R$ is a field does not apply over rings, see
	Remark~\ref{RM:why-is-Jacobson-difficult}, and we instead appeal
	to our Theorem~\ref{TH:finer-exactness}.

\medskip

	Finally, assume that $R$ is any regular domain%
	, possibly non-semilocal,
	let $S$ be a quadratic \'etale $R$-algebra
	and let $\theta$ be its standard involution  
	(see \ref{subsec:quad-etale}).
	When $S$ is a field, a famous theorem of Pfister \cite[Theorem~I.5.2]{Scharlau_1985_quadratic_and_hermitian_forms}
	states that the kernel of the restriction
	map $W (R )\to W (S )$ is generated by the diagonal quadratic form $\langle 1,-\alpha\rangle$,
	where $S=R[\sqrt{\alpha}]$. Using our main result and 
	Colloit-Th\'el\`ene and Sansuc's purity result
	\cite[Corollary~6.14]{Colliot_1979_quadratic_fiberations},
	we generalize Pfister's theorem to the case where $\dim R\leq 2$, showing that
	the sequence
	\[
	W_1(S,\theta)\xrightarrow{[g] \mapsto [\Tr_{S/R}\circ g]} W (R )\xrightarrow{[f]\mapsto [f_S]}
	W(S )
	\]
	is exact in the middle (Theorem~\ref{TH:kernel-of-rest-quads}).
	This result also applies in the generality of quadratic \'etale coverings of regular
	integral $2$-dimensional schemes.

\subsection*{Additional Results}

	The first two sections of this paper are concerned with generalizing
	many fundamental  results about central simple algebras with
	involution and hermitian forms over them 
	to the context of   Azumaya algebras with involution over semilocal rings.
	For example, 
	letting $(A,\sigma)$ denote an Azumaya algebra with involution
	over a semilocal ring $R$ with $2\in\units{R}$,
	and letting $(P,f)$ be a unimodular $\veps$-hermitian space over $(A,\sigma)$,
	it is shown that:
	\begin{itemize} 
		\item $A$ contains a \emph{full} idempotent
		$e\in A$ with 
		$\deg eAe=\ind A$ (Theorem~\ref{TH:index-description}).
		\item If $\sigma$ is orthogonal or unitary,
		then the idempotent $e$   can be chosen
		to satisfy $e^\sigma=e$ (Theorem~\ref{TH:invariant-primitive-idempotent}).
		\item $(P,f)$ cancels from orthogonal sums (Theorem~\ref{TH:Witt-cancellation});
		this is essentially due to Reiter
		\cite{Reiter_1975_Witts_extension_theorem} and		
		Keller \cite{Keller_1988_cancellation_of_quadratic_forms}.
		\item If the Witt class of $(P,f)$ is $0$, then $(P,f)$ 
		is hyperbolic. If $(P',f')$ is Witt equivalent to $(P,f)$
		and $P\cong P'$, then $(P,f)\cong (P',f')$
		(Theorem~\ref{TH:trivial-in-Witt-ring}).
		\item If $(\sigma,\veps)$ is orthogonal or unitary, 
		then $(P,f)$ is diagonalizable whenever $P$ is a free $A$-module
		(Proposition~\ref{PR:diagonalizable-herm-forms}).
		\item When $\Cent(A)$ is connected, the isometry group of $(P,f)$ acts transitively
		on the set of Lagrangians of $(P,f)$, provided
		it is nonempty (Lemma~\ref{LM:Lag-transitive-action}).
	\end{itemize}
	We note that the first result is false when $R$ is not semilocal, see \cite{Antieu_2014_unramified_division_algebras}.
	The second result is particularly convenient when hermitian Morita theory
	is needed.


\subsection*{Outline}

	Sections~\ref{sec:separable}
	and~\ref{sec:hermitian} are preliminary and recall
	Azumaya algebras with involution and hermitian forms, respectively.
	In Section~\ref{sec:octagon}, we construct the octagon \eqref{EQ:into-oct}, prove it is
	a chain complex,  
	and survey the proof of its exactness when $R$ is semilocal. 
	The proof itself is carried in Sections~\ref{sec:preparation}--\ref{sec:Eii}  and is   concluded
	in Section~\ref{sec:completion-of-proof}.
	Finally, the applications to the Grothendieck--Serre
	conjecture, the local purity
	conjecture and the generalizations of Jacobson and Pfister's theorems are given in Section~\ref{sec:applications}.
	
\subsection*{Acknowledgements}

We are grateful to Eva Bayer-Fluckiger for suggesting us the project at hand.
We further thank Eva Bayer-Fluckiger and Raman Parimala for many
useful conversations and suggestions.
The research was partially conducted at the department of mathematics
at University of British Columbia,
where the author was supported by a post-doctoral fellowship.
We thank Ori Parzanchevski for  encouragement and motivation.

We are also grateful to anonymous referees for   many
useful suggestions which have improved the exposition.

\section*{Notation and Conventions}

	Throughout this paper, a ring means a commutative (unital) ring.
	Algebras are unital and associative, but not necessarily commutative.
	\emph{We assume   that $2$ is invertible in all rings and algebras.}
	
	Unless otherwise indicated, $R$ denotes a   ring.
	Unadorned tensors and $\Hom$-sets are always taken over $R$.
	An $R$-ring means a commutative $R$-algebra.
	Given $\frakp\in \Spec R$, we let $k(\frakp)$ denote
	the fraction field of $R/\frakp$.

	Let $S$ be an $R$-ring.
	Given (right) $R$-modules $M$, $N$ and $f\in\Hom(M,N)$,
	we write $M_S:=M\otimes S$ and $f_S:=f\otimes \id_S\in\Hom_S(M_S,N_S)$.
	When $S=k(\frakp)$ for $\frakp\in \Spec R$,
	we write $M(\frakp)=M_{k(\frakp)}$ and $f(\frakp)=f_{k(\frakp)}$,
	and let $m(\frakp)$
	denote the image  of $m\in M$ in $M(\frakp)$.
	When $S=R_\frakp$, we write
	$M_\frakp=M_{R_\frakp}$ and $f_\frakp=f_{R_\frakp}$.
	
	Let  $M$ be a finite  (i.e.\ finitely
	generated) projective $R$-module.
	The $R$-rank of $M$, denoted $\rank_R M$,
	is   the function $\Spec R\to \Z_{\geq 0}$ sending $\frakp$
	to $\dim_{k(\frakp)} M(\frakp)$; it is locally
	constant relative to the Zariski topology  \cite[Theorem~2.3.5]{Ford_2017_separable_algebras}.
	Thus, when $R$ is connected, we shall freely regard $\rank_RM$ as an integer.
	
	Statements and operations involving locally constant functions from $\Spec R$ to $\Z$
	should be interpreted point-wise. 
	For example, the sum   of two such functions is taken point-wise,
	and relations such as ``$<$'' should be understood as holding after evaluation 
	at every $\frakp\in\Spec R$.
	
	We will   need
	to compare integer-valued functions defined on spectra
	of different rings. 
	To that end, given a ring homomorphism  $\iota:R\to S$ and  
	$f:\Spec R\to \Z$, define $\iota f:\Spec S\to \Z$
	by $(\iota f)(\frakq)=f(\iota^{-1}(\frakq))$.
	For example,  $\iota \rank_R M=\rank_S M_S$.
	In addition,   if $S$ is finite projective
	over $R$ and $N$ is a finite projective $S$-module of rank
	that is constant along the fibers of $\Spec S\to \Spec R$,
	then  
	\begin{align}\label{EQ:rank-change-of-rings}
	\rank_SN\cdot \iota\rank_RS=\iota\rank_RN.
	\end{align}

	Given an $R$-algebra $A$, the units, the center, the Jacobson radical and the opposite algebra of
	$A$ are denoted $\units{A}$, $\Cent(A)$, $\Jac A$ and $A^\op$, respectively.
	We write $\Cent_A(X)$ for the centralizer of a subset $X\subseteq A$ in $A$.
	The category of finite   projective right $A$-modules is denoted $\rproj{A}$.
	If $a\in \units{A}$, then $\Int(a)$ denotes the inner automorphism
	$x\mapsto axa^{-1}:A\to A$. Given  an $R$-ring $S$, and $P,Q\in\rproj{A}$,
	the natural map
	$\Hom_{A}(P,Q)\otimes S\to \Hom_{A_S}(P_S,Q_S)$ is an isomorphism
	\cite[Theorem~1.3.26]{Ford_2017_separable_algebras}%
	, and we shall freely
	identify these $S$-modules.
	
	In situations when an abelian group $M$ can be regarded as a module over multiple $R$-algebras,
	we shall sometimes write $M_A$ to denote ``$M$, viewed as a right $A$-module''.
	In particular, $A_A$ denotes ``$A$, viewed as a right module over itself''.
	Similar notation will be applied to left modules, but with the subscript written
	on the left, e.g., ${}_AA$.

	If $e\in A$ is an idempotent, we shall freely
	identify
	$\End_A(eA)$ with $eAe$, where $eAe$ acts on $eA$ via multiplication on the left.
	We say that $e$ is \emph{full} if $AeA=A$,
	or equivalently, if   $eA$ is a progenerator \cite[\S18B]{Lam_1999_lectures_on_modules_rings}.
	(A right $A$-module
	$M$ is called a progenerator
	if $M$ is finite projective and $A_A$ is isomorphic to a summand of $M^n$ for some $n\in\N$.) 
	The idempotent $e$ is called \emph{primitive} if $e\neq 0$ and $eAe$
	contains no idempotents except $0$ and $e$.
	
	
	An $R$-algebra with involution means a pair $(A,\sigma)$
	consisting of an $R$-algebra $A$ and an
	{\it $R$-linear} involution $\sigma:A\to A$. Involutions
	are applied exponentially to elements of $A$, i.e., $a^\sigma$ stands for $\sigma(a)$.
	Given $\veps\in\Cent(A)$ with
	$\veps^{\sigma}\veps=1$, we let
	$\Sym_\veps(A,\sigma)=\{a\in A\suchthat a=\veps a^\sigma\}$.

\section{Azumaya Algebras With Involution}
\label{sec:separable}

We recall the definition and some properties of Azumaya algebras
with involution, giving particular attention to the case
where the base ring $R$ is semilocal.
When $R$ is a field, all the material can be found in \cite[Chapter~I]{Knus_1998_book_of_involutions}.

\subsection{Separable Projective Algebras}

Recall that an $R$-algebra $A$ is called
separable if $A$ is projective when endowed with the right $A^\op\otimes A$-module
structure  determined by $a\cdot (x^\op\otimes y)=xay$, or equivalently,
if the
right $A^\op\otimes A$-module homomorphism   $x^\op\otimes y\mapsto xy:A^\op\otimes A\to A$ admits 
an $A^\op\otimes A$-linear section.

By definition, the \emph{Azumaya} $R$-algebras are the central separable $R$-algebras,
and the \emph{finite \'etale} $R$-algebras are the finite projective commutative separable  $R$-algebras.
There are many other equivalent definitions, see \cite{Ford_2017_separable_algebras} and \cite[III.\S5]{Knus_1991_quadratic_hermitian_forms}, for instance.

In the sequel, we shall often consider $R$-algebras $A$ such that
$A$ is Azumaya over $\Cent(A)$ and $\Cent(A)$ is finite \'etale over $R$.
The following proposition lists
a few   convenient equivalent characterizations of such algebras,
which we  call   \emph{separable projective} after condition \ref{item:sep-proj:trivial-def}.

\begin{prp}
	Let $A$ be an $R$-algebra.
	The following conditions are equivalent.
	\begin{enumerate}[label=(SP\arabic*)]
	\item \label{item:sep-proj:Az-over-fin-et} $A$ is Azumaya over $\Cent(A)$
	and $\Cent(A)$ is finite \'etale over $R$.
	\item \label{item:sep-proj:trivial-def} $A$ is projective as an $R$-module
	and separable as an $R$-algebra.
	\item \label{item:sep-proj:fibers-def} $A$ is finite projective as an $R$-module 
	and, for all $\frakm\in \Max R$, the $k(\frakm)$-algebra $A(\frakm)$
	is semisimple and its center is a   product
	of separable field extensions of $k(\frakm)$.
\end{enumerate}
\end{prp}

\begin{proof}
	\ref{item:sep-proj:Az-over-fin-et}$\implies$\ref{item:sep-proj:trivial-def} follows from 
	\cite[Theorem~II.3.4(iii), Theorem~II.3.8]{DeMeyer_1971_separable_algebras}.
	\ref{item:sep-proj:trivial-def}$\implies$\ref{item:sep-proj:Az-over-fin-et} follows from
	\cite[Theorem~II.3.8, Lemma~II.3.1]{DeMeyer_1971_separable_algebras}.
	\ref{item:sep-proj:trivial-def}$\implies$\ref{item:sep-proj:fibers-def} 
	follows from \cite[Proposition~II.2.1, Corollary~II.2.4]{DeMeyer_1971_separable_algebras}
	and the fact that \ref{item:sep-proj:trivial-def} continues to hold after base-change.
	\ref{item:sep-proj:fibers-def}$\implies$\ref{item:sep-proj:trivial-def}
	follows from 
	\cite[Theorem~II.7.1, Corollary~II.2.4]{DeMeyer_1971_separable_algebras}.
\end{proof}

We collect several facts about separable projective algebras. 

\begin{lem}[{\cite[Proposition~2.14]{Saltman_1999_lectures_on_div_alg}}]\label{LM:projective-transfer}
	Let $A$ be a separable   
	$R$-algebra and let $M$ be a right 
	$A$-module.
	If $M$ is projective over $R$, then $M$
	is projective over $A$. The  converse
	holds when $A$ is projective over $R$.
\end{lem}

\begin{lem}\label{LM:separable-subring-I}
	Let $A$ be a separable projective $R$-algebra and let
	$S\subseteq \Cent(A)$ be an $R$-subalgebra such
	that $S$ is separable   over $R$ or an $R$-summand of $\Cent(A)$.
	Then $A$ is separable projective over $S$
	and $S$ is separable projective over $R$.
\end{lem}

\begin{proof}	
	Suppose first that $S$ is separable over $R$.
	That $A$ is separable over $S$ follows
	from \cite[Proposition~II.1.12]{DeMeyer_1971_separable_algebras}.
	Since $A$ is projective over $R$ and $S$ is separable over $R$,
	the algebra $A$ is projective as an $S$-module by Lemma~\ref{LM:projective-transfer}.
	It is  faithful over $S$ since $S$ is a subring of $A$.
	Now, by \cite[Corollary II.4.2]{DeMeyer_1971_separable_algebras},
	$S$ is a summand of $A$,
	so $S$ is projective over $R$.
	
	If $S$ is summand of $Z:=\Cent(A)$, then
	$S\in\rproj{R}$. Thus, for every
	$\frakp\in \Spec R$, the map $S (\frakp)\to Z(\frakp)$
	is injective. Since $Z(\frakp)$ is a finite product
	of separable field extensions of $k(\frakp)$,
	the same holds for $S(\frakp)$, and we conclude
	that $S$ is also separable.
	Proceed as in the previous paragraph.
\end{proof}

\begin{lem}\label{LM:center-base-change}
	Let $A$ be a separable projective  $R$-algebra,
	let $B\subseteq A$ be a separable projective $R$-subalgebra
	and let $S$ be an $R$-ring.
	Then the   natural map $\Cent_A(B)\otimes S\to \Cent_{A_S}(B_S)$
	is an isomorphism.  In particular, 
	$\Cent(A)\otimes S= \Cent(A_S)$.
\end{lem}

\begin{proof}
	Write $C=B\otimes A^\op$ and view $A$ as a left $C$-module 
	by setting $(b\otimes a^\op)\cdot x=bxa$ ($a,x\in A$, $b\in B$).
	Since $C$ is separable   over $R$ and $A\in\rproj{R}$,
	Lemma~\ref{LM:projective-transfer} implies
	that $A$ is projective as a $C$-module.
	Thus, the natural map $\End_C(A)\otimes S\to\End_{C_S}(A_S)$
	is an isomorphism.
	However, $\End_C(A)\cong \Cent_A(B)$ via $\vphi\mapsto \vphi(1)$,
	and likewise   $\End_{C_S}(A_S)\cong \Cent_{A_S}(B_S)$.
	The resulting isomorphism 	$\Cent_B(A)\otimes S\to \End_C(A)\otimes S
	\to \End_{C_S}(A_S)\to \Cent_{A_S}(B_S)$
	is   the natural map $\Cent_A(B)\otimes S\to \Cent_{A_S}(B_S)$ and the proposition follows.
\end{proof}

\begin{lem}\label{LM:Jac-of-separable-alg}
	Let $A$ be a separable projective $R$-algebra.
	Then $\Jac A=\Jac R \cdot A=\bigcap_{\frakm\in \Max R}\frakm A$.
\end{lem}

\begin{proof}
	Since $A$ is finite over $R$, we have $\Jac R\cdot A\subseteq \Jac A$
	\cite[Corollary~II.4.2.4]{Knus_1991_quadratic_hermitian_forms}.
	In addition, for all $\frakm\in \Max R$, the ring $A/\frakm A= A(\frakm)$
	is semisimple by \ref{item:sep-proj:fibers-def}, hence $\Jac A\subseteq  \bigcap_{\frakm\in \Max R}\frakm A$.
	It remains to show that $\bigcap_{\frakm\in \Max R}\frakm A\subseteq \Jac R\cdot A$.

	Consider the  
	exact sequence of $R$-modules $0\to \Jac R\to R\to \prod_{\frakm\in \Max R} R/\frakm$.
	Since $A$ is a flat over $R$,
	tensoring with $A$  gives an exact sequence
	$0\to \Jac R\otimes A\to R\otimes A \to (\prod_{\frakm\in\Max R} R/\frakm)\otimes A$.
	Furthermore, since $A$ is finitely presented,
	the natural map  $(\prod_{\frakm\in\Max R} R/\frakm)\otimes A\to \prod_{\frakm\in \Max R}
	(R/\frakm)\otimes A\cong \prod_{\frakm\in \Max R}A/\frakm A$ is an isomorphism
	\cite[Proposition~4.44]{Lam_1999_lectures_on_modules_rings}.
	Thus, $0\to \Jac R\otimes A\to A\to\prod_{\frakm\in \Max R}A/\frakm A$
	is exact, and the exactness at $A$ means that $\bigcap_{\frakm\in \Max R}\frakm A= \Jac R\cdot A$.
\end{proof}

	We also record the following general lemmas.

	\begin{lem}\label{LM:invertability-test}
		Let $A$ be a finite $R$-algebra and let $a\in A$. 
		If   $a(\frakm)\in\units{A(\frakm)}$ for all $\frakm\in\Max R$, then $a\in\units{A}$.
	\end{lem}

	\begin{proof}
		Consider the map $\phi :A\to A$ given by $\phi(x)=ax$.
		Then $\phi (\frakm):A(\frakm)\to A(\frakm)$ is bijective for all $\frakm\in \Max R$.
		As $A$ is a finite $R$-module,   $\phi$ is surjective 
		(\cite[Tag \href{https://stacks.math.columbia.edu/tag/05GE}{05GE}]{DeJong_2018_stacks_project} 
		with $f=1$), so $aA=\im \phi=A$. Likewise, $Aa=A$, so $a\in \units{A}$.
	\end{proof}

    \begin{lem}\label{LM:semilocal-direct-sum-reduction}
        Let $A$ be a finite projective $R$-algebra,
        let $P\in\rproj{A}$  and let $U$ and $V$ be summands of $P$.
        Suppose that $P(\frakm)=U(\frakm)\oplus V(\frakm)$
        for all $\frakm\in \Max R$.
        Then $P=U\oplus V$.
    \end{lem}
	
    \begin{proof}
    	We need to show that  $\psi:U\times V\to P$
    	given by $\psi(u,v)=u+v$ is an isomorphism. 
    	By assumption, $\im\psi+P\frakm=P$ for all $\frakm\in \Max R$.
    	By Nakayama's Lemma $\ann_R (P/\im \psi) $ is not contained
    	in any maximal ideal, so it must be $R$. Thus, $P/\im\psi=0$ and $\psi$ is onto.
    	Let $K=\ker \psi$. Since $P$ is projective,
    	$\psi$ splits and $U\times V\cong P\times K$,
    	hence $\rank_R U+\rank_R V=\rank_RP+\rank_RP$.
    	Since  $P(\frakm)=U(\frakm)\oplus V(\frakm)$
    	for all $\frakm\in \Max R$,
    	we have $\rank_RU+\rank_R V=\rank_RP$,
    	so $\rank_RK=0$ and $\ker\psi=K=0$.
    \end{proof}	

\subsection{Azumaya Algebras}
\label{subsec:Azumaya-algebras}

	We refer the reader to 
	\cite[7.\S3]{Ford_2017_separable_algebras},
	\cite[III.\S5.3]{Knus_1991_quadratic_hermitian_forms}
	or \cite[Chapter~3]{Saltman_1999_lectures_on_div_alg}
	for  the definition of the \emph{Brauer group} of $R$.
	We denote it as $\Br R$ and write its binary operation additively.
	The \emph{Brauer class} of an Azumaya $R$-algebra $A$ is denoted $[A]$.
	
	As usual, the \emph{degree} of an Azumaya $R$-algebra $A$
	is $\deg A:=\sqrt{\rank_R A}$,
	and  its \emph{index}   is $\ind A:=\gcd\{\deg A'\where A'\in [A]\}$.  
	Recall that  both $\deg A$ and $\ind A$ are functions from $\Spec R$ to $\N$,
	and that the ``$\gcd$'' in the definition of $\ind A$ is evaluated point-wise.
	Since $A$ is a finite projective $R$-module, $\deg A$ is locally
	constant relative to the Zariski topology, and with a little more work, one sees that the same holds for $\ind A$.
	When $R$ is connected, both $\deg A$ and $\ind A$
	are constant and may be regarded as elements of $\N$.
	
	We alert the reader that in general,
	there may be no $A'\in [A]$ with $\deg A'=\ind A$; see \cite{Antieu_2014_unramified_division_algebras}.
	However, this is true
	when $R$ is semilocal, by  Theorem~\ref{TH:index-description} below. 
	
	\begin{thm}[Saltman {\cite{Saltman_1981_Brauer_group_is_torsion}}]\label{TH:period-divides-index}
		Let $A$ be an Azumaya  $R$-algebra of degree dividing $n\in \N$.
		Then   $n\cdot [A]=0$ in $\Br R$.
	\end{thm}

%

	Given $P\in\rproj{A}$, the reduced rank $\rank_RP$ is (point-wise) divisible by $\deg A$;
	indeed, by \cite[pp.~5--6]{Knus_1998_book_of_involutions}, $\deg A(\frakp)\mid \dim_{k(\frakp)} P(\frakp)$
	for all $\frakp\in\Spec R$.
	It is therefore convenient to introduce the \emph{reduced $A$-rank} of $P$, defined by
	\[
	\rrk_AP:=\rank_RP/\deg A  .
	\]
	This agrees with the \emph{reduced dimension} 
	defined in {\it op.\ cit.}	
	when $R$ is a field.
	For example, $\rrk_A (A_A)=\deg A$.
	If $\iota:R\to S$ is a ring homomorphism, then $\rrk_{A_S}P_S=\iota \rrk_AP$.
	In particular, $\deg A_S=\iota \deg A$.	
	
	\begin{remark}\label{RM:Azumaya-over-center}
		If $A$ is an $R$-algebra which is Azumaya over its center  $\Cent(A)$,
		then we regard $\deg A$, $\ind A$ and $\rrk_AP$ ($P\in \rproj{A}$)
		as functions from $\Spec \Cent(A)$ to $\Z$. Note also that
		$[A]$ is a member of $\Br \Cent(A)$, rather than $\Br R$.
	\end{remark}

	We record a number of properties of the reduced rank which will
	be used many times in the sequel.

	\begin{prp}\label{PR:progenerator-iff-pos-red-rank}
		Let $A$ be an Azumaya $R$-algebra and let $P\in\rproj{A}$.
		Then $\rrk_AP>0$ if and only if $P$ is a progenertor.
	\end{prp}
	
	\begin{proof}
		Since $\rrk_A(A_A)=\deg A>0$,  if $P$ is a progenerator,
		then $\rrk_A P>0$.
		
		To see the converse, let $T=\sum_{\phi} \im \phi$
		where $\phi$ ranges over $\Hom_A(P,A)$. It is enough
		to prove that $T=A$, see \cite[pp.~7--8]{Ford_2017_separable_algebras}. 
		Fix   $\frakm\in \Max R$. 
		Since $(\rrk_AP)(\frakm)>0$, the $A(\frakm)$-module
		$P(\frakm)$ is nonzero. Since $A(\frakm)$ is simple artinian, there exists
		$n\in\N$ and a surjection $\vphi:P(\frakm)^n\to A(\frakm)$.
		Since $P$ is projective, there exists $\hat{\vphi}\in \Hom_A(P^n, A)$
		such that $\vphi=\hat{\vphi}(\frakm)$,
		hence $\im(\hat{\vphi})+\frakm A=A$. Since $\im(\hat{\vphi})\subseteq T$,
		this means that $T+\frakm A=A$, or rather, $(A/T)\frakm=A/T$.
		By Nakayama's Lemma $\ann_R (A/T)$ is not contained in $\frakm$.
		As this holds for all $\frakm\in\Max R$,
		we must have $A/T=0$, so $T=A$.
	\end{proof}

	\begin{prp}\label{PR:degree-of-endo-ring}
		Let $A$ be an Azumaya $R$-algebra
		and suppose that $P\in\rproj{A}$ satisfies 
		$\rrk_A P>0$. Then:
		\begin{enumerate}[label=(\roman*)]
			\item  $B:=\End_A(P)$ is an Azumaya $R$-algebra, $\deg B=\rrk_AP$ and $[B]=[A]$.  
			\item For all $Q\in \rproj{B}$, we have $\rrk_B Q=\rrk_A (Q\otimes_BP)$.
			\item \gap{}For every $B'\in [A]$, there exists $P'\in \rproj{A}$
			with $B'\cong \End_A(P')$ and $\rrk_{A}P'=\deg B'>0$.
		\end{enumerate}
	\end{prp}

	\begin{proof}	
		(i)
		By Proposition~\ref{PR:progenerator-iff-pos-red-rank},
		$P_A$ is a progenerator, and in particular faithful.
		Thus,  $A^\op$ embeds as an $R$-subalgebra of $\End_R(P)$ via
		$a^\op\mapsto [x\mapsto xa]$, and $B=\Cent_{A^\op}(\End_R(P))$.
		Since  both $A^\op$ and $\End_R(P)$ are Azumaya $R$-algebras,
		$B$ is Azumaya
		over $R$ and
		$A^\op\otimes B\cong \End_R(P)$  \cite[Theorem II.4.3]{DeMeyer_1971_separable_algebras}.
		This 
		implies that $\deg A^\op\cdot\deg B  =\deg\End_R(P)=\rank_R P$. It
		follows that
		$\deg B=\rrk_AP$  and $[A^\op]+[B]=[\End_R(P)]=0$, so $[A]=[B]$.

		(ii) By   definition of the reduced rank, it enough
		to check the statement after specializing to $k(\frakp)$ for all $\frakp\in\Spec R$.
		(Recall that   $\psi\mapsto \psi\otimes \id_{k(\frakp)}: B(\frakp)=\End_A(P)\otimes k(\frakp)\to \End_{A(\frakp)}(P(\frakp))$
		is an isomorphism because $P\in\rproj{A}$.)
		Now that $R$ is a field, we may regard the reduced rank as an integer
		and further specialize to the algebraic closure, as it would not affect the reduced rank.
		When $R$ is algebraically
		closed, we may
		assume that $A= \nMat{R}{n}$, $P= \nMat{R}{m\times n}$, $B= \nMat{R}{m}$, $Q=\nMat{R}{t\times m}$,
		and checking that $\rrk_B Q=t=\rrk_A (Q\otimes_AP)$ is routine.
		
		(iii) \gap{}By a Theorem of Bass, see \cite[Theorem 9.2]{First_2015_morita_equiv_to_op},
		exists is an progenerator $P'\in \rproj{A}$ such that $B'\cong \End_A(P')$.
		The claim now follows from Proposition~\ref{PR:progenerator-iff-pos-red-rank}
		and (i). 
	\end{proof}

	\begin{cor}\label{CR:degree-of-endo-ring}
		Let $A$ be an Azumaya $R$-algebra and let $e\in A$
		be an idempotent. Then $e$ is full (i.e.\ $AeA=A$)
		if and only if $\rrk_A eA>0$. In this case,
		$eAe$ is an Azumaya $R$-algebra, $\deg eAe=\rrk_AeA$,
		$[A]=[eAe]$ and for all $P\in \rproj{A}$,
		we have $\rrk_AP=\rrk_{eAe}Pe$.
	\end{cor}
	
	\begin{proof}
		Recall that $e$ is full if and only if
		$eA_A$ is a progenerator, and this is equivalent
		to $\rrk_A eA>0$ by
		Proposition~\ref{PR:progenerator-iff-pos-red-rank}. 
		Since $eAe=\End_A(eA)$,
		the first three assertions   follow from
		Proposition~\ref{PR:degree-of-endo-ring}(i).
		For the last assertion, note that by Morita theory,
		$Ae\in\rproj{eAe}$ is a progenerator and
		$A=\End_{eAe}(Ae)$ \cite[Corollary~18.21]{Lam_1999_lectures_on_modules_rings}. 
		Applying  Proposition~\ref{PR:degree-of-endo-ring}(ii)
		with $eAe$, $A$, $Ae$ in place of $A$, $B$, $P$,
		we see that $\rrk_{eAe} Pe=\rrk_{eAe} (P\otimes_A Ae)=\rrk_AP$.
	\end{proof}
	
	\begin{cor}\label{CR:index-divides-rrk}
		Let $A$ be an Azumaya $R$-algebra
		and let $P\in\rproj{A}$. Then $\ind A\mid \rrk_AP$.
	\end{cor}
	
	\begin{proof}
		Since $\ind A\mid \deg A=\rrk_AA$, we may replace
		$P$ with $P\oplus A$ and assume that $\rrk_AP>0$.
		By Proposition~\ref{PR:degree-of-endo-ring}(i),
		$\rrk_AP=\deg \End_A(P)$ and $\End_A(P)\in [A]$, so
		$\rrk_AP$  is divisible by $\ind A$.
	\end{proof}

\begin{prp}
	\label{PR:centralized-in-Az-alg}
	Let $A$ be an Azumaya $R$-algebra, let $S$ be a 
	finite \'etale $R$-subalgebra of $A$ and
	let $\iota:R\to S$ be the inclusion map.  Then $B:=\Cent_A(S)$ is Azumaya over 
	$S$ and $[B]=[A\otimes S]$
	in $\Br S$. Furthermore, $A$ is projective as a right $S$-module,
	and
	if $\rank_S A_A$ is constant along
	the fibers of $\Spec S\to \Spec R$, then:
	\begin{enumerate}[label=(\roman*)]
	\item  $\deg B\cdot \iota\rank_R S =\iota\deg A$,  
	and $\rrk_BP=\iota\rrk_AP$ for all $P\in\rproj{A}$, and
	\item $\iota \rrk_A (Q\otimes_B A)=\iota \rank_RS\cdot \rrk_B Q$
	for all $Q\in\rproj{A}$
	such that $\rrk_BQ$ is constant
	along the fibers of $\Spec S\to \Spec R$.
	\end{enumerate}
\end{prp}

\begin{proof}
	That $B$ is Azumaya over $S$ and $[B]=[A\otimes S]$ is well-known,
	see \cite[Theorem~3.10]{Saltman_1999_lectures_on_div_alg}, for instance.
	That $A$ is projective as a right $S$-module follows from Lemma~\ref{LM:projective-transfer}.
	
	By Lemma~\ref{LM:center-base-change},  it
	is enough to prove (i) and (ii) after  base changing to every
	residue field of $R$, so assume $R$ is a field.
	In fact, we may further base-change to an algebraic closure
	of $R$ and assume that $R$ is algebraically closed.
	In this case, $S\cong R^t$ for $t=\rank_RS$ and $A=\nMat{R}{n}$ for $n=\deg A$.
	
	Let $e_{i j}\in \nMat{R}{n}$ denote the $n\times n$ matrix with $1$ at the $(i,j)$-entry
	and zeroes elsewhere.
	Let $f_1,\dots,f_t$ denote the   primitive idempotents
	of $S$. Since $\rrk_S A_S$ is constant, $\dim_R Af_i$ is independent
	of $i$, so ${}_A Af_1\cong\dots\cong {}_A Af_t$.
	This means that each $f_i$ is an idempotent of rank $s:=\frac{n}{t}$
	in $A=\nMat{R}{n}$. Since all such idempotents
	are conjugate, we may   choose
	the identification of $A$ with $\nMat{R}{n}$
	such that $f_i=\sum_{j=(i-1)s+1}^{is}e_{jj}$.
	Thus, 
	\[B=\left[
	\begin{matrix}
	\nMat{R}{s} & & \\
	& \ddots & \\
	& & \nMat{R}{s}
	\end{matrix}\right]\subseteq \nMat{R}{n}.
	\] 
	Furthermore, every right $A$-module is isomorphic to $\nMat{R}{m\times n}$
	for some $m\geq 0$ and any right $B$-module with constant $S$-rank
	is isomorphic to $\nMat{R}{\ell\times s } \times \dots \times\nMat{R}{\ell\times s } $
	($t$ times) for some $\ell \geq 0$. 
	Now, verifying (i) and (ii)     is straightforward.
\end{proof}

	The requirement that $\rank_S A_A$ is contstant along the fibers
	of $\Spec S\to \Spec R$ is guaranteed when $\rank_R S=\deg A$.
	
	\begin{cor}\label{CR:rank-over-max-etale}
		Let $A$ be an Azumaya $R$-algebra and let $S$ be a finite \'etale
		subalgebra of $A$ such that $\rank_RS=\deg A$.
		Then $\rank_S A_A=\iota\deg A$ ($\iota:R\to S$
		is the inclusion), $S=\Cent_A(S)$ and $[A_S]=0$.
	\end{cor}
	
	\begin{proof}
		We only need to show that $\rank_S A_A=\iota\deg A$.
		The remaining assertions then follow from Proposition~\ref{PR:centralized-in-Az-alg}.

		The algebra $A$ is an $(A,S)$-bimodule and hence
		a right module over $A^\op\otimes S$. 
		By Lemma~\ref{LM:projective-transfer}, $A_{A^\op\otimes S}$ is projective,
		so $\deg A_S\mid \rank_S A_A$. Furthermore,   $\rank_S A_A>0$ because $A$ is faithful
		as a right $S$-module. 
		Let $\frakp\in\Max R$ and write $S(\frakp)=\prod_{i=1}^t K_i$,
		where $K_i$ is a $k(\frakp)$-field.
		We need to show that $\dim_{K_i} (A\otimes_S K_i)=\deg A(\frakp)$.
		Write $n=\deg A(\frakp)$.
		Then  $n^2=\dim_{k(\frakp)}A(\frakp)=\sum_i [K_i:k(\frakp)] \dim_{K_i}(A\otimes_S K_i)$
		and $\sum_i [K_i:k(\frakp)]=\dim_{k(\frakp)}S(\frakp)=n$.
		Since $\dim_{K_i}(A\otimes_SK_i)$ is positive and divisible by $n$,
		we must have $\dim_{K_i}(A\otimes_SK_i)=n$ for all $i$,
		as required.
	\end{proof}

\subsection{Quadratic \'Etale Algebras}
\label{subsec:quad-etale}

Finite \'etale $R$-algebras of $R$-rank $2$ are
also called \emph{quadratic \'etale} algebras. 
Every such algebra $S$ admits a unique
$R$-involution $\theta:S\to S$
such that $R=\{s\in S\suchthat s^\theta=s\}$;
it is given by $x^\theta=\Tr_{S/R}(x)-x$ and satisfies
$\Nr_{S/R}(x)=x^\theta x$.\footnote{
	If $A$ is a  finite projective $R$-algebra of rank $n\in\N$, 
	then the \emph{trace} and \emph{norm} maps $\Tr_{A/R},\Nr_{A/R}:A\to R$
	take $a\in A$ to $-c_1(a)$ and $(-1)^n c_n(a)$, respectively,
	where $X^n+c_1(a)X^{n-1}+\dots+c_n(a)X^0$ is the characteristic polynomial
	of $[x\mapsto ax]\in \End_R(A)$ in the sense of \cite[Example~5.3.3]{Ford_2017_separable_algebras}.
} See
\cite[Proposition I.1.3.4]{Knus_1991_quadratic_hermitian_forms} for its uniqueness.
Following \cite[I.\S1.3]{Knus_1991_quadratic_hermitian_forms},
we call $\theta$ the \emph{standard $R$-involution} of $S$.

For example, the $R$-algebra  $R\times R$ is quadratic \'etale 
and its  standard involution
is the exchange involution $(x,y)\mapsto (y,x)$.
Furthermore, by our standing assumption that $2\in\units{R}$,
the $R$-algebra $R[x\where x^2=a]$ is quadratic \'etale
whenever $a\in\units{R}$ (use \ref{item:sep-proj:fibers-def} above), and its standard
involution is determined by $x^\theta=-x$.

\begin{lem}\label{LM:non-connected-S}
	Let $S$ be a quadratic \'etale $R$-algebra.
	If $R$ is connected and $S$ is not connected,
	then $S\cong R\times R$
	as $R$-algebras.
\end{lem}

\begin{proof}
	Let $\theta$ denote the standard $R$-involution of $S$,
	and let $e\in S$ be a nontrivial idempotent.
	Then $e^\theta e$ is a non-invertible idempotent of $R$, hence $e^\theta e=0$
	(because $R$ is connected).
	This means that $e+e^\theta$ is also an idempotent in $R$, and it is nonzero
	because $e(e+e^\theta)=e\neq 0$. Since $R$ is connected,   $e+e^\theta=1$.
	It is now routine to check that
	$r\mapsto er:R\to eS$ and $s\mapsto s+s^\theta: eS\to R$
	are mutually inverse. Since the former map is an 
	$R$-algebra homomorphism, we see that $R\cong eS$ as $R$-algebras,
	and similarly  $R\cong e^\theta S=(1-e)S$. The lemma
	follows because $S\cong eS\times (1-e)S$.
\end{proof}

\begin{lem}\label{LM:invs-of-quad-et-algs}
	Let $S$ be a quadratic \'etale $R$-algebra and
	let $\sigma:S\to S$ be an $R$-involution.
	Then there exists a factorization $R=R_1\times R_2$
	such that $\sigma_{R_1}:S_{R_1}\to S_{R_1}$ is the standard
	$R_1$-involution of $S_{R_1}$
	and $\sigma_{R_2}:S_{R_2}\to S_{R_2}$ is the identity.
	In particular, if $R$ is connected, then $\sigma$ is either
	the standard $R$-involution of $S$ or $\id_S$.
\end{lem}

\begin{proof}
	This is a restatement of  \cite[Proposition~III.4.1.2]{Knus_1991_quadratic_hermitian_forms}.
\end{proof}

\begin{lem}\label{LM:splitting-quad-et-algs}
	Let $S$ be a quadratic \'etale $R$-algebra.
	Then $S_S\cong S\times S$ as $S$-algebras.
\end{lem}

\begin{proof}
	This follows from the discussion in  \cite[III.\S4.1]{Knus_1991_quadratic_hermitian_forms}.
\end{proof}


\begin{lem}\label{LM:quad-etale-over-semilocal}
	Suppose that $R$ is semilocal
	and let $S$ be a quadratic \'etale $R$-algebra
	with standard involution $\theta$.
	Then there exists $\lambda\in S$
	such that $\lambda^2\in \units{R}$, $\lambda^\theta=-\lambda$
	and $\{1,\lambda\}$ is an $R$-basis of $S$.
\end{lem}

\begin{proof}	
	Since $2\in\units{R}$, we have $S=\Sym_1(S,\theta)\oplus \Sym_{-1}(S,\theta)=
	R\oplus \Sym_{-1}(S,\theta)$, and so 
	$\Sym_{-1}(S,\theta)$ is a rank-$1$ projective $R$-module.
	Since $R$ is semilocal,  $\Sym_{-1}(S,\theta)$ is free.
	Let $\lambda$ be a generator of $\Sym_{-1}(S,\theta)$.
	Then $\lambda^2=-\lambda\cdot\lambda^\theta=-\Nr_{S/R}(\lambda)\in R$ and $\{1,\lambda\}$
	is an $R$-basis of $S$.
	As a result, $S\cong R[x\where x^2-a]$, where $a=\lambda^2$.
	If $a\notin \units{R}$, then there exists $\frakm\in \Max R$ with $a\in\frakm$,
	and it follows that
	$S(\frakm)\cong k(\frakm)[x\where x^2=0]$ is not \'etale over $k(\frakm)$.
	Thus, we must have $\lambda^2=a\in\units{R}$.
\end{proof}

\subsection{Azumaya Algebras With Involution}
\label{subsec:Az-alg-inv}

	Recall our standing assumption that $2\in\units{R}$.
	An \emph{Azumaya  algebra with involution}\footnote{
		This should be understood as ``Azumaya algebra-with-involution''
		rather than ``Azumaya-algebra with involution''.
	} over   $R$
	is an $R$-algebra with involution $(A,\sigma)$
	such that $A$ is separable projective over $R$
	and the homomorphism $r\mapsto r\cdot 1_A:R\to A$
	identifies $R$ with the $\sigma$-fixed elements
	of $\Cent(A)$. {\it Note that $A$ is not necessarily
	Azumaya as an $R$-algebra.} Rather, $A$ is Azumaya over $\Cent(A)$, so that
	$\deg A$ is a function from $\Spec \Cent(A)$ to $\Z$ and $[A]\in \Br \Cent(A)$, 
	cf.\ Remark~\ref{RM:Azumaya-over-center}.
	
	If $(A,\sigma)$ is an Azumaya $R$-algebra with involution and $S$ is an $R$-ring,
	then $(A_S,\sigma_S)$ is an Azumaya $S$-algebra with involution.
	Indeed, $\Cent(A_S)=\Cent(A)\otimes S$ by Lemma~\ref{LM:center-base-change}, and
	the exact sequence $0\to R\xrightarrow{r\mapsto r\cdot 1_A} \Cent(A)\xrightarrow{a\mapsto a-a^\sigma} R\to 0$
	is split
	at $\Cent(A)$ because $\Cent(A)=R1_A\oplus \Sym_{-1}(\Cent(A),\sigma)$, 
	so it remains exact after tensoring with $S$.
	Together, this means that $s\mapsto s\cdot 1_A:S\to  \{a\in \Cent(A_S)\suchthat a-a^{\sigma_S}=0\}$ 
	is an isomorphism, hence our claim.

	\begin{example}\label{EX:Azumaya-over-fixed-subring}
		Let $A$ be a separable projective $R$-algebra,
		let $\sigma:A\to A$ be an $R$-involution
		and let $R_1:=\{s\in\Cent(A)\suchthat s^\sigma=s\}$.
		Then $(A,\sigma)$ is an Azumaya $R_1$-algebra with involution.
		Indeed, $R_1$ is a $R$-summand of $\Cent(A)$ because $2\in\units{R}$,
		so by Lemma~\ref{LM:separable-subring-I},
		$A$ is separable projective over $R_1$ and
		$R_1$ is finite \'etale over $R$.
	\end{example}
	
	When $R $ is a field $F$,
	an Azumaya $F$-algebra with involution, $(A,\sigma)$,
	is a central simple $F$-algebra with involution in the sense
	of \cite[pp.~13, 20]{Knus_1998_book_of_involutions}.
	The  center of $A$ is then either $F$
	or a quadratic \'etale   extension of $F$.
	In first case, $A$ is a central simple $F$-algebra
	and $\sigma$ can be either of \emph{orthogonal} or \emph{symplectic} type,
	see 
	\cite[\S2.A]{Knus_1998_book_of_involutions}. When $\sigma$
	is symplectic, $\deg A$ must be  even \cite[Proposition~2.6]{Knus_1998_book_of_involutions}.
	In the case $\Cent(A)\neq F$, the center is either $F\times F$
	or a quadratic separable field extension of $F$,
	and $\sigma$ is said to be of \emph{unitary} type, see  
	\cite[\S2.B]{Knus_1998_book_of_involutions}.
	
\medskip

	Returning to the case $R$ is arbitrary, we turn to
	define the \emph{type} of the involution $\sigma$.
	In fact, it will be   convenient  to define
	the type of a pair $(\sigma,\veps)$, where
	$\veps\in\Cent(A)$ satisfies $\veps^\sigma \veps=1$, with
	the type of $\sigma$ being the type of $(\sigma,1)$.
	
	To that end, suppose first that $R$ is a field. 
	We say that the type of $(\sigma,\veps)$
	is unitary if $\sigma$ is unitary, i.e., when $\Cent(A)\neq R$.
	Suppose now that $\Cent(A)=R$. Then $\veps\in \{\pm1\}$
	and $\sigma$ is either
	orthogonal or symplectic.  
	We say that $(\sigma,\veps)$ is of orthogonal type if
	either
	$\sigma$ is   orthogonal and $\veps=1$, or $\sigma$ is symplectic
	and $\veps=-1$. In all other cases, $(\sigma,\veps)$ is said to
	be of symplectic type. 
	
	When $R$ is arbitrary, the
	type of $(\sigma,\veps)$
	is the function from
	$\Spec R$ to the set $\{\text{orthogonal},\text{symplectic},\text{unitary}\}$
	assigning $\frakp$ the type of $(\sigma(\frakp),\veps(\frakp))$.
	The type of $\sigma$ is the type of $(\sigma,1)$;
	this agrees with the definition of \cite[III.\S8]{Knus_1991_quadratic_hermitian_forms}.
	We also say that $(\sigma,\veps)$ is orthogonal (resp.\ symplectic, unitary) at $\frakp$ if 
	$(\sigma(\frakp),\veps(\frakp))$
	is orthogonal (resp.\ symplectic, unitary).
	The pair  $(\sigma,\veps)$ is called orthogonal (resp.\ symplectic, unitary) if
	this holds at all primes $\frakp\in \Spec R$. 
	We remark that $(\sigma,\veps)$ is unitary
	if and only if $\sigma$ (i.e.\ $(\sigma,1)$) is unitary.

\medskip

	Recall that $\Sym_\veps(A,\sigma)=\{a\in A\suchthat \veps a^\sigma=a\}$.

	\begin{prp}\label{PR:types-of-involutions-Az}
		Let $(A,\sigma)$ be an   Azumaya $R$-algebra with involution,
		let $\veps\in \Cent(A)$ be an element
		satisfying $\veps^\sigma \veps=1$ and write $n=\deg A$.
		\begin{enumerate}[label=(\roman*)]
		\item  $(\sigma,\veps)$ is orthogonal if and only if $\rank_R\Sym_\veps(A,\sigma)=\frac{1}{2}n(n+1)$
		and $\Cent(A)=R$.
		\item   $(\sigma,\veps)$ is symplectic if and only if $\rank_R\Sym_\veps(A,\sigma)=\frac{1}{2}n(n-1)$
		and $\Cent(A)=R$.
		\item   $(\sigma,\veps)$ is unitary if and only if $\rank_R \Cent(A)=2$.
		In this case, $\Cent(A)$ is a quadratic \'etale
		$R$-algebra, $\sigma|_{\Cent(A)}$ is its standard involution
		and $\rank_R\Sym_\veps(A,\sigma)=n^2$.
		\item  There exists a factorization $R\cong R_o\times R_s\times R_u$
		such that $(\sigma_{R_o},\veps\otimes 1_{R_o})$
		is orthogonal,
		$(\sigma_{R_s},\veps\otimes 1_{R_s})$
		is symplectic and
		$(\sigma_{R_u},\veps\otimes 1_{R_u})$
		is unitary.
		\item 		
		If $R$ is connected, then $(\sigma,\veps)$ is either orthogonal, symplectic
		or unitary.
		\end{enumerate}
	\end{prp}
	
	\begin{proof}
		Suppose first that $R$ is a field.
		If $\Cent(A)=R$, then $\veps\in\{\pm 1\}$ 
		and (i)--(iii) follow from
		\cite[Proposition~2.6]{Knus_1998_book_of_involutions}.
		If $\Cent(A)\neq R$, then by  Hibert's Theorem 90, there exist $\delta\in \Cent(A) $
		such that $\delta^\sigma\delta^{-1}=\veps^{-1}$.
		One readily checks that
		$\delta \cdot \Sym_1(A,\sigma)=\calS_{\veps}(A,\sigma)$,
		so  $\dim_R\Sym_\veps(A,\sigma)=\dim_R\calS_{1}(A,\sigma)$,
		and the right hand side is $n^2$ by \cite[Proposition~2.17]{Knus_1998_book_of_involutions}.
		It follows that (i)--(iii) hold in this case as well.
		
		Parts (i)--(iii) for general $R$
		will follow from the field case if we show that
		the natural maps
		$ \Cent(A)(\frakp)\to \Cent(A(\frakp))  $
		and $(\Sym_\veps(A,\sigma))(\frakp)\to \Sym_\veps(A(\frakp),\sigma(\frakp))$
		are isomorphisms
		for all $\frakp\in \Spec R$.
		The former isomorphism is Lemma~\ref{LM:center-base-change}.
		To establish the second, note that the short exact
		sequence $  \Sym_\veps(A,\sigma)\to A\to \Sym_{-\veps}(A,\sigma)$
		in which the right arrow is given by $a\mapsto a-\veps a^\sigma$
		is split, because $2\in\units{R}$, and thus it remains exact after
		base-change along $R\to k(\frakp)$.
		
		Now, part (iv) follows readily from the fact
		that  $\rank_R \Cent(A)$ and $\rank_R \Sym_\veps(A,\sigma)$
		are locally constant functions, and part (v) follows  from (iv).
	\end{proof}

	\begin{cor}\label{CR:type-conjugation}
		Let $(A,\sigma)$ be an Azumaya $R$-algebra with involution and
		let $\veps\in \Cent(A)$ be an element
		satisfying $\veps^\sigma \veps=1$.
		\begin{enumerate}[label=(\roman*)]
			\item For every
			$\delta\in\Cent(A)$ satisfying $\delta^\sigma\delta=1$
			and every  $\mu\in \Sym_{\delta}(A,\sigma)\cap \units{A}$,
			the pair
			$(A,\Int(\mu)\circ \sigma)$ is an Azumaya
			$R$-algebra with involution and 
			the type of $(\Int(\mu)\circ \sigma,\delta\veps)$ is the same
			as the type of $(\sigma,\veps)$.
			\item For every idempotent $e\in A$ with $\rrk_AeA>0$ and $e^\sigma=e$,
			the pair $(eAe,\sigma|_{eAe})$ is an Azumaya
			$R$-algebra with involution and the type
			of $(\sigma|_{eAe},e\veps)$ is the same
			as the type of $(\sigma,\veps)$.		
		\end{enumerate}
	\end{cor}
	
	\begin{proof}
		(i) 
		Checking that $(A,\Int(\mu)\circ \sigma)$ is an Azumaya
		$R$-algebra with involution is straightforward.
		It is routine to check that 
		$x\mapsto \mu x:\Sym_\veps(A,\sigma)\to \Sym_{\delta\veps}(A,\Int(\mu)\circ \sigma)$
		is an $R$-module isomorphism,
		hence $\rank_R   \Sym_\veps(A,\sigma)=
		\rank_R \Sym_{\delta\veps}(A,\Int(\mu)\circ \sigma)$. By
		Proposition~\ref{PR:types-of-involutions-Az}, this means
		that $(\sigma,\veps)$ and $(\Int(\mu)\circ \sigma,\delta\veps)$
		have the same type.
		
		(ii) Write $\sigma_e:=\sigma|_{eAe}$.
		By Corollary~\ref{CR:degree-of-endo-ring}, $eAe$ is Azumaya over $\Cent(A)$.
		In particular, $a\mapsto ea$ defines
		an isomorphism $\Cent(A)\to \Cent(eAe)$.
		This isomorphism is compatible with $\sigma$,
		so $r\mapsto er:R\to eAe$ identifies $R$ with the $\sigma_e$-fixed
		elements in $\Cent(eAe)$. Thus, $(eAe,\sigma_e)$ is an Azumaya
		$R$-algebra with involution.
		
		Let $\frakp\in\Spec R$.
		Since $\rank_R \Cent(A)=\rank_R \Cent(eAe)$, Proposition~\ref{PR:types-of-involutions-Az}
		implies that  $(\sigma,\veps)$ is unitary at $\frakp$ if and only
		if $(\sigma_e,e\veps)$ is unitary at $\frakp$.
		Furthermore, by \cite[Proposition~2.12]{First_2015_Witts_extension_theorem},
		$(\sigma,\veps)$ is orthogonal at $\frakp$ if and only if
		$(\sigma_e,e\veps)$ is orthogonal at $\frakp$.
		Thus, $(\sigma,\veps)$ and $(\sigma_e,\veps e)$ have the same type. 
	\end{proof}

	For later reference, we   record the following easy consequence of Lemma~\ref{LM:Jac-of-separable-alg}
	and the Chinese Remainder Theorem.
	
	\begin{lem}\label{LM:mod-Jac-decomposition-inv}
		Let $(A,\sigma)$ be an Azumaya algebra with involution
		over a semilocal ring $R$.
		Write $\quo{A}=A/\Jac A$
		and let $\quo{\sigma}:\quo{A}\to \quo{A}$
		be the induced involution.
		Then $(\quo{A},\quo{\sigma})\cong \prod_{\frakm\in\Max R}(A(\frakm),\sigma(\frakm))$ 
		as $R$-algebras with involution,
		and each factor $(A(\frakm),\sigma(\frakm))$
		is a central simple $k(\frakm)$-algebra with involution.
	\end{lem}

\subsection{Azumaya Algebras Over Semilocal Rings}
\label{subsec:semilocal-Azumaya-algs}

	We now specialize to the case where $R$ is semilocal and establish
	several results about Azumaya algebras
	and Azumaya algebras with involution.

    \begin{lem}\label{LM:rank-determines}
    	Let $A$ be an Azumaya algebra
    	over a semilocal ring $R$  and let
    	$P,Q\in \rproj{A}$.
    	Then $P\cong Q$ if and only if $\rrk_A P=\rrk_AQ$.
        Furthermore, $P$ is isomorphic to a summand
        of $Q$ if and only if $\rrk_AP\leq \rrk_AQ$.
    \end{lem}

    \begin{proof}
    	The ``only if'' part of both statements
    	is clear.
    	
		We first prove the ``if'' part of the second statement.  
		Since $\rrk_AP\leq \rrk_AQ$,
		we have $\dim_{k(\frakm)}P(\frakm)\leq \dim_{k(\frakm)}Q(\frakm)$
		for all $\frakm\in\Max R$. 
		Since $A(\frakm)$ is a central simple $k(\frakm)$-algebra,
		this means that $P(\frakm)$ is isomorphic to an $A(\frakm)$-summand
		of $Q(\frakm)$.
		
		Write $S=R/\Jac R$. Since $R$ is semilocal, we have $S=\prod_{\frakm\in \Max R}k(\frakm)$,
		$A_S=\prod_{\frakm\in \Max R} A(\frakm)$, $P_S=\prod_{\frakm\in \Max R}P(\frakm)$
		and $Q_S=\prod_{\frakm\in \Max R}Q(\frakm)$; the products are all finite.
		By the previous paragraph there exists an $A$-module
		epimorphism $\vphi:Q_S\to P_S$.
		Since $P$ is projective, $\vphi$ lifts to an $A$-module
		homomorphism $\psi:Q\to P$.
		Since $\im \vphi=P_S$, we have $\im \psi+P\frakm=P$ for all $\frakm\in \Max R$.
		Thus, as in the proof of Lemma~\ref{LM:semilocal-direct-sum-reduction},
		$\im \psi =P$. Since $P$ is projective, this means
		that $P$ is isomorphic to a summand of $Q$.
		
		To prove the ``if'' part of the first statement,
		argue as above and note that $\ker \psi=0$, because $\rrk_AP=\rrk_AQ$.
    \end{proof}

    \begin{thm}\label{TH:index-description}
    	Let $A$ be an Azumaya algebra over a semilocal ring $R$.
    	Then
    	there exists an idempotent $e\in A$
    	such that $eAe\in [A]$
    	and  $\rrk_A eA=\deg eAe=\ind A$.
    \end{thm}

    \begin{proof}
		We first claim that there exists $P\in\rproj{A}$ with $\rrk_AP=\ind A$.		
		Write $R=\prod_{i=1}^tR_i$, where each $R_i$ is connected.
		By working over each factor separately, we may assume
		that $R$ is connected. As a result,
		$\rrk_AP$ is constant for all $P\in\rproj{A}$.
		
		Since every $B\in [A]$ is isomorphic
		to $\End_A(P)$ for some $P\in \rproj{A}$
		with $ \rrk_AP>0$ and
		$\deg B=\rrk_AP$ (Proposition~\ref{PR:degree-of-endo-ring}(iii)),
		we have
		$\ind A=\gcd\{\rrk_AP\where P\in\rproj{A},\rrk_AP>0\}$.
		Thus, in order to establish the existence of $P\in\rproj{A}$
		with $\rrk_AP=\ind A$,
		it is  enough to show that
        for any $P,Q\in \rproj{A}$ with $\rrk_AQ\leq \rrk_AP$,
        there exists $S\in\rproj{A}$ with $\rrk_AS=\rrk_AP-\rrk_AQ$.
        This follows readily from Lemma~\ref{LM:rank-determines}.

		Let $P\in\rproj{A}$ be a module with $\rrk_AP=\ind A$.
		By Lemma~\ref{LM:rank-determines}, $P$ is isomorphic
		to a summand of $A_A$, because $\rrk_AP\leq \deg A=\rrk_AA_A$.
		Therefore,
		there exists an idempotent $e\in A$
		such that $P\cong eA$.
		The theorem now follows from Corollary~\ref{CR:degree-of-endo-ring}.
    \end{proof}

	We now turn to consider Azumaya $R$-algebras with involution.

	\begin{lem}\label{LM:invertible-symmetric-elements}
		Let $(A,\sigma)$ be an Azumaya algebra with involution
		over a semilocal ring $R$ and let $\veps\in\Cent(A)$
		be an element with $\veps^\sigma\veps=1$.
		If for every $\frakm\in \Max R$,  the type of $(\sigma ,\veps )$  at $\frakm$ is not symplectic,
		or $ \deg A (\frakm)$ is even,
		then $\Sym_\veps (A,\sigma)\cap \units{A}\neq \emptyset$.
	\end{lem}

	\begin{proof}
		Suppose first that $R$ is a field and let $S=\Cent(A)$.
		Then either $S$ is a field, or $S=R\times R$.
		If $S$ is a field, then the map $s\mapsto \veps s^\sigma:S\to S$
		is an involution and its nonzero fixed points    are contained
		in $\Sym_\veps (A,\sigma)\cap \units{A}$. If there are no such points,
		then $s=-\veps s^\sigma$ for all $s\in S$, which   implies
		$\veps=-1$ (take $s=1$)
		and $\sigma|_S=\id_S$. In this case, $\Sym_\veps (A,\sigma)\cap \units{A}\neq\emptyset$
		by \cite[Corollary~2.8]{Knus_1998_book_of_involutions}.
		If $S=R\times R$, then $\sigma|_S$ is the exchange involution
		$(x,y)\mapsto (y,x)$ and $\veps=(\alpha,\alpha^{-1})$ for some $\alpha\in \units{R}$,
		so $(\alpha,1)\in \Sym_\veps (A,\sigma)\cap \units{A}$.

		For general $R$, let $\frakm_1,\dots,\frakm_t$
		denote the maximal ideals of $R$.
		By the previous paragraph, for each $i\in\{1,\dots,t\}$,
		there exists $a_i\in \Sym_{\veps}(A(\frakm_i),\sigma(\frakm_i))\cap \units{A(\frakm_i)}$.
		By the Chinese Remainder Theorem,
		there exists $a\in A$
		with $a(\frakm_i)=a_i$, for all $i$.
		Replacing $a$ with $\frac{1}{2}(a+\veps a^\sigma)$, we may
		assume that $a\in \Sym_\veps(A,\sigma)$.
		By Lemma~\ref{LM:invertability-test},
		$a\in\units{A}$, so we are done.
	\end{proof}
	
	We finish with showing that idempotent $e$ of Theorem~\ref{TH:index-description}
	can sometimes be chosen to be invariant under a given involution of
	$A$.

	\begin{prp}\label{PR:invariant-primitive-idempotent}
		Let $(A,\sigma)$ be a central simple
		algebra with involution over a field
		$F$ and let $n$ be
		a natural number divisible by $\ind A$
		and not exceeding $\deg A$. If
		$\sigma$ is symplectic, we further require that $n$
		is even. Then
		there exists an idempotent
		$e\in A$ such that $e^\sigma=e$ and $\deg eAe=\rrk_AeA=n$
	\end{prp}
	
	\begin{proof}
		If $A$ contains no $\sigma$-invariant idempotents
		other than $0$ and $1$,
		then \cite[Theorem~8.2]{First_2015_general_bilinear_forms}
		(for instance) implies that
		$e=1$ is the required idempotent.
		Suppose now that $u\in A$ is a nontrivial $\sigma$-invariant
		idempotent and let $v=1-u$. 
		Since $AuA$ is a nonzero two-sided ideal of $A$ invariant under $\sigma$,
		and since $(A,\sigma)$ is a simple ring with involution, $AuA=A$, and likewise
		$AvA=A$.		
		Now, by Corollary~\ref{CR:degree-of-endo-ring},
		$\deg uAu+\deg vAv=\rrk_A (uA\oplus vA)=\rrk_A A=\deg A$,
		$\ind A=\ind uAu=\ind vAv$,
		by Corollary~\ref{CR:type-conjugation}(ii),
		$\sigma|_{uAu}$ and $\sigma|_{vAv}$ have the same type as $\sigma$.
		Express $n$ as $n_1+n_2$ with $n_1\leq \deg uAu$, $n_2\leq \deg vAv$
		and such that $n_1$, $n_2$ are divisible by $\ind A$, or $\lcm\{2,\ind A\}$
		if $\sigma$ is symplectic. Applying induction to 
		$(uAu,\sigma|_{uAu})$ and $(vAv,\sigma|_{vAv})$,
		we get $\sigma$-invariant idempotents $e_1\in uAu$, $e_2\in vAv$
		with $\deg e_iAe_i=n_i$ ($i=1,2$). Take $e=e_1+e_2$.
	\end{proof}
	
%
	
	\begin{lem}\label{LM:inv-idempotent-lift}
		Let $(A,\sigma)$ be an $R$-algebra with involution,
		let $\quo{A}=A/\Jac A$ and let $\quo{\sigma}:\quo{A}\to \quo{A}$
		denote the induced involution.
		Let $\eta\in \quo{A}$ be a $\sigma$-invariant idempotent.
		If $\eta$ is the image of an idempotent in $A$,
		then $\eta$ is the image of a $\sigma$-invariant idempotent in $A$.
	\end{lem}
	
	\begin{proof}
		Denote the image of $a\in A$ in $\quo{A}$ as $\quo{a}$.
		Let $e\in A$ be an idempotent with $\quo{e}=\eta$.
		Since $\eta=\eta^\sigma$, we have $eA+(1-e)^\sigma A+\Jac A=A$,
		so $eA+(1-e)^\sigma A=A$ by Nakayama's Lemma.
		On the other hand, if $a\in eA\cap (1-e)^\sigma A$,
		then $(1-e)a=e^\sigma a=0$, hence $( 1-e +e^\sigma)a=0$.
		Since $\quo{1-e+e^\sigma}=\quo{1}$, we have $1-e+e^\sigma\in \units{A}$,
		so   $a=0$.
		Thus, $A=eA\oplus (1-e)^\sigma A$. Write $1=e_1+f_1$
		with $e_1\in eA$, $f_1\in (1-e)^\sigma A$.
		It is well-known that $e_1$ and $f_1$ are idempotents
		satisfying $e_1A=eA$ and $f_1=(1-e)^\sigma A$.
		Now, $e_1-e_1^\sigma e_1=(1-e_1)^\sigma e_1=f_1^\sigma e_1=
		((1-e)^\sigma f_1)^\sigma ee_1=f_1^\sigma (1-e)ee_1=0$,
		so $e_1=e_1^\sigma e_1$.
		It follows that $e_1^\sigma=(e_1^\sigma e_1 )^\sigma=e_1^\sigma e_1=e_1^\sigma$.
		Finally, since $\quo{e_1}\in \eta \quo{A}$ and $\quo{1-e_1}\in (1-\eta) \quo{A}$,
		we must have $\quo{e_1}=\eta$, because $\quo{A}=\eta \quo{A}\oplus (1-\eta)\quo{A}$
		and $\quo{1}=\eta+(1-\eta)$.
	\end{proof}
	
	\begin{lem}\label{LM:projective-covering-criterion}
		Let $A$ be a semilocal $R$-algebra,
		let $\quo{A}:=A/\Jac A$ and let $\eta\in \quo{A}$
		be an idempotent. Then there exists an idempotent
		$e\in A$ with $\quo{e}:=e+\Jac A=\eta$
		if and only if there exists $P\in\rproj{A}$
		such that $\quo{P}:=P/P\Jac A\cong \eta \quo{A}$
		as right $A$-modules.
	\end{lem}
	
	\begin{proof}
		For the ``only if'' part, take $P=eA$. We turn to prove the ``if'' part.

		Note that $P\to \quo{P}\cong \eta\quo{A}$ is a projective
		covering;  denote this map by $f$.
		Consider the surjective
		homomorphism $g:A_A\to \eta\quo{A}$ given by $g(a)= \eta \quo{a}$.
		Since $f:P\to \eta\quo{A}$ is a projective
		covering, there exists a factorization $A_A=P_1\oplus Q$
		and an isomorphism $P\to P_1$
		such that   
		the composition
		$P\to P_1\xrightarrow{g} \eta\quo{A}$ is $f$.
		In particular, $\quo{P_1}=\eta\quo{A}$. Choose
		an  idempotent $e_1\in A$ such that $P_1=e_1A$.
		Then $\quo{e_1}\quo{A}=\eta\quo{A}$ and $\quo{(1-e_1)}\quo{A}\cong
		\quo{A}/\quo{e_1}\quo{A}=\quo{A}/\eta\quo{A}\cong (1-\eta)\quo{A}$.
		Now, by \cite[Exercise 21.16]{Lam_1991_first_course},
		there exists $x\in \units{\quo{A}}$ with $x\quo{e'}x^{-1}=\eta$.
		Choose $y\in A$ with $\quo{y}=x$ and take $e=y e_1y^{-1}$.
	\end{proof}

	\begin{thm}\label{TH:invariant-primitive-idempotent}
		Let $(A,\sigma)$ be an Azumaya algebra with involution
		over a semilocal   ring $R$.
		Write $S:=\Cent(A)$  and let $n\in\Gamma(\Spec S,\N)$.
		Suppose that $n$ is invariant under $\sigma|_S$
		and satisfies  $\ind A\mid n$ and $n\leq \deg A $.
		If $\sigma$ is symplectic at  $\frakp\in \Spec R$,
		we also require that $ n(\frakp)$ is even.
		Then there exists an idempotent
		$e\in A$ such that $e^\sigma=e$ and $\deg eAe=\rrk_A eA=n$.
	\end{thm}
	
	We remark that $\ind A$ is a $\sigma|_S$-invariant function from $\Spec S$ to $\N$.
	
	\begin{proof}
		Let $\frakm_1,\dots,\frakm_t$
		denote the maximal ideals of $R$.
		We use Lemma~\ref{LM:mod-Jac-decomposition-inv}
		to identify
		$\quo{A}:=A/\Jac A$ with $\prod_{i=1}^t A(\frakm_i)$.
		Since  $\ind A(\frakm_i)\mid (\ind A)(\frakm_i)$,
		we may apply Proposition~\ref{PR:invariant-primitive-idempotent}
		to $(A(\frakm_i),\sigma(\frakm_i))$ and $n(\frakm_i)$
		and get a $\sigma$-invariant idempotent $\eta_i\in A(\frakm_i)$
		with $\rrk_{A(\frakm_i)} \eta_i A(\frakm_i)=n(\frakm_i)$.
		Let $ \eta =(\eta_i)_{i=1}^t\in \quo{A}$.
		Then $ \eta ^{\quo{\sigma}}=\eta$.
		
		By Theorem~\ref{TH:index-description},
		there exists $P\in\rproj{A}$ such that $\rrk_AP=n$.
		Comparing reduced ranks, one sees
		that $\quo{P}=P/P\Jac A\cong P\otimes_A\quo{A}$ is
		isomorphic to $\eta\quo{A}$. Thus, 
		by Lemmas~\ref{LM:inv-idempotent-lift}
		and~\ref{LM:projective-covering-criterion}, there exists
		a $\sigma$-invariant idempotent $e\in A$ projecting onto $\eta$.
		Since $\rrk_{A(\frakm_i)} eA(\frakm_i)=\rrk_{A(\frakm_i)}\eta_iA(\frakm_i)=n(\frakm_i)$
		for all $1\leq i\leq t$, and
		since $\rrk_AeA$ is locally constant,
		we must have $\rrk_AeA=n$.
	\end{proof}

\section{Hermitian Forms}
\label{sec:hermitian}

	This section concerns with hermitian forms, mainly
	over Azumaya
	algebras with involution, and related objects.
	See \cite[Chapter I]{Knus_1991_quadratic_hermitian_forms} for an extensive discussion
	of hermitian forms in general.

	Throughout this section, 
	$(A,\sigma)$ denotes an $R$-algebra with involution and
    $\veps$ is an element of $\Cent(A)$ satisfying $\veps^\sigma \veps=1$.
	Recall our standing assumption that   $2\in\units{R}$.

\subsection{Hermitian Forms}

	We define $\veps$-hermitian spaces over $(A,\sigma)$ in the usual way, i.e.,
	as pairs $(P,f)$ where $P\in\rproj{A}$
	and $f:P\times P\to A$ is a  biadditive map
	satisfying $f(xa,yb)=a^\sigma f(x,y)b$ and $f(x,y)=\veps f(y,x)^\sigma$
	($x,y\in P$, $a,b\in A$). We also say that $f$ is an $\veps$-hermitian
	form on $P$.

	Given  $\veps$-hermitian spaces $(P,f)$, $(P',f')$   over
	$(A,\sigma)$, an isometry $(P,f)\to (P',f')$ is
	an $A$-module isomorphism $\vphi:P\to P'$
	such that $f'(\vphi x,\vphi y)=f(x,y)$ ($x,y\in P$).
	If such an isometry exists, we write $(P,f)\cong (P',f')$  or   $f\cong f'$.
	The 
	group of isometries from $(P,f)$ into itself is denoted $U(f)$. 
	Orthogonal  
	sums of  hermitian spaces or hermitian forms are
	defined in the usual way and are written using the symbol $\oplus$.
	The $n$-fold orthogonal sum $(P,f)\oplus \dots \oplus (P,f)$ 
	is denoted $n\cdot (P,f)$.
	
	\begin{example}\label{EX:diagonal-forms}
		Let $a_1,\dots,a_n\in \Sym_{\veps}(A,\sigma)$. 
		Then the map $f:A^n\times A^n\to A$
		given by $f((x_i),(y_i))=\sum_i x_i^\sigma a_iy_i$ is
		an $\veps$-hermitian form over $(A,\sigma)$. We call $f$ a diagonal  
		form and denote it by
		$\langle a_1,\dots,a_n\rangle_{(A,\sigma)}$. A hermitian
		form which is
		isomorphic to a diagonal  form is called diagonalizable.
	\end{example}
	
	Given $P\in \rproj{A}$, let 
	$P^*$ denote $\Hom_A(P,A)$ endowed with the \emph{right} $A$-module structure
	given by $(\phi a)x=a^\sigma (\phi x)$ ($\phi\in P^*$, $a\in A$, $x\in P$).
	If $f$ is an $\veps$-hermitian form on $P$, then  the map $x\mapsto f(x,-):P\to P^*$ is 
	an $A$-module homomorphism.
	When it is an isomorphism, we say that $(P,f)$, or $f$, is  \emph{unimodular}\footnote{
	Some texts use ``regular'' or ``nondegenerate''.}.
	The category of unimodular $\veps$-hermitian spaces over $(A,\sigma)$ with isometries as 
	its morphisms is
    denoted
    \[\Herm[\veps]{A,\sigma} .\]

	We shall need the following versions of Witt's Cancellation Theorem
	and Witt's Extension Theorem.
	The cancellation is derived from cancellation
	results of  Reiter \cite[Theorem~6.2]{Reiter_1975_Witts_extension_theorem}
	and Keller \cite[Theorem~3.4.2]{Keller_1988_cancellation_of_quadratic_forms}.

	\begin{thm}
    	\label{TH:Witt-cancellation}
    	Suppose that $(A,\sigma)$ is an Azumaya $R$-algebra
    	with involution and $R$ is semilocal,
    	and
    	let $(P_1,f_1),(P_2,f_2),(Q,g)\in \Herm[\veps]{A,\sigma}$.
    	If $f_1\oplus g\cong f_2\oplus g$, then $f_1\cong f_2$.
    \end{thm}

    \begin{proof}
        Write $R=\prod_iR_i$ with each $R_i$ a connected semilocal ring. Working over each factor
        separately, we may assume that $R$ is connected.
       	Under this assumption, we may further assume that $\rank_R P_1>0$, because
       	otherwise $\rank_R P_1=\rank_RP_2=0$, which 
       	implies $P_1=P_2=0$ and $f_1\cong f_2$.
       	At this point, we claim that we may apply Keller's cancellation result 
       	\cite[Theorem~3.4.2(iii)]{Keller_1988_cancellation_of_quadratic_forms}
       	and conclude that $f_1\cong f_2$. 
		Indeed, in order to apply Keller's theorem, 
		we need to check that the number $r$ defined in {\it op.~cit.} is $0$
		for the hermitian space $(P_1,f_1)$. By Lemma~\ref{LM:mod-Jac-decomposition-inv},
		this is equivalent to having $P_1(\frakm)\neq 0$ for all $\frakm\in \Max R$,
		and this holds by our assumption that $\rank_R P_1>0$.
       	Alternatively, one can   use Reiter's version of
       	Witt's Extension Theorem \cite[Theorem~6.2]{Reiter_1975_Witts_extension_theorem},
       	which applies under similar conditions,
       	to conclude the proof.
    \end{proof}

	\begin{thm}\label{TH:Witt-extension}
		Suppose that $R$ is a henselian local ring
		and $(A,\sigma)$ is a finite
		$R$-algebra with involution.
		Let $(P,f)\in\Herm[\veps]{A,\tau}$
        and let $U,V$ be summands of $P$.
        Then any isometry
		$f|_{U\times U}\to f|_{V\times V}$ extends to an isometry
        of $f$.
    \end{thm}

    \begin{proof}
    	By a theorem of Azumaya
    	\cite[Theorem 24]{Azumaya_1951_maximally_central_algebras},
    	$A$ is a \emph{semiperfect} ring.
    	The theorem therefore follows from \cite[Corollary~4.9]{First_2015_Witts_extension_theorem}.
    \end{proof}

\subsection{The Witt Group}
\label{subsec:Witt-grp}

	As usual, a \emph{Lagrangian} of a unimodular hermitian space 
	$(P,f)\in \Herm[\veps]{A,\sigma}$ is a summand $L$ of $P$ such that
    $L=L^\perp:=\{x\in P \suchthat f(x,L)=0\}$. If $(P,f)$ admits a Lagrangian,
    it is called \emph{metabolic}.
    A convenient way to verify that an $A$-submodule $L\leq P$
    with $f(L,L)=0$
    is a Lagrangian is to exhibit another submodule
    $M\leq P$ such that $f(M,M)=0$ and $L\oplus M=P$.
    If such $L$ and $M$ exist, $(P,f)$ is called \emph{hyperbolic}.
    In this case, the map $x\mapsto f(x,-):M\to L^*$ is an isomorphism
	of $A$-modules, and the induced map $P=L\oplus M\to L\oplus L^*$
	is an isometry  from $(P,f)$
	to   $(L\oplus L^*,\Hyp[\veps]{L})$, where $\Hyp[\veps]{L}$
    is the $\veps$-hermitian form given by
    \[\Hyp[\veps]{L}(x\oplus \phi,x'\oplus \phi')=
    \phi x'+\veps (\phi' x)^\sigma
    \]
    ($x,x'\in P$, $\phi,\phi'\in P^*$). 
    Since we assume that $2\in\units{A}$, any Lagrangian $L$
    admits a Lagrangian $M$ with $L\oplus M=P$ 
    \cite[Proposition I.3.7.1]{Knus_1991_quadratic_hermitian_forms},
    so  metabolic spaces are hyperbolic.
    Therefore,
    we shall 
	only consider hyperbolic spaces in the sequel.

\medskip    
    
    Recall that the Witt group of $\veps$-hermitian forms over
    $(A,\sigma)$, denoted
    \[
    W_\veps(A,\sigma),
    \]
    is the Grothendieck group of $\Herm[\veps]{A,\sigma}$, relative to orthogonal sum,
    divided by the subgroup spanned by the (representatives of) hyperbolic spaces.
    The   class represented by $(P,f)$  in $W_\veps(A,\sigma)$
    is denoted $[P,f]$ or $[f]$.
	Two forms $f$, $f'$ representing the same element in $W_\veps(A,\sigma)$
	will be called Witt-equivalent; this   happens if and only
	if there exist hyperbolic forms $h,h'$
	such that $f\oplus h\cong f'\oplus h'$. Note that $-[f]=[-f]$
	because $f\oplus (-f)$ is hyperbolic.
	
	\begin{example}\label{EX:exchange-involution}
		We say that $\sigma:A\to A$ is an \emph{exchange involution}
		if there exists   an idempotent $\eta\in\Cent(A)$
		such that $\eta^\sigma=1-\eta$.
		For example, this is the case if $(A,\sigma)$
		is an Azumaya $R$-algebra with involution
		and $\Cent(A)=R\times R$, because $\sigma|_{\Cent(A)}$
		is the involution $(r,s)\mapsto (s,r)$ (see Proposition~\ref{PR:types-of-involutions-Az}).
		In this situation, 
		there exists an $R$-algebra $B$
		such that $(A,\sigma)\cong (B\times B^\op,(x,y^\op)\mapsto (y,x^\op))$,
		hence the name ``exchange involution''.
		Indeed, take $B=\eta A$; the required isomorphism $A\to B\times B^\op$
		is given
		by $a\mapsto (\eta a,(\eta a^\sigma)^\op)$.

		It easy to see that for any $P\in\rproj{A}$,
		we have $P=P\eta\oplus P\eta^\sigma$. Furthermore, if
		$f$ is a unimodular $\veps$-hermitian form on $P$,
		then $f(P\eta,P\eta)=f(P\eta^\sigma,P\eta^\sigma)=0$ (because
		$\eta^\sigma\eta=0$),
		so $f$ is hyperbolic
		and $f\cong \Hyp[\veps]{ P\eta}$.
		From this we see that every  $(P,f)\in\Herm[\veps]{A,\sigma}$
		is hyperbolic  and is determined up to isomorphism
		by the isomorphism class of $P$.
		In particular, $W_\veps(A,\sigma)=0$.
	\end{example}

	Recall that	
	an $\veps$-hermitian space $(P,f)\in\Herm[\veps]{A,\sigma}$
	is called \emph{isotropic} if $P$ admits a nonzero summand $M$
	such that $f(M,M)=0$. When no such $M$ exists,
	$(P,f)$  is called \emph{anisotropic}. 
	We alert the reader that at this level of generality,
	the existence of $0\neq x\in P$ such that $f(x,x)=0$ does not imply
	that $(P,f)$ is isotropic. However, when $A$ is semisimple artinian,
	$xA$ is a summand of $P$, so  $f$ is isotropic
	if and only if $f(x,x)=0$ for some nonzero $x\in P$.
	 
	\begin{prp}[{\cite[Proposition~I.3.7.9]{Knus_1991_quadratic_hermitian_forms}}] \label{PR:ansio-Witt-equivalent}
		Every $(P,f)\in\Herm[\veps]{A,\sigma}$ 
		can be written as an orthogonal sum
		of    an anisotropic space and a  hyperbolic space.
		In particular, $(P,f)$ is Witt equivalent to an anisotropic 
		$\veps$-hermitian space.
	\end{prp}

	We proceed with showing that if $(A,\sigma)$ is an Azumaya $R$-algebra with involution 
	and $R$ is semilocal,
	then every hermitian form representing $0$ in $W_\veps(A,\sigma)$ is hyperbolic.
	Furthermore, two hermitian spaces in the same Witt class having the same
	reduced rank are isomorphic.
	These  statements may already
	fail for $(A,\sigma)=(R,\id_R)$
	if $R$ is not semilocal;  see    \cite[Example~1.2.6]{Balmer_2005_Witt_groups}, for instance.

	\begin{lem}\label{LM:unimodular-implies-rrk-sigma-inv}
		Suppose that $(A,\sigma)$
		is an   Azumaya $R$-algebra with involution 
		and let $P\in \rproj{A}$.
		Then $\rank_RP=\rank_RP^*$ and $\rrk_AP^*=\sigma  \rrk_AP  $, i.e.,
		$(\rrk_AP^*)(\frakp)=(\rrk_AP )(\frakp^\sigma)$
		for all $\frakp\in \Spec \Cent(A)$.
		In particular, if there exists a unimodular
		$\veps$-hemritian form on $P$, then
		$\rrk_AP$ is $\sigma$-invariant.
	\end{lem}
	
	\begin{proof}	
		Write $S=\Cent(A)$. It is enough to prove lemma
		after specializing to the residue fields of $R$,
		so assume $R$ is a field. 
		
		If $S$ is connected,
		then $\sigma \rrk_AP =\rrk_AP$
		and $A$ is a simple artinian ring. Length considerations
		force $P\cong P^*$, hence $\rank_RP=\rank_RP^*$
		and  $\rrk_AP^*=\rrk_AP=\sigma \rrk_AP $. 
		
		If $S$ is not connected, then $S=R\times R$
		and $\sigma|_R$ is the exchange involution. Thus,
		as in Example~\ref{EX:exchange-involution},
		there exists a central simple $R$-algebra $B$
		such that
		$(A,\sigma)\cong (B\times B^\op,\tau)$
		where $(x,y^\op)^\tau=(y,x^\op)$.
		Identifying $A$ with $B\times B^\op$, we can write $P=P_1\times P_2$
		where $P_1\in\rproj{B}$ and $P_2\in\rproj{B^\op}$.
		Regarding $P_1$ and $P_2$ as $A$-modules, 
		one readily checks that $(1_B,0_B^\op)$ annihilates
		$P_1^*$ and  $(0_B,1_B^\op)$ annihilates $P_2^*$, so $P_1^*\in \rproj{B^\op}$
		and $P_2^*\in\rproj{B}$. Since $*:\rproj{A}\to \rproj{A}$ is a duality 
		and $\rproj{A}$ is  abelian semisimple,
		$P_1$ and $P_1^*$  have the same $A$-length,
		so $\operatorname{length}_B P_1=\operatorname{length}_{B^\op} P_1^*$.
		Likewise, $\operatorname{length}_{B^\op} P_2=\operatorname{length}_{B } P_2^*$.
		Since $B$ and $B^\op$ are central simple $R$-algebras of equal degree,
		this means that $\rank_RP=\rank_RP^*$
		and $\rrk_AP^*=\rrk_A(P_2^*\times P_1^*)=\sigma \rrk_AP $.
	
		Finally, if $P$ carries a unimodular hermitian form,
		then $P\cong P^*$, so $\rrk_AP=\rrk_A P^*=\sigma \rrk_AP $.
	\end{proof}

	\begin{lem}\label{LM:rank-determines-hyperbolic}
		Suppose that $(A,\sigma)$
		is an Azumaya algebra with involution over a semilocal ring $R$,
		and 
		let $(P_1,f_1),(P_2,f_2)\in\Herm[\veps]{A,\sigma}$ be hyperbolic.
		If $\rrk_A P_1\leq \rrk_A P_2$, then there is $V\in\rproj{A}$
		such that $f_1\oplus\Hyp[\veps]{V}\cong f_2$.
		In particular, if $\rrk_A P_1=\rrk_A P_2$, then $f_1\cong f_2$.
	\end{lem}
	
	\begin{proof}
		Write $S=\Cent(A)$.
		As in the proof of Theorem~\ref{TH:Witt-cancellation},
		we may assume that $R$ is connected.
		
		Suppose first that $S$ is not connected. By  Lemma~\ref{LM:non-connected-S},   $S=R\times R$.
		Put $\eta=(1_R,0_R)\in S$. Then by 
		Example~\ref{EX:exchange-involution}, 
		$f_1\cong \Hyp[\veps]{P_1\eta}$ and $f_2\cong \Hyp[\veps]{P_2\eta}$.
		The assumption $\rrk_A P_1\leq \rrk_A P_2$ means
		that $\rrk_A P_1\eta \leq \rrk_A P_2\eta$, so by Lemma~\ref{LM:rank-determines},
		there is $V\in\rproj{A}$ such that $P_1\eta\oplus V\cong P_2\eta$. Then
		$f_1\oplus\Hyp[\veps]{V}\cong \Hyp[\veps]{P_1\eta}\oplus\Hyp[\veps]{V}\cong
		\Hyp[\veps]{P_2\eta}\cong f_2$.
		
		Now assume   that $S$ is connected and
		write $f_1\cong \Hyp[\veps]{U_1}$
		and $f_2\cong \Hyp[\veps]{U_2}$ with $U_1,U_2\in\rproj{A}$.
		Lemma~\ref{LM:unimodular-implies-rrk-sigma-inv} 
		and the connectivity of $S$ imply that
		$2\rrk_AU_1=\rrk_AP_1\leq \rrk_AP_2=2\rrk_AU_2$. By Lemma~\ref{LM:rank-determines},
		there is $V\in\rproj{A}$ such that 
		$U_1\oplus V\cong U_2$.	Then $f_1\oplus\Hyp[\veps]{V} \cong f_2$.
	\end{proof}	
	
	\begin{thm}\label{TH:trivial-in-Witt-ring}
		Suppose that $(A,\sigma)$
		is an Azumaya algebra with involution over a semilocal ring $R$,
		and let $(P,f),(P',f')\in \Herm[\veps]{A,\sigma}$.
		\begin{enumerate}[label=(\roman*)]
			\item If $\rrk_AP\leq \rrk_A P'$, then there exists
			$V\in \rproj{A}$ such that $f\oplus\Hyp[\veps]{V}\cong f'$.
			\item If $[f]=0$, then $f$ is hyperbolic.
			\item If $[f]=[f']$ and $\rrk_AP=\rrk_A P'$, then $f\cong f'$.
		\end{enumerate}
	\end{thm}
	
	\begin{proof}
		(i) There are $U,W\in\rproj{A}$ such that
		$f\oplus \Hyp[\veps]{U}\cong f'\oplus \Hyp[\veps]{W}$.
		Then $\rrk_A (U\oplus U^*)-\rrk_A(W\oplus W^*)=\rrk_A P'-\rrk_AP\geq 0$,
		so by Lemma~\ref{LM:rank-determines-hyperbolic},
		there is $V\in\rproj{A}$ such that $\Hyp[\veps]{U}\cong \Hyp[\veps]{W}\oplus\Hyp[\veps]{V}$.
		(Caution: $U\cong W\oplus V$ is {\it a priori} not guaranteed.)
		Then $(f\oplus \Hyp[\veps]{V})\oplus \Hyp[\veps]{W}\cong f\oplus \Hyp[\veps]{U}\cong f'\oplus \Hyp[\veps]{W}$.
		By Theorem~\ref{TH:Witt-cancellation}, this means that $f\oplus\Hyp[\veps]{V}\cong f'$.
		
		(ii) Apply (i) with $(P,f)$ being the zero hermitian space.
		
		(iii) By (i), $f\oplus\Hyp[\veps]{V}\cong f'$ for some $V\in\rproj{A}$, and   $V=0$
		because $\rrk_AV=\rrk_AP'-\rrk_AP=0$.
	\end{proof}

	We also record the following useful corollary to Lemma~\ref{LM:unimodular-implies-rrk-sigma-inv}.
	
	\begin{cor}\label{CR:constant-even-ranks}
		Suppose that $(A,\sigma)$
		is an   Azumaya $R$-algebra with involution  
		and let $(P,f)\in \Herm[\veps]{A,\sigma}$.
		\begin{enumerate}[label=(\roman*)]
			\item If $R$ is connected, then $\rrk_AP$ is constant.
			\item 
			If $S:=\Cent(A)$ is connected and $f$ is hyperbolic,
			then there exists $V\in\rproj{A}$ with $\rrk_AP=2\rrk_AV$.
		\end{enumerate}
	\end{cor}
	
	\begin{proof}
		(i) 
		If $S=R$, then this is clear.
		If $S\neq R$, then $S$ is a quadratic \'etale $R$-algebra
		and $\sigma|_S$ is the standard $R$-involution of $S$ (see~\ref{subsec:quad-etale}).
		Thus, $\sigma$ acts transitively on every fiber of 
		$\Spec S\to \Spec R$.
		By Lemma~\ref{LM:unimodular-implies-rrk-sigma-inv},
		this means that
		$\rrk_AP$ is constant on the fibers of $\Spec S\to \Spec R$.
		Thus, by   \eqref{EQ:rank-change-of-rings},
		we have $\iota \rank_RP=\iota \rank_RS\cdot \rank_SP=2\rank_SP$, where $\iota:R\to S$
		is the inclusion.
		Since the left hand side is constant ($R$ is connected),
		$\rrk_AP$ is also constant.

		(ii) There exists $V\in\rproj{A}$ such that $P=V\oplus V^*$.
		By Lemma~\ref{LM:unimodular-implies-rrk-sigma-inv},
		$\rrk_AP=\rrk_AV+\sigma\rrk_AV$,
		and $\sigma\rrk_AV=\rrk_AV$ because $S$ is connected.
	\end{proof}

\subsection{Base Change}
\label{subsec:herm-base-change}

	Let $R\to S$ be a ring homomorphism.
	Given an $\veps$-hermitian space $(P,f)$ over $(A,\sigma)$,
	define
	its \emph{base change}   along $R\to S$ to be the $\veps$-hermitian
	space $(P_S,f_S)$ over $(A_S,\sigma_S)$, where $P_S=P\otimes S$
	and $f_S$ is determined by $f_S(x\otimes s,y\otimes t)=f(x,y)\otimes st$
	($x,y\in P$, $s,t\in S$). It is well-known that if $(P,f)$
	is unimodular, resp.\ hyperbolic, then so is $(P_S,f_S)$.
	When $S=k(\frakp)$ for $\frakp\in\Spec R$,
	we shall write $f(\frakp)$ instead of $f_{k(\frakp)}$.

\medskip

	Let $\rho: (B,\tau)\to (A,\sigma)$
	be a homomorphism of $R$-algebras with involution
	and let $\delta\in \Cent(B)$ be an element
	such that $\delta^\tau\delta=1$ and $\veps:=\rho(\delta)\in \Cent(A)$.
	We view $A$ as a left $B$-module  via $\rho$.
	For every $\delta$-hermitian space $(Q,g)$ over $(B,\tau)$,
	define   $\rho(Q,g)$ to be $(Q\otimes_BA,\rho g)$,
	where $\rho g:(Q\otimes_BA)\times (Q\otimes_BA)\to A$
	is the biadditive pairing determined by $\rho g(x\otimes a,x'\otimes a') =a^\sigma \cdot \rho(f(x,x'))\cdot a'$
	($x,x'\in P$, $a,a'\in A$).
	It is routine to check that $\rho(Q,g)$ is an $\veps$-hermitian space over $(A,\sigma)$.
	Furthermore, it is unimodular, resp.\ hyperbolic, when $(Q,g)$ is.
	The assignment $\rho$ extends to a functor
	$\rho:\Herm[\delta]{B,\tau}\to \Herm[\veps]{A,\sigma}$
	by setting $\rho\vphi=\vphi\otimes_B\id_A$.\footnote{
		We do not write $\rho(Q,g)$
		as $(P_A,f_A)$ because we
		reserve the subscript notation for base change relative to the base ring $R$.
	}

\subsection{Adjoint Involutions}
\label{subsec:adjoint-inv}

	Let $(P,f)\in\Herm[\veps]{A,\sigma}$.
	It is well-known that there exists a unique
	$R$-linear involution $\theta:\End_A(P)\to\End_A(P)$
	satisfying
	$f(\vphi x,y)=f(x,\vphi^\theta y)$ for all $\vphi\in\End_A(P)$, see
	\cite[I.\S9.2]{Knus_1991_quadratic_hermitian_forms}. 
	It is called the adjoint involution of $f$.
	Notice that $U(f)$ coincides with the group
	$U(\End_A(P),\theta):=\{\vphi\in \End_A(P)\suchthat \vphi^\theta \vphi= \vphi \vphi^\theta=1\}$.
	
	If $R\to S$ is a ring homomorphism, then $P\in\rproj{A}$ implies that the natural
	map $\End_A(P)\otimes S\to \End_{A_S}(P_S)$ is an isomorphism.
	Under this isomorphism, $\theta_S$ is the adjoint involution of $f_S$.
	
	\begin{example}\label{EX:adoint-inv-diag-form}
		Let $\alpha,\beta\in \calS_\veps(A,\sigma)\cap \units{A}$ and consider
		the diagonal binary $\veps$-hermitian form $\langle \alpha,\beta\rangle_{(A,\sigma)}$
		on $A^2$ (notation as in Example~\ref{EX:diagonal-forms}). 
		Direct computation shows that, upon realizing $\End_A(A^2_A)$
		as $\nMat{A}{2}$, the adjoint involution of $\langle \alpha,\beta\rangle_{(A,\sigma)}$
		is given by $[\begin{smallmatrix} x & y\\ z & w \end{smallmatrix}]\mapsto
		[\begin{smallmatrix} \alpha^{-1}x^\sigma\alpha &  \alpha^{-1} z^\sigma \beta \\ 
		\beta^{-1}  y^\sigma \alpha & \beta^{-1}w^\sigma\beta \end{smallmatrix}]$
		($x,y,z,w\in A$).
		When $\alpha,\beta\in\Cent(A)$, this simplifies
		into $[\begin{smallmatrix} x & y\\ z & w \end{smallmatrix}]\mapsto
		[\begin{smallmatrix}  x^\sigma  & \gamma z^\sigma   \\ 
		\gamma^{-1}  y^\sigma  &  w^\sigma  \end{smallmatrix}]$,
		where $\gamma=\alpha^{-1}\beta$ lives in $\Sym_1(\Cent(A),\sigma)$.
	\end{example}
	
	\begin{prp}\label{PR:type-of-adjoint}
		Suppose that $(A,\sigma)$ is an Azumaya $R$-algebra with involution. 
		Let $(P,f)\in\Herm[\veps]{A,\sigma}$
		and let $\theta:\End_A(P)\to\End_A(P)$ be the adjoint
		involution of $f$.
		If   $\rrk_AP>0$, then $(\End_A(P),\theta)$
		is  an Azumaya $R$-algebra with involution
		and $\theta$ and
		$(\sigma,\veps)$   
		have the same type.
	\end{prp}
	
	\begin{proof}
		Write $S=\Cent(A)$. 
		By Proposition~\ref{PR:degree-of-endo-ring}(i), 
		$\End_A(P)$ is Azumaya over $S$.
		It is easy to check that $\theta|_S=\sigma|_S$,
		hence $(\End_A(P),\theta)$
		is an Azumaya $R$-algebra with involution.
		To see that the types of $\theta$
		and $(\sigma,\veps)$ coincide, we need
		to check that they coincide at every $\frakp\in\Spec R$,
		so we may assume $R$ is a field.
		In this case, it is clear that $\theta$
		is unitary if and only if $(\sigma,\veps)$
		is unitary. For the orthogonal and
		symplectic cases, see \cite[Theorem 4.2(1)]{Knus_1998_book_of_involutions}.
	\end{proof}

	The converse of Proposition~\ref{PR:type-of-adjoint},
	namely, that every involution of $\End_A(P)$
	is adjoint to some hermitian form, holds when $R$ is a field;
	see \cite[Theorem 4.2]{Knus_1998_book_of_involutions}.
	In fact, it holds in general if one allows hermitian forms to take
	values in     $(A^\op,A)$-progenerators; see \cite{First_2015_general_bilinear_forms}
	and \cite[\S3--4]{First_2015_morita_equiv_to_op}.
	We shall need a special case of the latter observation.
	
	\begin{prp}\label{PR:involution-is-adjoint}
		Suppose that $(A,\sigma)$ is an Azumaya $R$-algebra with involution,
		$S:=\Cent(A)$ is semilocal  
		and   $A=\End_S(Q)$ for some $Q\in\rproj{S}$.
		Then there exists $\delta\in S$
		with $\delta^\sigma\delta=1$
		and  a unimodular $\delta$-hermitian
		form $g:Q\times Q\to S$ over $(S,\sigma|_S)$
		such that $\sigma$ is   adjoint to $g$. 
		One has $\delta=1$ when $\sigma$ is orthogonal
		and $\delta=-1$ when $\sigma$ is symplectic.
	\end{prp}
	
	\begin{proof}
		By \cite[Theorem~4.2]{Saltman_1978_Azumaya_algebras_w_involution}
		(or, alternatively, \cite[Proposition~4.6]{First_2015_morita_equiv_to_op}),
		there exist    $\delta_1\in S$
		with $\delta_1^\sigma \delta_1=1$,
		a rank-$1$ projective $S$-module  $L$,
		a $\sigma|_S$-linear involutive automorphism $\tau:L\to L$,
		and a unimodular $L$-valued $ \sigma|_S$-sesquilinear form
		$g:Q\times Q\to L$ satisfying $g(x,y)=\delta_1 g(y,x)^\tau$
		and having  $\sigma$ as its adjoint involution.
		(Here, unimodularity means that $x\mapsto g(x,-):P\to \Hom_S(P,L)$
		is bijective.) Since $S$ is semilocal, $L\cong S_S$,
		so we may assume $L=S$. 
		Put $\delta_2= (1_S)^\tau$. Then $\sigma\circ \tau\in \End_S(S)=S$
		maps $1_S$ to $\delta_2^\sigma$, and so it coincides
		with the $S$-endomorphism $s\mapsto \delta_2^\sigma s:S\to S$.
		As a result, $s^{\tau \sigma}=\delta_2^\sigma s$, or rather $s^\tau =\delta_2  s^\sigma$, for all $s\in S$.
		Taking $s=\delta_2$ and noting that $\delta_2^\tau=1$ (because $1^\tau=\delta_2$),
		we get $\delta_2\delta_2^\sigma =1$.
		Thus, $g:P\times P\to S$ is a unimodular $\delta_1\delta_2$-hermitian
		form over $(S,\sigma|_S)$ with adjoint involution $\sigma$. 
		Write $\delta=\delta_1\delta_2 $.

		By Proposition~\ref{PR:type-of-adjoint},
		the type of $\sigma$
		is the same as the type of $(\sigma|_S,\delta)$.
		Thus,  $\delta=1$ if $\sigma$ is orthogonal
		and $\delta=-1$ if $\sigma$ is symplectic.
	\end{proof}
	
	The following proposition can be proved
	directly, but we   use Theorem~\ref{TH:invariant-primitive-idempotent}
	together with adjoint involutions to give a short proof.
	
	\begin{prp}\label{PR:diagonalizable-herm-forms}
		Suppose that $(A,\sigma)$ is an Azumaya $R$-algebra with involution
		and $R$ is semilocal.
		Let $(P,f)\in\Herm[\veps]{A,\sigma}$
		and suppose that $P$ is a free $A$-module.
		If $(\sigma,\veps)$
		is symplectic at $\frakp\in \Spec R$,
		we also assume that $2\mid (\deg A)(\frakp)$.
		Then $(P,f)$ is   diagonalizable (see Example~\ref{EX:diagonal-forms}).
	\end{prp}
	
	\begin{proof}				
		If $P=0$, there is nothing to prove, so assume $P\neq 0$.
		Since $P$ is free, we have $\rrk_AP>0$.
		Let $B =\End_A(P)$
		and let $\theta:B\to B$ be the adjoint involution of
		$f$. By Proposition~\ref{PR:degree-of-endo-ring}(i), $\ind B=\ind A\mid \deg A$
		and $\deg A\leq \rrk_AP=\deg B$.
		Furthermore,
		if $(\sigma,\veps)$, equivalently
		$\theta$, is symplectic at $\frakp\in \Spec R$,
		then $2\mid (\deg A)(\frakp)$. Thus, by Theorem~\ref{TH:invariant-primitive-idempotent},
		there exists an idempotent $e\in B$ such that $e^\sigma =e$
		and $\rrk_BeB=\deg A$.
		Write $e'=1_B-e$.
		It is easy to check that $(P,f)=(eP,f|_{eP\times eP})\oplus
		(e'P,f|_{e'P\times e'P})$.
		Furthermore, by Proposition~\ref{PR:degree-of-endo-ring}(ii),
		$\rrk_A eP=\rrk_A eB\otimes_B P=\rrk_B eB=\deg A$,
		so $eP\cong A_A$ (Lemma~\ref{LM:rank-determines}).
		Thus, $f|_{eP\times eP}\cong\langle a\rangle$
		for some $a\in\Sym_{\veps}(A,\sigma)$. Proceed
		by induction on $(e'P,f|_{e'P\times e'P})$.
	\end{proof}

\subsection{The Isometry Group Scheme}
\label{subsec:isometry-group}

	Suppose that $A$  is finite projective over $R$
	and let $(P,f)\in\Herm[\veps]{A,\sigma}$.
	By \cite[Appendix]{Bayer_2016_patching_weak_approximation},
	the functor $S\mapsto U(f_S)$ from $R$-rings to groups
	is represented by a smooth affine group $R$-scheme,  
    denoted $\uU(f)$. Since $\uU(f)\to \Spec R$
    is smooth, by \cite[Corollaire~15.6.5]{EGA_iv_1967}, $\uU(f)$
    admits a unique open   subgroup, $\uU^0(f)$,
    such that $\uU^0(f)\to \Spec R$ is connected, i.e., the
    fiber $\uU^0(f)\times_R k(\frakp)$ over  $\Spec k(\frakp)$
    is connected for all $\frakp\in\Spec R$. Moreover, $\uU^0(f)\to \Spec R$
    is geometrically connected 
    \cite[\href{https://stacks.math.columbia.edu/tag/04KV}{Tag 04KV}]{DeJong_2018_stacks_project}. We 
	call $\uU^0(f)$ the neutral component of $\uU(f)$ and    
    write
    $U^0(f)=\uU^0(f)(R)$.

 	\begin{remark}\label{RM:Weil-rest-of-U}
		In case the base ring $R$ is not clear from the context, we shall
   		write $\uU_R(f)$, $\uU_R^0(f)$, $U_R^0(f)$ instead of $\uU(f)$, $\uU^0(f)$, $U^0(f)$.
   		However, somewhat conveniently, if $A$ is projective
   		over $R_1:=\Sym_1(\Cent(A),\sigma)$, e.g., when $A$ is separable projective
   		over $R$ (Example~\ref{EX:Azumaya-over-fixed-subring}), 
   		then $U^0(f)$ is independent of the base ring $R$.
   		
   		Indeed,
   		by \cite[Proposition~2.4.6(1)]{Ford_2017_separable_algebras},
   		$R_1$ is an $R_1$-summand of $A$, and therefore $R_1$ is finite projective
   		over $R$.
   		Note that $\uU_{R }(f)=\calR_{R_1/R}\uU_{R_1}(f)$,
		where $\calR_{R_1/R}$ is the Weil restriction;
		see \cite[\S7.6]{Bosch_1990_Neron_models}, for instance.
		By \cite[Proposition~7.6.2(i)]{Bosch_1990_Neron_models},
		the $R_1/R$-Weil restriction of an open immersion   is an open immersion,
		so $\calR_{R_1/R}\uU_{R_1}^0(f)$ is open
		in $\uU_{R }(f)$. In addition,
		since  $\uU_{R_1}^0(f)\to \Spec R_1$
		is  geometrically connected, affine and smooth,
		the fibers of $\calR_{R_1/R}\uU_{R_1}^0(f)\to\Spec R$
		are connected
		(\cite[Proposition~A.5.9]{Conrad_2015_pseudo_reductive_groups_ed_ii}).
		As a result,
		$\uU^0_R(f)$ conincides with $\calR_{R_1/R}\uU_{R_1}^0(f)$
		and  $U^0_R(f)=(\calR_{R_1/R}\uU_{R_1}^0)(R)=U_{R_1}^0(f)$.
		In particular, $U_R^0(f)$ is determined by $(A,\sigma)$ and is independent
		of $R$.
	\end{remark}
	
	\begin{remark}\label{RM:relation-to-U-A-sigma}
		Keeping the previous assumptions, write $E=\End_A(P)$
		and let   $\theta$ denote the adjoint involution of $f$.
		Then $\uU(f)$ coincides with $\uU(E,\theta)$, the group
		$R$-scheme representing the functor 
		$S\mapsto U(E_S,\theta_S):=\{x\in E_S\suchthat {x^\theta x=1}\}$.
		Indeed, for any $R$-ring $S$, we have $\uU(f)(S)=U(f_S)=U(E_S,\theta_S)$
		upon identifying $\End_{A_S}(P_S)$ with $\End_A(P)_S$.
		As a result, $\uU^0(f)$ is the neutral component of $\uU(E,\theta)\to \Spec R$,
		denoted $\uU^0(E,\theta)$.
	\end{remark}

	We describe $\uU^0(f)$ explicitly when $(A,\sigma)$
	is an Azumaya $R$-algebra with involution and $\rrk_AP>0$.
	Since we can factor  $R$ as $\prod_{i=1}^t R_i$
	such that $\rrk_{A\otimes R_i} P_{R_i}$ is constant for all $i$, it
	is enough to consider the case where $\rrk_AP$ is constant.
	By Proposition~\ref{PR:types-of-involutions-Az}(v),
	$(\sigma,\veps)$
	is now either orthogonal, symplectic or unitary.

	When $(\sigma,\veps)$ is symplectic or unitary, 
	the   fibers of $\uU(f)\to \Spec R$
	are well-known to be outer forms of $\uSp_{2n}$
	or $\uGL_n$, respectively; see \cite[\S23A]{Knus_1998_book_of_involutions}. Thus, they are connected and
	$\uU^0(f)=\uU(f)$.

	Suppose now that $(\sigma,\veps)$
	is orthogonal, let $E=\End_A(P)$  and let $\theta:E\to E$
	be the adjoint involution  of $f$. 
	Then $(E,\theta)$ is an Azumaya
	$R$-algebra with an orthogonal involution (Proposition~\ref{PR:type-of-adjoint}).
	Let $\umu_{2,R}\to \Spec R$
	denote the affine group $R$-scheme representing 
	the functor $S\mapsto \mu_2(S):=\{s\in S\suchthat s^2=1\}$.
	For every $R$-ring $S$, 
	the reduced norm map, $\Nrd_{E_S/S} :E_S\to S$ (see~\cite[p.~410]{Ford_2017_separable_algebras}),
	is compatible with base change and
	restricts to a group homomorphism $U(f_S)=U(E_S,\theta_S)\to \mu_2(S)$.\footnote{
		One can  show that $\Nrd_{E/R}$ maps $U(E,\theta)$ to $\mu_2(R)$, and similarly after
		base-changing to $S$, as follows:
		By
		\cite[III.\S8.5]{Knus_1991_quadratic_hermitian_forms} 
		or \cite[Theorems 5.17 \& 5.37, Examples~7.3 \& 7.4]{First_2020_involutions_of_Azumaya_algs},
		there exists  a faithfully flat $R$-ring $R'$ 
		such that $(E_{R'},\theta_{R'})\cong (\nMat{R'}{n},\mathrm{t})$,
		where $\mathrm{t}$ is the transpose involution. 
		Now, for all $x\in U(\nMat{R'}{n},\mathrm{t})$, we have
		$\Nrd(x)^2=\det(x)^2=\det(x^{\mathrm{t}}x)=1$,
		so $\Nrd(x)\in \mu_2(R')$.		
		As a result,  $\Nrd_{E/R}$ maps
		$U(E,\theta)$ to $R\cap \mu_2(R')=\mu_2(R)$.
	}
	Thus, it determines a morphism of affine group $R$-schemes
	\[\Nrd:\uU(f)\to \umu_{2,R}.\]
	The scheme-theoretic kernel of this morphism --- call it $\bfK$ --- is   $\uU^0(f)$.
	Indeed, $\bfK$ is   open in $\uU(f)$
	because the trivial group $R$-scheme $\mathbf{1}$ is open in $\umu_{2,R}$ (recall
	that $2\in\units{R}$),
	and the fiber  of $\bfK $ over $\frakp\in \Spec R$ is $\uU(E(\frakp),\theta(\frakp))$,
	which is a form of $\uSO_n$ for $n=\deg A(\frakp)$ \cite[\S23B]{Knus_1998_book_of_involutions}, hence
	connected.

	We conclude the previous discussion with:
	
	\begin{prp}\label{PR:U-zero-description}
		Let $(A,\sigma)$ be an Azumaya $R$-algebra with involution
		and
		let $(P,f)\in \Herm[\veps]{A,\sigma}$. Assume 
		that $\rrk_AP>0$. If $(\sigma,\veps)$
		is symplectic or unitary, then $\uU^0(f)=\uU(f)$.
		If $(\sigma,\veps)$
		is orthogonal, then $\uU^0(f)=\ker(\Nrd:\uU(f)\to \umu_{2,R})$. 
	\end{prp}

	The following lemma is convenient 
	for verifying equalities in $\mu_2(R)$.
	
	\begin{lem}\label{LM:mu-two-check}
		Let $\alpha,\beta\in \mu_2(R)$. Then $\alpha=\beta$
		if and only if $\alpha(\frakm)=\beta(\frakm)$ for all
		$\frakm\in \Max R$.
	\end{lem}
	
	\begin{proof}
		Write $\gamma=\alpha^{-1}\beta$.
		It is enough to prove that if $\gamma(\frakm)=1$
		for all $\frakm\in\Max R$, then $\gamma=1$.
		Note that $\frac{1}{2}(1-\gamma)$ is an idempotent.
		If $\gamma(\frakm)=1$ for all $\frakm\in \Max R$, then $1-\gamma\in \Jac R$,
		so the idempotent $\frac{1}{2}(1-\gamma)$ must be $0$ and $\gamma=1$.
	\end{proof}

	Following are  two theorems  that  will play a major role in the sequel.

    \begin{thm}\label{TH:U-zero-mapsto-onto-closed-fibers}
        Suppose that  $(A,\sigma)$ is  an Azumaya  $R$-algebra with involution and $R$ is semilocal,
        and let $(P,f)\in\Herm[\veps]{A,\sigma}$.
        Then the specialization map 
        $U^0(f)\to \prod_{\frakm\in\Max R} U^0(f(\frakm))$
        is surjective. 
    \end{thm}
    
    \begin{proof}
		Write $R=\prod_{i=1}^t R_i$ with each $R_i$ connected. 
		Working over each factor separately, we may assume $R$
		is connected. We may further assume that $\rrk_AP> 0$.
		
		Let $E$ and $\theta$ be as in Remark~\ref{RM:relation-to-U-A-sigma}
		and write $U^0(E,\theta)=\bfU^0(E,\theta)(R)=U^0(f)$.
		Then $(E,\theta)$ is Azumaya over $R$ (Proposition~\ref{PR:type-of-adjoint}),
		$\theta$ is either orthogonal,  symplectic, or  unitary
		(Propositions~\ref{PR:types-of-involutions-Az}(v)),
		and
		the theorem is equivalent to  $U^0(E,\theta)\to\prod_{\frakm\in \Max R}U^0(E(\frakm),\theta(\frakm))$
		being surjective.
		This holds by \cite[Theorem~2]{First_2020_orthogonal_group}
		(and Proposition~\ref{PR:U-zero-description})
		when $\theta$ is orthogonal and by \cite[Theorem~6]{First_2020_orthogonal_group}
		when $\theta$ is not orthogonal.
    \end{proof}

    \begin{remark}
    	Under the assumptions of Theorem~\ref{TH:U-zero-mapsto-onto-closed-fibers},
		the specialization map $U(f)\to\prod_{\frakm\in \Max R} U(f(\frakm))$
		may fail to be surjective in general.
		For example, take
		$R$ to be a connected semilocal ring with two maximal
		ideals, let $(A,\sigma)=(R,\id_R)$ 
		and consider the $1$-hermitian form $f(x,y)=xy$ on $R$.
    \end{remark}

	\begin{thm}  \label{TH:criterion-for-det-one}
		Suppose  that $(A,\sigma)$ is an Azumaya $R$-algebra with    involution,
		$(\sigma,\veps)$ is orthogonal  and 
		$R$ is    semilocal.
		Let $(P,f)\in\Herm[\veps]{A,\sigma}$ be   hermitian space with $\rrk_A P>0$.
		Then $\Nrd:U(f)\to\mu_2(R)$ is surjective
		if and only if $[A]=0$.
	\end{thm}
	
	\begin{proof}
		As in the proof of Theorem~\ref{TH:U-zero-mapsto-onto-closed-fibers},
		we may assume
		that $R$ is connected, and hence, $\mu_2(R)=\{\pm 1\}$.
		Now, the theorem follows by
		applying  \cite[Theorem~1]{First_2020_orthogonal_group}
		to the adjoint involution of $f$. Note that
		$\End_A(P)$ is Azumaya over $R$ because $\rrk_AP>0$ (Proposition~\ref{PR:degree-of-endo-ring}(i)).
	\end{proof}

   	We finish with the following well-known theorem.	
	
	\begin{thm}\label{TH:rational-variety}
		Suppose that $R$ is a field and $A$ is finite dimensional over $R$,
		and	
		and let $(P,f)\in \Herm[\veps]{A,\sigma}$.
		Then $\uU^0(f)\to \Spec R$ is a rational variety.
	\end{thm}
	
	\begin{proof}[Proof (sketch)]
		We already know that $\uU^0(f)\to \Spec R$ is irreducible. 
		Let $\theta$
		be the adjoint involution of $f$
		and let $\bfV$ denote
		the affine $R$-variety
		representing the functor $S\mapsto \Sym_{-1}(\End_{A_S}(P_S),\theta_S)$;
		it is isomorphic to $\mathbb{A}^n_R$ for $n=\dim_R \Sym_{-1}(\End_A(P),\theta)$.
		A birational equivalence between
		$\bfV$ and $\uU^0(f)$ is given
		by the Cayley transform, $y\mapsto (1+y)(1-y)^{-1}:\bfV\dashrightarrow \uU^0(f)$,
		and its inverse,
		$x\mapsto - (1+x)^{-1}(1-x):\uU^0(f)\dashrightarrow \bfV$.
	\end{proof}

\subsection{More on Lagrangians}
\label{subsec:Lagrangians}

	Throughout this subsection, $(A,\sigma)$ 
	denotes an Azumaya $R$-algebra with involution.
	Given $(P,f)\in \Herm[\veps]{A,\sigma}$,
	let
	\[
	\Lag(f)=\{L\subseteq P\suchthat
	\,\text{$L$ is a Lagrangian of $f$ and $\rrk_AL=\textstyle{\frac{1}{2}}\rrk_AP$}\}.
	\]
	Recall from~\ref{subsec:Witt-grp} that if
	$L$ is a Lagrangian of $f$, then $P=L\oplus L^*$, hence $\rrk_A P=\rrk_AL+\sigma \rrk_AL$
	by  Lemma~\ref{LM:unimodular-implies-rrk-sigma-inv}.
	Thus, if $\sigma|_{\Cent(A)}=\id_{\Cent(A)}$, or if $\Cent(A)$ is connected, then $\Lag(f)$
	consists of all Lagrangians of $f$.

	In this   subsection, we collect  several facts about
	the action of $U^0(f)$ on $\Lag(f)$.
	Some of the results will require the use of sheaves,
	and we refer the reader to \cite[Chapter III.\S2]{Knus_1991_quadratic_hermitian_forms}
	for a scheme-free introduction, or \cite{Milne_1980_etale_cohomology} for an extensive treatment.

\medskip
	
	The map $S\mapsto \Lag(f_S)$
	naturally extends to a functor, $\uLag(f)$, from $R$-rings to
	sets. 
	It is routine to check that
	Lagrangians descend along along faithfully flat ring homomorphisms.
	That is, if $S\to T$ is a faithfully flat map of $R$-rings, $i_1,i_2:T\to T\otimes_ST$
	are the maps $t\mapsto t\otimes 1$ and $t\mapsto 1\otimes t$,
	and
	$L\in \Lag(f_{T})$ satsfies $\uLag(i_1)(L)=\uLag(i_2)(L)$,
	then there exists a unique $L_0\in \Lag(f_S)$ with $(L_0)\otimes_S T=L$;
	consult  \cite[III.\S\S1--2]{Knus_1991_quadratic_hermitian_forms}.
	Thus, $\uLag(f)$ is sheaf relative to the fppf topology
	on the category
	of affine $R$-schemes, denoted $(\Aff/R)_{\fppf}$.\footnote{
		With the appropriate definitions, this functor also extends
		to a sheaf on the site of all $R$-schemes
		with the fpqc topology. 	
	}
	(In fact, it can be shown that $\uLag(f)$ is represented
	by a  non-affine  $R$-scheme, but this fact will not
	be needed in this work.)
	The group $U(f)$ acts on $\Lag(f)$ in a way which is compatible
	with base change, thus giving rise to an
	action of $\uU(f)$ on $\uLag(f)$.

\medskip

	For the next results, given $P,Q\in \rproj{A}$ and $f\in\Hom_A(P,Q)$,
	recall that the dual homomorphism $f^*\in\Hom_A(Q^*,P^*)$ is defined   by $f^*\phi=\phi \circ f$
	($\phi\in Q^*$).

	\begin{lem}\label{LM:Lag-transitive-action}
		Suppose that $(A,\sigma)$
		is an Azumaya $R$-algebra with  involution
		and $R$ is semilocal.
		Let $(P,f)\in \Herm[\veps]{A,\sigma}$.
		Then $U(f)$ acts transitively on $\Lag(f)$, provided it is nonempty.
	\end{lem}
	
	\begin{proof}
		Let $L_1,L_2\in \Lag(f)$.
		As explained in \ref{subsec:Witt-grp},
		we can find isometries
		$\vphi_i:\Hyp[\veps]{L_i}\to f$ ($i=1,2 $)
		such that $\vphi_i$ restricts to the identity
		on $L_i$.
		Since $\rrk_A L_1=\frac{1}{2}\rrk_AP=\rrk_A L_2$,
		there is an $A$-module isomorphism $\psi:L_1\cong L_2$
		(Lemma~\ref{LM:rank-determines}).
		Then
		$\hat{\psi}:=\psi\oplus (\psi^*)^{-1}: \Hyp[\veps]{L_1}\to \Hyp[\veps]{L_2}$
		is an isometry
		taking $L_1$ to $L_2$.
		Now, $\vphi_2\hat{\psi}\vphi_1^{-1}$ is an element of $U(f)$
		taking $L_1$ to $L_2$.
	\end{proof}
	
	\begin{prp}\label{PR:transitive-action-of-uUf}
		Suppose that $(A,\sigma)$ 
		is an Azumaya $R$-algebra with involution,
		and let $(P,f)
		\in \Herm[\veps]{A,\sigma}$.
		When viewed as sheaves on $(\Aff/R)_{\Zar}$ --- the category
		of affine $R$-schemes with the Zariski topology --- the group $\uU(f)$
		acts transitively on $\uLag(f)$, provided $\uLag(f)\neq \boldsymbol{\emptyset}$.
	\end{prp}
	
	\begin{proof}
		The statement  means that
		for every $R$-ring $S$
		and $L,M\in  \Lag(f_{S})$, there
		exist $\alpha_1,\dots,\alpha_t\in S$
		and $\vphi_i\in U(f_{S_{\alpha_i}})$ ($i=1,\dots,t$)
		such that $S=\sum_i\alpha_iS$ and $\vphi_i (L_{S_{\alpha_i}})=M_{S_{\alpha_i}}$
		for all $i$. Here, $S_{\alpha_i}$ denotes
		the localization of $S$ with respect to $\{1,\alpha_i,\alpha_i^2,\dots\}$.
		
		Fix some $\frakp\in\Spec S$. By Lemma~\ref{LM:Lag-transitive-action},
		there exists an isometry $\psi\in U(f_{S_\frakp})$
		with $\psi (L_{S_\frakp})=M_{S_\frakp}$.
		It is easy to see that there exists $\alpha=\alpha^{(\frakp)}\in S\setminus \frakp$
		and $\vphi=\vphi^{(\frakp)}\in U(f_{S_\alpha})$ such
		that $\psi=\vphi_{S_\frakp}$ and $\vphi(L_{S_\alpha})=M_{S_\alpha}$.
		Now, since  $\sum_\frakp\alpha_\frakp S=S$,
		there exist $\frakp_1,\dots,\frakp_t\in S$
		such that $\sum_{i=1}^t\alpha^{(\frakp_i)}S=S$.
		The elements $\alpha_i:=\alpha^{(\frakp_i)}$
		and the isometries $\vphi_i=\vphi^{(\frakp_i)}$
		fulfill all the requirements.
	\end{proof}

	Given $P\in\rproj{A}$ and $b\in\Hom_A(P,P^*)$,
	write $b^\trans$ for the element of $\Hom_A(P,P^*)$
	determined by $(b^\trans x)y=((by)x)^\sigma$ ($x,y\in P$).
	It is straightforward to check that $b^{\trans\trans}=b$
	and $(b\circ \psi)^\trans=\psi^*\circ b^\trans$ for all $\psi\in\End_A(P)$.
	We set $\Sym_\veps(P)=\{b\in\Hom_A(P,P^*)\suchthat b=\veps b^\trans\}$.

	\begin{lem}\label{LM:Borel-subgroup-structure}
		Let $L\in \rproj{A}$ and 
		let $B$ denote
		the subgroup of $U(\Hyp[\veps]{L})$ consisting of
		isometries $\vphi$ satisfying $\vphi (0\oplus L^*)= 0\oplus L^* $. Then,
		writing elements of $\End_A({L\oplus L^*})$ as $2\times 2$ matrices, we have
		\[
		B=\left\{
		\left[
		\begin{matrix}
		a & 0 \\
		b & (a^*)^{-1}
		\end{matrix}
		\right]
		\suchthat 
		a\in \Aut_A(L),\,
		b\in  \Hom_A(L, L^*),
		a^*\circ b\in\Sym_{-\veps}(L) 
		\right\}.
		\]
	\end{lem}
	
	\begin{proof}
		That elements of $B$ live in $U(\Hyp[\veps]{L})$ and preserve $0\oplus L^*$ is routine.
		Conversely, every element
		$\vphi\in U(f)$ satisfying  $\vphi (0\oplus L^*)= 0\oplus L^* $
		can be written as $[\begin{smallmatrix} a & 0 \\ b & c\end{smallmatrix}]$
		with $a\in\Aut_A(L)$, $b\in\Hom_A(L,L^*)$, $c\in \Aut_A(L^*)$.
		Let $x,x'\in L$ and $\phi\in L^*$.
		Unfolding the equality
		$\Hyp[\veps]{L}([\begin{smallmatrix} 0 \\ \phi \end{smallmatrix}],
		[\begin{smallmatrix} x' \\ 0 \end{smallmatrix}])=
		\Hyp[\veps]{L}(\vphi[\begin{smallmatrix} 0 \\ \phi \end{smallmatrix}],
		\vphi[\begin{smallmatrix} x' \\ 0 \end{smallmatrix}])$
		gives $\phi x'=(c\phi)(ax')=(a^*(c\phi))x'$, so $a^*c=\id_{L^*}$,
		or rather, $c=(a^*)^{-1}$.
		Unfolding
		$\Hyp[\veps]{L}([\begin{smallmatrix} x \\ 0 \end{smallmatrix}],
		[\begin{smallmatrix} x' \\ 0 \end{smallmatrix}])=
		\Hyp[\veps]{L}(\vphi[\begin{smallmatrix} x \\ 0 \end{smallmatrix}],
		\vphi[\begin{smallmatrix} x' \\ 0 \end{smallmatrix}])$
		gives
		$0=(bx')(ax)+\veps((bx)(ax'))^\sigma=
		(a^*(bx'))x+\veps (b^\trans(ax'))x$, so $a^*\circ b+\veps b^\trans \circ a=0$,
		which means that $a^*\circ b\in \Sym_{-\veps}(L)$.
	\end{proof}

	The following proposition provides
	information about the $\uU^0(f)$-orbits in $\uLag(f)$ when $(A,\sigma)$
	is an Azumaya $R$-algebra with  involution and $(\sigma,\veps)$
	is orthogonal. It will   feature a number of times in the sequel.

	\begin{prp}\label{PR:partition-of-Lag}
		Suppose that $(A,\sigma)$ 
		is an Azumaya $R$-algebra with   involution
		and  $(\sigma,\veps)$ is orthogonal. 
		Let $(P,f)\in\Herm[\veps]{A,\sigma}$, let $L\in\Lag(f)$ and suppose that $\rrk_AP>0$.
		Then there exists a unique $\uU(f)$-equivariant
		natural transformation of functors from $R$-rings to sets,
		\[
		\Phi_L=\Phi_L^{(f)}:\uLag(f)\to\umu_{2,R} ,
		\]
		such that $\Phi_L(L)=1$;
		here, $\uU(f)$ acts on $\umu_{2,R} $ via $\Nrd:\uU(f)\to\umu_{2,R} $.
		The map $\Phi_L$ has the following additional
		properties:
		\begin{enumerate}[label=(\roman*)]
			\item $\Phi_L(M)\Phi_M(K)=\Phi_L(K)$
			and $\Phi_L(M)=\Phi_M(L)$ for all $L,M,K\in \Lag(f)$.
			\item Given $(P',f')\in \Herm[\veps]{A,\sigma}$ 
			and $L'\in \Lag(f')$,
			we have $\Phi_{L\oplus L'}(M\oplus M')=\Phi_L(M)  \Phi_{L'}(M')$
			for all $M\in \Lag(f)$, $M'\in \Lag(f')$.
		\end{enumerate}
	\end{prp}
	
	\begin{proof}
		We may assume without loss of generality
		that $(P,f)=(L \oplus L^*,\Hyp[\veps]{L})$ and identify
		$L$ with its copy in $P=L\oplus L^*$. A sheaf
		means a sheaf on the site $(\Aff/R)_{\fppf}$.
		
		Given an $R$-ring $S$, let $B=B(L)$
		be as in Lemma~\ref{LM:Borel-subgroup-structure}.
		Let $\bfB$ denote the subfunctor of $\uU(f)$
		determined by $\bfB(S)=B(L_S)$.
		It is routine to check that $\bfB$ is a group subsheaf of $\uU(f)$.
		We let $\uU(f)/\bfB$ denote the quotient sheaf (note that
		$(\uU(f)/\bfB)(S)$ is in general larger than $U(f_S)/B(L_S)$).
		By definition, $\bfB$ is  the stabilizer of the global section
		$0\oplus L^*$ of $\uLag(f)$ under the action of $\uU(f)$.
		Thus, we have an induced morphism $\Psi:\uU(f)/\bfB\to \uLag(f)$,
		which is an isomorphism by Proposition~\ref{PR:transitive-action-of-uUf}.

		For every $a\in \End_A(L)$, we have $\Nrd(a^*)=\Nrd(a)$.
		Indeed, $a\mapsto a^*:\End_A(L)\to \End_A(L^*)^\op$ is an isomorphism of
		Azumaya $R$-algebras, and thus respects the reduced norm.
		This implies readily that $\bfB\subseteq \ker(\Nrd:\uU(f)\to \umu_{2,R})$.
		As a result, there is an induced $\uU(f)$-equivariant map $\quo{\Nrd}:\uU(f)/\bfB\to \umu_{2,R}$.
		Let $\Phi_0$ denote the composition $\quo{\Nrd}\circ \Psi^{-1} :\uLag(f) \to \umu_{2,R}$.
		Then $\Phi_0$ is $\uU(f)$-equivariant.
		Writing $\xi=\Phi_0(L)\in\mu_2(R)$ and defining $\Phi_L=\xi \cdot \Phi_0$,
		we see that $\Phi_L$ is
		$\uU(f)$-equivariant and satisfies $\Phi_L(L)=1$.

		Suppose that
		$\Phi':\uLag(f)\to \umu_{2,R}$ is another $\uU(f)$-equivariant natural
		transformation satisfying $\Phi'(L)=1$. Let $R'$ be
		an $R$-ring and let $M\in \Lag(f_{R'})$. By Proposition~\ref{PR:transitive-action-of-uUf},
		there exists a faithfully flat   $R'$-algebra $R''$
		and $\vphi\in U(f_{R''})$ 
		such that $\vphi (L_{R''})=M\otimes_{R'}R''$.
		Thus, $\Phi'(M\otimes_RR'')=\Nrd(\vphi)\Phi'(L)=\Nrd(\vphi)\Phi_L(L)=\Phi_L(M\otimes_RR'')$
		in $\mu_2(R'')$. Since $R'\to R''$ is faithfully flat, this means
		that $\Phi'(M)=\Phi_L(M)$ in $\mu_2(R')$, and we have shown that $\Phi'=\Phi_L$.
		
\medskip
		
		We turn to prove   (i) and (ii):
		
		(i)
		We apply Proposition~\ref{PR:transitive-action-of-uUf}  
		to assert the existence of a faithfully flat   $R$-algebra $R'$
		and $\vphi,\psi\in U(f_{R'})$
		such that $\vphi(L_{R'})=M_{R'}$ and $\psi(M_{R'})=K_{R'} $.
		Note that $\Phi_L(M)=\Nrd(\vphi)\Phi_L(L)=\Nrd(\vphi)$
		in $\mu_2(R')$,  and similarly, $\Phi_M(K)=\Nrd(\psi)$,
		$\Phi_L(K)=\Nrd(\psi\vphi)$ and $\Phi_M(L)=\Nrd(\vphi)^{-1}$.
		The identities
		in (i)   follow readily from these equalities
		and the fact that $\mu_2(R)$ is $2$-torsion.
		
		(ii) By Proposition~\ref{PR:transitive-action-of-uUf},
		there exists a faithfully
		flat  $R$-algebra $S$ and $\vphi\in U(f_{S})$, $\vphi'\in U(f'_S)$
		such that $\vphi L=M$ and $\vphi'L'=M'$.
		Then $\Phi_{L\oplus L'}(M\oplus M')=\Phi_{L\oplus L'}((\vphi\oplus \vphi')(L\oplus L'))=
		\Nrd(\vphi\oplus \vphi')=\Nrd(\vphi)\Nrd(\vphi')=\Phi_L(M)\cdot \Phi_{L'}(M')$.
	\end{proof}

\subsection{Conjugation and  Transfer}
\label{subsec:conjugation}

	We now recall two special instances of hermitian Morita equivalence
	that will be used repeatedly in the sequel. We address them simply 
	as ``$\mu$-conjugation'' and ``$e$-transfer''.

\medskip

	Recall that $\veps\in \Cent(A)$ satisfies
	$\veps^\sigma\veps=1$. Let $\delta\in \Cent(A)$ be another element satisfying
	$\delta^\sigma\delta=1$ and let $\mu\in\Sym_\delta(A,\sigma)\cap \units{A}$.
	One readily checks that $\Int(\mu)\circ \sigma$ is also an $R$-involution
	and $(\delta\veps)^{\Int(\mu)\circ \sigma}(\delta\veps)=1$.
	Given $(P,f)\in\Herm[\veps]{A,\sigma}$,
	define $\mu f:P\times P\to A$ by
	$(\mu f)(x,y)=\mu\cdot f(x,y)$.
	Then $\mu f$ is an $\veps\delta$-hermitian form over $(A,\Int(\mu)\circ \sigma)$
	and   
	\[(P,f)\mapsto (P,\mu f):\Herm[\veps]{A,\sigma}\to \Herm[\delta\veps]{A,\Int(\mu)\circ\sigma}\]
	is an equivalence of categories; morphisms are mapped to themselves.
	We call this equivalence \emph{$\mu$-conjugation}.  It has the following properties:
	\begin{enumerate}[label=(c\arabic*)]
		\item \label{item:conj-first} For every $R$-ring $S$, we have $\mu(f_S)=(\mu f)_S$.
		\item $U(f)=U(\mu f)$.
		If  $A$ is finite projective over $R$, then   $\uU(f)=\uU(\mu f)$, $U^0(f)=U^0(\mu f)$
		and $\uU^0(f)=\uU^0(\mu f)$.
		\item The forms $f$ and $\mu f$ have the same Lagrangians. In particular,
		$f$ is hyperbolic if and only if $\mu f$ is hyperbolic.
	\end{enumerate}
	
	Suppose further that $(A,\sigma)$ is an Azumaya $R$-algebra
	with involution. Then, by Corollary~\ref{CR:type-conjugation}(i), $(A,\Int(\mu)\circ \sigma)$ 
	is also an Azumaya $R$-algebra with involution, and the
	types of $(\sigma,\veps)$ and $(\Int(\mu)\circ \sigma,\delta\veps)$
	are equal. When $(\sigma,\veps)$ is orthogonal, we have:
	\begin{enumerate}[resume, label=(c\arabic*)]
		\item $\Lag(f)=\Lag(\mu f)$ and $\uLag(f)=\uLag(\mu f)$.
		\item \label{item:conj-and-Phi} 
		\label{item:conj-last} For every $L\in \Lag(f)=\Lag(\mu f)$,
		the maps $\Phi_L^{(f)}:\uLag(f)\to \umu_{2,R}$ 
		and $\Phi_L^{(\mu f)}:\uLag(\mu f)\to \umu_{2,R}$
		of Proposition~\ref{PR:partition-of-Lag} coincide.
	\end{enumerate}
	(Item \ref{item:conj-and-Phi} follows from the uniqueness part in Proposition~\ref{PR:partition-of-Lag}.)
		
	Items \ref{item:conj-first}--\ref{item:conj-last} 
	allow us to rephrase certain claims about $\veps$-hermitian forms
	over $(A,\sigma)$ as claims about $\delta\veps$-hermitian forms over $(A,\Int(\mu)\circ \sigma)$.
	We shall address this process as  $\mu$-conjugation  in the sequel.

\medskip

	Next, let $e\in A$ be an idempotent such that $e^\sigma=e$
	and $eA_A$ is a progenerator, or equivalently,
	$AeA=A$. 
	When $A$ is Azumaya over its center,
	this is also equivalent to having $\rrk_AeA>0$ (Proposition~\ref{PR:progenerator-iff-pos-red-rank}).
	By Morita theory, 
	the functor $\rproj{A}\to \rproj{eAe}$ sending a module
	$P$ to $Pe$ and a morphism $\vphi:P\to Q$ to $\vphi_e:=\vphi|_{Pe}$ is an equivalence;
	see \cite[Example~18.30]{Lam_1999_lectures_on_modules_rings}.
	
	Write $\sigma_e:=\sigma|_{eAe}$ and note that $(e\veps )^{\sigma_e}(e\veps )=1$.
	Given $(P,f)\in \Herm[\veps]{A,\sigma}$, let $f_e=f|_{Pe\times Pe}$.
	It is well-known, see \cite[Proposition~2.5, Remark~2.1]{First_2015_Witts_extension_theorem}  for instance, that 
	\[(P,f)\mapsto (Pe,f_e)\in \Herm[\veps]{A,\sigma}\to \Herm[e\veps ]{eAe,\sigma_e}\] 
	defines an equivalence of categories; isometries
	$\vphi$ are mapped to $\vphi_e$. We call this equivalence
	\emph{$e$-transfer}. It has the following additional properties:
	\begin{enumerate}[label=(t\arabic*)]
		\item \label{item:e-transfer-first} For every $R$-ring $S$, there is a natural isomorphism $(f_S)_e\cong (f_e)_S$.
		\item The map $\vphi\mapsto \vphi_e$
		defines an isomorphism  $U(f)\to U(f_e)$.
		If $A$ is finite projective over $R$, then it 
		also defines isomorphisms $\uU(f)\to\uU(f_e)$, $U^0(f)\to U^0(f_e)$
		and $\uU^0(f)\to \uU^0(f_e)$.
		\item \label{item:e-transfer-hyperbolic}
		The map $L\mapsto Le$ defines a bijection between
		the Lagrangians of $f$ and the Lagrangians of $f_e$. In particular,
		$f$ is hyperbolic if and only if $f_e$ is hyperbolic.
	\end{enumerate}
	
	Suppose further that $(A,\sigma)$ is an Azumaya $R$-algebra
	with involution. By Corollary~\ref{CR:type-conjugation}(ii), $(eAe,\sigma_e)$
	is also an Azumaya $R$-algebra with involution and the types
	of $(\sigma,\veps)$
	and $(\sigma_e,e\veps)$ are the same. When $(\sigma,\veps)$ is orthogonal,
	we   have:
	\begin{enumerate}[resume, label=(t\arabic*)]
		\item \label{item:e-transfer-and-Nrd} The isomorphism $\vphi\mapsto \vphi_e:\uU(f)\to \uU(f_e)$
		respects the reduced norm. 
		\item \label{item:e-transfer-Lags}
		The map $L\mapsto Le$ defines isomorphisms $\Lag(f)\to \Lag(f_e)$ and $\uLag(f)\to \uLag(f_e)$;
		its inverse is $L'\mapsto L'A$.
		\item \label{item:e-transfer-and-Phi} 
		\label{item:e-transfer-last}
		The composition $\uLag(f)\xrightarrow{{\sim}} \uLag(f_e)\xrightarrow{\Phi_{Le}}\umu_{2,R}$
		coincides
		with $\Phi_L$ (see Proposition~\ref{PR:partition-of-Lag}).
	\end{enumerate}
	(Item \ref{item:e-transfer-and-Nrd} follows from the fact
	that $\vphi\mapsto \vphi_e:\End_A(P)\to\End_{eAe}(Pe)$ is an isomorphism
	of Azumaya algebras and so preserves the reduced norm.
	Item \ref{item:e-transfer-Lags} follows from
	\ref{item:e-transfer-hyperbolic} Corollary~\ref{CR:degree-of-endo-ring}.
	Item \ref{item:e-transfer-and-Phi} follows
	from \ref{item:e-transfer-and-Nrd} and the uniqueness
	part of Proposition~\ref{PR:partition-of-Lag}.)
	Note also that $e$-transfer preserves reduced rank by Corollary~\ref{CR:degree-of-endo-ring}.
	
	Items \ref{item:e-transfer-first}--\ref{item:e-transfer-last} allow us to rephrase 
	certain claims about $\veps$-hermitian forms
	over $(A,\sigma)$ as claims about $e\veps$-hermitian forms over $(eAe,\sigma_e)$.
	We shall address this process as  $e$-transfer in the sequel.

\medskip
	
	As a first example of using conjugation and  transfer, we
	prove the following result, which provides an alternative  way  to
	evaluate $\Phi_L$.

	\begin{prp}\label{PR:computation-of-Phi}
		With the notation of Proposition~\ref{PR:partition-of-Lag},
		let $L,M\in \Lag (f)$.
		For every $\frakp\in \Spec R$,
		let $I_\frakp$ denote the intersection
		of $L(\frakp)$ and $M(\frakp)$ in $P(\frakp)$.
		Then
		\[
		\Phi_L(M)(\frakp)=(-1)^{\rrk_{A(\frakp)}L(\frakp)-\rrk_{A(\frakp)}I_\frakp} 
		\] 
		in $\mu_2(k(\frakp))$. In particular,
		if $P=L\oplus M$, then $\Phi_L(M)=(-1)^{\rrk_AL}$ in $\mu_2(R)$.
	\end{prp}

	We alert the reader that $I_\frakp$
	is in general not the image of $L\cap M$ in $P(\frakp)$.	
	
	\begin{proof}
		It is enough to prove the proposition
		when $R$ is a field  and $\frakp=0$ (the last assertion
		will follow  by virtue of Lemma~\ref{LM:mu-two-check}).
		Note further that base-changing from the field
		$R$ to an algebraic closure does not affect
		the $R$-dimension of $A$, $L$, $M$ and $I_0=L\cap M$,
		and thus $\rrk_A M$, $\rrk_AL$ and $\rrk_A I_0$
		remain unchanged.
		This allows us to further restrict to the case where $R$ is an algebraically
		closed field. In particular, $[A]=0$ in $\Br R$.

		If $\sigma$ is symplectic, then $\veps=-1$ (because $(\sigma,\veps)$ is orthogonal)
		and $\deg A$
		is even. By 
		Lemma~\ref{LM:invertible-symmetric-elements},	
		there exists $\mu\in \Sym_{-1}(A,\sigma)\cap \units{A}$.
		Applying $\mu$-conjugation,
		we may replace $\sigma$, $\veps $, $f$ 
		with $\Int(\mu)\circ \sigma$, $-\veps $, $\mu f$ and
		assume that $\sigma$ is orthogonal and $\veps=1$.

		Now, by Proposition~\ref{PR:invariant-primitive-idempotent},
		there exists an idempotent $e\in A$ with $e^\sigma =e$ and $\deg eAe=1$.
		Applying $e$-transfer, we may 
		replace $A$, $\sigma$, $P$, $f$, $L$, $M$
		with $eAe$, $\sigma_e$, $Pe$, $f_e$, $Le$, $Me$ and
		assume that $A=R$ and $\sigma=\id_R$ henceforth.
		
		Write $I=I_0=L\cap M$ and fix $R$-subspaces
		$W\subseteq L$, $W'\subseteq M$
		such that $L=I\oplus W$ and $M=I\oplus W'$.
		Let $N=(W\oplus W')^\perp$ 
		and fix a basis $\{x_1,\dots,x_n\}$ to $W$.
		The kernel of $y\mapsto f(y,-):M\to L^*$ is $M\cap L^\perp=M\cap L=I$,
		so $y\mapsto f(y,-):W'\to L^*$ is injective. Since any element 
		in the image of this map vanishes on $I$,
		it follows that $y\mapsto f(y,-):W'\to W^*$ is also  injective,
		and thus bijective by conisdering $R$-dimensions.
		This means that there exists a
		basis $\{y_1,\dots,y_n\}$ to $W'$
		satisfying $f(x_i,y_j)=\delta_{ij}$.
		Consequently, $f|_{W\oplus W'}$ is unimodular, and thus  $P=N\oplus W\oplus W'$.
		Let  $\vphi\in \End_R(P)$ denote the endomorphism
		exchanging $x_i$ and $y_i$ and fixing $N$.
		Then $\vphi\in U(f)$ 
		and $\Nrd(\vphi)=(-1)^{\dim_R W }=(-1)^{\rrk_A L-\rrk_A I_0}$.
		Since   $\vphi L=\vphi(W+I)=W'+I=M$,
		we have $\Phi_L(M)=\Nrd(\vphi)$, so
		we are done.
	\end{proof}

\subsection{The Discriminant}
\label{subsec:disc}

	Classically, the discriminant of a nondegenerate symmetric bilinear space $(V,b)$ over a field
	$F$  is
	the coset in $\units{F}/(\units{F})^2$ represented
	by
	$(-1)^{\frac{1}{2}\dim V(\dim V-1)}$ times the determinant of some Gram matrix
	of $b$,
	see \cite[p.~80]{Knus_1998_book_of_involutions}. 
	If $F$ carries a nontrivial involution 
	$\sigma:F\to F$ with a fixed subfield $F_0$, then the discriminant 
	of a unimodular $1$-hermitian space 
	$(V,h)$ over $(F,\sigma)$ is defined similarly, but this time it 
	is regarded as an element of $\units{F_0}/\Nr_{F/F_0}(\units{F})$
	\cite[p.~114]{Knus_1998_book_of_involutions}.
	These definitions do not generalize naively to hermitian forms over $R$-algebras with involution $(A,\sigma)$ because
	projective $A$-modules need not be free. 
	However, in \cite[\S7, \S8, \S10]{Knus_1998_book_of_involutions}, a discriminant invariant  
	was defined for $1$-hermitian forms
	over central simple algebras with an orthogonal or unitary involution. 
	It agrees with the classical discriminant and is compatible
	with extending the base field.
	Moreover, it is invariant
	under conjugation and $e$-transfer (see \ref{subsec:conjugation}),
	because it is defined as an invariant of the adjoint involution of the hermitian space,
	which is unaffected by these operations.
	
	Suppose henceforth that $(A,\sigma)$ is an Azumaya $R$-algebra with involution.
	We will need a generalization of the discriminant defined in 
	\cite[\S7, \S8, \S10]{Knus_1998_book_of_involutions} to $\veps$-hermitian forms over $(A,\sigma)$
	when $(\sigma,\veps)$ is orthogonal or unitary.
	Unfortunately, such a definition seems missing in the literature, and
	introducing one is out of the scope of this work.
	We therefore give an ad hoc generalization of the definition 
	in 
	{\it op.\ cit.}
	to some specific $R,A,\sigma$
	that will be needed in this work, and prove that
	it has desired properties such as
	being invariant under    conjugation and $e$-transfer.
	Specifically, we shall restrict to  rings $R$ which are connected semilocal
	and consider only the cases where (1) $(\sigma,\veps)$ is orthogonal, or
	(2) $\sigma$ is unitary and $[A]=0$ in $\Br \Cent(A)$.

\medskip

	Suppose first that $(\sigma,\veps)$ is orthogonal and $R$ is connected   semilocal.
	Let $(P,f)\in \Herm[\veps]{A,\sigma}$ be a hermitian space such that $n:=\rrk_AP$ is even and positive.
	Write $E=\End_A(P)$ and let $\theta$ denote the adjoint involution
	of $f$. By Lemma~\ref{LM:invertible-symmetric-elements},
	there exists $\vphi\in \Sym_{-1}(E,\theta)\cap \units{E}$. Following \cite[\S7]{Knus_1998_book_of_involutions},
	we define the discriminant of $f$ to be
	\[
	\disc(f)=(-1)^{n/2}\Nrd(\vphi)\cdot (\units{R})^2\in \units{R}/(\units{R})^2.
	\]
	This is well-defined by the following proposition. 
	The discriminant
	of the zero form  is defined to be the trivial class $(\units{R})^2 $.

	\begin{prp}\label{PR:disc-orth-basic-props}
		Under the previous assumptions:
		\begin{enumerate}[label=(\roman*)]
			\item $\disc(f)$ is well-defined, i.e., it is independent of the choice of $\vphi$.
			\item Isomorphic forms have equal discriminants.
			The discriminant is unchanged under $\mu$-conjugation
			and $e$-transfer (see~\ref{subsec:conjugation}).
			\item If $(P',f')\in \Herm[\veps]{A,\sigma}$
			and $\rrk_AP'$ is even, then $\disc (f\oplus f')=\disc(f)\disc(f')$.
			\item If $(A,\sigma)=(R,\id_R)$, $\veps=1$
			and $f=\langle \alpha_1,\dots,\alpha_{2n}\rangle_{(R,\id_R)}$,
			then $\disc(f)\equiv (-1)^n\prod_i\alpha_i \bmod (\units{R})^2$.
			\item If $d:=\deg A$ is even,  $u\in\Sym_{-\veps}(A,\sigma)\cap \units{A}$
			and $a_1,\dots,a_n\in \Sym_{\veps}(A,\sigma)$,
			then $\disc(\langle a_1,\dots,a_n\rangle_{(A,\sigma)})\equiv (-1)^{nd/2}\Nrd(u)^n\prod_{i=1}^n\Nrd(a_i) 
			\bmod (\units{R})^2$.
		\end{enumerate}
	\end{prp}
	
	\begin{proof}
		(i) See \cite[Proposition~7.1]{Knus_1998_book_of_involutions}
		for the case $R$ is a field.
		The same proof works when $R$ is general; for the definition
		of the \emph{Pfaffian} over general rings and a  proof that its square
		is the reduced characteristic polynomial, see
		\cite[p.~3]{Knus_1988_Pfaffians_quad_forms}.

		(ii) The definition of $\disc(f)$ depends only on
		the isomorphism class of $(E,\theta)$ and this remains
		unchanged if we replace $f$ with an isomorphic form
		or perform $\mu$-conjugation or $e$-transfer.

		(iii) Write $n'=\rrk_AP'$, let $\theta'$ be the adjoint involution of
		$f'$ and let $\vphi'\in \Sym_{-1}(\End_A(P'),\theta')$.
		One readily  checks
		that 
		\[(f\oplus f')((\vphi\oplus \vphi')(x\oplus x'),y\oplus y')=
		-(f\oplus f')(x\oplus x',(\vphi\oplus \vphi')(y\oplus y'))\]
		for all $x,y\in P$ and $x',y'\in P'$.
		Thus, the adjoint involution of $f\oplus f'$
		takes $\vphi\oplus \vphi'$ to $-(\vphi\oplus \vphi')$,
		and,
		by definition,
		$\disc(f\oplus f')\equiv (-1)^{(n+n')/2}\Nrd(\vphi\oplus \vphi')\equiv 
		(-1)^{n/2}\Nrd(\vphi)(-1)^{n'/2}\Nrd(\vphi')\equiv
		\disc(f)\disc(f')$ modulo $(\units{R})^2$.
		
		(iv) The proof of \cite[Proposition~7.3(3)]{Knus_1998_book_of_involutions}
		applies verbatim.
		
		(v) By (iii), it is enough to prove the case $n=1$.
		Writing $a=a_1$,
		and identifying $\End(A_A)$ with $A$ via
		$\vphi\mapsto \vphi(1_A)$,
		the adjoint involution of $\langle a\rangle$ is  
		$\theta:=\Int(a^{-1})\circ \sigma$.
		Thus, $u a\in \Sym_{-1}(A,\theta)\cap\units{A}$
		and $\disc \langle a\rangle\equiv (-1)^{d/2}\Nrd(u a)$ modulo
		$(\units{R})^2$.
	\end{proof}

	Given a quadratic \'etale  $R$-algebra $S$ with standard involution $\theta$
	(see~\ref{subsec:quad-etale}), we define
	the norm form $n_{S/R} :S\times S\to R$  
	by $n_{S/R}(x,y)= \frac{1}{2}(x^\theta y+y^\theta x)$; it is a $1$-hermitian
	form over $(R,\id_R)$.
	When $R$ is semilocal,  there is $\lambda\in S$ such
	that $\{1,\lambda\}$ is an $R$-basis of $S$,  
	$\lambda^2\in\units{R}$ and $\lambda^\theta=-\lambda$  (Lemma~\ref{LM:quad-etale-over-semilocal}).
	Using this basis to identify $S$ with $R^2$, one finds that 
	$n_{S/R}\cong \trings{1,-\lambda^{2}}_{(R,\id)}$.
	In this case, we define
	\[
	\disc(S/R):=\disc(n_{S/R})=\lambda^2(\units{R})^2.
	\]

	Keeping our assumption that $R$ is connected semilocal, we now
	proceed with defining a discriminant for $\veps$-hermitian spaces
	over $(A,\sigma)$ when  $\sigma$ is unitary
	and $[A]=0$ in $\Br \Cent(A)$.
	Note that the reduced rank
	of any $(P,f)\in \Herm[\veps]{A,\sigma}$ 
	is constant by Corollary~\ref{CR:constant-even-ranks}(i).	
	Write $S=\Cent(A)$ and let $\Nr_{S/R}:S\to R$
	denote the norm map; it is given by $\Nr_{S/R}(x)=x^\sigma x$ because
	$\sigma|_S$ is the standard $R$-involution of $S$.
	
	Suppose first that $\deg A=1$
	and let $(P,f)\in\Herm[\veps]{A,\sigma}$. Then $A=S$ and $P$ is free. 
	Let $\{x_i\}_{i=1}^n$ be an $S$-basis of $P$ and let $g=(f(x_i,x_j))_{i,j}$
	denote the corresponding Gram matrix.
	Since $g$ is $(\sigma,\veps)$-hermitian,
	$\det g=\veps^n (\det  g)^\sigma$.
	When $n=\rrk_AP$ is even, this means that $(-\veps)^{-n/2}\det g= ((-\veps)^{-n/2}\det  g)^\sigma$,
	so $(-\veps)^{-n/2}\det g\in\units{R}$.
	In this case, the discriminant
	of $f$ is defined to be   
	\[
	\disc(f)=(- \veps)^{-n/2}\det g\cdot \Nr_{S/R}(\units{S})\in \units{R}/\Nr_{S/R}(\units{S}).
	\]
	It is easy to see that this is independent of the basis $\{x_i\}_{i=1}^n$. 
	Moreover, isomorphic forms have the same discriminant. 

	We extend this to any $A$ with $[A]=0$
	as follows:
	Use Theorem~\ref{TH:invariant-primitive-idempotent}
	to choose an idempotent $e\in A$ with $e^\sigma=e$
	and $\rrk_A eA=1$. Noting that $eAe\cong S$,
	we define 
	\[
	\disc(f):=\disc(f_e)\in \units{R}/\Nr_{S/R}(\units{S})
	\]
	for every $(P,f) \in \Herm[\veps]{A,\sigma}$ with $\rrk_AP$ even.
	Here, $f_e:Pe\times Pe\to eAe$ is the $e$-transfer of $f$, see \ref{subsec:conjugation}.
	This is well-defined by the following proposition.
	
	\begin{prp}\label{PR:disc-unitary-basic-props}
		Under the previous assumptions:
		\begin{enumerate}[label=(\roman*)]
			\item $\disc(f)$ is well-defined, i.e., it is independent of the choice of $e$.
			\item Isomorphic forms have equal discriminants.
			The discriminant is unchanged under $\mu$-conjugation
			and $e'$-transfer (see~\ref{subsec:conjugation}).
			\item If $(P',f')\in \Herm[\veps]{A,\sigma}$
			and $\rrk_AP'$ is even, then $\disc (f\oplus f')=\disc(f)\disc(f')$.
		\end{enumerate}
	\end{prp}
	
	\begin{proof}
		(i) Let $e'\in A$ be another idempotent with $e'^\sigma=e'$
		and $\rrk_A e'A=1$. Then $eA\cong e'A$ (Lemma~\ref{LM:rank-determines}).
		Every $A$-module homomorphism $eA\to e'A$ is given
		by multiplication on the left with a unique element
		in $e'Ae$, so there exist  $u\in e'Ae$ and $v\in eAe'$
		such that $uv=e'$ and $vu=e$.
		We also see that $v^\sigma v$ is invertible in $e'Ae'=\End_A(e'A)$.
		Since $\deg e'Ae'=1$, we have $e'Ae'=Se'$.
		Write $v^\sigma v=\alpha e'$ with $\alpha\in \units{S}$.
		Then $\alpha\in \units{R}$  because $v^\sigma v$ is fixed under $\sigma$.
		Identifying $eAe$ and $e'Ae'$ with $S=\Cent(A)$,
		it is   routine to check that
		$x\mapsto xv:Pe\to Pe'$ defines an isometry
		from $\alpha f_e$ to $ f_{e'}$ (its inverse is $y\mapsto yu:Pe'\to Pe$).
		Since $\rrk_{eAe} Pe=\rrk_A P$ is even,
		$\disc(\alpha f_{e'})= \disc( f_{e'})$ and
		it follows that $\disc(f_e)=\disc( f_{e'})$.
		
		(ii) 
		If $(P,f)\cong (P',f')$, then $f_e\cong f'_e$,
		so $\disc(f)=\disc(f_e)=\disc(f'_e)=\disc(f')$.
		
		Let $\mu\in\Sym_{\delta}(A,\sigma)$, where $\delta\in\Cent(A)$
		satisfies $\delta^\sigma\delta=1$,
		and write $\tau=\Int(\mu)\circ \sigma$.
		Then $\tau$ is also unitary, and so there exists an idempotent
		$e'\in A$ with $\rrk_A e'A=1$ and $e'^\tau=e'$.
		As in the proof of (i), choose $u\in e'Ae$, $v\in eAe'$
		such that $uv=e'$ and $vu=e$.
		We have $\mu v^\sigma v,u u^\sigma \mu^{-1}\in e'Ae'$
		because  $e'\mu v^\sigma v=\mu \mu^{-1} e' \mu v^\sigma v=
		\mu e'^\sigma v^\sigma v=\mu (ve')^\sigma v=\mu v^\sigma v$ 
		and similarly
		$u u^\sigma \mu^{-1}e'=u u^\sigma \mu^{-1}$.
		Furthermore, $\mu v^\sigma v \cdot u u^\sigma \mu^{-1}=
		\mu v^\sigma e u^\sigma \mu^{-1}=
		\mu (uev)^\sigma \mu^{-1}=\mu e'^\sigma \mu^{-1}=e'^\tau=e'$,
		hence $\mu v^\sigma v\in \units{(e'Ae')}=e'\units{S}$.
		Write $\mu v^\sigma v=\alpha e'$ with $\alpha\in\units{S}$.
		As in the proof of (i), identifying $eAe$ and $e'Ae'$ with $S $, 
		we see that $x\mapsto xv:Pe\to Pe'$ is
		an isometry from $\alpha f_e$ to $(\mu f)_{e'}$.
		Thus, 
		\[\disc((\mu f)_{e'})=\alpha^{2n}\delta^{-n} \disc(f_e) ,\]
		where   $\rrk_AP=2n$.
		Straightforward
		computation shows that $\delta (\mu v^\sigma v)^\tau=\mu v^\sigma v$.
		Since $\tau|_S=\sigma|_S$, this means
		that $\delta \alpha^\sigma=\alpha$, or rather,
		$\alpha^2=\delta\Nr_{S/R}(\alpha)$.
		Thus, $\disc(\mu f)=\disc((\mu f)_{e'})=\alpha^{2n}\delta^{-n} \disc(f_e)=
		\disc(f_e)=\disc(f)$.

		Next, let $e'\in A$ be an idempotent with $e'^\sigma=e'$
		and $\rrk_A e'A>0$. Then, using Theorem~\ref{TH:invariant-primitive-idempotent},
		we can choose an idempotent $e\in e'Ae'$ with $e^\sigma=e$
		and $\rrk_{e'Ae'} eAe'=1$.
		By Corollary~\ref{CR:degree-of-endo-ring},
		$\rrk_{A}eA=\rrk_{e'Ae'}eAe'=1$, so
		$\disc(f)=\disc (f_e)=\disc(f_{e'})$.
		
		(iii) We may replace $f$ and $f'$
		with $f_e$ and $f'_e$ and assume that $A=S$.
		The statement is now straightforward.
	\end{proof}

	We continue to assume that
	$[A]=0$ and $(\sigma,\veps)$ is unitary.
	Let $S=\Cent(A)$ and $\theta=\sigma|_S$.
	Recall that with every $\alpha\in \units{R}$,
	we can associate a crossed produced $R$-algebra
	\[
	(S/R, \alpha).
	\]
	Its underlying $R$-module is the free right $S$-module
	with basis $\{1,u\}$ and its multiplication is determined
	by the product in $S$ and the rules $u^2=\alpha$
	and $su=us^\theta$ for all $s\in S$.
	It is well-known that $(S/R, \alpha)$ is a quaternion
	(i.e.\ degree-$2$) Azumaya $R$-algebra.
	Moreover, the map
	\[
	\alpha\mapsto [(S/R, \alpha)] 
	\]
	determines a group homomorphism from $\units{R}/\Nr_{S/R}(\units{S})$
	to $\Br R$; see \cite[Theorem~7.1a]{Saltman_1999_lectures_on_div_alg} or
	\cite[Lemma III.5.4.1, Corollary~III.5.4.6]{Knus_1991_quadratic_hermitian_forms}.
	
	Following \cite[\S10]{Knus_1998_book_of_involutions}, 
	given  $(P,f)\in \Herm[\veps]{A,\sigma}$
	of even reduced rank, we write
	\[
	D(f)=(S/R, \disc f)
	\]
	and define the \emph{discriminant Brauer class} of $f$
	to be $[D(f)]$.\footnote{
		When $R$ is a field, our definition of $D(f)$ does not agree
		with the definition given in \cite[\S10]{Knus_1998_book_of_involutions}.
		However, both definitions give the same Brauer class by
		\cite[Corollary~10.35]{Knus_1998_book_of_involutions}.
	}
	We remark that since $R$ is semilocal, using the discriminant Brauer class
	instead of the discriminant causes no loss of information:
	
	\begin{prp}
		Assume $R$ is semilocal, let $S$ be a quadratic \'etale
		$R$-algebra 
		and let $\alpha,\beta\in \units{R}$.
		Then $[(S/R,\alpha)]=[(S/R,\beta)]$ if and only
		if $\alpha\equiv \beta\bmod \Nr_{S/R}(\units{S})$.
	\end{prp}
	
	\begin{proof}
		We only need to check the ``only if'' part.
		Write $A=(S/R, \alpha)$, $B=(S/R, \beta)$
		and let $\theta$ denote the standard $R$-involution of $S$.
		We have $A=S\oplus uS$ with $u^2=\alpha$
		and $B=S\oplus vS$ with $v^2=\beta$.
		Since $R$ is semilocal and $\deg A=\deg B$,
		there exists an $R$-algebra isomorphism $\iota:A\to B$, see
		\cite[Corollary~3.3]{Roy_1967_derivations_of_Azumaya_algebras}. 
		
		We claim that there is an isomorphism $\psi:A\to B$
		which restricts to the identity on $S$.
		To see this, view $A$ as a right $A_S$-module
		via $x\cdot (a\otimes s)=sxa$ and
		$B$ as a right $A_S$-module via
		$y\cdot (a\otimes s)=sy\cdot \iota a$
		($x\in A$, $y\in B$, $a\in A$, $s\in S$).
		Since $\rank_S ({}_AA)=2=\rank_S ({}_BB)$,
		we have $\rrk_{A_S} A=\rrk_{A_S} B$.
		By Lemma~\ref{LM:rank-determines},
		there exists an $A_S$-module isomorphism
		$\xi:A\to B$. It induces an isomorphism $\Int(\xi):\End_A(A_A)\to \End_B(B_B)$.
		Identifying $\End_A(A_A)$ with $A$
		and $\End_B(B_B)$ with $B$ via $\vphi\mapsto \vphi(1)$,
		we get an isomorphism $\psi:A\to B$.
		Now, for all $s\in S$,
		we have $\psi(s)=(\xi \circ [x\mapsto sx]\circ\xi^{-1})(1_B)=\xi(s\cdot\xi^{-1}(1_B))=
		s\cdot \xi^{-1}\xi(1_B)=s$.
		
		Let $E_A=\{a\in A\suchthat \text{$sa=a s^\theta$ for all $s\in S$}\}$
		and define $E_B$ similarly.
		One readily checks that $E_A=uS$ and $E_B=vS$.
		Since $\psi$ fixes $S$,
		we have $\psi(E_A)\subseteq E_B$.
		Thus, $\psi(u)=vs$ for some $s\in\units{S}$. 
		Now, $\alpha =u^2=\psi(u^2)=(vs)^2=\Nr_{S/R}(s)v^2= \Nr_{S/R}(s)\beta$.
	\end{proof}

\section{An   Octagon of Witt Groups}
\label{sec:octagon}	

In this section, we introduce an $8$-periodic chain complex  --- an octagon, for short ---
of Witt groups of Azumaya
algebras with involution, generalizing a similar octagon defined by 
Grenier-Boley and Mahmoudi   for central simple
algebras with involution  \cite{Grenier_2005_octagon_of_Witt_grps}.
By the end of Section~\ref{sec:completion-of-proof}, we will show that
this octagon is exact when the base ring $R$ is  semilocal.

\subsection{The Octagon}
\label{subsec:octagon}

	Recall that $R$ denotes a ring with $2\in\units{R}$.
	Suppose we are given the following data:
	\begin{enumerate}[label=(G\arabic*)]
		\item \label{item:octagon-given-first}
		\label{item:octagon:Az-alg-with-inv}
		$(A,\sigma)$ is an Azumaya $R$-algebra with involution (see \ref{subsec:Az-alg-inv}),
		\item $\veps\in \Cent(A)$ satisfies $\veps^\sigma \veps=1$,
		\item \label{item:octagon-given-last} $\lambda,\mu \in \units{A} $ satisfy
		$\lambda^\sigma=-\lambda$, $\mu^\sigma=-\mu$, $\lambda\mu=-\mu\lambda$
		and  $\lambda^2\in \Cent(A)$.
	\end{enumerate}
	Define the following:
	\begin{enumerate}[label=(N\arabic*)]
	\item \label{item:octagon-notation-first} $S=\Cent(A)$,
	\item  $B$ is the commutant of $\lambda$ in $A$,
	\item  $T=\Cent(B)$,
	\item  $\tau_1:=\sigma|_B$,
	\item  $\tau_2:=\Int(\mu^{-1})\circ \sigma|_B$, i.e.\ $x^{\tau_2}=\mu^{-1}x^\sigma \mu$
	\end{enumerate}
	Note that $R\subseteq S\subseteq T\subseteq B\subseteq A$ and $\tau_1,\tau_2$
	are $R$-linear involutions on $B$. Also, $\lambda^2\in R$ because $(\lambda^2)^\sigma=(-\lambda)^2=\lambda^2$ and $R=\Sym_1(\Cent(A),\sigma)$.
	
	\begin{lem}\label{LM:octagon-basic-properties}
		In the previous notation, the following hold:
		\begin{enumerate}[label=(\roman*)]
			\item \label{item:T-quad-etale} 
			$T$ is a quadratic \'etale $S$-algebra,
			$\{1,\lambda\}$ is an $S$-basis of $T$
			and $\rank_{T}(A_A)$
			is constant along the fibers of $\Spec T\to \Spec S$.
			\item \label{item:B-Azumaya}  $B$ is an Azumaya $T$-algebra,
			$[B]=[A\otimes_S T]$ in $\Br T$ and
			$\deg B=\frac{1}{2}\iota\deg A$, where $\iota:S\to T$
			is the inclusion map.
			\item \label{item:B-plus-muB} $A=B\oplus \mu B$, $\mu B= B\mu$ and $\mu^2\in B$.
		\end{enumerate}
	\end{lem}

	\begin{proof}
		Write $T'=S[\lambda]$.
		Since $\lambda^2\in S$, we have $T'=S+\lambda S$.
		If $a\in S\cap \lambda S$,
		then $a $ commutes and anti-commutes with $\mu$,
		so $2a\mu=0$ and  $a=0$.
		Since $\ann_S\lambda=0$, this means that 
		$T'\cong S[x]/(x^2-\lambda^2)$.
		Thus, $T'$ is a quadratic \'etale $S$-algebra (see~\ref{subsec:quad-etale}).

		We claim that $\rank_{T'}A_{A}$
		is constant along the fibers of $\Spec T'\to \Spec S$.
		To see this, note that $\mu \lambda \mu^{-1}=-\lambda$,
		hence $\Int(\mu)|_{T'}$ coincides with the standard $S$-involution of $T'$,
		call it $\theta$.
		This involution acts transitively on every fiber  of $\Spec T'\to \Spec S$,
		so it is enough to show that $\rank_{T'}A_{A}=\theta \rank_{T'}A_{A}$.
		However, this follows from the fact that $\Int(\mu):A\to A$ 
		defines $\theta$-linear isomorphism from $A$,
		viewed as a right $T'$-module, to itself.

		Now, by Proposition~\ref{PR:centralized-in-Az-alg}, 
		$B=\Cent_A(T')$ is an Azumaya $T'$-algebra, $[B]=[A\otimes_S T']$,
		and $2\deg B=\deg B\cdot\iota \rank_S T'=\iota\deg A$,
		where $\iota:S\to T'$ is the inclusion map.
		Since   $T=\Cent(B)=T'$, we have established \ref{item:T-quad-etale}
		and \ref{item:B-Azumaya}.

		To prove \ref{item:B-plus-muB},
		let $E$ denote the set of elements of $A$
		which anti-commute with $\lambda$.
		One readily checks that $\mu B\subseteq E$ and $\mu^{-1}E\subseteq B$,
		hence $E= \mu B$. Likewise, $E=B\mu$, so $\mu B=B\mu$.
		Furthermore,   $B\cap \mu B=B\cap E$ consists of
		elements which commute
		and anti-commutes with $\lambda$, so $B\cap \mu B=0$
		(because   $2\lambda\in\units{A}$).
		Finally, every $a\in A$ can be written
		as $\frac{1}{2}(a+\lambda^{-1}a\lambda)+\frac{1}{2}(a-\lambda^{-1}a\lambda)$.
		It is easy to check, using $\lambda^2\in\Cent(A)$, that $	a+\lambda^{-1}a\lambda\in B$
		and $a-\lambda^{-1}a\lambda\in E$, so $A=B+E=B+\mu B$. We conclude
		that $A=B\oplus \mu B$.	Finally, $\mu^2\in B$ because $\mu^2\lambda=-\mu\lambda\mu=\lambda\mu^2$.
	\end{proof}
	
	Before proceeding, let us present 
	some examples where  
	\ref{item:octagon-given-first}--\ref{item:octagon-given-last} hold.
	
	\begin{example}
		(i) 
		Let $\alpha,\beta\in \units{R}$.
		Take $A$ to be the quaternion Azumaya $R$-algebra
		$R\trings{\lambda,\mu\where \lambda\mu=-\mu\lambda,\lambda^2=\alpha,\mu^2=\beta}$
		and $\sigma$ to be the $R$-involution of $A$ determined by $\lambda^\sigma=-\lambda$
		and
		$\mu^\sigma=-\mu$. Then $A,\sigma,\lambda,\mu$ and any
		$\veps\in \mu_2(R)$ satisfy \ref{item:octagon-given-first}--\ref{item:octagon-given-last}.
		In this case, $\{1,\lambda,\mu,\lambda\mu\}$ is an $R$-basis of $A$,
		and it is routine to check that
		$S=R$,
		$B=R+\lambda R=R[\lambda]$, 
		$\tau_1$ is the standard $R$-involution of $T=B$,
		and $\tau_2=\id_B$.
		
		(ii) Write $A_0,\sigma_0,\lambda_0,\mu_0$ for $A,\sigma,\lambda,\mu$
		defined in (i) and let $(A_1,\sigma_1)$
		be another Azumaya $R$-algebra with involution. Then
		$(A,\sigma):=(A_0\otimes A_1,\sigma_0\otimes \sigma_1)$
		is   also Azumaya over $R$ because
		$\Cent(A)=\Cent(A_0)\otimes \Cent(A_1)=R \otimes \Cent(A_1)$
		(see \cite[Propositio~5.3.10(ii)]{Rowen_1988_ring_theory_II} for the first equality),
		which means that $R=\{a\in \Cent(A)\suchthat a^{\sigma}=a\}$.
		Then $A,\sigma,\lambda:=\lambda_0\otimes 1,\mu:=\mu_0\otimes 1$
		and any $\veps\in\mu_2(R)$
		satisfy \ref{item:octagon-given-first}--\ref{item:octagon-given-last}.
		Writing $S_1=\Cent(A_1)$ and letting $\theta$
		denote the standard $R$-involution of
		$R[\lambda_0]$,
		we have $S=R\otimes S_1$, $B=R[\lambda_0]\otimes A_1$, $T=R[\lambda_0]\otimes S_1$
		(Lemma~\ref{LM:center-base-change}), $\tau_1=\theta\otimes\sigma_1$,
		and $\tau_2=\id_{R[\lambda_0]}\otimes \sigma_1$.
	\end{example}
	
	Using Lemma~\ref{LM:octagon-basic-properties}\ref{item:B-plus-muB}, we can define the following maps:
	\begin{enumerate}[label=(N\arabic*), start=6]
	\item  $\pi_1,\pi_2:A\to B$  are defined by
	$\pi_i(b_1+\mu b_2)=b_i$
	($b_1,b_2\in B$, $i\in \{1,2\}$).
	\end{enumerate}
	(For a definition of $\pi_1$ not involving $\mu$, see Lemma~\ref{LM:dfn-of-pi-A-B}(i)
	below.)
	We now introduce four functors:
	\begin{enumerate}[label=(N\arabic*), resume]
	\item For $i=1,2$, let 
	$\pi^{(\veps)}_i \! : \Herm[\veps]{A,\sigma}\to \Herm[(-1)^{i+1}\veps]{B,\tau_i}$
	be defined by $\pi_i^{( \veps)} (P,f)=(P,\pi_i f)$, where $\pi_i f=\pi_i\circ f$;
	morphisms are mapped to themselves.
	\item For $i=1,2$, let $\rho_i^{(\veps)}:\Herm[\veps]{B,\tau_i}\to \Herm[-\veps]{A,\sigma}$
	be defined by
	$\rho_i^{( \veps)}(Q,g)=({Q\otimes_BA},\rho_i g)$, 
	where $\rho_ig :(Q\otimes_BA)\times (Q\otimes_BA)\to A$ is determined by
	\begin{align*}
	(\rho_1 g)(x\otimes a,x'\otimes a')&= a^\sigma \lambda g(x,x')a',\\
	(\rho_2 g)(x\otimes a,x'\otimes a')&= a^\sigma (\lambda\mu) g(x,x')a'
	\end{align*}
	($x,x'\in Q$, $a,a'\in A$); for a morphism $\vphi$, set  $\rho_i^{( \veps)}\vphi =\vphi\otimes_B\id_A$.
	\end{enumerate}
	When there is no
    risk of confusion,
	we will  drop the superscript ``$(\veps)$''.
	The functors $\pi_1,\pi_2,\rho_1,\rho_2$
	are well-defined by the following lemma.

	\begin{lem}\label{LM:maps-in-oct-are-well-def}
		The assignments   $\pi_1^{(\veps)}$, $\pi_2^{(\veps)}$,
		$\rho_1^{(\veps)}$, $\rho_2^{(\veps)}$ are functors.
		Moreover, they take hyperbolic hermitian forms to hyperbolic hermitian forms.
	\end{lem}

	\begin{proof}
		Everything is straightforward   
		except the fact that $\pi_1,\pi_2,\rho_1,\rho_2$
		take unimodular hermitian forms to unimodular hermitian forms.
		We verify this fact case-by-case.

		The inclusion $B\to A$
		induces a homomorphisms of $R$-algebras with involution
		$(B,\tau_1)\to (A,\sigma)$, which we denote by $\rho$.
		Given $(Q,g)\in \Herm[\veps]{B,\tau_1}$,
		the hermitian space $\rho_1(Q,g)$ is just $\rho(Q,\lambda g)$, where
		$\lambda g$ is the $\lambda$-conjugation of $g$, see \ref{subsec:conjugation},
		and $\rho$ is base change in the sense of  \ref{subsec:herm-base-change}.
		Since both $\lambda$-conjugation and base change preserve unimodularity, $\rho_1(Q,g)$
		is unimodular.
		
		Similarly,
		to see that $\rho_2(Q,g)$ is unimodular,
		let $\sigma_2:=\Int((\lambda\mu)^{-1})\circ \sigma$
		and note that the inclusion $B\to A$ also defines a morphism
		$\rho':(B,\tau_2)\to (A,\sigma_2)$.
		It is straightforward to check that $\rho_2g=(\lambda\mu)(\rho'g)$,
		so $\rho_2 g$ is unimodular.

		We proceed with checking that $\pi_1f$ is unimodular for all $(P,f)\in\Herm[\veps]{A,\sigma}$.
		If $x\in P$ satisfies $\pi_1f(x,P)=0$, then $f(x,P)$
		is a right ideal of $A$ contained in $\ker \pi=B\mu$. Thus,
		$f(x,P)\subseteq  B\mu \cap   (B\mu)\mu^{-1}=0$, 
		and $x=0$ because $f$ is unimodular.
		Suppose now that $\phi\in\Hom_B(P,B)$.
		Define $\hat{\phi}:P\to A$
		by $\hat{\phi}x = \phi x+\phi(x\mu)\mu^{-1}$.
		It is routine to check that $\hat{\phi}\in\Hom_A(P,A)$
		and $\pi_1\circ \hat{\phi}=\phi$.
		Since $f$ is unimodular, there exists $x\in P$
		with $\hat{\phi}=f(x,-)$, so $\phi=\pi_1f(x,-)$.
		
		That $\pi_2f$ is unimodular is shown similarly;
		define $\hat{\phi}$ by $\hat{\phi}x=\mu\cdot \phi(x\mu)\cdot \mu^{-1}+\mu \phi x$. 
	\end{proof}
	
	Lemma~\ref{LM:maps-in-oct-are-well-def} implies
	that  $\pi_1^{(\veps)}$, $\pi_2^{(\veps)}$,
	$\rho_1^{(\veps)}$, $\rho_2^{(\veps)}$ induce
	maps between the relevant Witt groups.
	These maps can be arranged in an octagon-shaped diagram:
	\begin{equation}\label{EQ:octagon}
    \xymatrix{
    W_\veps(A,\sigma) \ar[r]^{\pi_1^{(\veps)}} & 
    W_\veps(B,\tau_1) \ar[r]^{\rho_1^{(\veps)}} & 
    W_{-\veps}(A,\sigma) \ar[r]^{\pi_2^{(-\veps)}} & 
    W_{\veps}(B,\tau_2) \ar[d]^{\rho_2^{(\veps)}}\\
    W_{-\veps}(B,\tau_2) \ar[u]^{\rho_2^{(-\veps)}} & 
    W_{\veps}(A,\sigma)  \ar[l]^{\pi_2^{(\veps)}} & 
    W_{-\veps}(B,\tau_1) \ar[l]^{\rho_1^{(-\veps)}} & 
    W_{-\veps}(A,\sigma) \ar[l]^{\pi_1^{(-\veps)}} 
    }
    \end{equation}
	We will  see  in Proposition~\ref{PR:octagon-is-a-complex}
	below  that the
    octagon 
    is  a chain complex of abelian groups.

	The octagon is known to be exact when $R$ is a field
    \cite{Grenier_2005_octagon_of_Witt_grps}; see the Introduction
    for the history of this result.
    The purpose of this paper is to extend the exactness of the octagon
    to   semilocal rings.
    Specifically, we prove:

    \begin{thm}\label{TH:exactness}
		Suppose that $R$ is semilocal.
		Then the octagon \eqref{EQ:octagon} is exact.
    \end{thm}

	The proof  will  occupy the following four sections
	and   be concluded in Section~\ref{sec:completion-of-proof}; its highlights are
	given
	in \ref{subsec:proof-overview}.
	In the course of the proof, we will also  determine 
	the images of the functors $\pi_1,\pi_2,\rho_1,\rho_2$
	when $T$ is connected  semilocal (the exactness
	of the octagon
	answers this only up to Witt equivalence), see Theorem~\ref{TH:finer-exactness}.
	This finer version will be required
	for some of the applications.
	
\medskip

	The remainder of this section is dedicated to proving that the octagon is a complex,
	providing equivalent conditions for its exactness, and surveying
	how these conditions will be proved under the assumption that $R$ is semilocal.

\subsection{Equivalent Conditions for Exactness}

	Keep the assumptions of \ref{subsec:octagon}.
	In this subsection, we show that  the exactness of the octagon
	\eqref{EQ:octagon} is equivalent
	to a certain list of conditions involving $R,A,\sigma,\veps,\lambda,\mu$. 
	The proof generally follow the same lines
	as the corresponding arguments given in \cite[\S3]{Grenier_2005_octagon_of_Witt_grps} and 
	\cite[Appendix]{Bayer_1995_Serre_conj_II}, both addressing
	the case $S$ is a field.

	Given a right $B$-module $Q$ and $a\in A$, we write $Q\otimes a$ for
	the subset $\{q\otimes a\where q\in Q\}$ of $Q\otimes_B A$; it is a $B$-submodule
	of $Q\otimes_BA$ if $aB\subseteq Ba$.
	If $M$ is a subset of a right $A$-module $P$,   we write $MA$ for the $A$-submodule
	generated by $M$.

\medskip
	
	We begin by showing	 that the octagon is a chain complex for any ring $R$.

	\begin{prp}\label{PR:octagon-is-a-complex}
		In the notation of \ref{subsec:octagon},  \eqref{EQ:octagon} is
		a chain complex of abelian groups.
	\end{prp}
	
	\begin{proof}
		By symmetry, we only need to consider the top row of 	\eqref{EQ:octagon}.
		
\medskip
	
		\noindent\textit{\eqref{EQ:octagon} is a complex at  $W_\veps(A,\sigma)$.}
        Let $(Q,g)\in\Herm[-\veps]{B,\tau_2}$. Then
        $\pi_1\rho_2(Q,g)=(Q\otimes_BA,\pi_1\rho_2 g)$.
        Straightforward calculation shows that the $B$-sumodules $M_1:=Q\otimes 1$ and 
        $M_2:=Q\otimes \mu$
        satisfy $\pi_1\rho_2g(M_1,M_1)=\pi_1\rho_2g(M_2,M_2)=0$
        and $M_1+M_2=Q\otimes_BA$. Thus,
        $\pi_1\rho_2(Q,g)$ is hyperbolic.

\medskip

        \noindent\textit{\eqref{EQ:octagon} is a complex at $W_{-\veps}(A,\sigma)$.}
        The proof for $W_{\veps}(A,\sigma)$ applies  verbatim.

\medskip

        \noindent\textit{\eqref{EQ:octagon} is a complex at $W_\veps(B,\tau_1)$.}
        Let $(P,f)\in\Herm[\veps]{A,\sigma}$.
    	Then $\rho_1\pi_1(P,f)=(P\otimes_B A,\rho_1\pi_1 f)$. Define
    	\begin{align*}
    	L_1&=\{x\mu\otimes 1+ x \otimes \mu\where x\in P\},\\
    	L_2&=\{x\mu\otimes 1- x\otimes \mu \where x\in P\}.
    	\end{align*}
        It is easy to check that $L_1$ and $L_2$ are $B$-submodules of $P\otimes_AB$
        and that $L_1+L_2=P\otimes_AB$ (recall that $2\in\units{R}$).
        We claim that $\rho_1\pi_1 f(L_1,L_1)=\rho_1\pi_1 f(L_2,L_2)=0$.
        Indeed, let $x,y\in P$ and write $f(x,y)=\alpha+\mu\beta$ with $\alpha,\beta\in B$.
        Then
        \begin{align*}
        \rho_1\pi_1f(x &\mu\otimes 1+ x \otimes \mu, y\mu\otimes 1+ y\otimes \mu)\\
        &=  \rho_1\pi_1f(x\mu\otimes 1,y\mu\otimes 1)+\rho_1\pi_1f(x\mu \otimes 1,y\otimes \mu)
        +\rho_1\pi_1f(x\otimes\mu,y\mu\otimes 1)\\
        &\qquad +\rho_1\pi_1f(x\otimes \mu,y\otimes \mu)\\
        &= \lambda\pi_1(\mu^\sigma(\alpha+\mu\beta)\mu)+\lambda\pi_1(\mu^\sigma(\alpha+\mu\beta))\mu
        +\mu^\sigma\lambda\pi_1((\alpha+\mu \beta)\mu)\\
        &\qquad +\mu^\sigma\lambda\pi_1(\alpha+\mu\beta)\mu\\
        &=\lambda\mu^\sigma\alpha\mu+\lambda\mu^\sigma\mu\beta\mu
        +\mu^\sigma\lambda\mu\beta\mu+\mu^\sigma\lambda\alpha\mu\\
        &=-\lambda\mu\alpha\mu-\lambda\mu^2\beta\mu+\lambda\mu^2\beta\mu+\lambda\mu\alpha\mu=0 ,
        \end{align*}
        hence $\rho_1\pi_1 f(L_1,L_1)=0$.
        Likewise, $\rho_1\pi_1f(L_2,L_2)=0$, 
        so  $\rho_1\pi_1 f$ is hyperbolic.

\smallskip

        \noindent\textit{\eqref{EQ:octagon} is a complex at   $W_\veps(B,\tau_2)$.}
        This is similar to the proof of the case $W_\veps(B,\tau_1)$; define
        $L_1$ and $L_2$ in the same manner.
	\end{proof}

	We now give equivalent conditions for the exactness of the octagon
	\eqref{EQ:octagon}.
	
	\begin{thm}\label{TH:exactness-equiv-conds}
		With the notation of \ref{subsec:octagon}, consider the following conditions:
		\begin{enumerate}[label=(E\arabic*)]
			\item  \label{item:Eii} For every $(P,f)\in \Herm[\veps]{A,\sigma}$ such that
    		$[\pi_1(P,f)]=0$ in $W_\veps(B,\tau_1)$, there exists $(P',f')$ in the Witt
    		class of $(P,f)$ and a Lagrangian   $M$ of $\pi_1(P',f')$
    		with  
			$M\cdot A=P'$.
			\item   \label{item:Ei} For every 
			$(Q,g)\in \Herm[\veps]{B,\tau_1}$ such that
    		$[\rho_1(Q,g)]=0$ in $W_{-\veps}(A,\sigma)$, there exists $(Q',g')$ in the Witt
    		class of $(Q,g)$
    		and   a Lagrangian $L$ of $\rho_1(Q',g')$
    		with $L\oplus (Q'\otimes 1)=Q'\otimes_BA$ as $B$-modules.
    		\item  \label{item:Eii-tag} 
    		For every $(P,f)\in \Herm[-\veps]{A,\sigma}$ such that
    		$[\pi_2(P,f)]=0$ in $W_\veps(B,\tau_2)$, there exists $(P',f')$ in the Witt
    		class of $(P,f)$ and a Lagrangian   $M$ of $\pi_2(P',f')$
    		with  
			$M\cdot A=P'$.
    		\item \label{item:Ei-tag} For every 
			$(Q,g)\in \Herm[\veps]{B,\tau_2}$ such that
    		$[\rho_2(Q,g)]=0$ in $W_{-\veps}(A,\sigma)$, there exists $(Q',g')$ in the Witt
    		class of $(Q,g)$
    		and   a Lagrangian $L$ of $\rho_2(Q',g')$
    		with $L\oplus (Q'\otimes 1)=Q'\otimes_BA$  as $B$-modules.
		\end{enumerate}
		Then the exactness of
		 \eqref{EQ:octagon} at
		the terms 
		$W_{\veps}(A,\sigma)$, $W_\veps(B,\tau_1)$, $W_{-\veps}(A,\sigma)$, $W_\veps(B,\tau_2)$    
		on the top row is equivalent
		to the conditions \ref{item:Eii}, \ref{item:Ei},  
		 \ref{item:Eii-tag}, \ref{item:Ei-tag},  respectively. 
	\end{thm}

	\begin{remark}\label{RM:Es-for-field}
		Conditions \ref{item:Eii}--\ref{item:Ei-tag}  
		are not difficult to verify when $R$ is a field.
		We illustrate this for \ref{item:Ei} and
		\ref{item:Eii}.
		
		In the context of \ref{item:Ei},
		using    Proposition~\ref{PR:ansio-Witt-equivalent},
		choose $(Q',g')$ to be anisotropic.
		Since $B$ is semisimple artinian, this means that
		$g'(x,x)\neq 0$ whenever $x\neq 0$.
		Now,  if $L$ is a Lagrangian of $\rho_1 g'$, then every $x\in L\cap (Q'\otimes 1)$
		satisfies $g'(x,x)=0$, hence $L\cap (Q'\otimes 1)=0$.
		On the other hand, by Lemma~\ref{LM:unimodular-implies-rrk-sigma-inv},
		we have
		$2\dim_R L=\dim_RL +\dim_RL^*=\dim_R (Q'\otimes_BA)$,
		and since  $Q'\otimes_BA=(Q'\otimes 1)\oplus (Q'\otimes \mu)$,
		we have $\dim_R (Q'\otimes 1)=\frac{1}{2}\dim_R (Q'\otimes_BA)$,
		so  $R$-dimension considerations 
		force $L\oplus (Q'\otimes 1)=Q'\otimes_BA$.
		
		Similarly, in the context of \ref{item:Eii},
		we may choose $(P',f')$
		so that $f'(x,x)\neq 0$ whenever $x\neq 0$. 
		If $M$ is a Lagrangian of $\pi_1f'$,
		then every $x\in M\cap M\mu$ satisfies
		$\pi_1f(x,x)=\pi_1f(x\mu^{-1},x)=0$, which means
		that $f(x,x)=0$. Thus, $M\cap M\mu=0$. 
		Since $P\cong M\oplus M^*$ (the dual is taken relative to $B$),
		$\rank_{T_0} P=\rank_{T_0} M+\rank_{T_0} M^*$,
		where  $T_0=\Sym_1(T,\tau)$. By Lemma~\ref{LM:unimodular-implies-rrk-sigma-inv},
		$2\rank_{T_0} M=\rank_{T_0} P$, so $2\dim_R M=\dim_R P$
		and $R$-dimension considerations
		force $M+M\mu=P$. In particular, $M  A=P$.
	
    	This argument relies critically
    	on the fact that $R$ is a field, and thus cannot be naively generalized
    	to more general rings.
    \end{remark}

	Before proving Theorem~\ref{TH:exactness-equiv-conds},
	we first prove the following lemma.
	
	\begin{lem}\label{LM:Eii-equiv-formulation}
		With notation as in \ref{subsec:octagon},
		let $(P,f)\in\Herm[\veps]{A,\sigma}$
		and let $M$ be a Lagrangian of $\pi_1 f$, resp.\ $\pi_2f$.
		Then $MA=P$ if and only if $M\oplus M\mu=P$.
	\end{lem}

	\begin{proof}
		It is clear that $M\oplus M\mu=P$ implies $MA=P$, so we   prove
		the converse.
		As both $M$ and $M\mu$ are $R$-summands of $P$, 
		the lemma will follow from Lemma~\ref{LM:semilocal-direct-sum-reduction}
		if we show that $M(\frakm)\oplus M\mu(\frakm)=P(\frakm)$ for all $\frakm\in \Max R$.
		We may therefore assume that $R$ is a field (the setting of \ref{subsec:octagon}
		is preserved under base-change by Lemma~\ref{LM:center-base-change}).
		Since $MA=P$ and $A=B+B\mu$, we have $M+ M\mu=P$.
		We observed in Remark~\ref{RM:Es-for-field} that $2\dim_R M=\dim_R P$,
		so this means that  $M\oplus M\mu=P$.
	\end{proof}

\begin{proof}[Proof of Theorem~\ref{TH:exactness-equiv-conds}]
	We   showed that the octagon is a chain complex in Proposition~\ref{PR:octagon-is-a-complex}.
	Moreover, the proof of that proposition shows that if $(Q,g)=\pi_1(P,f)$ for 
	$(P,f)\in \Herm[\veps]{A,\sigma}$,
	then $\rho_1(Q,f)$ admits a Lagrangian $L$ --- $L_1$ or $L_2$ in
	the notation of that proof --- with $L\oplus (Q\otimes 1)=Q\otimes_BA$.
	Thus, condition \ref{item:Ei} follows from the exactness of the octagon at $W_\veps(B,\tau_1)$,
	and,
	in a similar manner, the exactness of the octagon
	at  $W_{\veps}(A,\sigma)$, $W_{-\veps}(A,\sigma)$, $W_\veps(B,\tau_2)$
	implies   \ref{item:Eii}, \ref{item:Eii-tag}, \ref{item:Ei-tag}, respectively.
	It remains to show the converse.

\medskip

	\noindent \textit{\ref{item:Eii} implies exactness at $W_{\veps}(A,\sigma)$ (top row).}
		Suppose that $(P,f)\in \Herm[\veps]{A,\sigma}$
		satisfies $[\pi_1(P,f)]=0$ in $W_\veps(B,\tau_1)$.
        By \ref{item:Eii} and Lemma~\ref{LM:Eii-equiv-formulation}, we may replace
        $(P,f)$ with a Witt equivalent hermitian space
        to assume that $\pi_1f$ admits a Lagrangian $M$ with
        $P=M\oplus M\mu$. Let $g=(\lambda\mu)^{-1} f|_{M\times M}$. Since $\pi_1f(M,M)=0$, 
        we have
        $g(M,M)\subseteq (\lambda\mu)^{-1}\ker\pi_1= -\lambda^{-1}\mu^{-1}\mu B
        =B$.
        We claim that $g:M\times M\to B$ is a $(-\veps)$-hermitian
        form over $(B,\tau_2)$. Indeed, the sesquilinearily is straightforward, and for all $x,y\in M$,
        we have
        \begin{align*}
            -\veps g(y,x)^{\tau_2}&= 
            -\veps \mu^{-1}f(y,x)^{\sigma}((\lambda\mu)^{-1})^\sigma\mu\\
            &=-\mu^{-1} f(x,y)\lambda^{-1}\mu^{-1}\mu=\mu^{-1}\lambda^{-1}f(x,y)=g(x,y)
        \end{align*}
        (note that  $f(x,y)\in\ker \pi_1= \mu B$  and $\lambda$ anti-commutes with elements
        from $\mu B$).
        
        Next, we claim that $(M,g)$ is unimodular. Suppose
        that $g(x,M)=0$. Then $f(x,M)=0$, hence $f(x,P)=f(x,M+M\mu)=0$, and $x=0$
        by the unimodularity of $f$. Now, let $\phi\in \Hom_B(M,B)$.
        Using $P=M\oplus M\mu$, 
        define $\psi:P\to A$ by $\psi(x+y\mu)=\lambda\mu\cdot  \phi x+ \lambda\mu \cdot \phi y\cdot \mu$ 
        for all $x,y\in M$.
        It is routine to check that $\psi\in\Hom_A(P,A)$. Thus, there exists
        $x\in P$ such that $\psi y=f(x,y)$ for all $y\in P$.
        Furthermore, we have $\pi_1 f(x,M)= \pi_1(\lambda\mu \cdot \phi(M))=0$, 
        hence $x\in M^{\perp(\pi_1f)}=M$.
        Since   $g(x,y)= (\lambda\mu)^{-1} f(x,y)=(\lambda\mu)^{-1} \psi y=\phi y$ for all $y\in M$,
        we have shown that $x\mapsto g(x,-):M\to \Hom_B(M,B)$ is bijective.

		Finally, we show that $\rho_2^{(-\veps)}(M,g) = (M\otimes_BA,\rho_2 g)$
		is isomorphic to $(P,f)$.
		Let $\vphi$ denote the $A$-module homomorphism $x\otimes a\mapsto xa: M\otimes_BA\to P$;
		it is clearly surjective, and it is straightforward to check that $f(\vphi x,\vphi y)=\rho_2g(x,y)$
		for all $x,y\in M\otimes_BA$. The latter means that if $x\in \ker \vphi$, then
		$\rho_2g(x,M\otimes_B A)=0$, and thus  $x=0$ because $\rho_2g$ is unimodular.
		As a result, $\vphi$
		is injective, and therefore an isometry from $\rho_2^{(-\veps)}(M,g)$
		to $(P,f)$.

\medskip
		
		\noindent \textit{\ref{item:Eii-tag} implies exactness at $W_{-\veps}(A,\sigma)$ (top row).}
		This similar to the exactness at $W_{\veps}(A,\sigma)$ with the difference that
		one defines $g=\lambda^{-1} f|_{M\times M}$.
		
\medskip

	\noindent \textit{\ref{item:Ei} implies exactness at $W_\veps(B,\tau_1)$.}
        Let $(Q,g)\in \Herm[\veps]{B,\tau_1}$ be a
       	hermitian space such that 
        $[\rho_1(Q,g)]=0$ in $W_{-\veps}(A,\sigma)$.
        By \ref{item:Ei}, we may replace
		$(Q,g)$ with a Witt-equivalent space
		and        
        assume that $\rho_1 g$ admits a Lagrangian $L$ such that
        such that $L\oplus (Q\otimes 1)=Q\otimes_B A$. Multiplying both sides
        with $\mu$ yields $L\oplus (Q\otimes \mu)=Q\otimes_BA$.
		This means that every $x\in Q$ admits
        a unique element $Jx\in Q$ such that $x\otimes 1+Jx\otimes \mu\in L$.

        The map $J:Q\to Q$ is easily seen to satisfy:
        \begin{align}\label{EQ:J-i}
            J^2x&= x\mu^{-2}\\
            J(xb)&=(Jx)(\mu b\mu^{-1}) \nonumber
        \end{align}
        for all $x\in Q$, $b\in B$. We make $Q$ into an $A$-module by  setting
        \[
        x(b_1+\mu b_2)=xb_1+(Jx)\mu^2b_2\qquad (x\in Q,\, b_1,b_2\in B ) .
        \]
        Using \eqref{EQ:J-i}, it is easy to see that this indeed defines an $A$-module
        structure (one has to verify the identities  $(x\mu)\mu=x(\mu^2)$ and $(xb)\mu=(x\mu)(\mu^{-1}b\mu)$ for
        $b\in B$).

        Since $\rho_1g(L,L)=0$, we have $\rho_1 g(x\otimes 1+Jx\otimes \mu,y\otimes 1+Jy\otimes \mu)=0$
        for all $x,y\in Q$. Since $A=B\oplus\mu B$, this means that
        \begin{align}\label{EQ:J-ii}
        & g(x,y)+\mu g(Jx,Jy)\mu=0,\\
        & g(x,Jy)\mu+\mu g(Jx,y)=0  \nonumber
        \end{align}
        for all $x,y\in Q$.
        Define $f:Q\times Q\to A$ by
        \[
        f(x,y)=g(x,y)-\mu g(Jx,y)\ .
        \]
        We claim that $f$ is an $\veps$-hermitian form over $(A,\sigma)$.
        After unfolding the definitions, this comes down to checking that
        $f(xb,y)=b^\sigma f(x,y)$, $f(x,yb)=f(x,y)b$, $f(x\mu,y)=\mu^{\sigma}f(x,y)$, $f(x,y\mu)=f(x,y)\mu$
        and $f(x,y)=\veps f(y,x)^\sigma$ for all $x,y\in Q$, $b\in B$.
        The first three identities follow easily from \eqref{EQ:J-i}. For the fourth and fifth identities we also
        use \eqref{EQ:J-ii}:
        \begin{align*}
        f(x,y\mu)
        &=g(x,y\mu)-\mu g(Jx,y\mu)\\
        &=g(x,(Jy)\mu^2)-\mu g(Jx,(Jy)\mu^2)\\
        &=g(x,Jy)\mu^2-\mu g(Jx,Jy)\mu^2\\
        &=-\mu g(Jx,y)\mu+g(x,y)\mu\\
        &=(g(x,y)-\mu g(Jx,y))\mu=f(x,y)\mu\ ,
        \end{align*}
        \begin{align*}
        \veps f(y,x)^\sigma
        &= \veps g(y,x)^\sigma-\veps  g(Jy,x)^\sigma\mu^{\sigma}\\
        &= g(x,y) +g(x,Jy)\mu\\
        &=g(x,y)- \mu g(Jx,y)=f(x,y)\ .
        \end{align*}

		We  claim that $(Q_A,f)$ is unimodular.
		Indeed, suppose $f(x,Q)=0$.
		Then, the definition of $f$ implies that $g(x,Q)=0$,
		so $x=0$ by the unimodularity of $g$.
		Now let $\phi\in \Hom_A(Q,A)$. Then $\pi_1\circ \phi\in\Hom_B(Q,B)$,
		hence there exists $x\in Q$
		such that $\pi_1\phi y=g(x,y)$ for all $y\in Q$.
		Define $\psi y:=\phi y-f(x,y)$.
		Then $\psi \in \Hom_A(Q,A)$ satisfies $\im\psi\subseteq \ker \pi_1=\mu B$.
		Since $\im \psi$ is a right ideal in $A$ and $b\mu\notin \mu B$
		for all $0\neq b\in \mu B$, it follows that $\psi=0$ and  
		$\phi y=f(x,y)$ for all $y\in Q$.

        Finally, it is clear that $\pi_1(Q,f)=(Q,g)$, so we
        have verified the exactness at  $W^\veps(B,\tau_1)$.
        
\medskip

	\noindent \textit{\ref{item:Ei-tag} implies exactness at $W_\veps(B,\tau_2)$.}
		Let $(Q,g)\in W^\veps(B,\tau_2)$ be a hermitian space such that
		$[\rho_2(Q,g)]=0$ in $W_\veps(A,\sigma)$.
        By \ref{item:Ei-tag}, we can
        replace $(Q,g)$ with a Witt equivalent space 
        and  assume that $\rho_2 g$ admits a Lagrangian $L$ such that
        such that $L\oplus (Q\otimes 1)=Q\otimes_B A$.
        
		Define $\sigma_2=\Int(\mu^{-1})\circ \sigma$.
		Then $\mu^{-1}$-conjugation (see \ref{subsec:conjugation}) 
		induces a group
		isomorphism $(P,f)\mapsto (P,\mu^{-1}f):W_{-\veps}(A,\sigma)\to W_\veps(A,\sigma_2)$.
		Furthermore, one readily checks that
		$\pi_2(P,f)=\pi_1(P,\mu^{-1}f)$.
		Thus, it is enough to show that there exists $(P,f)\in W_\veps(A,\sigma_2)$
		with $\pi_1(P,f)=(Q,g)$. This can be shown exactly as in
		the proof that \ref{item:Ei} implies exactness at $W_\veps(B,\tau_1)$.
	\end{proof}

	\begin{remark}\label{RM:finer-exactness}
		In the course of proving Theorem~\ref{TH:exactness-equiv-conds},
		we also showed:
		\begin{enumerate}[label=(\roman*)]
			\item \label{item:RM:finer:rho-two-image}
    		Given $(P,f)\in \Herm[\veps]{A,\sigma}$,
    		there exists 
    		$(Q,g)\in \Herm[-\veps]{B,\tau_2}$
    		with $\rho_2g\cong f$
    		if and only if $\pi_1f$ admits a Lagrangian   $M$ 
    		for which
			$M\cdot A=P$.
			\item \label{item:RM:finer:pi-one-image}
			Given $(Q,g)\in \Herm[\veps]{B,\tau_1}$,
    		there exists $(P,f)\in\Herm[\veps]{A,\sigma}$ 
    		with $\pi_1f\cong g$ if and only if 
    		$\rho_1f$ admits a Lagrangian $L$ for which $L\oplus (Q\otimes 1)=Q\otimes_BA$.
    		\item  \label{item:RM:finer:rho-one-image}
    		Given $(P,f)\in \Herm[-\veps]{A,\sigma}$,  
    		there exists $(Q,g)\in\Herm[\veps]{B,\tau_1}$ 
    		with $\rho_1g\cong f$
    		if and only if $\pi_2 f$ admits    a Lagrangian   $M$ 
    		for which
			$M\cdot A=P$.
    		\item  \label{item:RM:finer:pi-two-image}
    		Given $(Q,g)\in \Herm[\veps]{B,\tau_2}$,
    		there exists $(P,f)\in\Herm[-\veps]{A,\sigma}$ 
    		with $\pi_2f\cong g$
    		if and only if $\rho_2f$ admits a Lagrangian $L$ 
    		for which $L\oplus (Q\otimes 1)=Q\otimes_BA$.
		\end{enumerate}
	\end{remark}
	
\subsection{Overview of The Proof of Theorem~\ref{TH:exactness}}
\label{subsec:proof-overview}

	Keep the notation of \ref{subsec:octagon} and suppose that
	$R$ is semilocal.
	Thanks to Theorem~\ref{TH:exactness-equiv-conds}
	and the antipodal symmetry of \eqref{EQ:octagon},
	in order to prove Theorem~\ref{TH:exactness},
	it is enough to establish the conditions \ref{item:Eii}--\ref{item:Ei-tag}.
	The proof is somewhat involved,  so we   outline the argument first.
	
\medskip
	
	Let us consider condition \ref{item:Ei}:
	We are given 
	$(Q,g)\in\Herm[\veps]{B,\tau_1}$  such that $[\rho_1 g]=0$
	in $\Herm[-\veps]{A,\sigma}$ and need
	to find a Lagrangian $L$ of $\rho_1g$ such that $(Q\otimes 1)\oplus L= Q\otimes_B A$,
	possibly after replacing $(Q,g)$ with a Witt equivalent hermitian space.
	We abbreviate $Q\otimes_BA$ to $QA$ and identify $Q$ with its copy $Q\otimes 1$ in $QA$.
	
	Fix a Lagrangian $L' $ of $\rho_1g$; it exists by Theorem~\ref{TH:trivial-in-Witt-ring}(ii).
	We assume that $\rrk_A L'=\frac{1}{2}\rrk_AP$ for simplicity, so
	that $L' \in\Lag(\rho_1 g)$ (see \ref{subsec:Lagrangians}).	
	Let $\frakm_1,\dots,\frakm_t$ denote the maximal ideals of $R$.
	Suppose that we can find, for every $1\leq i\leq t$, an isometry
	$\vphi_i \in U^0(\rho_1g(\frakm_i))$ such that $Q(\frakm_i)\oplus \vphi_i (L' (\frakm_i))=QA(\frakm_i)$.
	Then, by Theorem~\ref{TH:U-zero-mapsto-onto-closed-fibers},
	there exists $\vphi\in U^0(\rho_1g)$ with $\vphi(\frakm_i)=\vphi_i$,
	and by Lemma~\ref{LM:semilocal-direct-sum-reduction}, ${Q\oplus \vphi(L' )}=QA$.
	We may therefore take $L=\vphi(L')$ 
	and the proof of \ref{item:Ei} reduces into proving the existence of $\vphi_1,\dots,\vphi_t$.
	
	Write $k_1=k(\frakm_1)$, $g_1 =g(\frakm_1)$, $Q_1=Q(\frakm_1)$, $L'_1=L'(\frakm_1)$ and so on.
	In ideal circumstances, e.g., when $\sigma_1$ is unitary,
	we have $U^0(\rho_1g_1)=U(\rho_1g_1)$ (Proposition~\ref{PR:U-zero-description}),
	and the existence of $\vphi_1$ can be shown by  proving the existence of some   $L_1\in\Lag(\rho_1 g_1)$
	with $Q_1\oplus L_1=Q_1A_1$ and then using 
	Lemma~\ref{LM:Lag-transitive-action} 
	to assert the existence of $\vphi_1\in U(\rho_1g_1)$ with $\vphi_1(L'_1)=L_1$.
	
	To prove the existence of $L_1$, we    write $g_1$ as a sum of an anisotropic
	form and a hyperbolic form (Proposition~\ref{PR:ansio-Witt-equivalent})
	and treat each case separately. In fact, the anisotropic
	case   has already been addressed in Remark~\ref{RM:Es-for-field}, so only the
	hyperbolic case should be treated.
	In addition, when $k_1$ is infinite, one can   use the rationality of
	the the $k_1$-variety $\uU^0(f_1)$  (see Theorem~\ref{TH:rational-variety})
	to reduce to the case where $k_1$ is algebraically closed (Proposition~\ref{PR:rationality-for-Ei}). 
	Assuming $k_1$ is algebraically
	closed or finite, we have $[A_1]=0$ in $\Br S_1$ and $[B_1]=0$ in $\Br T_1$, in which case  we  can
	further use   $\mu$-conjugation and $e$-transfer (see~\ref{subsec:conjugation})
	to reduce to the case where  $\deg B_1=1$ and $\deg A_1=2$ 
	(Reduction~\ref{RD:common-reduction}). After these reductions, establishing the existence of   $L_1\in \Lag(\rho_1 g_1)$
	with $Q_1\oplus L_1=Q_1A_1$
	becomes a technical   check.
	
	
	Unfortunately, it can happen that $\vphi_1$ does not exist. 
	Specifically, in the context of \ref{item:Ei},
	this can happen when $(\sigma,\veps)$ is orthogonal, $[B]=0$
	and $[A]\neq 0$. 
	In order to understand what goes wrong, it is instructive to view $\uLag(\rho_1g )$ as an $R$-scheme
	on which $\uU^0(\rho_1g )$ acts. Suppose that $R$ is connected 
	for simplicity. 
	Then Propositions~\ref{PR:transitive-action-of-uUf} and~\ref{PR:partition-of-Lag}
	imply that $\uLag(\rho_1g)$ is the disjoint union of two   components,
	both being homogeneous $\uU^0(\rho_1g)$-spaces.
	When $\rrk_BQ$ is even, it turns out that  $\vphi_1$
	exists   when $L'$ lives in   one of these two components,
	but not when it lives in the other   (Corollary~\ref{CR:Ei-good-Lag-is-in-Lag-zero}). 
	Moreover, the former  component  may have no $R$-points.
	To overcome this, we put considerable work  into effectively identifying the ``good'' component 
	of  $\uLag(\rho_1g)$ and 
	understanding when does it have $R$-points --- if it is does not, then  
	$(Q,g)$ must be replaced with a Witt equivalent hermitian space.
	When $\rrk_BQ$ is odd, $\vphi_1$ never exists, but then one
	can prove that $g$ must be hyperbolic (Proposition~\ref{PR:Ei-orthII-odd-rank-forms}) and thus
	Witt equivalent to the zero form.
	
\medskip

	The proofs of conditions  \ref{item:Eii}, 
	\ref{item:Eii-tag} and \ref{item:Ei-tag} follow a similar strategy and share
	similar complications, notably when $(\sigma,\veps)$ or $(\tau_2,\veps)$
	are orthogonal. 
	The   cases where   $\vphi_1,\dots,\vphi_t$
	exist are precisely
	the  ones featuring in  parts
	(i)--(iv) of Theorem~\ref{TH:finer-exactness} below.
	
\medskip

	The argument we outlined is carried in Sections~\ref{sec:preparation}--\ref{sec:completion-of-proof}:
	Section~\ref{sec:preparation} collects some preliminary results, Section~\ref{sec:Ei}
	establishes conditions \ref{item:Ei} and \ref{item:Ei-tag},
	Section~\ref{sec:Eii} establishes
	conditions \ref{item:Eii} and \ref{item:Eii-tag}, and the proof of Theorem~\ref{TH:exactness}
	is concluded in Section~\ref{sec:completion-of-proof}, which also brings some of its by-products.
	
	In order to address \ref{item:Ei} and \ref{item:Ei-tag},
	resp.\ \ref{item:Eii} and \ref{item:Eii-tag}, simultaneously, {\it we 
	replace the setting of \ref{subsec:octagon} 
	with the more robust   Notation~\ref{NT:proof-of-Es} below, and use the latter throughout Sections
	\ref{sec:preparation}--\ref{sec:Eii}. We return to the setting of
	\ref{subsec:octagon}  in Section~\ref{sec:completion-of-proof}.}

\section[Preparation]{Preparation for The Proof of Theorem~\ref{TH:exactness}}
\label{sec:preparation}

	This section collects  preliminary results
	that will be used in proving conditions  \ref{item:Eii}, \ref{item:Ei},
	\ref{item:Eii-tag},  \ref{item:Ei-tag} of
	Theorem~\ref{TH:exactness-equiv-conds} when $R$ is    semilocal.
	
	We begin with 
	replacing the setting of \ref{subsec:octagon} 
	with a new one --- Notation~\ref{NT:proof-of-Es} ---
	that will be in use until the end of Section~\ref{sec:Eii}. The reason for
	the change of notation is two-fold: first,
	it will ultimately allow us to treat \ref{item:Ei} and \ref{item:Ei-tag},
	resp.\ \ref{item:Eii} and \ref{item:Eii-tag}, simultaneously,
	and second, the new notation is amenable to $\mu$-conjugation and $e$-transfer
	in the sense of \ref{subsec:conjugation} (see \ref{subsec:simultaneous} for a precise statement).
	We will explain how Notation~\ref{NT:proof-of-Es} specializes to that of \ref{subsec:octagon}
	(in a few possible ways) in Section~\ref{sec:completion-of-proof},
	where we prove Theorem~\ref{TH:exactness}.

	\begin{notation}\label{NT:proof-of-Es}
		Let $(A,\sigma)$ be an Azumaya $R$-algebra with involution
		and let $\veps\in \Cent(A)$ be an element 
		satisfying $\veps^\sigma=\veps$.
		Write $S=\Cent(A)$
		and let $T$ be a quadratic \'etale $S$-subalgebra of $A$
		such that 
		$T^\sigma=T$
		and $\rank_T{A_A}$ is constant
		along the fibers of $\Spec T\to\Spec S$ (we have $A \in\rproj{T}$
		by Lemma~\ref{LM:projective-transfer}).
		Write $B=\Cent_A(T)$ and $\tau=\sigma|_B$.
		The inclusion $S\to T$ is denoted $\iota$,
		and we let $T_0=\{t\in T\suchthat t^\sigma=t\}$.

		We let $\rho$ denote the inclusion map $B\to A$,
		viewed as a homomorphism of $R$-algebras
		with involution $(B,\tau)\to (A,\sigma)$.
		Given $Q\in\rproj{B}$, we abbreviate $Q\otimes_BA$ to $QA$
		and identify $Q$ as a $B$-submodule of $QA$ via $x\mapsto x\otimes 1$
		(this map is injective because $Q_B$ is flat).
		If $(Q,g)\in\Herm[\veps]{B,\tau}$, we define
		$\rho(Q,g)=(QA,\rho g)$ as in \ref{subsec:herm-base-change}.
		
		We let $\pi:A\to B$ denote a homomorphism
		of $(B,B)$-bimodules such that $\pi|_B=\id_B$.
		Given $(P,f)\in \Herm[\veps]{A,\sigma}$,
		we write $\pi(P,f)=(P,\pi f)$, where $\pi f=\pi \circ f$.
		We shall see below (Lemma~\ref{LM:dfn-of-pi-A-B})
		that   $\pi$ exists, is unique,
		and satisfies
		$\pi\circ \sigma=\tau \circ \pi$.
		Moreover, $\pi (P,f)\in\Herm[\veps]{B,\tau}$.
	\end{notation}

	By Proposition~\ref{PR:centralized-in-Az-alg},
	$B$ is Azumaya over
	$T$, $[B]=[A\otimes_S T]$ in $\Br T$ and $\deg B=\frac{1}{2}\iota\deg A$.
	This means that $B$
	is separable projective over $R$, so by Example~\ref{EX:Azumaya-over-fixed-subring},
	$(B,\tau)$ is Azumaya over $T_0$.
	Proposition~\ref{PR:centralized-in-Az-alg} also tells us that
	\[
	\iota\rrk_A P=\rrk_B P
	\qquad\text{and}\qquad
	\iota\rrk_A QA =2\rrk_B Q
	\]
	for all $P\in\rproj{A}$ and $Q\in\rproj{B}$ such that $\rrk_BQ$
	is constant on the fibers 
	of $\Spec T\to \Spec S$. In addition, $A_B\in\rproj{B}$ by Lemma~\ref{LM:projective-transfer}.
	These facts will be used freely and without comment.
	
	We further note that the assumptions of Notation~\ref{NT:proof-of-Es}
	continue to hold if we base change along a ring homomorphism $R\to S$,
	thanks to Lemma~\ref{LM:center-base-change}.

\subsection{Existence and Uniqueness of $\pi$}

	\begin{lem}\label{LM:dfn-of-pi-A-B}
		With Notation~\ref{NT:proof-of-Es}, the following hold:
		\begin{enumerate}[label=(\roman*)]
			\item 
			There exists a unique $(B,B)$-bimodule homomorphism
			$\pi=\pi_{A,B}:A\to B$ such that $\pi|_B=\id_B$.	
			\item If there exists $\lambda\in T$
			such that $\lambda^2\in \units{S}$
			and $T=S\oplus \lambda S$, then $\pi a=\frac{1}{2}(a+\lambda^{-1}a\lambda)$
			for all $a\in A$.
			\item $\pi\circ \sigma=\tau\circ \pi$.
			\item $E:=\ker \pi$ satisfies $A=B\oplus E$, $E\cdot E=B$, $EA=A$ and $\rrk_B E_B=\deg B$.
			\item When $R$ is semilocal, there exists $\lambda$ as in (ii)
			and $\mu\in\units{A}$ such that $E=\mu B=B\mu$, $\lambda\mu=-\mu\lambda$
			and $\pi(b_1+\mu b_2)=b_1$ for all $b_1,b_2\in B$.\footnote{
				Note that  in contrast
				with \ref{subsec:octagon}, 
				we do no require $\lambda^\sigma=-\lambda$ and $\mu^\sigma=-\mu$.
				Indeed, this cannot be guaranteed in general. The situations
				considered in cases (i) and (ii) of Lemma~\ref{LM:strcture-of-quat}
				below are such examples, the reason being that $\sigma|_B=\id_B$ or $\sigma|_E=\id_E$.
			} If $\deg A=2$,
			then we also have $\mu^2\in\units{S}$.
			\item For all $(P,f)\in\Herm[\veps]{A,\sigma}$,
			we have $(P,\pi f)\in\Herm[\veps]{B,\tau}$, where $\pi f:=\pi \circ f$.
		\end{enumerate}
	\end{lem}
	
	\begin{proof}
		(i) Write $T^e=T^\op\otimes_ST$ and $B^e=B^\op\otimes_S B$.
		We view $A$ and $B$ as right $B^e$-modules using their
		evident $(B,B)$-bimodule structure.
		
		Since $T$ is a separable $S$-algebra, 
		the map $\mu:T^e\to T$ sending $x^\op\otimes y$ to $xy$ is split
		as a morphism of $T^e$-modules.
		Let $\xi:T\to T^e$ denote such a splitting,
		and let $e:=\xi(1_T)$. 
		It is well-known that   $e^2=e$
		and $e(1\otimes t)=e(t^\op \otimes 1)$ for all $t\in T$,
		see \cite[Lemma~III.5.1.2]{Knus_1991_quadratic_hermitian_forms}
		and its proof.
		
		Note that $e$ is a central idempotent in $B^e$.
		Thus,  $A=Ae\oplus A(1-e)$,
		and both summands are $B^e$-modules.
		Writing $e=\sum_i u_i^\op\otimes v_i$ with $\{u_i,v_i\}_i\subseteq T$,
		we see that for all $b\in B$, we have
		\[
		b  e=\sum_i u_ibv_i=b\sum_iu_iv_i= b\cdot 1_T=b,
		\]
		because $\sum_iu_iv_i=\mu(e)=1_T$.
		On the other hand, if $a\in Ae$, then for all
		$t\in T$, we have 
		\[
		ta=a(t^\op \otimes 1)=ae(t^\op \otimes 1)=ae(1\otimes t)=
		a(1\otimes t)=at,
		\]
		hence $a\in B$. We conclude that $B=eA$.
		This in turn means that $\pi:a\mapsto  ae$ is a $B^e$-module
		homomorphism, or equivalently,
		a $(B,B)$-bimodule homomorphism, which splits the inclusion $B\to A$.
		
		If $\pi':A\to B$ is another $(B,B)$-module homomorphism splitting
		$B\to A$,
		then $B=Ae\subseteq \ker(\pi-\pi')$. On the other hand, 
		since multiplying on the right by $e$ annihilates $A(1-e)$ while fixing $B$,
		we have   $A(1-e)\subseteq \ker \pi'$.
		Since $\ker \pi=A(1-e)$,
		this means that $\ker (\pi-\pi')\supseteq Ae+A(1-e)=A$,
		so $\pi=\pi'$.
		
		(ii) Using $\lambda^2\in S$ 
		and $B=\Cent_A(T)=\Cent_A(\lambda)$, it is routine to check 
		that $a\mapsto \frac{1}{2}(a+\lambda^{-1}a\lambda)$
		is a $(B,B)$-bimodule homomorphism from $A$ to $B$ 
		which restricts to the identity on $B$.
		This map must  be $\pi$ by (i). 
		
		(iii) The uniqueness of $\pi$   implies that
		$\tau\circ \pi\circ\sigma=\pi$, or rather,
		$\pi\circ \sigma=\tau\circ \pi$.

		(iv) That $A=B\oplus E$ follows from the fact that $\pi:A\to B$ splits
		the inclusion  $A\to B$. 
		Since $\rrk_B A_B=\iota\deg A=2\deg B$,
		this means that $\rrk_B E=\deg B$.

		We proceed with checking that 
		$E\cdot E\subseteq B$. It is enough to prove that this
		holds after localizing at $\frakp$ for all $\frakp\in\Spec R$,
		so we may assume that $R$ is local.
		In this case, by Lemma~\ref{LM:quad-etale-over-semilocal},
		there exists $\lambda\in T$
		such that $\lambda^2\in\units{S}$
		and $T=S\oplus \lambda S$. By (ii), 
		$E$ consists of the elements which anti-commute
		with $\lambda$ while
		$B=\Cent_A(\lambda)$. Thus, $E\cdot E\subseteq B$.
		
		Now, since $A=B\oplus E$ and $BEB\subseteq E$, the set $E^2+E$ is a two sided	
		ideal of $A$. Since $A$ is Azumaya over $S$, there exists
		$I\idealof S$ such that $E^2+E=IA=IB+IE$ \cite[Lemma II.3.5]{DeMeyer_1971_separable_algebras}.
		It follows that  $E=IE$ and $E^2=IB$.
		If $I\neq S$, then $\ann_S E\neq 0$ by Nakayama's Lemma,
		which is impossible
		because $\rrk_BE=\deg B>0$.
		Thus, $I=S$ and $E^2=IB=B$.
		This also means that   $EA\supseteq E^2A=BA=A$. 
		
		(v) The existence of $\lambda$ follows from Lemma~\ref{LM:quad-etale-over-semilocal}.
		By (iv), $\rrk_B E_B=\deg B$,
		so 
		$E_B\cong B_B$ by Lemma~\ref{LM:rank-determines}.
		Let $\mu$ be a generator of $E_B$.
		By (ii), $0=\pi(\mu)=\frac{1}{2}(\mu+\lambda^{-1}\mu\lambda)$,
		so $\lambda\mu=-\mu\lambda$. By  (iv),
		$B=E^2=\mu B\mu B\subseteq \mu A$, so $\mu$ is invertible
		on the right. This means that $\Nrd_{A/S}(\mu)\in\units{S}$, hence  $\mu\in\units{A}$.
		Now, since $B=\Cent_A(\lambda)$ and $E$ is the set of elements
		in $A$ which anti-commute with $\lambda$ (by (ii)), we have $E=\mu B=B\mu$.
		This means that the map $A=B\oplus \mu B\to B$ sending $b_1+\mu b_2$ to $b_1$ ($b_1,b_2\in B$)
		is a $(B,B)$-bimodule homomorphism which restricts to the
		identity on $B$. Therefore, it must coincide with $\pi$.
		Finally, if $\deg A=2$, then $\deg B=1$, so $B=T$
		and $A$ is generated as an $S$-algebra by $\lambda$ and $\mu$.
		Since $\mu^2$ commutes with both $\lambda$ and $\mu$, 
		we have $\mu^2\in\units{\Cent(A)}=\units{S}$.
		
		(vi) 
		It is straightforward to check that $\pi f$ is
		an $\veps$-hermitian form. We need to show that $\pi f$ is unimodular.
		Using (iv),
		choose $\{u_i,v_i\}_{i=1}^t\subseteq E$
		such that $\sum_i u_iv_i=1$. Given $\phi\in \Hom_B(P,B)$,
		define $\hat{\phi}:P\to A$
		by $\hat{\phi}x=\phi x+\sum_i\phi(xu_i)v_i$.
		We claim that $\hat{\phi}\in \Hom_A(P,A)$.
		Indeed,  $\hat{\phi}$ is additive, and for all
		$b\in B$, $b'\in E$ and $x\in P$, we have
		$\hat{\phi}(xb)=\phi (xb)+\sum_i\phi(x bu_i)v_i=
		\phi x\cdot b+\sum_{i,j}\phi(xv_ju_jbu_i)v_i=
		\phi x\cdot b+\sum_{i,j}\phi(xv_j)u_jbu_iv_i=
		\phi x\cdot b+\sum_{j}\phi(xv_j)u_jb=\hat{\phi} x\cdot b$
		and 
		$\hat{\phi}(x b')=\phi (xb')+\sum_i\phi(x b'u_i)v_i=
		\sum_i \phi(xu_iv_ib')+\sum_i\phi x\cdot b'u_iv_i=
		\sum_i\phi (xu_i)v_ib'+\phi x\cdot b'
		=\hat{\phi} x\cdot b'$.
		A similar computation shows that $\phi\mapsto \hat{\phi}:\Hom_B(P,B)\to \Hom_A(P,A)$
		defines an inverse to $\xi\mapsto \pi\circ \xi:\Hom_A(P,A)\to \Hom_B(P,B)$.
		The composition of the latter map with $x\mapsto f(x,-):P\to \Hom_A(P,A)$
		is precisely $x\mapsto \pi f(x,-):P\to \Hom_B(P,B)$,
		so this map is bijective and $\pi f$ is unimodular.
	\end{proof}
	
\subsection{Some Structural Results}

	\begin{lem} \label{LM:strcture-of-quat}
		With Notation~\ref{NT:proof-of-Es},
		suppose that $R$ is semilocal, $S=R$ and
		$\deg B=1$. Then there exist
		$\lambda,\mu\in A$ such that 
		$\lambda^2,\mu^2\in\units{R}$, $\lambda\mu=-\mu\lambda$,
		$\{1,\lambda\}$ is an $R$-basis of $B$,
		$\{1,\lambda,\mu,\mu\lambda\}$ is an $R$-basis of $A$,
		and:
		\begin{enumerate}[label=(\roman*)]
			\item $\lambda^\sigma=\lambda$ and $\mu^\sigma=-\mu$ if $\tau=\id_T$;
			\item $\lambda^\sigma=-\lambda$ and $\mu^\sigma=\mu$ if $\tau$
			is unitary 
			and $\sigma$ is orthogonal;
			\item $\lambda^\sigma=-\lambda$ and $\mu^\sigma=-\mu$ if 
			$\sigma$ is symplectic.
		\end{enumerate}
		Furthermore, 
		$\pi(b_1+\mu b_2)=b_1$ for all $b_1,b_2\in B$.
	\end{lem}

	\begin{proof}
		Let $\lambda$ and $\mu$ be as in Lemma~\ref{LM:dfn-of-pi-A-B}(v).
		Then all the requirements are fulfilled except, maybe, (i)--(iii).
		Note also that we may replace $\mu$ with any element of $\mu \units{B}$.
		
		Let $E=\ker \pi$. Since $E^\sigma=E$, $B^\sigma=B$ and $A=B\oplus E$,
		we have 
		$\Sym_{1}(A,\sigma)=\Sym_{1}(B,\sigma)\oplus \Sym_1(E,\sigma)$ and
		$\Sym_{-1}(A,\sigma)=\Sym_{-1}(B,\sigma)\oplus \Sym_{-1}(E,\sigma)$.
		
		Suppose that $\tau=\id_T$. Then $\lambda^\sigma=\lambda$
		and
		$\Sym_{-1}(A,\sigma)=\Sym_{-1}(B,\sigma)\oplus \Sym_{-1}(E,\sigma)=
		\Sym_{-1}(E,\sigma)$. By Lemma~\ref{LM:invertible-symmetric-elements},
		there exists $\mu'\in \Sym_{-1}(A,\sigma)\cap \units{A}$.
		Since $\mu'\in E=\mu B$, we have $\mu'=\mu t$ for some $t\in T$,
		so we may replace $\mu$ with $\mu'$ and finish.
		
		Suppose that $\tau$ is unitary
		and $\sigma$ is orthogonal. Then $\tau$ is the standard
		$R$-involution of $T$, so $\lambda^\sigma=-\lambda$.
		By Proposition~\ref{PR:types-of-involutions-Az},
		we have $1=\rank_R \Sym_{-1}(A,\sigma)=
		\rank_R \Sym_{-1}(B,\sigma)+ \rank_R\Sym_{-1}(E,\sigma)=
		1+ \rank_R\Sym_{-1}(E,\sigma)$,
		hence $\Sym_{-1}(E,\sigma)=0$.
		Since $E=\Sym_{1}(E,\sigma)\oplus \Sym_{-1}(E,\sigma)$, this means
		that $E=\Sym_1(E,\sigma)$ and $\mu^\sigma=\mu$.
		
		Finally, when $\sigma$ symplectic,  
		using Proposition~\ref{PR:types-of-involutions-Az}
		and the fact that $R$ is a summand of $B$,
		one finds that
		$1=\rank_R \Sym_{1}(A,\sigma)\geq 
		\rank_R \Sym_{1}(R,\sigma)+ \rank_R\Sym_{1}(E,\sigma)=
		1+ \rank_R\Sym_{1}(E,\sigma)$, hence $\Sym_1(E,\sigma)=0$.
		This means that $E=\Sym_{-1}(E,\sigma)$, so $\mu^\sigma=-\mu$
		and $(\lambda\mu)^\sigma=-\lambda\mu$.
		Now, $\lambda=-\mu\lambda\mu^{-1} =(\mu\lambda)^\sigma\mu^{-1}=
		\lambda^\sigma \mu^\sigma\mu^{-1}=-\lambda^{\sigma}$.
	\end{proof}

	\begin{lem}\label{LM:idempotent-in-T}
		With Notation~\ref{NT:proof-of-Es},		
		suppose that 
		$T\cong S\times S$ as $S$-algebras,
		and let $e,e'\in T$ correspond to $(1_S,0_S), (0_S,1_S)$
		under this isomorphism. Then:
		\begin{enumerate}[label=(\roman*)]
			\item $B=eAe+e'Ae'$.
			\item $\rrk_AeA>0$ and $AeA=A$.
			\item 
			$eB$ is an Azumaya $S$-algebra 
			of degree $\frac{1}{2}\deg A$ and $[eB]=[A]$ in $\Br S$.
			\item 
			For every $Q\in\rproj{B}$, we have
			$\rrk_A QA=\rrk_{eB}Qe+\rrk_{e'B}Qe'$.
			\item $\pi:A\to B$ is given by $\pi a=eae+e'ae'$.
		\end{enumerate}
	\end{lem}

	\begin{proof}
		(i) 
		Since $T=S[e]$, we have $B=\Cent_A(e)$.
		Using the 	Peirce decomposition of $A$ relative to $e$
		(i.e.\ $A=eAe\oplus eAe'\oplus e'Ae\oplus e'Ae'$ as abelian groups),
		it routine to check   that 	
		$\Cent_A(e)=eAe+e'Ae'$.
		
		(ii) 
		Since
		${}_BB$ is a summand of ${}_BA$ (Lemma~\ref{LM:dfn-of-pi-A-B}(iv))
		and $T$ is a $T$-summand of $B$ 
		\cite[Proposition~2.4.6(1)]{Ford_2017_separable_algebras},
		$eT\cong S$ is an $S$-summand of $eA$, hence  $\rrk_AeA>0$.
		That $AeA=A$ follows from 
		Corollary~\ref{CR:degree-of-endo-ring}.		
		
		(iii)
		By (i), $eB=e(eAe+e'Ae')=eAe$,
		and
		by Corollary~\ref{CR:degree-of-endo-ring} and (ii),
		$eAe$ is Azumaya over $S$ and $[eAe]=[A]$.
		That $\deg eB= \frac{1}{2}\deg A$
		follows  from  
		$\deg B=\frac{1}{2}\iota\deg A$.

		(iv) Since $Q=Qe\oplus Qe'$ as $B$-modules, $QA=QeA\oplus Qe'A$,
		so 
		it is enough to check that $\rrk_{eB} Qe=\rrk_A QeA$. 
		By Corollary~\ref{CR:degree-of-endo-ring} and (ii),
		$\rrk_A QeA=\rrk_{eAe} QeAe=\rrk_{eB}Qe$.

		(v) Using (i), it is easy to check that $a\mapsto eae+e'ae'$
		is a homomorphism of $(B,B)$-bimodules
		which splits $B\to A$. Thus, it must
		coincide with $\pi$. 
	\end{proof}

\subsection{The Types of $(\sigma,\veps)$ and $(\tau,\veps)$}

	\begin{lem}\label{LM:sigma-unit-means-tau-unit}
		With Notation~\ref{NT:proof-of-Es},
		if $\sigma$ is unitary, then so is $\tau$.
	\end{lem}
	
	\begin{proof}
		It is enough to prove this when $R$ is a field.
		Recall that  
		$T_0=\Sym_1(T,\sigma)$
		and let
		$T_1=\Sym_{-1}(T,\sigma)$.
		Since $2\in\units{R}$, we have $T=T_0\oplus T_1$.
		Since $\sigma$ is unitary, $S$ is quadratic
		\'etale over $R$ and $\sigma|_S$ is the standard
		$R$-involution of $S$.
		By Lemma~\ref{LM:quad-etale-over-semilocal},
		there exists $\lambda\in \units{S}$ such that $\lambda^\sigma=-\lambda$.
		One readily checks that $t\mapsto\lambda t:T_0\to  T_1$
		is a $T_0$-module isomorphism, so $\rank_{T_0}T=2$. 
		We observed in the comment after Notation~\ref{NT:proof-of-Es}
		that $(B,\tau)$ is Azumaya over $T_0$.
		Since $\rank_{T_0}T=2$, Proposition~\ref{PR:types-of-involutions-Az}(iii)
		implies that $\tau$ is unitary.
	\end{proof}

	\begin{lem}\label{LM:tau-type-is-constant}
		With Notation~\ref{NT:proof-of-Es},
		if $R$ is connected, then the type
		of  $(\tau,\veps)$ is constant on $\Spec T_0$ (see~\ref{subsec:Az-alg-inv}).
	\end{lem}

	\begin{proof}
		By Proposition~\ref{PR:types-of-involutions-Az}(v),
		$(\sigma,\veps)$ is either orthogonal, symplectic or unitary.
		If $(\sigma,\veps)$ is unitary, then the
		lemma follows from Lemma~\ref{LM:sigma-unit-means-tau-unit}.
		Suppose that $(\sigma,\veps)$ is orthogonal or symplectic.
		Then $S=R$. If $T$ is connected, then the type
		of $(\tau,\veps)$ is constant (Proposition~\ref{PR:types-of-involutions-Az}(v)),
		so assume that $T$ is not connected.
		Now, by Lemmas~\ref{LM:non-connected-S} and~\ref{LM:invs-of-quad-et-algs},
		$T=R\times R$ and $\tau|_T$ is either the standard
		$R$-involution of $T$, or $\id_T$.
		In the former case, $(\tau,\veps)$ is unitary,
		so assume $\tau|_T=\id_T$.
		Let $e=(1_R,0_R)$ and $e'=(0_R,1_R)$.
		By
		Lemma~\ref{LM:idempotent-in-T}(i),
		$(B,\tau)=(eAe,\sigma|_{eAe})\times (e'Ae',\sigma|_{e'Ae'})$,
		and
		by Corollary~\ref{CR:type-conjugation}(ii),
		$(\sigma|_{eAe},e\veps)$ and $(\sigma|_{e'Ae'},e'\veps)$ have the same
		type as $(\sigma,\veps)$, so	
		the type of $(\tau,\veps)$ is   constant.
	\end{proof}

%

\subsection{Simultaneous  Conjugation and  Transfer}
\label{subsec:simultaneous}

	We check that the setting of Notation~\ref{NT:proof-of-Es}
	is compatible with $\mu$-conjugation and $e$-transfer
	(see \ref{subsec:conjugation}).

	\begin{prp}\label{PR:simult-conj}
		With Notation~\ref{NT:proof-of-Es},
		let  $(P,f)\in\Herm[\veps]{A,\sigma}$
		and $(Q,g)\in\Herm[\veps]{B,\tau}$.
		Let $\delta\in S$ 
		satisfy $\delta^\sigma\delta=1$
		and let $\mu\in \Sym_\delta(B,\tau)\cap\units{B}$.
		Then:
		\begin{enumerate}[label=(\roman*)]
			\item The assumptions
			of Notation~\ref{NT:proof-of-Es} continue to hold
			upon replacing $\sigma,\tau,\veps$
			with $\Int(\mu)\circ \sigma,\Int(\mu)\circ \tau,\delta\veps$.
			\item
			$\rho (\mu g)=\mu (\rho g)$ and $\pi (\mu f)=\mu (\pi f)$
			(notation as in \ref{subsec:conjugation}).
		\end{enumerate}
	\end{prp}
	
	\begin{proof}
	
		(i) We need to check that $\Int(\mu)\circ \sigma$ is an involution
		of $A$ which restricts to an involution of $T$, and
		$(\delta\veps)^{\Int(\mu)\circ \sigma}(\delta \veps)=1$. Noting
		that $\mu\in B=\Cent_A(T)$, $\mu^\sigma=\delta^{-1}\mu$ 
		and $T^\sigma=T$, this follows by straightforward computation.
		
		(ii) Let $y,y'\in Q$ and $a,a'\in A$. Then
		$\rho(\mu g)(y\otimes a,y'\otimes a')=
		a^{\Int(\mu)\circ \sigma}\mu g(y,y') a'=
		\mu a^\sigma \mu^{-1}\mu g(y,y')a'=\mu (\rho g)(y\otimes a,y'\otimes a')$,
		so $\rho (\mu g)=\mu (\rho g)$.
		Now let $x,x'\in P$. Then
		$\pi(\mu f)(x,x')=\pi(\mu \cdot f(x,x'))=\mu\cdot \pi(f(x,x'))=\mu(\pi f)(x,x')$,
		where the second equality holds because $\pi$ is a left $B$-module homomorphism
		and $\mu\in B$.
	\end{proof}
	
	\begin{lem}\label{LM:simult-e-transfer-tensor}
		With Notation~\ref{NT:proof-of-Es},
		let $e\in B$ be an idempotent with $\rrk_B eB>0$
		and let $Q\in\rproj{B}$.
		Then the map $\xi=\xi_Q:Qe\otimes_{eBe}eAe\to Q\otimes_BAe$
		determined by $\xi(x\otimes a)=x\otimes a$ ($x\in Qe$, $a\in eAe$)
		is an isomorphism of $eAe$-modules.
	\end{lem}

	\begin{proof}
		By Proposition~\ref{CR:degree-of-endo-ring}, $BeB=B$.
		Choose elements $\{u_i,v_i\}_{i=1}^t\subseteq B$
		such that $\sum_i u_iev_i=1$,
		and consider
		the map    $\psi:Q\otimes_BAe\to Qe\otimes_{eBe}eAe$
		determined by
		$x\otimes a\mapsto \sum_i xu_i e\otimes ev_ia$ ($x\in Q$, $a\in Ae$).
		It is well-defined because
		for all $b\in B$, 
		we have $\psi(xb\otimes a)=\sum_ixbu_ie\otimes ev_ia=
		\sum_{i,j}xu_jev_jbu_ie\otimes ev_ia=
		\sum_{i,j}xu_je\otimes ev_jbu_iev_ia=
		\sum_jxu_je\otimes ev_jba=\psi(x\otimes ba)$.
		A similar computation shows that $\psi$ is an inverse of $\xi$.
	\end{proof}
	
	\begin{prp}\label{PR:simult-e-transfer}
		With Notation~\ref{NT:proof-of-Es},
		let  $(P,f)\in\Herm[\veps]{A,\sigma}$
		and $(Q,g)\in\Herm[\veps]{B,\tau}$.
		Let
		$e\in B$ be an idempotent
		such that $e^\sigma =e$ and $\rrk_BeB$ is positive
		and constant along the fibers of $\Spec T\to \Spec S$.
		Write $\tau_e=\tau|_{eBe}$, $\sigma_e=\sigma|_{eAe}$,
		$\rho_e=\rho|_{eBe}$, $\pi_e=\pi|_{eAe}$, $f_e=f|_{Pe\times Pe}$,
		$g_e=g|_{Qe\times Qe}$ (see \ref{subsec:conjugation}).
		Then:
		\begin{enumerate}[label=(\roman*)]
			\item The assumptions
		of Notation~\ref{NT:proof-of-Es} apply
		upon replacing $A,\sigma,\veps,T,B,\tau,\rho,\pi$
		with $eAe,\sigma_e,e\veps, eT,eBe,\tau_e,\rho_e,\pi_e$.

			\item 	
			Upon identifying $Q\otimes_{eBe}eAe$
			with $Q\otimes_B Ae$ as in Lemma~\ref{LM:simult-e-transfer-tensor},
			we have $\rho_e g_e=(\rho g)_e$.
			Furthermore, 
		the map $L\mapsto Le$
		defines a bijection from the Lagrangians of $\rho g$
		to the Lagrangians of $\rho_e g_e$ and, for a Lagrangian $L$ of $\rho g$, we have 
		$L\oplus Q=QA$ (as $B$-modules) if and only if
		$Le\oplus Qe=QAe$.

		\item $\pi_e f_e =(\pi f)_e$,
		the map $M\mapsto Me$ is a bijection between
		the Lagrangians of $\pi f$ and the Lagrangians of $\pi_e f_e$
		and, for a Lagrangian $M$ of $\pi f$,
		we have   $MA=P$ if and only if $Me\cdot eAe=Pe$.
		\end{enumerate}
	\end{prp}
	
	\begin{proof}		
		By Proposition~\ref{CR:degree-of-endo-ring},
		$BeB=B$, hence $AeA=ABeBA=ABA=A$, and so $eA_A$ is a progenerator.
		Thus, we can use  \ref{item:e-transfer-first}--\ref{item:e-transfer-last} in
		\ref{subsec:conjugation}
		for both $(B,\tau)$ and $(A,\sigma)$.

		(i) Everything is straightforward
		except the fact that 		
		$\rank_{ T} (eAe_{ eAe})$
		is constant along the fibers of $\Spec  T\to \Spec S$.
		To see this, we use Corollary~\ref{CR:degree-of-endo-ring}
		to get $\rank_{ T} (eAe_{ eAe})=\deg eBe\cdot \rrk_{eBe}(eAe_{eBe})=
		\deg eBe\cdot \rrk_B eA_B=\rrk_B eB\cdot \iota \rrk_A eA=
		\rrk_B eB\cdot 2 \rrk_B eB=2(\rrk_BeB)^2$.
		Since $\rrk_B eB$ is constant along the fibers of $\Spec T\to \Spec S$,
		so is $\rank_T(eAe_{eAe})$.
		
		(ii) 
		This is straightforward; use facts \ref{item:e-transfer-first}--\ref{item:e-transfer-last} in
		\ref{subsec:conjugation} and Morita theory.

		(iii) That $(\pi f)_e=\pi_e f_e$ is straightforward.
		The second assertion is \ref{item:e-transfer-Lags} in \ref{subsec:conjugation}.
		For the third assertion, note that $MA=P$ implies
		$Me\cdot eAe=M(AeA)e=MAe=Pe$, and conversely,
		$Me\cdot eAe=Pe$ implies $MA=M(AeAeA)=Me\cdot eAe\cdot eA=Pe\cdot eA=P(AeA)=PA=P$.
	\end{proof}
	
\subsection{Two Important Reductions}
	
	\begin{reduction}\label{RD:common-reduction}	
		With Notation~\ref{NT:proof-of-Es}, suppose 
		that $R$ is connected   semilocal and $[B]=0$ in $\Br T$
		(note   that $[A]=0$ in $\Br S$ implies $[B]=[A\otimes_S T]=0$).
		We claim that verifying statements 
		about hermitian forms over $(A,\sigma)$ and
		$(B,\tau)$
		which are amenable to  conjugation and
		transfer  (in the sense of \ref{subsec:conjugation}),
		e.g., the statements
		$L\oplus Q=QA$ and $MA=P$
		from parts
		(ii) and (iii) of Proposition~\ref{PR:simult-e-transfer}, 
		can be reduced into verifying them in the following
		setting, and without affecting the
		the types of $(\sigma,\veps)$ and $(\tau,\veps)$:
		\begin{itemize}
			\item $\deg B=1 $,  i.e.\ $B=T$,
			\item $\tau$ is orthogonal or unitary,
			\item $\sigma$ is orthogonal or unitary.
		\end{itemize}

		This is done as follows: 
		Applying Proposition~\ref{PR:types-of-involutions-Az}(v)
		and	Lemma~\ref{LM:tau-type-is-constant} with $\veps=1$,
		we see that the types of $\sigma$ and $\tau$ are constant.
		If $\tau$ is not orthogonal or unitary,
		then it is symplectic. In this case, by 
		Lemma~\ref{LM:invertible-symmetric-elements},
		there exists $\mu\in \Sym_{-1}(B,\tau)\cap \units{B}$.
		By Proposition~\ref{PR:simult-conj},
		we may apply $\mu$-conjugation 
		and replace $\sigma,\tau,\veps$ with $\Int(\mu)\circ \sigma,\Int(\mu),-\veps$,
		thus changing $\tau$ into an orthogonal involution.
		
		Next, by Theorem~\ref{TH:invariant-primitive-idempotent},
		there exists an idempotent $e\in B$ such that $e^\tau=e$
		and $\rrk_B eB=\deg eBe=\ind B=1$.
		By Proposition~\ref{PR:simult-e-transfer},
		we may apply $e$-transfer
		and replace $A,\sigma,\veps,T,B,\tau,\rho,\pi$
		with $eAe,\sigma_e,e\veps, eT,eBe,\tau_e,\rho_e,\pi_e$
		and get $\deg B=\ind B=1$.
 		
		Finally, if $\sigma$ is not orthogonal or unitary,
		then it is symplectic.
		In this case, by Lemma~\ref{LM:strcture-of-quat},
		there exists $\lambda\in\units{T}$ with $\lambda^\sigma=-\lambda$.
		By Proposition~\ref{PR:simult-conj}, we can apply
		$\lambda$-conjugation
		and replace  $\sigma, \veps$ with $\Int(\lambda)\circ \sigma,\Int(\lambda),-\veps$,
		turning $\sigma$ into an orthogonal involution and leaving  $\tau$ unchanged. 		 		
	\end{reduction}
	
	\begin{reduction}\label{RD:common-reduction-II}
		Assume that $R$ is connected semilocal and $[A]=0$ in $\Br S$.
		After performing   Reduction~\ref{RD:common-reduction},
		Proposition~\ref{PR:involution-is-adjoint}
		implies that $\sigma$
		is  adjoint to a unimodular  binary $\delta$-hermitian form over $(S,\sigma|_S)$,
		with $\delta=1$ if $\sigma$ is orthogonal.
		This form can be diagonalized by Proposition~\ref{PR:diagonalizable-herm-forms},
		so, by Example~\ref{EX:adoint-inv-diag-form},
		we may
		assume that $A=\nMat{S}{2}$ and $\sigma$ is given by 
		$[\begin{smallmatrix}a & b \\ c & d\end{smallmatrix}]
		\mapsto[\begin{smallmatrix} a^\sigma & \alpha c^\sigma \\ \alpha^{-1} b^\sigma & 
		d^\sigma \end{smallmatrix}]$
		for some $\alpha\in\Sym_1(S,\sigma)\cap\units{S}=\units{R}$. 
	\end{reduction}

\subsection{Miscellaneous Results}

	\begin{lem}\label{LM:tau-orth-implies-sigma-orth}
		With Notation~\ref{NT:proof-of-Es},
		if $(\tau,\veps)$ is orthogonal,
		then $(\sigma,\veps)$ is orthogonal.
	\end{lem}
	
	\begin{proof}
		It is enough to prove the lemma after specializing 
		$R$ to the algebraic closure of each of its residue fields, so assume
		$R$ is an algebraically closed field.
		Then $[B]=0$. 
		We apply Reduction~\ref{RD:common-reduction}
		to assume that   $\deg A=2$ and $\tau$ is orthogonal. 
		Since $(\tau,\veps)$ is orthogonal, $\veps=1$.		
		By Lemma~\ref{LM:strcture-of-quat}(i),
		$\dim_R \Sym_{1}(A,\sigma)=3$, so  $(\sigma,\veps)$ is orthogonal
		by Proposition~\ref{PR:types-of-involutions-Az}.
	\end{proof}
	
	\begin{lem}\label{LM:B-endo-of-P}
		With Notation~\ref{NT:proof-of-Es}, let $P\in\rproj{A}$. 
		\begin{enumerate}[label=(\roman*)]
			\item The map $\End_A(P)\otimes_S T\to\End_B(P)$
			given by sending $\psi\otimes t$ to $[x\mapsto \psi x\cdot t]$
			is an isomorphism of $T$-algebras.
			\item For all $\psi\in\End_A(P)$,
			we have $\Nrd_{\End_A(P)/S}(\psi)=\Nrd_{\End_B(P)/T}(\psi)$ in $T$.
		\end{enumerate}	
	\end{lem}
	
	\begin{proof}
		We may assume that $\rrk_AP>0$, otherwise
		write $R=R_0\times \ann_RP$ (use \cite[Proposition~1.1.15]{Ford_2017_separable_algebras}) 
		and work over $R_0$.
	
		(i)  
		By Proposition~\ref{PR:degree-of-endo-ring}(i),
		$\End_A(P)$ is Azumaya over $S$,
		$\End_B(P)$ is Azumaya over $T$
		and
		$\deg \End_B(P)=\rrk_BP=\iota \rrk_AP=\iota\deg \End_A(P)=
		\deg \End_A(P)\otimes_ST$.
		Thus, the map $\End_A(P)\otimes_ST\to \End_B(P)$
		is a homomorphism of Azumaya $T$-algebras of equal degrees.
		By 
		\cite[Corollary~III.5.1.18]{Knus_1991_quadratic_hermitian_forms},		
		such a homomorphism is always an isomorphism.
		
		(ii)  This follows from (i) and the fact that reduced norm is preserved under base-change.
	\end{proof}

	\begin{prp}\label{PR:rationality-for-Ei}
		With Notation~\ref{NT:proof-of-Es},
		suppose that $R$ is an infinite field, 
		and
		let $\quo{R}$
		be an algebraic closure of $R$.
		Let $(Q,g)\in \Herm[\veps]{B,\tau}$ and let
		$L$ be a Lagrangian of $\rho g$.
		Let  $(P,f)\in\Herm[\veps]{A,\sigma}$ and let $M$ be a Lagrangian of $\pi f$.
		Then:
		\begin{enumerate}[label=(\roman*)]
			\item  
			If there exists $\vphi\in U^0(\rho g_{\quo{R}})$
			such that $Q_{\quo{R}}\oplus \vphi (L_{\quo{R}})=QA_{\quo{R}}$,
			then there exists $\psi\in U^0(\rho g)$
			such that $Q\oplus \psi L=QA$.
			\item  
			If there exists $\vphi\in U^0(\pi f_{\quo{R}})$
			such that $\vphi M\cdot A_{\quo{R}}=P_{\quo{R}}$,
			then there exists $\psi\in U^0(\pi f)$
			such that $\psi M\cdot A=P$.
		\end{enumerate}
		(See \ref{subsec:isometry-group} for the definition
		of $U^0(-)$.)
	\end{prp}

	\begin{proof}
		(i) 
		Consider  $\uU^0(\rho g)$ as a functor from $R$-rings
		to groups and define a subfunctor
		$R_1\mapsto \bfX(R_1)$ of $\uU^0(\rho g)$ by
		\[
		\bfX(R_1)=\{\psi\in U^0(\rho g_{R_1})\suchthat Q_{R_1}\oplus \psi (L_{R_1})=QA\}.
		\]
		We claim that $\bfX$ is represented by open affine subscheme of $\uU^0(\rho g)$,
		also denoted   $\bfX$.
		
		To see this, fix $R$-bases $\{x_i\}_{i=1}^r$, $\{v_i\}_{i=1}^s$, $\{y_i\}_{i=1}^{r+s}$
		to $Q$, $L$, $QA$, respectively. These will also be viewed
		as $R_1$-bases of $Q_{R_1}$, $L_{R_1}$, $QA_{R_1}$.
		The group $U^0(\rho g_{R_1})$
		is the zero locus of certain polynomial functions
		on $\End_{R_1}(QA_{R_1})\cong R_1^{(r+s)^2}$ with coefficients
		in $R$. Thus, it is enough to show that there exists a
		polynomial   $\xi\in R[x_{11},x_{12},\dots,x_{(r+s)(r+s)}]$ 
		such that $\bfX(R_1)=\{\psi \in U^0(\rho g_{R_1})\suchthat \xi(\psi)\in \units{R_1}\}$.
		To that end, given $y\in QA_{R_1}$,
		let $[y]$ denote the vector $(\alpha_1,\dots,\alpha_{r+s})\in R_1^{r+s}$
		for which $y=\sum_iy_i\alpha_i$. Then
		the function sending $a\in R_1^{r^2}\cong  \End_{R_1}(Q_{R_1})$
		to the determinant of the $(r+s)\times (r+s)$
		matrix with columns $[v_1],\dots,[v_s],[ax_1],\dots,[ax_r]$
		is   a polynomial $\xi\in R[x_{11},x_{12},\dots,x_{(r+s)(r+s)}]$
		having the desired property.
		
		By Theorem~\ref{TH:rational-variety},
		the irreducible $R$-variety $\uU^0(\rho g)$
		is  rational.
		By the previous paragraph, $\bfX$ is an open subvariety
		of $\uU^0(\rho g)$ and it is nonempty because $\vphi\in \bfX(\quo{R})$.
		Thus, $\bfX$ is also rational.
		Since  rational varieties
		have points over any infinite field,  
		$\bfX(R)\neq \emptyset$ and the existence of $\psi$ follows.
		
		(ii) 	
		This is similar to (i), but one
		uses an open subscheme of $\uU^0(\pi f)$
		defined as follows:
		Write $r=\dim_RP$.
		Since $\vphi  M_{\quo{R}} \cdot A_{\quo{R}}=P_{\quo{R}}$,
		there exist   pairs
		$\{(m_i,a_i)\}_{i=1}^r\subseteq M\times A$
		such that $\{\vphi m_i\cdot a_i\}_{i=1}^r$
		forms an $\quo{R}$-basis to $P_{\quo{R}}$.
		Given an $R$-ring $R_1$,
		define $\bfX(R_1)$
		to be the set of $\psi\in U^0(\pi f_{R_1})$
		such that $\{ \psi m_i\cdot a_i\}_{i=1}^r$
		is an $R_1$-basis to $P_{R_1}$. 
	\end{proof}

\section{Verification of (E2) and (E4)}
\label{sec:Ei}

	Keep the assumptions of Notation~\ref{NT:proof-of-Es}.
	The purpose of this section is to prove:

	\begin{thm}\label{TH:Ei-holds}
		With Notation~\ref{NT:proof-of-Es}, suppose that $R$
		is semilocal and
		let $(Q,g)\in\Herm[\veps]{B,\tau}$.
		Assume that   $[\rho g]=0$ in $W^\veps(A,\sigma)$.
		Then:
		\begin{enumerate}[label=(\roman*)]
		\item When $T$ is connected,
		there exists a Lagrangian $L$ of $\rho g$
		such that $L\oplus Q=QA$ if and only if:
			\begin{enumerate}[label=(\arabic*)]
			\item
			$(\sigma,\veps)$ is not orthogonal, or 
			\item  
			$(\tau,\veps)$ is not unitary, or
			\item $[A]=0$ in $\Br S$, or 
			\item $[B]\neq 0$ in $\Br T$, or
			\item $(\tau,\veps)$ is unitary, 
			$[B]=0$, $\rrk_BQ$ is even
			and $[D(g)]=\frac{\rrk_BQ}{2}  [A]$;
			here, $D(g)$ is the discriminant algebra of $g$, see \ref{subsec:disc}.
			\end{enumerate}
			When none of 
			(1)--(5) hold, $\rrk_BQ$ is even,  $[D(g)]=(\frac{\rrk_BQ}{2}+1)\cdot [A]$
			and $g$ is isotropic.
		\item There exists $(Q',g')\in\Herm[\veps]{B,\tau}$ with $[g]=[g']$
    	and   a Lagrangian $L$ of $\rho g'$
    	such that  $L\oplus Q'=Q'A$.
		\end{enumerate}
	\end{thm}
	
	In Section~\ref{sec:completion-of-proof}, we will use Theorem~\ref{TH:Ei-holds}
	to establish    conditions \ref{item:Ei}
	and \ref{item:Ei-tag} of Theorem~\ref{TH:exactness-equiv-conds} when
	$R$ is semilocal. The reader can skip to the next section without loss of continuity.
	
\medskip
	
	It is enough
	to prove Theorem~\ref{TH:Ei-holds}
	when $R$ is connected.
	Indeed, we can write $R$ as a finite product
	of connected semilocal rings and work over each factor separately.
	In this case, by Proposition~\ref{PR:types-of-involutions-Az}(v)
	and Lemma~\ref{LM:tau-type-is-constant}, 
	exactly one of the following hold:
	\begin{enumerate}[label=(\arabic*)]
		\item  \label{item:unit:Ei} $(\sigma,\veps)$ is unitary or symplectic,
		\item  \label{item:orthI:Ei} $(\sigma,\veps)$ is orthogonal and  
		$(\tau,\veps)$ is orthogonal or symplectic,
		\item  \label{item:orthII:Ei} $(\sigma,\veps)$ is orthogonal and 
		$(\tau,\veps)$ is unitary.
	\end{enumerate}
	The first two cases will be
	handled in Theorem~\ref{TH:Ei-holds-unit-syp}
	and the third case will be treated in  
	Theorem~\ref{TH:Ei-holds-orth-II}.

\subsection{Cases (1) and (2)}

	We begin by establishing some special cases of Theorem~\ref{TH:Ei-holds}
	in the context of case  \ref{item:unit:Ei}.

	\begin{prp}\label{PR:Ei-unitary-S-not-field}
		With Notation~\ref{NT:proof-of-Es},
		suppose that
		$R$ is a field, $S=R\times R$  
		and  
		$[A]=0$.
		Let $(Q,g)\in \Herm[\veps]{B,\tau}$
		be a hermitian space such that
		$\rrk_BQ$ is constant.
		Then $ \rho g $ admits a Lagrangian $L$
		satisfying   
		$Q\oplus L=QA$.
	\end{prp}
	
	\begin{proof}
		We may apply Reduction~\ref{RD:common-reduction} to
		assume that $B=T$ and $\deg A=2$.
	
		Let $\eta$ denote a nontrivial idempotent
		of $S$. Then $\eta^\sigma=1-\eta$.
		By Example~\ref{EX:exchange-involution}, we
		may assume that $T= T_1\times T_1$,
		$B=B_1\times B_1^\op$ and $A=A_1\times A_1^\op$,
		with $T_1\subseteq B_1\subseteq A_1$,
		and  under these identifications, $\sigma$ is the exchange
		involution $(x,y^\op)\mapsto (y,x^\op)$.
		Furthermore, all hermitian
		forms over $(B,\tau)$   are hyperbolic, and
		every hermitian space is determined up to isomorphism by
		its underlying module.

		Write $\veps$ as $(\alpha,\alpha^{-1})\in R\times R$
		and
		consider the $\veps$-hermitian form $g_1:{B\times B}\to B$
		given by $g_1((x_1,x^\op_2),(y_1,y_2^\op))= (\alpha  x_2y_1,(y_2x_1)^\op)$.
		It is easy to see that $(B,g_1)\in \Herm[\veps]{B,\tau}$.
		Since $\rrk_BB=1$ and $\rrk_BQ$
		is constant, we have $Q\cong B^n$ for $n=\rrk_BQ$ (Lemma~\ref{LM:rank-determines}).
		As we noted above, this means that
		$(Q,g)\cong n\cdot (B,g_1)$. 
		It is therefore  enough to prove the proposition for $(Q,g)=(B,g_1)$.
		In this case, the isomorphism $b\otimes a\mapsto ba: B\otimes_BA\to A$
		is an isometry from  $(QA,\rho g)$ to  $(A,f_1)$,
		where
		$f_1$ is given by the same formula as $g_1$.
		
		Fix an identification $A_1\cong \nMat{R}{2}$.
		Since $B_1=T_1$ is a quadratic \'etale $R$-algebra,
		there exists $t\in T_1$
		such that
		$r:=t^2\in\units{R}$ and $T_1=R\oplus t R$ (Lemma~\ref{LM:quad-etale-over-semilocal}).
		Thus, $t$ is conjugate to $[\begin{smallmatrix} 0 & r \\ 1 & 0 \end{smallmatrix}]$
		in $A_1$. Using this, we choose the identification $A_1\cong  \nMat{R}{2}$
		to satisfy   $t=[\begin{smallmatrix} 0 & r \\ 1 & 0 \end{smallmatrix}]$.
		Now,
		$
		B_1=T_1=[\begin{smallmatrix} 1 & 0 \\ 0 & 1 \end{smallmatrix}]R+
		[\begin{smallmatrix} 0 & r \\ 1 & 0 \end{smallmatrix}]R 
		$ and one readily checks
		that $L=[\begin{smallmatrix} 1 & 0 \\ 0 & 0 \end{smallmatrix}]A_1\times 
		(A_1 [\begin{smallmatrix} 0 & 0 \\ 0 & 1 \end{smallmatrix}])^\op$ is 
		a Lagrangian of $f_1=\rho g_1$ satisfying $B\oplus L=A$.
	\end{proof}
	
	\begin{prp}\label{PR:Ei-unitary-S-is-field-A-split}
		With Notation~\ref{NT:proof-of-Es},
		suppose that $S$
		is a field, $[A]=0$
		and $(\sigma,\veps)$ is symplectic or unitary.
		Let $(Q,g)\in \Herm[\veps]{B,\tau}$
		be a hyperbolic hermitian  space such that
		$\rrk_B Q$ is constant.
		Then $ \rho g $ admits a Lagrangian $L$
		satisfying $Q\oplus L=QA$.
	\end{prp}
	
	\begin{proof}
		By Reduction~\ref{RD:common-reduction},
		we may assume that both $\sigma$
		and $\tau$ are orthogonal or unitary and
		$\deg B=1$. Thus, $B=T$ and $\deg A=2$.
		We now split into cases.

\medskip

		\noindent {\it Case I.    $\rrk_B Q$ is even.}	
		We apply Reduction~\ref{RD:common-reduction-II}
		to assume that
		$A=\nMat{S}{2}$ and $\sigma$ is given by 
		$[\begin{smallmatrix}a & b \\ c & d\end{smallmatrix}]
		\mapsto[\begin{smallmatrix} a^\sigma & \alpha c^\sigma \\ \alpha^{-1} b^\sigma & 
		d^\sigma \end{smallmatrix}]$
		for some $\alpha\in\units{R}$.		
		
		Consider the $\veps$-hermitian form
		$g_1:B^2\times B^2\to B$
		given by $g_1((x_1,x_2),(y_1,y_2))=x_1^\sigma y_2+\veps x_2^\sigma y_1$.
		Then $(B^2,g_1)$ is a hyperbolic.
		Since $\deg B=1$ and $n:=\rrk_BQ$ is constant and even,
		we have, by Lemma~\ref{LM:rank-determines-hyperbolic},
		$(Q,g)\cong \frac{n}{2}\cdot (B^2,g_1)$.
		It is therefore enough to prove the proposition for
		$(Q,g)=(B^2,g_1)$. In this case, $(QA,\rho g)$
		can be identified with $(A^2,f_1)$,
		where $f_1:A^2\times A^2\to A$ is given by the same formula as $g_1$.

		Given an $R$-subspace $E$ of $A$, let $\Sym_\veps(E)=\{a\in E\suchthat \veps a^\sigma= a\}$.
		Suppose that there exists $s\in \Sym_{-\veps}( A )\setminus S_{-\veps}( B )$ such
		that $s\in \units{A}$. It is routine to check
		that $L=\{(a,sa)\where a\in A\}$ is a Lagrangian of $A$ satisfying
		$L\cap B^2=0$, which, by $R$-dimension considerations,
		implies $L\oplus B^2=A^2$.
		It is therefore enough to establish the existence of $s$.
		To that end, we split into subcases.

\medskip

		\noindent {\it Subcase I.1.   $(\sigma,\veps)$ is symplectic.}
		This means that $S=R$, $\sigma$
		is orthogonal and $\veps=-1$.
		Let $E_1=\{[\begin{smallmatrix} a & 0 \\ 0 & a \end{smallmatrix}]\where a\in S\}$
		and $E_2=\{[\begin{smallmatrix}0  &  \alpha c \\ c & 0 \end{smallmatrix}]\where c\in S\}$.
		Then $E_1$ and $E_2$ are $1$-dimensional $S$-subspaces of $\Sym_1( A )$.
		If   $E_i\cap \Sym_1( B )=0$ for some $i\in\{1,2\}$, then we can 
		take any $0\neq s\in E_i$. Otherwise, since $\dim_S  B = 2$, we have $ B =E_1+E_2$,
		so  take $s=[\begin{smallmatrix}1  &  \alpha   \\ 1 & 0\end{smallmatrix}]$.

\medskip

		\noindent {\it Subcase I.2.   $(\sigma,\veps)$ is unitary.}
		Then $S$ is quadratic \'etale over $R$ and $\tau$ is unitary 
		(Lemma~\ref{LM:sigma-unit-means-tau-unit}).		
		By Hilbert's Theorem 90, there exists $\delta\in \units{S}$
		such that $\delta^{-1}\delta^\sigma =\veps$.
		Since $\Sym_{-\veps}( A )=\delta^{-1} \Sym_{-1}( A )$,
		we reduce into verifying
		the existence of $s'\in \Sym_{-1}( A )\setminus \Sym_{-1}( B )$
		with $s'\in \units{ A }$.

		Since $\sigma $ and $\tau$ are unitary involutions, we have
		$\dim_R \Sym_{-1}( A )=(\deg  A )^2=4$
		and $\dim_R \Sym_{-1}( B )=(\deg  B )^2\cdot \rank_RS=2$
		(Proposition~\ref{PR:types-of-involutions-Az}).
		Let $E_1=\{[\begin{smallmatrix} a & 0 \\ 0 & a \end{smallmatrix}]\where a\in \Sym_{-1}(S)\}$
		and $E_2=\{[\begin{smallmatrix}0  & -\alpha c^\sigma \\ c & 0 \end{smallmatrix}]\where c\in S\}$.
		Then $E_1$ and $E_2$ are   $R$-subspaces of $\Sym_{-1}(A)$
		of dimensions $1$ and $2$, respectively,
		and $E_2\setminus \{0\}$ consists   of invertible elements.
		If $ \Sym_{-1}(B) \neq E_2$, then take any $s'\in E_2\setminus \Sym_{-1}(B)$.
		If $\Sym_{-1}(B)=E_2$, then $\Sym_{-1}(B)\cap E_1=E_2\cap E_1=0$ and we can choose
		any nonzero $s'\in E_1$.

\medskip

		\noindent {\it Case II. 
		 $\rrk_B Q$ is odd.}
		Writing $(Q,g)\cong (U\oplus U^*,\Hyp[\veps]{U})$ with $U\in\rproj{B}$,
		Lemma~\ref{LM:unimodular-implies-rrk-sigma-inv}
		implies that  $\rrk_B Q=\rrk_B U+\sigma(\rrk_B U)$.
		Since $\rrk_B Q$ is odd, $\rrk_B U$ cannot be $\sigma$-invariant.
		In particular, $\rrk_B U$ is non-constant, forcing $T=S\times S$.	
			
		Let $e$ denote a nontrivial idempotent of $T$.
		We identify
		$A$ with  $\nMat{S}{2}$  in such a way that the idempotent $e$ corresponds
		to $[\begin{smallmatrix} 1 & 0 \\ 0 & 0 \end{smallmatrix}]$.
		Under this identification, $B$ is the subalgebra of diagonal matrices.

		We have $e^\sigma\in \{e,1-e\}$.
		Since $e^\sigma=e$ implies that $\rrk_B U$ is fixed by $\sigma$,
		we must have $e^\sigma=1-e$. 
		We conclude that
		$B=T=S\times S$ and $\tau$ is the exchange involution.
		Now, by Example~\ref{EX:exchange-involution},
		every unimodular $\veps$-hermitian form over $(B,\tau)$ is 
		hyperbolic and   
		determined up to isomorphism by its underlying module.

		At this point, we claim that we may assume that $\veps=1$.
		Indeed, if $\sigma$ is unitary, then
		$S$ is a quadratic \'etale $R$-algebra and $\sigma|_S$
		is its standard involution. Thus, by Hibert's Theorem~90,
		there exists $\mu\in \units{S}$ with $\mu(\mu^{-1})^\sigma =\veps^{-1}$,
		or rather, $\mu\in \Sym_{\veps^{-1}}(S,\sigma|_S)\cap\units{S}$.
		Applying  $\mu$-conjugation, see \ref{subsec:conjugation} and Proposition~\ref{PR:simult-conj},
		we may assume that $\veps=1$.
		If $\sigma$ is not unitary, then $R=S$, $\sigma$ is orthogonal and $\veps=-1$, so we can repeat
		the previous argument with $\mu:=(1_S,-1_S)\in S\times S=T$; this will turn $\sigma$ into a symplectic involution.

		Define $g_1:B\times B\to B$ by $g_1(x,y)=x^\tau y$.
		Then $g_1$ is a hyperbolic $1$-hermitian form.
		Since $\rrk_B B=\deg B=1$, we have $(Q,g)\cong n\cdot (B,g_1)$ for $n=\rrk_BQ$,
		and so it enough to prove the proposition when $(Q,g)=(Q,g_1)$.
		In this case,
		$b\otimes a\mapsto ba:BA\to A_A$ is
		an isomorphism   under which
		$f_1:=\rho g_1$ is given by $f_1(x,y):=x^\sigma y$.
		Again, we  split into subcases.

\medskip

		\noindent {\it Subcase II.1.   $(\sigma,\veps)$ is symplectic.}
		Since $\veps=1$, the involution $\sigma$
		is   the unique symplectic involution of $\nMat{S}{2}$, given
		by 	$
		[\begin{smallmatrix} a & b \\ c & d  \end{smallmatrix}]^\sigma
		=[\begin{smallmatrix} d  & -b \\ -c & a  \end{smallmatrix}] 
		$ \cite[Proposition~2.21]{Knus_1998_book_of_involutions}.	
		Now, it is routine to check that $L=
		\{[\begin{smallmatrix} \alpha & \beta \\ \alpha & \beta  \end{smallmatrix}]\where
		\alpha,\beta\in S\}$ is a Lagrangian of $\rho f$ satisfying
		$B\oplus L=A$.

\medskip

		\noindent {\it Subcase II.2.   $(\sigma,\veps)$ is unitary.}
		Since $e^\sigma=1-e$, there are $\sigma|_S$-linear
		automorphisms $\sigma_2,\sigma_3:S\to S$ such that
		$\sigma:A\to A$ is given by
		$
		[\begin{smallmatrix} a & b \\ c & d  \end{smallmatrix}]^\sigma
		=[\begin{smallmatrix} d^\sigma & \sigma_2b \\ \sigma_3c & a^\sigma  \end{smallmatrix}] 
		$. 	Furthermore, $\sigma_2\circ \sigma_2=\sigma_3\circ\sigma_3=\id_S$.
		
		Since $S$ is a quadratic field extension of $R$
		and $2\in\units{S}$, there exists $\delta\in \units{S}$
		with $\delta^\sigma=-\delta$.
		Choose some  $c \in \units{S}$.
		Then $c=\frac{1}{2}(c-\sigma_3 c)+\frac{1}{2}\delta^{-1}\delta(c+\sigma_3 c)$,
		hence at least one of $ (c-\sigma_3 c)$, $\delta(c+\sigma_3 c)$ is nonzero.
		Replacing $c$ with $ (c-\sigma_3 c)$ or $\delta(c+\sigma_3 c)$,
		we may assume that $\sigma_3 c=-c$ and $c\neq 0$.
		Now, it is straightforward to check
		that  $L=[\begin{smallmatrix} 1 & 0 \\ c & 0  \end{smallmatrix}]A=
		\{[\begin{smallmatrix} \alpha & \beta \\ c\alpha & c\beta  \end{smallmatrix}]\where
		\alpha,\beta\in S\}$ is a Lagrangian of $\rho f$ satisfying
		$B\oplus L=A$.
		This completes the proof.
	\end{proof}
	
	\begin{lem}\label{LM:rank-of-Lagrangian-Ei}
		With Notation~\ref{NT:proof-of-Es},
		suppose that $S$ is  field.
		Let $(Q,g)\in\Herm[\veps]{B,\tau}$
		be an anisotropic hermitian space
		such that $ \rho g $ is hyperbolic.
		Then $\rrk_BQ$ is constant.
	\end{lem}
	
	\begin{proof}
		This is clear if $T$ is a field, so assume  
		$T=S\times S$ and let $e$ denote a nontrivial idempotent
		in $T$. Then $e^\tau\in \{e,1-e\}$.
		By Lemma~\ref{LM:unimodular-implies-rrk-sigma-inv},
		$\rrk_BQ$ is $\tau$-invariant, so it constant
		when $e^\tau=1-e$. 
		It remains to consider the case $e^\tau=e$.
		Writing $e':=1-e$, we need to show that $\rrk_{eB} Qe=\rrk_{e'B}Qe'$.

		By Lemma~\ref{LM:idempotent-in-T}, $eB=eAe$,
		$e'B=e'Ae'$ and  
		$AeA=A$. 
		Moreover, $(B,\tau)=(eB,\tau|_{eB})\times (e'B,\tau|_{e'B})$,
		because $e^\tau=e$.
		Thus, we may   consider
		hermitian forms over $(eB,\tau|_{eB})$, resp.\ $(e'B,\tau|_{e'B})$,
		as hermitian forms over $(B,\tau)$.
		Given $(V,h)\in \Herm[\veps]{eB,\tau|_{eB}}$ (resp.\ 
		$(V,h)\in \Herm[\veps]{e'B,\tau|_{e'B}}$), we write $(VA,\rho h)\in \Herm[\veps]{A,\sigma}$
		for the hermitian space obtained by regarding $(V,h)$ as a hermitian space
		over $(B,\tau)$ and then base-changing along the inclusion morphism
		$\rho:(B,\tau)\to (A,\sigma)$.

		For the sake of contradiction, suppose that $\rrk_{eB}Qe\neq \rrk_{e'B}Qe'$.
		By applying Lemma~\ref{LM:idempotent-in-T}(iv)  to $Qe_B$ and $Qe'_B$, we
		see that $\rrk_AQeA\neq \rrk_A Qe'A$.
		Viewing $g_e$ and $g_{e'}$ (notation as in \ref{subsec:conjugation}) 
		as hermitian forms over $(B,\tau)$,
		we have $(Q,g)=(Qe,g_e)\oplus (Qe',g_{e'})$.
		Thus,    $g_e$ and $g_{e'}$
		are anisotropic. Furthermore, $[\rho (g_e)]+[\rho (g_{e'})]=[\rho g]=0$ in
		$W_\veps(A,\sigma)$, so  $\rho (g_e)$ and $-\rho (g_{e'})$
		are Witt equivalent. Since the  underlying modules of $\rho (g_e)$ and $-\rho (g_{e'})$,
		namely,
		$QeA$ and $Qe'A$, are not isomorphic, either $\rho (g_e)$ or $\rho (g_{e'})$
		is isotropic  \cite[\S3.4(2)]{Quebbemann_1979_hermitian_categories} (for instance).
		Without loss of generality, suppose that $V$ is a nonzero $A$-summand of $QeA$
		such that $\rho (g_e)(V,V)=0$. Then $Ve$ is a nonzero $B$-module
		(because $VeA=VAeA=VA=V$) and summand of   
		$QeAe=Qe$
		such that $g(Ve,Ve)=0$, contradicting our assumption that $g$ is anisotropic.
	\end{proof}
	
	\begin{lem}\label{LM:Ei-Lagrangians-are-iso}
		With Notation~\ref{NT:proof-of-Es},
		let $(Q,g)\in\Herm[\veps]{B,\tau}$, and
		let $L$ be a Lagragian of $\rho g$
		satisfying $Q\oplus L=QA$. 
		Suppose that   $\rrk_BQ$ is constant.		
		Then $\rrk_AQA$ is even and
		$\rrk_A L=\frac{1}{2}\rrk_A QA$.
	\end{lem}

	\begin{proof}
		We have $\rrk_B QA=\iota \rrk_A QA=2\rrk_B Q$,
		hence $\rrk_B L= \rrk_B QA-\rrk_B Q=\frac{1}{2}\rrk_B QA$.
		The lemma follows because $\rrk_B L=\iota \rrk_A L$
		and $\rrk_B QA=\iota \rrk_A QA$.
	\end{proof}
	
	\begin{prp}\label{PR:Ei-unitary-R-is-field}
		With Notation~\ref{NT:proof-of-Es},
		suppose that $R$ is a field and  $(\sigma,\veps)$ is symplectic or unitary.
		Let $(Q,g)\in\Herm[\veps]{B,\tau}$ and assume that $\rrk_BQ$ is constant
		and $ \rho g $ admits a Lagrangian $L$
		with $\rrk_AL=\frac{1}{2}\rrk_AQA$.
		Then there exists $\vphi\in U^0(\rho g)$
		such that $Q\oplus \vphi L=QA$.
	\end{prp}
	
	\begin{proof}
		By Proposition~\ref{PR:rationality-for-Ei}(i),
		when $R$ is infinite, it is enough to prove
		the proposition after base-changing to an algebraic closure
		of $R$, in which case $[A]=0$.
		On the other hand, if $R$ is finite, then
		$[A]=0$ by Wedderburn's theorem. We may therefore assume 
		that $[A]=0$.

		Suppose first that $S$ is a field.
		Using Proposition~\ref{PR:ansio-Witt-equivalent},
		write $(Q,g)=(Q_1,g_1)\oplus (Q_2,g_2)$
		with   $g_1$   anisotropic
		and $g_2$ hyperbolic.
		Then $[\rho g_{1}]=[\rho g]=0$ in $W_\veps(A,\sigma)$,
		so $\rho g_1$ is hyperbolic by Theorem~\ref{TH:trivial-in-Witt-ring}(ii).
		Let $L_1$ be a Lagrangian of  $\rho g_1$. 
		By arguing as in Remark~\ref{RM:Es-for-field},
		we see that $Q_1\oplus L_1=Q_1A$.	
		Now, by Lemma~\ref{LM:rank-of-Lagrangian-Ei},
		$\rrk_BQ_1$ is constant, and thus, so is $\rrk_BQ_2$.
		With this at hand, Proposition~\ref{PR:Ei-unitary-S-is-field-A-split}
		says that
		$\rho g_2$ admits a Lagrangian $L_2$
		such that $Q_2\oplus L_2=Q_2A$.
		
		Let $L'=L_1\oplus L_2$.
		Then $Q\oplus L'=QA$. 
		By Lemmas~\ref{LM:Ei-Lagrangians-are-iso}
		and~\ref{LM:Lag-transitive-action}, 
		there exists $\vphi\in U(\rho g)$
		such that $\vphi L=L'$. 
		Since $(\sigma,\veps)$ is symplectic or unitary,
		we have $U(\rho g)=U^0(\rho g)$ (Proposition~\ref{PR:U-zero-description}),
		so we are done.

		If $S$ is not a field,  
		Proposition~\ref{PR:Ei-unitary-S-not-field}
		implies that there exists $L'\in \Lag(\rho g)$
		with $Q\oplus L'=QA$ and  we can finish
		the proof as in the previous paragraph. 	
	\end{proof}
	
	We proceed with showing that Proposition~\ref{PR:Ei-unitary-R-is-field}
	also holds 
	in the context of Case~\ref{item:orthI:Ei}, namely,	
	when $(\sigma,\veps)$ is orthogonal 
	and $(\tau,\veps)$ is orthogonal or symplectic.
	This  is similar to the proof of Proposition~\ref{PR:Ei-unitary-R-is-field}, but a few modifications 
	are required in order to account for the possibility 
	that $U^0(\rho g)$ is smaller than $U(\rho g)$.
	
	\begin{prp}\label{PR:Ei-orthogonal-I-R-field-A-split}
		With Notation~\ref{NT:proof-of-Es},
		suppose $S$
		is a field, $[A]=0$,
		$(\sigma,\veps)$ is orthogonal,
		and $\tau|_T=\id_T$.  
		Let $(Q,g)\in \Herm[\veps]{B,\tau}$
		be a  hyperbolic hermitian  space such that
		$\rrk_BQ$ is constant and positive, and let $L$ be a Lagrangian
		of $\rho g$.
		Then there exist 
		$\vphi_1,\vphi_{-1}\in U(\rho g)$
		such that $\Nrd(\vphi_1)=1$, $\Nrd(\vphi_{-1})=-1$
		and
		$Q\oplus \vphi_1L=Q\oplus \vphi_{-1}L=QA$.
	\end{prp}
	
	\begin{proof}
		As noted in~\ref{subsec:Lagrangians},
		since $(\sigma,\veps)$ is orthogonal,
		all  Lagrangians of $\rho g$
		have reduced rank $\frac{1}{2}\rrk_AP$
		and are thus isomorphic as $A$-modules (Lemma~\ref{LM:rank-determines}).
		Thus, by   
		Lemma~\ref{LM:Lag-transitive-action},
		$U(\rho g)$ acts
		transitively on $\Lag(\rho g)$.
		It is therefore enough to prove the proposition for a single   Lagrangian $L_0$ of our choice.

		By Reductions~\ref{RD:common-reduction} and~\ref{RD:common-reduction-II},
		we may assume that $B=T$, $A=\nMat{S}{2}$, $\veps=1$
		and $\sigma$ is orthogonal
		and given by  
		$[\begin{smallmatrix} a &  b \\ c & d \end{smallmatrix}]^{\sigma}=
		[\begin{smallmatrix} a &   \alpha c \\  \alpha^{-1} b & d \end{smallmatrix}]
		$ for some $\alpha\in\units{S}$.

		As noted in Case II of the proof of Proposition~\ref{PR:Ei-unitary-S-is-field-A-split},
		since $\sigma|_T=\id_T$, the reduced rank of $Q$
		is even.
		Thus, arguing as in Case I of that proof,
		we may assume
		that $(Q,g)=(B^2,g_1)$ with $g_1$ given
		by $g_1((x_1,x_2),(y_1,y_2))=x_1^\sigma y_2+ x_2^\sigma y_1$.
		We identify $QA$ with $A^2$ and $\End_A(QA)$ with $\nMat{A}{2}$ 
		in the obvious way. The form
		$f_1:=\rho g_1$ is   defined by same formula as $g_1$
		and we take $L_0:=A\times 0$ as our fixed Lagrangian.

\medskip

		\noindent {\it Existence of $\vphi_1$.}
		Let $s=[\begin{smallmatrix} 0 & -\alpha \\ 1 & 0\end{smallmatrix}]\in \nMat{S}{2}=A$.
		Then $s^\sigma=-s$, $s\in \units{A}$ and $s\notin B=T$ because $\sigma|_T=\id_T$.
		It is routine to check that
		$\vphi_1:=[\begin{smallmatrix} 1 &  0 \\ s & 1 \end{smallmatrix}]\in \nMat{A}{2}$
		is an isometry
		of $\rho g$ and $\vphi_1L_0=\{(a,sa)\where a\in A\}$.
		Now, as in Case I of the proof of Proposition~\ref{PR:Ei-unitary-S-is-field-A-split},
		we have $Q\oplus \vphi_1L_0=QA$.

\medskip

		\noindent {\it Existence of $\vphi_{-1}$.}
		Since $B=T$ is a quadratic \'etale 
		$S$-algebra, we can write $B=S\oplus \lambda S$
		with $\lambda^2\in \units{S}$.
		The assumption $\sigma|_T=\id_T$  
		allows us to write $\lambda= [\begin{smallmatrix} x_1 & \alpha x_2 \\  x_2 & x_3 \end{smallmatrix}]$
		with $x_1,x_2,x_3\in S$.
		Let 
		\[\psi :=\left[
		\begin{smallmatrix}
		& & 1 & \\
		& 1 & & \\
		1 & & & \\
		& & & 1
		\end{smallmatrix}
		\right]
		\in\nMat{S}{4}=\nMat{A}{2}
		\]
		It is routine to check that $\psi\in U(\rho g)$,
		$\Nrd(\psi)=-1$,
		and 
		\[\psi L_0=\{(	[\begin{smallmatrix} 0 &  0 \\  c & d \end{smallmatrix}],
		[\begin{smallmatrix} a &  b \\  0 & 0 \end{smallmatrix}])\where
		a,b,c,d\in S\} .\]
		Now, if $x_2\neq 0$, then   $B^2\oplus \psi L_0=A^2$
		and   we can take $\vphi_{-1}=\psi$.
		On the other hand, if $x_2=0$,
		then $B=[\begin{smallmatrix} S &  0 \\ 0 & S \end{smallmatrix}]$,
		and hence $\psi (B^2)=B^2$.
		This means that $B^2+\psi \vphi_1L_0=\psi(B^2+\vphi_1L_0)=\psi (A^2)=A^2$,
		so  we can take $\vphi_{-1}=\psi\vphi_1$.
	\end{proof}
	
	\begin{prp}\label{PR:Ei-orthogonal-I-R-field}
		With Notation~\ref{NT:proof-of-Es},
		suppose that $S$
		is a field,  
		$(\sigma,\veps)$ is orthogonal,
		and $\tau|_T=\id_T$.  
		Let $(Q,g)\in \Herm[\veps]{B,\tau}$
		and assume that
		$\rrk_BQ$ is constant  
		and  $\rho g$
		admits  a Lagrangian
		$L$.
		Then there exists 
		$\vphi \in U^0(\rho g)$		
		such that  
		$Q\oplus \vphi L=QA$.	
	\end{prp}

	\begin{proof}		
		As in the proof of 	Proposition~\ref{PR:Ei-unitary-R-is-field},
		we can reduce to the case where $[A]=0$  
		and write $(Q,g)=(Q_1,g_1)\oplus (Q_2,g_2)$
		with  $g_1$  anisotropic  and  $g_2$ hyperbolic.
		Furthermore, $\rrk_BQ_2$ is constant,
		$\rho g_1$ is hyperbolic
		and any Lagrangian  $U$ of $\rho g_1$ satisfies $Q_1\oplus U=Q_1A$.

		If $Q_2=0$, then $L$ is a Lagrangian of $(P_1,f_1)=(P,f)$ and we can take $\vphi=\id_P$.
		
		Assume $Q_2\neq 0$,
		let $U$ be a Lagrangian of $\rho g_1$
		and let $V$ be a Lagrangian  of $\rho g_2$. By Proposition~\ref{PR:Ei-orthogonal-I-R-field-A-split},
		there exist $\vphi_1,\vphi_{-1}\in U(\rho g_2)$
		such that $Q_2\oplus \vphi_i V =Q_2A$ and $\Nrd(\vphi_i)=i$
		for $i\in\{\pm 1\}$.
		Then $L_i:=U\oplus \vphi_i V$ ($i=\pm 1$) is a Lagrangian of $\rho g$ 
		having the same reduced rank as $L$ (see \ref{subsec:Lagrangians})
		and satisfying $Q\oplus L_i=QA$.
		By  
		Lemma~\ref{LM:Lag-transitive-action},
		there exists $\psi\in U(\rho g)$
		such that $\psi  L=L_1$. If
		$\Nrd(\psi)=1$,   take $\vphi=\psi$.
		On the other hand, if
		$\Nrd(\psi)=-1$, then we can take $\vphi:=(\id_{P_1}\oplus \vphi_{-1}\vphi_1^{-1})\psi$,
		because
		$\vphi  L=U\oplus \vphi_{-1}\vphi_1^{-1}(\vphi_1V)=L_{-1}$
		and $\Nrd(\vphi)=\Nrd(\vphi_{-1}\vphi_1)^{-1}\Nrd(\psi)=1$.
	\end{proof}
	
	We can now establish Theorem~\ref{TH:Ei-holds}
	in cases~\ref{item:unit:Ei} and~\ref{item:orthI:Ei}

	\begin{thm}\label{TH:Ei-holds-unit-syp}
		Assuming $R$ is connected,
		Theorem~\ref{TH:Ei-holds}
		holds  when $(\sigma,\veps)$ is symplectic or unitary, or
		$(\tau,\veps)$ is orthogonal or symplectic.
	\end{thm}
	
	\begin{proof}
		We only prove part (ii). It will be clear
		from the proof that we can take $(Q',g')=(Q,g)$ when $T$ is connected,
		which is exactly what we need to show in order to prove (i).
	
		Recall that we are given $(Q,g)\in \Herm[\veps]{B,\tau}$
		such that $[\rho g]=0$ in $W_\veps(A,\sigma)$.
		We need
		to show that $\rho g$ admits a Lagrangian $L$
		such that $Q\oplus L=QA$, possibly after replacing $g$ with
		a Witt equivalent form.
		By Theorem~\ref{TH:trivial-in-Witt-ring}(ii), $\rho g$ is hyperbolic.

		We may   assume that $S$ is connected. If not, then  
		$S=R\times R$ (Lemma~\ref{LM:non-connected-S}),
		and by Example~\ref{EX:exchange-involution},
		$g$ is hyperbolic. We may therefore replace $g$
		with the zero form and take $L=0$.

		We   may also assume that $\rrk_BQ$ is constant.
		Indeed, if $\rrk_BQ$ is not constant, then $T$ is not connected.
		Now, by Lemma~\ref{LM:non-connected-S} and our assumption
		that  $S$ is connected, $T\cong S\times S$, so 
		$T$ has exactly two primitive idempotents, denoted
		$e$ and $e'$.
		If $\sigma$ swaps $e$ and $e'$, 
		then $\rrk_BQ$ is constant because it is $\sigma$-invariant
		(Lemma~\ref{LM:unimodular-implies-rrk-sigma-inv}), so it
		must be the case that $\sigma$ 
		fixes $e$ and $e'$.
		By Lemma~\ref{LM:idempotent-in-T},
		$B=eAe\oplus e'Ae'$ and $[eB]= [A]$.
		Thus,
		$\ind eB =\ind A$,
		and similarly, $\ind e'B=\ind A$.
		Let $U$ be a finite projective $eB$-module of reduced rank $\ind eB$
		and let $V$ be a finite projective $e'B$-module of reduced rank $\ind e'B$;
		they exist by Theorem~\ref{TH:index-description}.
		By Lemma~\ref{LM:rank-determines} and Corollary~\ref{CR:index-divides-rrk}, there are $r,s\in \Z$
		such that $Q\cong U^r\oplus V^s$, and
		by Lemma~\ref{LM:idempotent-in-T}(iv), $\rrk_AQA=\rrk_{eB}U^r+\rrk_{e'B} V^r=(r+s)\ind A$. 
		Applying Corollary~\ref{CR:constant-even-ranks}(ii) to $(QA,\rho g)$,
		we see that
		there exists $W\in\rproj{A}$ with $\rrk_A QA=2\rrk_A W$.
		Since $\ind A\mid\rrk_AW$ 
		(Corollary~\ref{CR:index-divides-rrk}), $\rrk_A QA$ is an even  multiple
		of $\ind A$, and so $r\equiv s\bmod 2$.
		Now, if $r>s$, we can replace 
		$g$
		with 
		$g\oplus (\frac{r-s}{2})\cdot\Hyp[\veps]{V}$
		and if $r<s$, we can replace 
		$g$
		with 
		$g\oplus (\frac{s-r}{2})\cdot\Hyp[\veps]{U}$.
		After this modification, we get $r=s$, which means that $\rrk_BQ$ is constant.
		
		Fix a Lagrangian $L$ of $\rho g$.
		Since $S$ is connected, $\rrk_A L=\frac{1}{2}\rrk_AQA$ (see~\ref{subsec:Lagrangians}).
		Let $\frakm_1,\dots,\frakm_t$
		denote the maximal ideals of $R$.
		By Propositions~\ref{PR:Ei-unitary-R-is-field} and~\ref{PR:Ei-orthogonal-I-R-field},
		for every $1\leq i\leq t$,
		there exists $\vphi_i\in U^0(\rho g(\frakm_i))$
		such that $Q(\frakm_i)\oplus \vphi_i(L(\frakm_i))=QA(\frakm_i)$.
		By 
		Theorem~\ref{TH:U-zero-mapsto-onto-closed-fibers}, 
		there exists $\vphi \in U^0(\rho g)$
		such that $\vphi(\frakm_i)=\vphi_i$.
		This means that $Q(\frakm_i)\oplus (\vphi L)(\frakm_i)=QA(\frakm_i)$
		for all $i$, so by Lemma~\ref{LM:semilocal-direct-sum-reduction}, we have
		$Q\oplus (\vphi L)=QA$. Since $\vphi L$ is a Lagrangian of $QA$, we are done.
	\end{proof}

\subsection{Case (3)}

	We now turn to prove Theorem~\ref{TH:Ei-holds}
	in Case~\ref{item:orthII:Ei}, namely,
	when $R$ is connected, $(\sigma,\veps)$ is orthogonal and 
	$(\tau,\veps)$ is unitary. Note that $S=R$.
	
	This case
	is more subtle than  Cases \ref{item:unit:Ei} and~\ref{item:orthI:Ei}
	because the key Propositions~\ref{PR:Ei-unitary-R-is-field}
	and~\ref{PR:Ei-orthogonal-I-R-field} 
	no longer hold.	
	The proof will therefore consist 
	of characterizing  when these propositions
	fail, and bypassing the failure when they do.

\medskip

	We begin with treating the case where $\rrk_BQ$ is odd;
	this case is degenerate.

	\begin{prp}\label{PR:Ei-orthII-odd-rank-forms}
		With Notation~\ref{NT:proof-of-Es},
		suppose that $R$ is connected semilocal,
		$(\sigma,\veps)$ is orthogonal and $ \tau $
		is unitary.
		Let $(Q,g)\in\Herm[\veps]{B,\tau}$ 
		and assume that $\rho g$ is hyperbolic
		and $\rrk_BQ$ is not constant or not even.
		Then  
		$T$ is not connected and $g$ is hyperbolic.
	\end{prp}
	
	\begin{proof}
		If $T$ is not connected,
		then 
		$T\cong R\times R$ by Lemma~\ref{LM:non-connected-S},
		and   $g$ is hyperbolic by Example~\ref{EX:exchange-involution} (applied to $(B,\tau)$).
		It is therefore enough to show that $T$ is not connected.

		For the sake of contradiction, suppose that $T$ is connected.
		Then $\rrk_BQ$ is constant and odd.
		Furthermore, by Corollary~\ref{CR:constant-even-ranks}(ii),
		there exists $V\in \rproj{A}$
		such that $2\iota\rrk_A V= \iota\rrk_AQA=2\rrk_BQ$.
		Thus, $n:=\rrk_AV$ is odd.
		By Corollary~\ref{CR:index-divides-rrk},
		$\ind A\mid n$, so by Theorem~\ref{TH:period-divides-index},
		$n [A]=0$ in $\Br R$.
		On the other hand, since $A$ has an $R$-involution, $2[A]=0$, so $[A]=0$.

		We   now apply   Reductions~\ref{RD:common-reduction}
		and~\ref{RD:common-reduction-II}
		to assume that $B=T$, $A=\nMat{S}{2}$, $\veps=1$
		and $\sigma:A\to A$ is orthogonal and given
		by 	$[\begin{smallmatrix} a & b \\ c & d\end{smallmatrix}]^\sigma=
		[\begin{smallmatrix} a & \eta c \\ \eta^{-1}b & d\end{smallmatrix}]$
		for some $\eta\in \units{S}$.
		
		By Lemma~\ref{LM:quad-etale-over-semilocal},
		there exists $\lambda\in \units{T}$
		such that $\lambda^\sigma=-\lambda$ and $T=R\oplus \lambda R$.
		Then   $\lambda=[\begin{smallmatrix} 0 &  \eta c \\ - c & 0\end{smallmatrix}]$
		for some $c\in\units{S}$, and consequently
		$ T=R[\begin{smallmatrix} 1 &  0 \\ 0  & 1\end{smallmatrix}]
		\oplus R[\begin{smallmatrix} 0 &  \eta \\ -1  & 0\end{smallmatrix}]$.
		Furthermore,  
		by Proposition~\ref{PR:diagonalizable-herm-forms},
		$g$ is diagonalizable, so  there 
		exist $\alpha_1,\dots,\alpha_n\in \Sym_{1}(T,\tau) \cap\units{T}=\units{R}$
		such that 
		$g\cong \langle \alpha_1,\dots,\alpha_n\rangle_{(T,\tau )}$
		(notation as in Example~\ref{EX:diagonal-forms}). Note that $n=\rrk_BQ$ is odd.

		Let $e:=[\begin{smallmatrix} 1 & 0 \\  0 & 0\end{smallmatrix}]$.
		Then $e$ is an idempotent satisfying $e^\sigma=e$, $eAe=eR$ and $AeA=A$,
		hence $e$-transfer (see~\ref{subsec:conjugation})
		induces an group isomorphism
		$[f]\mapsto [f_e]:W_1(A,\sigma)\to W_1(R,\id_R)$.
		It is routine to check
		that upon identifying
		$Ae$ with $R^2$ via
		$[\begin{smallmatrix} a & 0 \\  c & 0\end{smallmatrix}]\mapsto (a,c)$,
		the bilinear form $(\rho g)_e: A^ne\times A^ne\to eR\cong R$
		is just
		$\langle \alpha_1,\eta \alpha_1,\dots,\alpha_n,\eta \alpha_n \rangle_{(R,\id_R)}$.
		By assumption, this form is hyperbolic, 
		so it is isomorphic to $n\langle 1,-1\rangle_{(R,\id_R)}$
		(Lemma~\ref{LM:rank-determines-hyperbolic}).
		Comparing discriminants (using Proposition~\ref{PR:disc-orth-basic-props}(iv)),
		we find that $(-\eta)^n$ is a square
		in $\units{R}$.
		Since $n$ is odd, this means that $-\eta$ is a square in $\units{R}$,
		say $-\eta=r^2$.
		Then $\frac{1}{2}[\begin{smallmatrix} 1 & r \\  r^{-1} & 1\end{smallmatrix}]
		=\frac{1}{2}[\begin{smallmatrix} 1 & 0 \\  0 & 1\end{smallmatrix}]-
		\frac{1}{ 2r}[\begin{smallmatrix} 0 & \eta \\ -1  & 0\end{smallmatrix}]$
		is a nontrivial idemptonent in $T$, contradicting our 
		assumption that 
		$T$ is   connected.
	\end{proof}

	Recall from \ref{subsec:Lagrangians} 
	that $\Lag(f)$ denotes
	the set of Lagrangians $L$
	of $(P,f)\in\Herm[\veps]{A,\sigma}$ with $\rrk_AL=\frac{1}{2}\rrk_AP$.
	When $(\sigma,\veps)$ is orthogonal, $\Lag(f)$ consists
	of all the Lagrangians of $f$, and when $R$
	is semilocal,  any two   Lagrangians in $\Lag(f)$ are   isomorphic (Lemma~\ref{LM:rank-determines}).
	These facts will be used without comment in the sequel.

	\begin{prp}\label{PR:Ei-Lag-partition}
		With Notation~\ref{NT:proof-of-Es},
		suppose that $(\sigma,\veps)$ is
		orthogonal
		and $\tau$ is unitary.
		Let $(Q,g)\in\Herm[\veps]{B,\tau}$
		and assume that $ \rho g $ is hyperbolic and $\rrk_BQ$ is even.
		Then there exists a unique $\uU(\rho g)$-equivariant
		natural transformation  of functors from $R$-rings
		to sets,
		\[
		\Phi_g:\uLag(\rho g)\to \umu_{2,R},
		\]
		such that for any $R$-ring $R_1$
		and any idempotent $e_1\in T_{R_1}$ with $e_1+e_1^\sigma=1$,
		one has $\Phi_g(P_{R_1}e_1A_{R_1})=1$.
		The map $\Phi_g$ has the following additional properties:
		\begin{enumerate}[label=(\roman*)]
			\item If there exists an idempotent
			$e\in T$
			such that $e^\sigma+e=1$, then $\Phi_g=\Phi_{PeA}$ (notation
			as in Proposition~\ref{PR:partition-of-Lag}).
			\item If $n:=\rrk_BQ$
			is constant and $M\in \Lag(g)$, then $\Phi_g(MA)=(-1)^{\frac{n}{2}}$.
			\item If $(Q',g')\in \Herm[\veps]{B,\tau}$ is another
			hermitian space
			such that $\rho g'$  is hyperbolic and $\rrk_BQ'$
			is even, then $\Phi_{g\oplus g'}(L\oplus L')=\Phi_g(L)\cdot \Phi_{g'}(L')$
			for all $L\in\Lag(\rho g)$, $L'\in\Lag(\rho g')$.
			\item If $e\in B$ is an idempotent
			such that $e^\tau=e$
			and $\rrk_BeB$ is positive and constant
			on the fibers of $\Spec T\to\Spec R$,
			then $\Phi_g(L)=\Phi_{g_e}(Le)$ for all $L\in\Lag(\rho g)$
			(notation as in \ref{subsec:conjugation}).
		\end{enumerate}
	\end{prp}
	
	Note that  
	$Q_{R_1}e_1$
	is a Lagrangian of $g_{R_1}$ (see Example~\ref{EX:exchange-involution}),
	and therefore  $Q_{R_1}e_1A_{R_1}$ is a Lagrangian of $\rho g_{R_1}$.
	We alert the reader that $\Phi_g$ is not defined
	when $\rrk_BQ$ is not even.
	
	\begin{proof}
		Fix a Lagrangian $L_0$ of $\rho g$
		and let $\Phi_0:=\Phi_{L_0}$ be as in Proposition~\ref{PR:partition-of-Lag}.
		
		Let $R_1$ be an $R$-ring
		and let $e,e'\in T_{R_1}$
		be two idempotents
		satisfying $e+e^\sigma=e'+e'^\sigma=1$.
		We claim that $\Phi_0(Q_{R_1}eA_{R_1})=\Phi_0(Q_{R_1}e'A_{R_1})$,
		or rather, $\Phi_{Q_{R_1}eA_{R_1}}(Q_{R_1}e'A_{R_1})=1$
		(Proposition~\ref{PR:partition-of-Lag}(i)).
		Base changing along $R\to R_1$,
		we may assume that $R_1=R$.
		Now, by Lemma~\ref{LM:mu-two-check}, it is enough to show that $\Phi_{QeA}(Qe'A)(\frakm)=1$
		in $k(\frakm)$ for all $\frakm\in\Max R$,
		so assume that $R$ is a field. Since $e\in T$
		is 
		an 
		idempotent satisfying $e^\sigma+e=1$,
		it is nontrivial and hence $T=R\times R$.
		Similarly, $e'$ is nontrivial, so
		$e=e'$ or $e=1-e'$.
		In the first case, we have $QeA=Qe'A$
		and $\Phi_{QeA}(Qe'A)=1$.
		In the second case, $QA=QeA\oplus Qe'A$, so
		$\Phi_{QeA}(Qe'A)=(-1)^{\rrk_A QeA}=(-1)^{\rrk_{eB} Qe}=1$
		by Proposition~\ref{PR:computation-of-Phi},
		Lemma~\ref{LM:idempotent-in-T}(iv), and the fact $\rrk_B Q$
		is even. This proves the claim.

		Write $R_0=T$. 
		Then $T_{R_0}\cong R_0\times R_0$ by Lemma~\ref{LM:splitting-quad-et-algs}.
		Let $e_0\in T_{R_0}$ correspond to $(1_{R_0},0_{R_0})$
		under this isomorphism. Since $\tau|_T$
		is the standard $R$-involution of $T$,  we have $e_0+e_0^\sigma=1$.

		Write $\theta:=\Phi_0(P_{R_0}e_0A_{R_0})\in \mu_2(R_0)$.
		We claim that $\theta$ is in fact in $\mu_2(R)$.
		Let $i_1,i_2:R_0\to R_0\otimes R_0$
		denote the maps $r\mapsto r\otimes 1$ and $r\mapsto 1\otimes r$
		respectively.
		By what we have shown above,
		$i_1\theta=\Phi_0(P_{R_0\otimes R_0}(i_1e_0)A_{R_0\otimes R_0})=
		\Phi_0(P_{R_0\otimes R_0}(i_2e_0)A_{R_0\otimes R_0})=i_2\theta$.
		Since $\umu_{2,R}$ is a sheaf on $(\Aff/R)_{\fpqc}$,
		and since $R\to R_0$ is faithfully flat, 
		this means that $\theta\in \mu_2(R)$.

		Define $\Phi_g:=\theta^{-1}\cdot\Phi_0$.
		It is clear that $\Phi_g$ is $\uU(\rho g)$-equivariant.
		Let $R_1$ and $e_1\in T_{R_1}$ be as in the proposition.
		Then, in $\mu_2(R_0\otimes R_1)$,
		we have $\Phi_g(Q_{R_0\otimes R_1}e_1A_{R_0\otimes R_1})=
		\Phi_g(Q_{R_0\otimes R_1}e_0A_{R_0\otimes R_1})=\Phi_g(Q_{R_0}e_0A_{R_0})=\theta^{-1}
		\theta=1$.
		Since $R_1\to R_0\otimes R_1$ is faithfully flat,
		this means that $\Phi_g(Q_{R_1}e_1A_{R_1})=1$
		in $\mu_2(R_1)$.

		Suppose that $\Phi':\uLag(\rho g)\to \umu_{2,R}$
		also satisfies the conditions of the Proposition. Then,
		by Proposition~\ref{PR:partition-of-Lag}, both
		$\Phi'$ and $\Phi$ must coincide
		with $\Phi_{Q_{R_0}e_0 A_{R_0}}$
		on the subcategory of $R_0$-rings.
		Since  $\umu_2$ and $\uLag(\rho g)$
		are sheaves over $(\Aff/R)_{\fpqc}$
		and since $R\to R_0$ is faithfully
		flat, this forces $\Phi'=\Phi_g$.
		
\medskip
		
		We finish with verifying (i)--(iv).
		Since $R\to R_0$ is faithfully
		it is enough to prove
		these statements after base-changing to $R_0$.
		We may therefore assume that $T=R\times R$
		and there exists
		an idempotent $e_0\in T$ with $e_0^\sigma+e_0=1$.
		
		(i) This is immediate from the uniqueness part of Proposition~\ref{PR:partition-of-Lag}.
		
		(ii) We have $Me_0=M\cap Qe_0$. Since ${}_BA$ is flat,
		this means that $Me_0A=MA\cap Qe_0A$.
		By (i) and Proposition~\ref{PR:computation-of-Phi},
		$\Phi_g(MA)=       
		(-1)^{\rrk_AQe_0A-\rrk_AMe_0A}$,
		and $\rrk_AQe_0A-\rrk_AMe_0A=\rrk_{e_0B} Qe_0-\rrk_{e_0B}Me_0 =\frac{n}{2} $
		by Lemma~\ref{LM:idempotent-in-T}(iv).
		
		(iii) This follows from (i) and Proposition~\ref{PR:partition-of-Lag}(ii).
		
		(iv) 
		This follows readily
		from (i), item \ref{item:e-transfer-and-Phi} in \ref{subsec:conjugation}
		and Proposition~\ref{PR:simult-e-transfer}(i).
	\end{proof}

	\begin{prp}\label{PR:Ei-orthogonal-II-R-field-A-split}
		With Notation~\ref{NT:proof-of-Es},
		suppose that $R$ is a field, $[A]=0$, $(\sigma,\veps)$ is orthogonal
		and $\tau$ is unitary.
		Let $(Q,g)\in\Herm[\veps]{B,\tau}$
		be a   hyperbolic hermitian  space such that
		$\rrk_BQ$ is constant and even. Then:
		\begin{enumerate}[label=(\roman*)]
			\item 
			There exists $L\in\Lag(\rho g)$
			such that $Q\oplus L=QA$
			and $\Phi_g(L)=1$.
			\item 
			There is no $L\in\Lag(\rho g)$
			such that $Q\oplus L=QA$
			and $\Phi_g(L)=-1$.
		\end{enumerate}
	\end{prp}
	
	\begin{proof}
		By Reduction~\ref{RD:common-reduction} and Proposition~\ref{PR:Ei-Lag-partition}(iv),
		we may assume that $B=T$, $A=\nMat{R}{2}$, $\veps=1$
		and $\sigma$ is orthogonal.

\medskip

		(i) By Reduction~\ref{RD:common-reduction-II},
		we may assume that $\sigma$
		is given by $[\begin{smallmatrix} a &  b \\ c & d \end{smallmatrix}]^{\sigma}=
		[\begin{smallmatrix} a &   \alpha c \\  \alpha^{-1} b & d \end{smallmatrix}]
		$ for some $\alpha\in\units{R}$.
		Arguing as in Case I of  the proof of Proposition~\ref{PR:Ei-unitary-S-is-field-A-split},
		we may assume that $(Q,g)=(B^2,g_1)$,
		where $g_1((x_1,x_2),(y_1,y_2))=x_1^\sigma y_2+ x_2^\sigma y_1$.
		By Lemma~\ref{LM:quad-etale-over-semilocal}, there exists
		$\lambda\in T$ such that $T=R\oplus \lambda R$ and $\lambda^\tau=-\lambda$.
		This 
		forces
		$B=T=[\begin{smallmatrix} 1 &  0 \\ 0 & 1 \end{smallmatrix}]R+
		[\begin{smallmatrix} 0 &  \alpha \\ -1  & 0 \end{smallmatrix}]R$.
		Now, it is routine to check that
		$L=\{(	[\begin{smallmatrix} 0 &  0 \\  c & d \end{smallmatrix}],
		[\begin{smallmatrix} a &  b \\  0 & 0 \end{smallmatrix}])\where
		a,b,c,d\in S\}$
		is a Lagrangian of $\rho g$ satisfying $B^2\oplus L=A^2$.
		That $\Phi_g(L)=1$ will follow once we prove (ii).

\medskip	

		(ii) \noindent {\it Step 1.} 
		It is enough to prove the statement 
		after base-changing to an algebraic closure of $R$,
		so assume that $R$ is an algebraically closed field.
		In this case, $B=T=R\times R$ and $\tau$ is the exchange involution.
		By Example~\ref{EX:exchange-involution},
		this means that $g\cong n\langle 1\rangle_{(B,\tau)}$,
		where $n=\rrk_B Q$ and $\langle 1\rangle_{(B,\tau)}$
		is the hermitian form $(x,y)\mapsto x^\tau y $ on $B$.
		We may therefore assume that $(Q,g)=(B^n,n\langle 1\rangle_{(B,\tau)})$.
		
		Arguing as in Subcase II.2 of the proof of  Proposition~\ref{PR:Ei-unitary-S-is-field-A-split},
		we may identify
		$A$ with $\nMat{R}{2}$ in such a way that $B$
		is the algebra of diagonal matrices and $\sigma$
		is given by $ [\begin{smallmatrix} a &  b \\ c & d \end{smallmatrix}]^{\sigma}
		=[\begin{smallmatrix} d &  \sigma_2b \\ \sigma_3c & a \end{smallmatrix}]^{\sigma}$,
		where $\sigma_2$, $\sigma_3$
		are $R$-linear automorphisms of $R$ of order $2$.
		Since $\sigma_2,\sigma_3\in\{\pm\id_R\}$
		and $\dim_R \Sym_{-1}(A,\sigma)=1$ (Proposition~\ref{PR:types-of-involutions-Az}), 
		we must have $\sigma_2=\sigma_3=\id_S$,
		hence $\sigma$ is given by $ [\begin{smallmatrix} a &  b \\ c & d \end{smallmatrix}]^{\sigma}
		=[\begin{smallmatrix} d &   b \\ c & a \end{smallmatrix}]$.

\medskip
		
		\noindent {\it Step 2.}
		Let $\tilde{A}=\nMat{A}{n}$
		and let $\tilde{\sigma}:\tilde{A}\to\tilde{A}$
		be given by $(a_{ij})^{\tilde{\sigma}}=(a_{ji}^\sigma)$.
		Define $\tilde{B}$ and $\tilde{\tau}$
		similarly and let $\tilde{\rho}:\tilde{B}\to\tilde{A}$
		denote the inclusion map.
		Let $e\in\tilde{B}=\nMat{B}{n}$ denote the matrix
		with $1$ in the $(1,1)$-entry and $0$ elsewhere,
		and let $\tilde{g} :\tilde{B}\times \tilde{B}\to 
		\tilde{B}$ denote the 
		diagonal hermitian form $\langle 1\rangle_{(\tilde{B},\tilde{\tau})}$
		(see Example~\ref{EX:diagonal-forms}).
		It is easy to check that the assumptions
		of Notation~\ref{NT:proof-of-Es}
		apply to $\tilde{A}$, $\tilde{\sigma}$, $T$ (embedded diagonally in $\tilde{A}=\nMat{A}{n}$)
		and $\tilde{B}$.
		Furthermore, under the evident isomorphisms
		$e\tilde{B}e\cong B$,
		$\tilde{B}e\cong B^n$, one finds that,
		the $e$-transfer $\tilde{g}_e$
		(see  
		\ref{subsec:conjugation})
		is just $g$.
		Thus, by Propositions~\ref{PR:simult-e-transfer}(ii)
		and~\ref{PR:Ei-Lag-partition}(iv), it is enough to prove that
		$\tilde{\rho}\tilde{g}$ 
		admits no Lagrangians $\tilde{L}$
		with $\tilde{B}\oplus \tilde{L}=\tilde{A}$ and $\Phi_{\tilde{g}}(\tilde{L})=-1$.
		
		Note that $(\tilde{A},\tilde{\sigma})\cong(\nMat{R}{2},\sigma)\otimes(\nMat{R}{n},\trans )$,
		where $\mathrm{t}$ denotes the transpose involution.
		Thus, we may identify $\tilde{A}$ with $\nMat{R}{2n}$ in such a way that $\tilde{\sigma}$
		is given by
		\[
		\SMatII{a}{b}{c}{d}^{\tilde{\sigma}}=\SMatII{d^\trans}{b^\trans}{c^\trans}{a^\trans}, 
		\]
		where $a,b,c,d\in\nMat{R}{n}$.
		Under this identification, $\tilde{B}=\{[\begin{smallmatrix} a & \\ & d\end{smallmatrix}]\where
		a,d\in\nMat{R}{n}\}$ and $T=\{[\begin{smallmatrix} \alpha 1_n & \\ & \beta 1_n \end{smallmatrix}]\where
		\alpha,\beta\in R\}$,
		where $1_n$ is the $n\times n$ identity
		matrix.
		
		Overriding previous notation, let $e=[\begin{smallmatrix}   1_n & 0\\0 & 0 \end{smallmatrix}]$.
		Then $e\tilde{A}=\tilde{B}e\tilde{A}$ is a Lagrangian of $\tilde{\rho}\tilde{g} $,
		and $\Phi_{\tilde{g} }({e}\tilde{A})=1$ by the defining property
		of $\Phi_{\tilde{g} }$.
		Since $U(\tilde{\rho}\tilde{g} )$ acts transitively
		on $\Lag(\tilde{\rho}\tilde{g} )$ 
		(Lemma~\ref{LM:Lag-transitive-action}),
		it is enough to prove that for every
		$\vphi \in U(\tilde{\rho}\tilde{g})$
		with $\Nrd(\vphi)=-1$, we have $\tilde{B}+\vphi e\tilde{A}\neq \tilde{A}$.
		Identifying $\End_{\tilde{A}}(\tilde{A}_{\tilde{A}})$
		with $\tilde{A}$ (acting on the left on itself)
		and writing
		$\vphi=[\begin{smallmatrix}{x}&{x'}\\{y}&{y'}\end{smallmatrix}]$ with $x,x',y,y'\in\nMat{R}{n}$,
		we get $\vphi e\tilde{A}=\{[\begin{smallmatrix} xa & xb \\ ya & yb\end{smallmatrix}]\where
		a,b\in \nMat{R}{n}\}$, from which it follows readily that
		\[
		\tilde{B}+\vphi e\tilde{A}=\tilde{A}
		\qquad\iff\qquad
		x,y\in\nGL{R}{n}.
		\]
		
	\medskip
		
		\noindent {\it Step 3.}
        Recall that $R$ is assumed to be algebraically closed.
        We shall view all
        finite dimensional $R$-vector spaces and 
        the group $U(\tilde{A},\tilde{\sigma})=
        U(\tilde{\rho}\tilde{g} )$
        as   varieties over $R$ in the obvious way. 
        Recall from
        Proposition~\ref{PR:U-zero-description} that $U(\tilde{A},\tilde{\sigma})$ has two
        (Zariski) connected components --- 
        $U^0(\tilde{A},\tilde{\sigma}):=
        U^0(\tilde{\rho}\tilde{g} )$ and $U^1(\tilde{A},\tilde{\sigma}):=U(\tilde{A},\tilde{\sigma})
        \setminus U^0(\tilde{A},\tilde{\sigma})$.
        
		Consider the morphism $\psi:U(\tilde{A},\tilde{\sigma})\to \nMat{R}{n}\times \nMat{R}{n}$
        given by $[\begin{smallmatrix}{x}&{x'}\\{y}&{y'}\end{smallmatrix}]\mapsto (x,y)$.
        By Step 2, we need to show
        that $\psi(U^1(\tilde{A},\tilde{\sigma}))$
		does not meet
		$\nGL{R}{n}\times\nGL{R}{n}$.
		Since $\nGL{R}{n}\times\nGL{R}{n}$ is Zariski open in $\nMat{R}{n}\times \nMat{R}{n}$,
		it is enough to verify  
		this 
		after replacing 
		$U^1(\tilde{A},\tilde{\sigma})$ with a Zariski dense subset.

	\medskip
		
		\noindent {\it Step 4.  }
		In what follows,
        we shall write matrices $a\in\nMat{R}{n}$ in $2\times 2$ block form 
        $\smallSMatII{a_{11}}{a_{12}}{a_{21}}{a_{22}}$,
        where $a_{11}$ is a $1\times 1$ matrix. 
        With this notation, let
        \[
        u:=\SMatII{\smallSMatII{0}{0}{0}{1_{n-1}}}{\smallSMatII{1}{0}{0}{0}}{\smallSMatII{1}{0}{0}{0}}{\smallSMatII{0}{0}{0}{1_{n-1}}}
        \]
        and note that $u\in U(\tilde{A},\tilde{\sigma})$
        and $\Nrd(u)=-1$.
		
        For all $a,b\in\Sym_{-1}(\nMat{R}{n},\trans)$, $c\in\nGL{R}{n}$, define
        \begin{align*}
        \xi(a,b,c)&=
        u\cdot 
        \SMatII{1}{0}{a}{1}\SMatII{c}{0}{0}{(c^{\trans})^{-1}}\SMatII{1}{b}{0}{1}
        =u\cdot \SMatII{c}{cb}{ac}{acb+(c^{\trans})^{-1}}.
        \end{align*}
        It is easy to check that $\xi$ is a morphisms of
        $R$-varieties from $\Sym_{-1}(\nMat{R}{n},\trans)\times
        \Sym_{-1}(\nMat{R}{n},\trans)\times\nGL{R}{n}$
        to $U^1(\tilde{A},\tilde{\sigma})$ that is injective on $R$-points.
		Since $U^1(\tilde{A},\tilde{\sigma})\cong U^0(\tilde{A},\tilde{\sigma})$
		as $R$-varieties, and since $U^0(\tilde{A},\tilde{\sigma})$
		is just $\uSO_{2n}(R)$, it follows
		that the source and target of $\xi$ have the same dimension (i.e.\ 
		$\frac{1}{2}(2n)(2n-1)=\frac{1}{2}n(n-1)+\frac{1}{2}n(n-1)+n^2$).
		Thus, by   Chevalley's Theorem,  $\im(\xi)$ is dense in $U^1(\tilde{A},\tilde{\sigma})$.
	    
	    Writing
        $a=\smallSMatII{0}{a_{12}}{a_{21}}{a_{22}}\in\Sym_{-1}(\nMat{S}{n},\trans)$
        and $c=\smallSMatII{c_{11}}{c_{12}}{c_{21}}{c_{22}}\in\nGL{R}{n}$, one readily checks that
        \[
        \xi(a,b,c)=\SMatII{\smallSMatII{a_{12}c_{21}}{a_{12}c_{22}}{c_{21}}{c_{22}}}{*}
        {*}{*}\ .
        \]
        Since $\smallSMatII{a_{12}c_{21}}{a_{12}c_{22}}{c_{21}}{c_{22}}$ is never invertible
        (multiply by $\smallSMatII{1}{-a_{12}}{0}{1}$ on the left), 
        we see that $\psi(\im(\xi))$ does not meet $\nGL{R}{n}\times\nGL{R}{n}$.
        Since $\im(\xi)$ is dense in $U^1(\tilde{A},\tilde{\sigma})$,
        this completes the proof.
	\end{proof}

	\begin{remark}
		In Proposition~\ref{PR:Ei-orthogonal-II-R-field-A-split},
		one can similarly show
		that if $\rrk_BQ$ is constant and odd, then
		there is no $L\in\Lag(\rho f)$
		such that $Q\oplus L=QA$:
		Replace $\xi$
		with the maps 
		$\xi_0(a,b,c)=\smallSMatII{1}{0}{a}{1}
		[\begin{smallmatrix}
		c & 0 \\ 0 & (c^\trans)^{-1}
		\end{smallmatrix}]
		\smallSMatII{1}{b}{0}{1}$
		and $\xi_1(a,b,c)=[\begin{smallmatrix}
		0 & 1_n \\ 1_n & 0
		\end{smallmatrix}]\xi_0(a,b,c)$
		and note that $a$ cannot be invertible when $n$ is odd.
	\end{remark}

	\begin{cor}\label{CR:Ei-good-Lag-is-in-Lag-zero}
		With Notation~\ref{NT:proof-of-Es},
		suppose  that
		$(\sigma,\veps)$ is orthogonal
		and $\tau $ is unitary.
		Let $(Q,g)\in\Herm[\veps]{B,\tau}$ and assume
		that $\rho g$ hyperbolic and $\rrk_BQ$ is constant and even.
		If $L\in \Lag(\rho g)$
		satisfies $Q\oplus L=QA$, then $\Phi_g(L)=1$.
	\end{cor}

	\begin{proof}
		Let $K$ be an algebraically closed
		$R$-field.
		Then $T_K\cong K\times K$.
		Thus, by Example~\ref{EX:exchange-involution},
		$g_K$ is hyperbolic. 
		Now, by Proposition~\ref{PR:Ei-orthogonal-II-R-field-A-split},
		$\Phi_g(L_K)=1$.
		Thanks to Lemma~\ref{LM:mu-two-check},
		$\Phi_g(L)=1$ follows by 
		letting $K$ range over the algebraic closures of 
		the residue fields of $R$.
	\end{proof}
	
	Now we can  prove an analogue to Propositions~\ref{PR:Ei-unitary-R-is-field}
	and~\ref{PR:Ei-orthogonal-I-R-field} in Case~\ref{item:orthII:Ei}.

	\begin{prp}\label{PR:Ei:orthII:good-Lag-Phi-one}
		With Notation~\ref{NT:proof-of-Es},
		suppose  that
		$R$ is a field, 
		$(\sigma,\veps)$ is orthogonal
		and $\tau $ is unitary.
		Let $(Q,g)\in \Herm[\veps]{B,\tau}$,
		let $L\in \Lag(\rho g)$ and assume 
		that $\rrk_BQ$ is constant and even.
		Then  there exists $\vphi\in U^0(\rho g)$
		such that $Q\oplus \vphi L=QA$
		if and only if
		$\Phi_g(L)=1$.
	\end{prp}
	
	\begin{proof}
		If $Q\oplus \vphi L=QA$ for $\vphi\in U^0(\rho g)$,
		then $\Phi_g(L)=\Nrd(\vphi)\Phi_g(L)=\Phi_g(\vphi L)=1$
		by Corollary~\ref{CR:Ei-good-Lag-is-in-Lag-zero}. We turn
		to prove the converse.	
	
		As in the proof of 	Proposition~\ref{PR:Ei-unitary-R-is-field},
		we can reduce to the case where $[A]=0$  
		and write $(Q,g)=(Q_1,g_1)\oplus (Q_2,g_2)$
		with   $g_1$ anisotropic  and  $g_2$ hyperbolic.
		Furthermore,
		there exists a Lagrangian  $L_1$ of $\rho g_1$ such that $Q_1\oplus L_1=Q_1A$.
		
		By Proposition~\ref{PR:Ei-orthII-odd-rank-forms}, $\rrk_BQ_1$ is  constant and even, and
		hence so is $\rrk_BQ_2$.
		Thus, by Proposition~\ref{PR:Ei-orthogonal-II-R-field-A-split}(i),
		there exists $L_2\in\Lag(\rho g_2)$ with $Q_2\oplus L_2=Q_2A$.
		
		Let $L':=L_1\oplus L_2$. Then $L'\in\Lag(\rho g)$
		and $Q\oplus L'=QA$.
		By  
		Lemma~\ref{LM:Lag-transitive-action},
		there exists $\vphi\in U(\rho g)$ such that $L'=\vphi L$.
		By Corollary~\ref{CR:Ei-good-Lag-is-in-Lag-zero},
		$1=\Phi_g(L')=\Nrd(\vphi)\Phi_g(L)=\Nrd(\vphi)$,
		so $\Nrd(\vphi)=1$ and the proposition follows. 
	\end{proof}
	
	From Proposition~\ref{PR:Ei:orthII:good-Lag-Phi-one},
	we see that in order to apply the proof of Theorem~\ref{TH:Ei-holds-unit-syp}
	to our situation, we have to find   $L\in \Lag(\rho g)$
	satisfying $\Phi_g(L)=1$. The purpose of the following propositions
	is to characterize precisely when
	such $L$ exists.
	
	We begin by noting that, in many cases,
	$\Phi_g$ is constant on the set $\Lag(\rho g)$.
	
	\begin{prp}\label{PR:Ei:Orth-Phi-constant}
		With Notation~\ref{NT:proof-of-Es},
		suppose  that
		$R$ is connected  semilocal, 
		$(\sigma,\veps)$ is orthogonal and 
		$ \tau $ is unitary.
		Let $(Q,g)\in\Herm[\veps]{B,\tau}$
		be an $\veps$-hermitian space
		such that $\rho g$ is hyperbolic
		and $\rrk_BQ$ is even.
		Then $\Phi_g:\Lag(\rho g)\to \mu_2(R)=\{\pm 1\}$
		is onto if and only if $[A]=0$ and $Q\neq 0$.
	\end{prp}
	
	\begin{proof}
		The proposition is clear when $Q=0$, so assume $Q\neq 0$.
	
		Suppose that $\Phi_g$ is onto.
		Then there are $L_0,L_1\in \Lag(\rho g)$
		such that $\Phi_g(L_0)=1$ and $\Phi_g(L_1)=-1$.
		By Lemma~\ref{LM:Lag-transitive-action},
		there exists $\vphi\in U(\rho g)$
		such that $\vphi L_0=L_1$,
		hence $\Nrd(\vphi)=\Nrd(\vphi)\Phi_g(L_0)=\Phi_g(L_1)=-1$.
		By Theorem~\ref{TH:criterion-for-det-one},
		this means that $[A]=0$.

		Conversely, if $[A]=0$, then Theorem~\ref{TH:criterion-for-det-one}
		implies the existence of $\vphi \in U(\rho g)$
		with $\Nrd(\vphi)=-1$.
		Choose some $L\in \Lag(\rho g)$.
		Then $\Phi_g(\vphi L)=-\Phi_g(L)$, hence $\Phi_g$
		is onto.
	\end{proof}

	\begin{prp}\label{PR:Ei:Orth-B-non-split-Phi-one}
		Under the assumptions of Proposition~\ref{PR:Ei:Orth-Phi-constant},
		if $T$ is connected and $[B]\neq 0$, then $\Phi_g(L)=1$ for all $L\in\Lag(\rho g)$.
	\end{prp}
	
	\begin{proof}
		For the sake of contradiction, suppose that there exists $L\in \Lag(\rho g)$
		with $\Phi_g(L)=-1$.
		By Lemma~\ref{LM:splitting-quad-et-algs},
		$T_T\cong T\times T$ 
		as $T$-algebras. Since $\tau_T$ is unitary, there exists $e\in T_T$
		such that $e^\tau+e=1$. By the definition of $\Phi_g$,
		we have $\Phi_g(Q_TeA_T)=1$, so $\Phi_g$ is not constant
		on $\Lag(\rho g_T)$.
		Now, applying Proposition~\ref{PR:Ei:Orth-Phi-constant} to $g_T$ (here we need $T$ to be connected),
		we get
		$[B]=[A_T]=0$, a contradiction.
	\end{proof}
	
	The next lemmas and proposition 
	concern with the case $[B]=0$. They will
	only be needed in proving part (i) of Theorem~\ref{TH:Ei-holds}.
	We shall make use of the discriminant algebra $D(g)$   defined in~\ref{subsec:disc}.
	
	\begin{lem}
		\label{LM:Ei-qaut-decomposition}
		With Notation~\ref{NT:proof-of-Es},
		suppose that $R$ is semilocal, $\deg B=1$,
		$\sigma$ is orthogonal, $\tau$ is unitary and $\veps=1$.
		Define $\lambda,\mu$ as in Lemma~\ref{LM:strcture-of-quat}(ii)
		(so $\lambda^\sigma=-\lambda$ and $\mu^\sigma=\mu$).
		Let $(Q,g)\in \Herm[1]{B,\tau}$
		and assume that $\rho g$ is hyperbolic
		and $\rrk_BQ$ is constant and even.
		Let $x_1,x_2\in Q$ and write $x=x_1+x_2\mu\in QA$.			
		\begin{enumerate}[label=(\roman*)]
			\item If $\rho g(x,x)=0$ and $g(x_1,x_1)\in\units{B}$,
			then $Q_1:=x_1B+x_2B$ is a summand of $Q$ with $B$-basis
			$\{x_1,x_2\}$.
			Writing $g_1=g|_{Q_1\times Q_1}$,
			the form $g_1$ is unimodular, $xA\in \Lag(\rho g_1)$,
			$Q_1\oplus xA=Q_1A$ and $[D(g_1)]=[A]$ in $\Br R$.
			\item If $\rrk_BQ\geq 4$, then there exist  $x_1,x_2,x$ as in (i). 
		\end{enumerate}  
	\end{lem}
	
	\begin{proof}
		(i) Write $\alpha:=g(x_1,x_1)\in\units{B}$.
		Since $g$ is $1$-hermitian
		and $\mu b=b^\sigma \mu$
		for all $b\in B$,
		we have $0=\rho g(x,x)=
		g(x_1,x_1)+2\mu g(x_2,x_1)+\mu^2 g(x_2,x_2)$,
		so $g(x_1,x_2)=0$ and $g(x_2,x_2)=-\mu^2g(x_1,x_1)$.
		By examining the Gram matrix of $g$ relative to $\{x_1,x_2\}$,
		we see that $\{x_1,x_2\}$ is a $g$-orthogonal basis to $Q_1$
		and $g_1$ is unimodular and isomorphic to $\langle \alpha,-\mu^2\alpha\rangle_{(B,\tau)}$.
		Thus, $D(g)=(B/R,\mu^2\alpha^2)\cong (B/R,\mu^2)\cong A$ (see~\ref{subsec:disc}).
		Let $x'=x_1-\mu x_2$. One readily checks that $\rho g(x',x')=0$
		and $xA\oplus x'A= Q_1A$, hence $xA\in \Lag(\rho g_1)$.
		
		We finish by checking that $Q_1\oplus xA=Q_1A$.
		If $y\in Q_1\cap xA$, then there is $a\in A$
		such that $y=xa=x_1a+x_2\mu a$. Since $\{x_1,x_2\}$ is a $B$-basis of $Q_1$,
		$\{x_1,x_2\}$ is an $A$-basis of $Q_1A$, so $y\in Q_1$
		implies that
		$a,\mu a\in B$. As a result $a\in B\cap \mu^{-1}B= B\cap \mu(\mu^{-2}B)
		\subseteq B\cap \mu B=0$ (because $\mu^{-2}\in \Cent_A(\lambda)= B$), and $y=xa=0$.
		This means that $Q_1\cap xA=0$.
		On the other hand,  $Q_1+xA\supseteq x_1B+x_2B+(x_1+x_2\mu)B+(x_1\mu+x_2\mu^2)B\supseteq
		x_1B+x_2B+x_1\mu B+x_2\mu B=Q_1A$, so $Q_1\oplus xA=Q_1A$.

\medskip		
		
		(ii) \Step{1}
		We first prove the claim when    $R$ is a field.
		Using Proposition~\ref{PR:ansio-Witt-equivalent},		
		write $(Q,g)=(Q_1,g_1)\oplus (Q_2,g_2)$ with $g_1$
		anisotropic and $g_2$ hyperbolic. Since $[\rho g_1]=[\rho g]=0$,
		the form $\rho g_1$
		is hyperbolic   (Theorem~\ref{TH:trivial-in-Witt-ring}(ii)).
		Since $\rrk_BQ\geq 4$, either $Q_1\neq 0$ or $\rrk_BQ_2\geq 4$.
		
		Assume
		$Q_1\neq 0$. Since $\rho g_1$ is hyperbolic,
		there exists $0\neq x\in Q_1A$ such that $\rho g_1(x,x)=0$.
		Write $x=x_1+x_2\mu$ with $x_1,x_2\in Q_1$.
		If $x_1=0$, replace $x$ with $x\mu$.
		Since $g_1$ is an anisotropic and $T$ is semisimple artinian, $g_1(x_1,x_1)\in \Sym_1(B,\tau)\setminus\{0\}=
		\units{R}$,
		so $x_1,x_2$ satisfy the requirements.
		
		Assume $\rrk_B Q_2\geq 4$.
		Since $g_2$ is hyperbolic,
		it has an orthogonal summand
		isomorphic to the hyperbolic form $\langle 1,-1,\mu^2,-\mu^2\rangle_{(B,\tau)}$
		(Lemma~\ref{LM:rank-determines-hyperbolic}).
		Now take $x_1$ and $x_2$ to be the elements corresponding to
		$(0,0,0,1)$ and $(1,0,0,0)$ in $Q$.
		
\medskip

		\Step{2}
		We continue to assume that $R$ is  field.
		Let $x,y\in QA$ be two elements such that
		$\rho g(x,x)=\rho g(y,y)=0$ and $\rrk_A xA=\rrk_A yA=2$.
		We claim that there exists $\vphi\in U^0(\rho g)$
		such that $\vphi x=y$.
		
		By Theorem~\ref{TH:Witt-extension},
		there exists $\psi\in U(\rho g)$ such that
		$\psi x=y$. If $\Nrd(\psi)=1$,
		then we can take $\vphi=\psi$, so assume
		$\Nrd(\psi)=-1$. In this case, $[A]=0$ by Theorem~\ref{TH:criterion-for-det-one}.
		
		Since $\rho g$ is unimodular and $xA $ is a free summand of $QA$,
		there exists $x'\in QA$
		such that $\rho g(x,x')=1$.
		Write $V=xA+x'A$. Since $\rho g(x,x)=0$,
		the restriction of $\rho g$ to
		$V$ is unimodular (the matrix $[\begin{smallmatrix}
		\rho g(x,x) & \rho g(x,x') \\ \rho g(x',x) & \rho g(x',x')\end{smallmatrix}]=
		[\begin{smallmatrix}
		0 & 1 \\ 1 & *\end{smallmatrix}]$
		is invertible), so $QA=V\oplus V^\perp$.
		Let $h=\rho g|_{V^\perp\times V^\perp}$.
		Since $\iota \rrk_AQA=2\rrk_B Q\geq 8$ and $\rrk_A V=4$,
		we have $V^\perp\neq 0$.
		Thus, by Theorem~\ref{TH:criterion-for-det-one},
		there exists $\psi_1\in U(h)$ with $\Nrd(\psi_1)=-1$.
		Take $\vphi=\psi\circ (\id_V\oplus \psi_1)$. 
		
\medskip

		\Step{3} We finally establish the existence of $x_1,x_2$
		in general.
		Let $L\in \Lag(\rho g)$.
		Then $\iota\rrk_A L=\frac{1}{2}\iota \rrk_A QA=\rrk_BQ\geq 4$,
		so $L$ admits a summand isomorphic to $A_A$ (Lemma~\ref{LM:rank-determines}).
		Let $y$ be a generator of such a summand.
		
		Let $\frakm_1,\dots,\frakm_t$
		denote the maximal ideals of $R$.
		By Step  1,
		for all $1\leq i\leq t$,
		there exists
		$x_i=x_{1i}+x_{2i}\mu$,
		with $x_{1i},x_{2i}\in Q(\frakm_i)$,
		such that $\rho g(\frakm_i)(x_i,x_i)=0$ and
		$g(\frakm_i)(x_{1i},x_{1i})\in \units{B(\frakm_i)}$.
		We observed in the proof of (i)  that $\rrk_{A(\frakm_i)} x_iA(\frakm_i)=2$,
		so by Step 2, there exists
		$\vphi_i\in U^0(\rho g(\frakm_i))$
		such that $\vphi_i(y(\frakm_i))=x_i$.		
		By Theorem~\ref{TH:U-zero-mapsto-onto-closed-fibers},
		there exists $\vphi\in U^0(\rho g)$
		such that
		$\vphi(\frakm_i)=\vphi_i$ for all $i$.
		Let $x=\vphi y$ and write $x=x_1+x_2\mu$
		with $x_1,x_2\in Q$.
		Since $QA=Q\oplus Q\mu$ (because $A=B\oplus B\mu$),
		we have  $x_1(\frakm_i)=x_{1i}$ for all $i$,
		hence  $g(x_1,x_1)\in \units{B}$ (Lemma~\ref{LM:invertability-test}).
		Since   $\rho g(x,x)=\rho g(y,y)=0$, we are done.
	\end{proof}

	\begin{lem}\label{LM:invertible-isotropic-vector}
		With Notation~\ref{NT:proof-of-Es},
		suppose that $R$ is semilocal, $\deg A=2$,
		$\sigma$ is orthogonal and $\veps=1$.
		Let $\alpha,\beta\in \units{R}$.
		If $f:=\langle\alpha,\beta\rangle_{(A,\sigma)}$
		is hyperbolic, then there exists $x\in\units{A}$
		such that $x^\sigma x=-\alpha\beta^{-1}$
	\end{lem}
	
	\begin{proof}
		The claim is equivalent to the existence
		of $x=(x_1,x_2)\in \units{A}\times \units{A}$
		such that $f(x,x)=\alpha x_1x_1^\sigma+\beta x_2x_2^\sigma=0$.
		Note that if the equality holds, then $x_1$ 
		is invertible if and only if $x_2$ is invertible.
		Since $f$ is hyperbolic, there exists an $A$-basis
		$\{u,v\}$ to $A^2$
		such that $f(u,u)=0$. Write $u=(u_1,u_2)\in A^2$.

\medskip

		\Step{1} Suppose $R$ is a field. We   claim that
		there exists $\vphi\in U^0(f)$
		such that $\vphi u\in \units{A}\times \units{A}$.
		If $u_1\in\units{A}$ or $u_2\in\units{A}$,
		then we can take $\vphi=\id_{A^2}$,
		so assume that both $u_1$ and $u_2$
		are not invertible. In particular, $A$ cannot
		be a division algebra, hence $A\cong \nMat{R}{2}$
		and $\rrk_A Au_1$ and $\rrk_A Au_2$ cannot exceed $1$.
		Since $u$ can be completed to an $A$-basis of $A^2$,
		we must have $Au_1+Au_2=A$. Length consideration
		now force $u_1$ and $u_2$ to be rank-$1$ matrices
		with $Au_1\cap Au_2=0$.
		Since $\alpha u_1^\sigma u_1=-\beta u_2^\sigma u_2$,
		this means that $u_1^\sigma u_1=0$.

		Arguing as in Reduction~\ref{RD:common-reduction-II},
		we may identify $A$ with $\nMat{R}{2}$
		in such a way that $\sigma$
		is given by $[\begin{smallmatrix} a & b\\ c & d\end{smallmatrix}]^\sigma=
		[\begin{smallmatrix} a & \gamma c\\ \gamma^{-1} b & d\end{smallmatrix}]$
		for some $\gamma\in \units{R}$.
		The condition $u_1^\sigma u_1=0$
		is easily seen to imply that $-\gamma$ is a square.
		Write $-\gamma=\delta^2$ with $\delta\in \units{R}$
		and let $c:=-\alpha\beta^{-1}$,
		\[x_1=  1_A ,\qquad
		x_2=\left[\begin{matrix} \frac{c+1}{2} & \frac{\delta(c-1)}{2} 
		\\ \frac{c-1 }{2\delta} & \frac{c+1}{2}\end{matrix}\right],
		\]
		and
		$x=(x_1,x_2)\in A^2$.
		It is routine to check that $\det x_2=c$ and $f(x,x)=0$.
		Since $xA$ is a summand of $A^2_A$,
		there exists $\vphi\in U(f)$
		such that $\vphi u=x$ (Theorem~\ref{TH:Witt-extension}). If
		$\Nrd \vphi=1$, we are done. If not,
		replace $x_2$
		with $[\begin{smallmatrix} 1 & 0 \\ 0 & -1\end{smallmatrix}]x_2$
		and $\vphi$ with $\psi\vphi$
		where $\psi\in U(f)$
		is given by $\psi(z_1,z_2)=(z_1,[\begin{smallmatrix} 1 & 0 \\ 0 & -1\end{smallmatrix}]z_2)$.
		
\medskip

		\Step{2} We now prove the general case.
		Let $\frakm_1,\dots,\frakm_t$
		denote the maximal ideals of $R$.
		By Step~1, for all $1\leq i\leq t$,
		there exists $\vphi_i\in U^0(f(\frakm_i))$
		such that $\vphi_i (u(\frakm_i))\in \units{A(\frakm_i)}\times \units{A(\frakm_i)}$.
		By Theorem~\ref{TH:U-zero-mapsto-onto-closed-fibers},
		there exists $\vphi\in U^0(f)$ with $\vphi(\frakm_i)=\vphi_i$
		for all $i$. Now, by Lemma~\ref{LM:invertability-test},
		$x=(x_1,x_2):=\vphi u\in \units{A}\times \units{A}$ and $f(x,x)=f(u,u)=0$.
	\end{proof}
	
	\begin{prp}
		\label{PR:Ei-orth-quat-case}
		With Notation~\ref{NT:proof-of-Es},
		suppose that $R$ is semilocal, $T$ is connected, $[A]\neq 0$, $[B]=0$, 
		$(\sigma,\veps)$
		is orthogonal and $\tau$ is unitary.
		Let $(Q,g)\in \Herm[\veps]{B,\tau}$
		be a hermitian space such that $\rho g$ is hyperbolic
		and $n:=\rrk_B Q$
		is  even.
		Then $\Phi_g(L)=1$ for some $L\in \Lag(\rho g)$
		if and only if $[D(g)]=\frac{n}{2}\cdot [A]$.
		When this fails, $[D(g)]=(\frac{n}{2}+1)\cdot [A]$ and $g$ is isotropic.
	\end{prp}
	
	\begin{proof}
		By Reduction~\ref{RD:common-reduction} and Proposition~\ref{PR:Ei-Lag-partition}(iv),
		we may assume that $\deg B=1$, $\deg A=2$,
		$\sigma$ is orthogonal and $\veps=1$.	
		Let $\lambda,\mu\in\units{A}$
		be as in Lemma~\ref{LM:strcture-of-quat}(ii) (so $\lambda^\sigma=-\lambda$
		and $\mu^\sigma =\mu$).
		By Proposition~\ref{PR:Ei:Orth-Phi-constant},
		$\Phi_g$ is constant on $\Lag(\rho g)$; we shall
		denote the value that it attains by $\bar{\Phi}_g\in \{\pm 1\}$.
		The proposition clear if $n=0$, so assume $n>0$.
		
		Suppose $n\geq 4$. Then by Lemma~\ref{LM:Ei-qaut-decomposition},
		we can write $(Q,g)=(Q_1,g_1)\oplus (Q_2,g_2)$,	
		where $\rrk_B Q_1=2$, $[D(g_1)]=[A]$ and there exists
		$L\in\Lag(\rho g_1)$ with $Q_1\oplus L=Q_1A$.
		By Corollary~\ref{CR:Ei-good-Lag-is-in-Lag-zero},
		$\bar{\Phi}_{g_1}=1$.
		Since $\bar{\Phi}_{g}=\bar{\Phi}_{g_1}\bar{\Phi}_{g_2}$
		(Proposition~\ref{PR:Ei-Lag-partition}(iii)),
		$[D(g)]=[D(g_1)]+[D(g_2)]$ (Proposition~\ref{PR:disc-unitary-basic-props}(iii))
		and $[\rho_2 g]=[\rho g]=0$ in $W_\veps(A,\sigma)$ (so $\rho_2g$
		is hyperbolic by Theorem~\ref{TH:trivial-in-Witt-ring}(ii)),
		the proposition will hold for $(Q,g)$
		if it holds  
		for $(Q_2,g_2)$. Repeating this process,
		we reduce to the case $n=2$.
		
		Suppose henceforth that $n=2$. By Proposition~\ref{PR:diagonalizable-herm-forms},
		we may assume that $g = \langle \alpha,\beta\rangle_{(B,\tau)}$
		for some $\alpha,\beta\in \units{B}\cap\Sym_1(B,\tau)=\units{R} $,
		and by Lemma~\ref{LM:invertible-isotropic-vector},
		there exists $x\in\units{A}$
		such that $x^\sigma x=-\alpha\beta^{-1}$.
		Note that 
		$\disc(g)\equiv -\alpha\beta\equiv -\alpha\beta^{-1}\bmod \Nr_{T/R}(\units{T})$,
		hence
		$[D(g)]=[(B/R,-\alpha\beta^{-1})]$ (see~\ref{subsec:disc}).
		
		Write $x=b_1+\mu b_2$ with $b_1,b_2\in B$.
		Since $\mu^\sigma=\mu$
		and $\mu b=b^\sigma \mu$ for all $b\in B$, we have
		\[
		-\alpha\beta^{-1}=x^\sigma x=(b_1^\sigma b_1+\mu^2 b_2^\sigma b_2)+2\mu b_1b_2.
		\]
		Thus, $b_1b_2=0$ and $b_1^\sigma b_1+\mu^2b_2^\sigma b_2=\alpha\beta^{-1}$.
		Arguing as in~\cite[Example~9.4]{Reiner_2003_maximal_orders_reprint}
		(for instance), we see that $\Nrd_{A/R}(x)=b_1^\sigma b_1-\mu^2 b_2^\sigma b_2$.
		Since $\Nrd_{A/R}(x)\in\units{R}$, this means
		that  $b_1B+b_2B=B$.
		
		We claim that $b_1=0$  
		or $b_2=0$. Indeed,
		$b_1B=b_1(b_1B+b_2B)=b_1^2B$, so there exists
		$c\in B$ with $b_1=b_1^2c$. In particular, $b_1c$
		is an idempotent. Since $T$ is connected,
		$b_1c=0$ or $b_1c=1$. In the first case,
		$b_1=b_1^2c=0$, whereas in the second case,
		$b_1\in\units{B}$, so $b_2=0$ because $b_1b_2=0$.
		
		Assume $b_1=0$. Then $x=\mu b_2\in\units{A}$ and $-\alpha\beta^{-1}=\mu^2 b_2^\sigma b_2$,
		hence $[D(g)]=[(B/R,\mu^2)]=[A]$.
		Let $L= [\begin{smallmatrix} 1 \\ \mu b_2 \end{smallmatrix}]A$
		and $L'=[\begin{smallmatrix} -1 \\ \mu b_2 \end{smallmatrix}]A$.
		One readily checks that $\rho g(L,L)=\rho g(L',L')=0$
		and $L\oplus L'=A^2$, hence $L\in\Lag(\rho g)$.
		Furthermore, $B^2\cap L=0$ and $
		[\begin{smallmatrix} 0\\ 1 \end{smallmatrix}],	
		[\begin{smallmatrix} 1\\ 0 \end{smallmatrix}],		
		[\begin{smallmatrix} \mu \\ \mu^2 b_2^\sigma \end{smallmatrix}],
		[\begin{smallmatrix} b_2^{-1} \\ \mu  \end{smallmatrix}]\in B^2+L$,
		so $B^2\oplus L=A^2$ and $\bar{\Phi}_g=1$ by Corollary~\ref{CR:Ei-good-Lag-is-in-Lag-zero}.

		Assume $b_2=0$. Then $x=b_1$ and $-\alpha\beta^{-1}=b_1^\sigma b_1$,
		hence $[D(g)]=[(B/R, 1)]=0$. 
		Now, Theorem~\ref{TH:period-divides-index}  and 
		and our assumption that $[A]\neq 0$ imply that $[D(g)]=(\frac{2}{2}+1)[A]\neq [A]$.
		Furthermore, it is  routine to check that  $M=[\begin{smallmatrix}  1 \\ b_1 \end{smallmatrix}]B$
		is a Lagrangian of $g$, so $g$ is hyperbolic (and in particular
		isotropic) and $\bar{\Phi}_g=-1$
		by Proposition~\ref{PR:Ei-Lag-partition}(ii).
		
		Since we cannot have both 
		$b_1=0$ and $b_2=0$ (because $x\in\units{A}$),
		the proposition follows.
	\end{proof}

	We  finally complete the proof Theorem~\ref{TH:Ei-holds} by
	establishing   case \ref{item:orthII:Ei}.
	 
	\begin{thm}\label{TH:Ei-holds-orth-II}
		Theorem~\ref{TH:Ei-holds}
		holds when $R$ is connected, $(\sigma,\veps)$ is orthogonal
		and $\tau $ is unitary.
	\end{thm}
	
	\begin{proof}
		Recall that we are given $(Q,g)\in\Herm[\veps]{B,\tau}$
		such that $[\rho g]=0$ in $W_\veps(A,\sigma)$.
		By Theorem~\ref{TH:trivial-in-Witt-ring}(ii), $\rho g$
		is hyperbolic.	
	
		(i) Since $T$ is connected, $n:=\rrk_BQ$
		is even by Proposition~\ref{PR:Ei-orthII-odd-rank-forms}.
		
		Suppose that there exists $L\in\Lag(\rho g)$
		with $Q\oplus L=QA$.
		By Corollary~\ref{CR:Ei-good-Lag-is-in-Lag-zero},
		$\Phi_g(L)=1$. Now, if  $[A]\neq 0$ and $[B]=0$, then we must
		have $[D(g)]=\frac{n}{2}\cdot [A]$  by Proposition~\ref{PR:Ei-orth-quat-case},
		as required.

		Conversely, suppose   that $[A]=0$, or $[B]\neq 0$, or $[D(g)]=\frac{n}{2}\cdot [A]$.
		If there exists $L \in\Lag(\rho g)$ with $\Phi_g(L )=1$,
		then we can argue as in 
		the last paragraph of the proof of Theorem~\ref{TH:Ei-holds-unit-syp},
		using Proposition~\ref{PR:Ei:orthII:good-Lag-Phi-one} instead
		of Propositions~\ref{PR:Ei-unitary-R-is-field}
		and~\ref{PR:Ei-orthogonal-I-R-field}, to prove
		the existence of $L'\in \Lag(\rho g)$ with $Q\oplus L'=QA$.
		The existence of $L$ follows from Proposition~\ref{PR:Ei:Orth-Phi-constant} if
		$[A]= 0$, from Proposition~\ref{PR:Ei:Orth-B-non-split-Phi-one} if $[B]\neq 0$,
		and from Proposition~\ref{PR:Ei-orth-quat-case} if $[A]\neq 0$ and $[B]=0$.
		Proposition~\ref{PR:Ei-orth-quat-case}, also tells
		us that  $[D(g)]=(\frac{n}{2}+1)[A]$ 
		and $g$ is isotropic when $[A]\neq 0$, $[B]=0$ and $[D(g)]\neq \frac{n}{2}\cdot [A]$.

		(ii) We need to prove the existence of $L\in \Lag(\rho g)$
		with $Q\oplus L=QA$, possibly
		after replacing $(Q,g)$ with a Witt-equivalent hermitian space.
		
		If $T$ is not connected,
		then, as explained in the proof of Proposition~\ref{PR:Ei-orthII-odd-rank-forms},
		$g$ is   hyperoblic and  can thus be replaced with   zero form.
		
		Suppose that $T$ is connected. As in the proof of (i), $\rrk_BQ$ is even
		and
		it is enough to find $L\in \Lag(\rho g)$ with $\Phi_g(L)=1$.
		Moreover, we showed that $L$ exists if $[B]\neq 0$ in $\Br T$, so we 
		only need to consider
		the case where $[B]=0$. Let $L \in \Lag(\rho g)$.
		If $\Phi_g(L)=1$, we are done, so assume $\Phi_g(L)=-1$.
		Since $[B]=0=[T]$,
		there exists $N\in\rproj{B}$
		with $\rrk_B N=\deg T=1$ (Proposition~\ref{PR:degree-of-endo-ring}(iii)).
		Consider $(Q',g'):=(Q,g)\oplus (N\oplus N^*,\Hyp[\veps]{N})$,
		which is Witt-equivalent to $(Q,g)$.
		By parts (ii) and (iii) of Proposition~\ref{PR:Ei-Lag-partition},
		we have $\Phi_{g'}(L \oplus NA)=\Phi_g(L )\cdot \Phi_{\Hyp[\veps]{N}}(NA)=
		(-1)\cdot (-1)=1$.  We may therefore replace
		$(Q,g)$, $L$ with $(Q',g')$, $L\oplus N$ and finish.
	\end{proof}

\section{Verification of (E1) and (E3)}
\label{sec:Eii}

	Keep the assumptions of Notation~\ref{NT:proof-of-Es}.
	The purpose of this section is to prove:

	\begin{thm}\label{TH:Eii-holds}
		With Notation~\ref{NT:proof-of-Es}, suppose that $R$
		is semilocal and let  $(P,f)\in \Herm[\veps]{A,\tau}$ be a hermitian
		spaces such that $[\pi f]=0 $ in $W_\veps(B,\tau)$.  
		\begin{enumerate}[label=(\roman*)]
		\item Assume that $T$ is connected and $(\tau,\veps)$
		is not orthogonal.
		Then there exists a Lagrangian $M$ of $\pi f$
		such that $MA=P$ if and only if 
		$(\tau,\veps)$ not unitary, or $(\sigma,\veps)$ is not symplectic,
		or  $4\mid \rrk_AP$. 
		When these conditions fail, $[A]=0$ in $\Br S$ and $f$ is hyperbolic.
		\item Assume that $T$ is connected and $(\tau,\veps)$
		is   orthogonal. 
		Then $(\sigma,\veps)$ is orthogonal.
		Moreover, 
		there exists a Lagrangian $M$ of $\pi f$
		such that $MA=P$ if and only if  
		$[B]\neq 0$ in $\Br T$, or 
		$\rrk_AP$ is even and
		$\disc (f)=\disc(T/R)^{\frac{1}{2}\rrk_AP}$
		(see \ref{subsec:disc}).
		When these conditions fail, $[A]=0$ in $\Br S$,
		$\rrk_AP$ is even, $\disc (f)=\disc(T/R)^{\frac{1}{2}\rrk_AP+1}$
		and $f$ is isotropic.
		\item 
		There exists $(P',f')\in \Herm[\veps]{A,\sigma}$ with $[f]=[f']$
		and a Lagrangian $M$ of  $\pi f'$ such that
		$M  A=P'$.
		\end{enumerate}
	\end{thm}
	
	In Section~\ref{sec:completion-of-proof}, we will use this theorem
	to establish    conditions \ref{item:Eii}
	and \ref{item:Eii-tag} of Theorem~\ref{TH:exactness-equiv-conds} when
	$R$ is semilocal. The reader can skip to the next section without loss of continuity.
	
\medskip
	
	As with Theorem~\ref{TH:Ei-holds}, it is enough
	to prove the theorem when $R$ is connected. 
	In this case, by Lemma~\ref{LM:tau-type-is-constant},
	exactly one of the following hold:
	\begin{enumerate}[label=(\arabic*)]
		\item \label{item:Eii:sym-or-unit} $(\tau,\veps)$ is symplectic or unitary,
		\item \label{item:Eii:orth} $(\tau,\veps)$ is orthogonal.
	\end{enumerate}
	These cases are treated in Theorems~\ref{TH:Eii-unitary-symplectic}
	and~\ref{TH:Eii-holds-orth}, respectively.

\subsection{Non-Connected Cases}
\label{subsec:Eii-T-not-contd}

	We begin by addressing the  simpler case
	where $T$ is not connected.
	Some of the observations made here will be used
	later.
	
\medskip

	First, we consider the case where $S$ is not connected.
	
	\begin{prp}\label{PR:Eii-S-not-connected}
		Theorem~\ref{TH:Eii-holds}(iii)
		holds when $R$ is connected and $S$ is not connected.
	\end{prp}
	
	\begin{proof}
		In this case,  $S=R\times R$ (Lemma~\ref{LM:non-connected-S})
		and $\tau|_S$ is the exchange involution.
		As observed in Example~\ref{EX:exchange-involution},
		this means that every $(P,f)\in \Herm[\veps]{A,\sigma}$
		is hyperbolic. Replacing $(P,f)$ with zero
		form,  Theorem~\ref{TH:Eii-holds}(iii)
		holds trivially.
	\end{proof}
	
	If $S$ is connected and $T$ is not connected,
	then  $T\cong S\times S$ (Lemma~\ref{LM:non-connected-S}),
	hence $T$ admits two nontrivial idempotents, call
	them $e$ and $e':=1-e$.
	We have $e^\sigma=e$ or $e^\sigma=e'$.
	We devote some attention to the case $e^\sigma=e$, working
	in slightly greater generality for later reference.
	
\medskip	
	
 	With Notation~\ref{NT:proof-of-Es}, suppose that $T= S\times S$ (but not
	that $S$ is connected),
	let $e:=(1_S,0_S)$ and assume $e^\sigma=e$.
	By Lemma~\ref{LM:idempotent-in-T},  
	we have $A=AeA$,  $B=eB\oplus e'B=eAe\oplus e'Ae'$ and $\pi:A\to B$
	is given by $a\mapsto eae+e'ae'$.
	Since $e^\sigma=e$, we may
	view $(B,\tau)$ as $(eB,\tau_e)\times (e'B,\tau_{e'})$,
	where $\tau_e=\tau|_{eB}$ and $\tau_{e'}=\tau|_{e'B}$.
	Thus, every hermitian space $(Q,g)\in \Herm[\veps]{B,\tau}$
	factors as $(Qe,g_e)\times (Qe',g_{e'})$, where $g_e=g|_{Qe\times Qe}$
	and $g_{e'}=g|_{Qe'\times Qe'}$ are $\veps$-hermitian forms
	over $(eB,\tau_e)$ and $(e'B,\tau_{e'})$, respectively.
	The following simple observation will be important in the sequel.
	
	\begin{prp}\label{PR:Eii:pi-f-T-not-connected}
		Under the previous assumptions and identifications,
		for every $(P,f)\in \Herm[\veps]{A,\sigma}$,
		the hermitian space $(P,\pi f)$
		is $(Pe,f_e)\times (Pe',f_{e'})$,
		where $f_e$
		denotes the $e$-transfer of $f$ (see~\ref{subsec:conjugation}), and likewise for $f_{e'}$.
	\end{prp}
	
	\begin{proof}
		This is straightforward.
	\end{proof}

	\begin{prp}\label{PR:Eii-T-not-connected}
		Theorem~\ref{TH:Eii-holds}(iii) holds when $S$ is connected and $T$ is not connected.
	\end{prp}
	
	\begin{proof}
		Write $T=S\times S$ and let 
		$e=(1_S,0_S)$.
		By Lemma~\ref{LM:idempotent-in-T}(ii),
		$\rrk_AeA>0$ and   $AeA=A$.
		Since $S$ is connected, $e^\sigma=e$ or $e^\sigma=1-e$.

		If $e^\sigma=e$, then $f_e$ is hyperbolic  by Proposition~\ref{PR:Eii:pi-f-T-not-connected}
		and the fact that $\pi f$ is hyperbolic (Theorem~\ref{TH:trivial-in-Witt-ring}(ii)).
		By   item \ref{item:e-transfer-hyperbolic} in \ref{subsec:conjugation}, this means
		that   $f$ is hyperbolic, so we may replace $(P,f)$ with zero form and finish.
		
		Suppose that $e^\sigma=1-e$. Then $M:=Pe$ is a Lagrangian of $\pi f$
		(Example~\ref{EX:exchange-involution}).
		Since $MA=PeA=PAeA=PA=P$, we are done.
	\end{proof}

\subsection{Case (1)}

	We now prove   Theorem~\ref{TH:Eii-holds}
	in case~\ref{item:Eii:sym-or-unit},
	namely, when  $R$ is connected  and $(B,\tau)$ is unitary
	or symplectic.
	As with Theorem~\ref{TH:Ei-holds}, we first establish some
	special cases.
	
	\begin{prp}\label{PR:Eii-unitary-A-split-S-split}
		With Notation~\ref{NT:proof-of-Es},
		suppose that $R$ is a field,   $S\cong R\times R$  and $[A]=0$.
		Let $(P,f)\in \Herm[\veps]{A,\sigma}$
		be a hermitian space such that 
		$\rrk_AP$ is even.
		Then there exists $M\in \Lag(\pi f)$
		such that $MA=P$ and $\rrk_BM=\frac{1}{2}\rrk_BP$.
	\end{prp}
	
	\begin{proof}
		By Reduction~\ref{RD:common-reduction}, we may
		assume that $B=T$ and $\deg A=2$. Note that $\rrk_AP$
		is constant by Corollary~\ref{CR:constant-even-ranks}(i).
	
		By assumption, there exists a  nontrivial idempotent
		$e\in S$
		such that $e+e^\sigma=1$.
		By Example~\ref{EX:exchange-involution},
		we may assume that $A=A_1\times A_1^\op$
		for a  central simple $R$-algebra $A_1$
		and that $\sigma$ is the exchange involution
		$(x,y^\op)\mapsto (y,x^\op)$.
		It is easy to see that there are $R$-subalgebras
		$T_1\subseteq B_1\subseteq A_1$
		such that  $T=T_1\times T_1^\op$, $B=B_1\times B_1^\op$
		and $B_1=\Cent_{T_1}(A_1)$.
		Furthermore, $\veps=(\veps_1,\veps_1^{-1})$
		for some $\veps_1\in \units{R}$.

		Consider the $\veps$-hermitian form
		$f_1:A\times A\to A$
		given by $f_1((x_1,x_2^\op),(y_1,y_2^\op))=(\veps_1 x_2y_1,(y_2x_1)^\op)$.
		Since $\rrk_AP$ is even and constant, and since $\deg A=2$,
		there exists
		$n\in \N$
		such that $P\cong A_A^n$.
		By Example~\ref{EX:exchange-involution},
		this means that $(P,f)\cong n\cdot (A,f_1)$.
		It is therefore enough to prove the proposition
		for $(P,f)=(A,f_1)$.
		
		Let $\pi_1:=\pi_{A_1,B_1}:A_1\to B_1$ be as in Lemma~\ref{LM:dfn-of-pi-A-B}.
		The uniqueness of $\pi$ forces $\pi(x,y^\op)=(\pi_1x,(\pi_1y)^\op)$ for all
		$x,y\in A_1$.
		Let $E_1=\ker \pi_1$. 
		Then $M:=E_1\times B_1^\op $
		and $M':=B_1\times E_1^\op$
		are submodules of $A_B$
		satisfying $f_1(M,M)=f_1(M',M')=0$ and $M\oplus M'=A$.
		Thus, $M$ is a Lagrangian of $\pi f_1$.
		In addition, Lemma~\ref{LM:dfn-of-pi-A-B}(iv) tells us that $E_1A_1=A_1$,
		so $MA=E_1A_1\times  A_1^\op B_1^\op=A$, as required.
	\end{proof}
	
	The following proposition holds without restrictions on
	the type of $(\tau,\veps)$.
	
	\begin{prp}\label{PR:Eii:rrk-divisible-by-four}
		With Notation~\ref{NT:proof-of-Es},
		suppose that $S$ is a field and $[A]=0$.
		Let $(P,f)\in \Herm[\veps]{A,\sigma}$
		be a hyperbolic hermitian space.
		If $4\mid\rrk_AP$,
		then $\pi f$ admits a Lagrangian $M$
		such that $MA=P$ and $\rrk_B M=\frac{1}{2}\rrk_BP$.
	\end{prp}
	
	\begin{proof}
		By Reduction~\ref{RD:common-reduction},
		we may assume that $B=T$
		and $\deg A=2$.
		Consider
		the hyperbolic $\veps$-hermitian
		form $f_1:A^2\times A^2\to A$ given 
		by $f_1((x_1,y_1),(x_2,y_2))=x_1^\sigma y_2+\veps x_2^\sigma y_1$.
		Writing $n=\frac{1}{4}\rrk_AP$, we
		have $n\cdot (A^2,f_1)\cong (P,f)$ by Lemma~\ref{LM:rank-determines-hyperbolic}.
		It is therefore enough to prove the proposition
		for $(A^2,f_1)$.
		Write $E=\ker \pi$
		and let $M:=B\times E$ and $M':=E\times B$;
		both $M$ and $M'$ are right $B$-submodules of $A^2$. One readily
		checks  that $\pi f_1(M,M)=\pi f_1(M',M')=0$
		and $A^2_A=M\oplus M'$. Thus, $M$ is a Lagrangian
		of $\pi f$.  By Lemma~\ref{LM:dfn-of-pi-A-B}(iv), 
		we have $MA=A^2$ and $\rrk_BM=2\rrk_BB=2=\frac{1}{2}\rrk_BA^2_B$.		
	\end{proof}
	
	\begin{prp}\label{PR:Eii-unitary-A-split-S-field}
		With Notation~\ref{NT:proof-of-Es},
		suppose that $S$ is a field, $[A]=0$ and
		$(\tau,\veps)$ is symplectic or unitary.
		If  $\tau$ is unitary, we also assume
		that $(\sigma,\veps)$ is not symplectic.
		Let $(P,f)\in \Herm[\veps]{A,\sigma}$
		be a hyperbolic hermitian space.
		Then there exists a Lagrangian $M$ of $\pi f$
		such that $MA=P$ and $\rrk_B M=\frac{1}{2}\rrk_BP$.
	\end{prp}
	
	\begin{proof}
		By Reductions~\ref{RD:common-reduction} and~\ref{RD:common-reduction-II},
		we may assume that $B=T$, $A=\nMat{S}{2}$, 
		$\tau$ is orthogonal or unitary,
		and $\sigma$ is orthogonal or unitary
		and given by $[\begin{smallmatrix} a & b \\
		c & d\end{smallmatrix}]\mapsto[\begin{smallmatrix} a^\sigma & \alpha c^\sigma \\
		\alpha^{-1}b^\sigma & d^\sigma\end{smallmatrix}] $
		for some $\alpha\in \units{R}$. 
		By Corollary~\ref{CR:constant-even-ranks}(ii), $\rrk_AP$ is even.
		
		Let $(P_1,f_1)$
		be a hyperbolic hermitian space
		such that $\rrk_AP_1=2$.
		Then, by Lemma~\ref{LM:rank-determines-hyperbolic}, $(P,f)\cong n\cdot (P_1,f_1)$
		for some $n\in \N$.
		It is therefore enough to prove the proposition
		for $(P_1,f_1)$.
		We now split into cases, making different choices
		of $(P_1,f_1)$ in each case.

\medskip

	\noindent {\it Case I.   $\tau|_T$ is not
		the standard $S$-involution of $T$.}
		We  may assume that $\veps=-1$.
		This already holds when $\sigma|_S=\id_S$,
		because  then $\tau$ is orthogonal while
		$(\tau,\veps)$ is symplectic.
		When $\sigma|_S\neq\id_S$, by
		Hilbert's Theorem 90, there exists
		$\eta\in S$ with $\eta^\sigma\eta^{-1}=-\veps$
		and we can apply $\eta$-conjugation
		(see \ref{subsec:conjugation}, Proposition~\ref{PR:simult-conj})
		to   replace $f$, $\veps$ with  {$\eta f$}, $-1$.

		Consider
		the hyperbolic $(-1)$-hermitian form
		$f_1:A\times A\to A$
		given by $f_1(x,y)=x^\sigma[\begin{smallmatrix} 0 & -\alpha \\ 1 & 0\end{smallmatrix}] y$;
		note that $\rrk_A A=2$ and
		$[\begin{smallmatrix} S & S \\ 0 & 0\end{smallmatrix}]$
		is a Lagrangian of $f_1$.
		
		We claim that there exists
		$\lambda\in T$
		such that $\lambda^2\in \units{S}$, $T=S\oplus \lambda S$
		and $\lambda^\sigma =\lambda$.
		If $\sigma|_S=\id_S$, then $\sigma|_T=\id_T$ 
		and the existence of $\lambda$ follows from Lemma~\ref{LM:quad-etale-over-semilocal}.
		If $\sigma|_S\neq \id_S$, then $\tau$ is unitary (Lemma~\ref{LM:sigma-unit-means-tau-unit}),
		so $T$ is quadratic \'etale over $T_0:=\Sym_1(T,\tau)$,
		which is in turn quadratic \'etale over $R$ and satisfies
		$T_0\cap S=\Sym_1(S,\sigma)=R$. Applying Lemma~\ref{LM:quad-etale-over-semilocal}
		to $T_0$, we see that
		there exists $\lambda\in\units{T_0}\setminus R$   and $\lambda^2\in\units{R}$.
		Thus, $\lambda^\sigma=\lambda$, and $T=S\oplus \lambda S$ because $\lambda\notin S$
		and $\dim_S T=2$.

		The  conditions
		$\lambda=\lambda^\sigma$ and $\lambda^2\in \units{S}$
		force 
		$\lambda=[\begin{smallmatrix} a & \alpha b \\ b & -a\end{smallmatrix}]$
		for some $a,b\in S$
		such that $a^\sigma=a$ and $a^2+\alpha b^2\in \units{S}$.				
		Using Lemma~\ref{LM:dfn-of-pi-A-B}(ii),
		it is routine to check that 
		\[
		\pi ([\begin{smallmatrix} 0 & -\alpha \\ 1 & 0\end{smallmatrix}])
		=\textstyle\frac{1}{2}[\begin{smallmatrix} 0 & -\alpha \\ 1 & 0\end{smallmatrix}]
		+\frac{1}{2}[\begin{smallmatrix} a & \alpha b \\ b & -a\end{smallmatrix}]^{-1}
		[\begin{smallmatrix} 0 & -\alpha \\ 1 & 0\end{smallmatrix}]		
		[\begin{smallmatrix} a & \alpha b \\ b & -a\end{smallmatrix}]
		=0.
		\]	
		Thus, for all $x,y\in B$, we have
		\[
		\pi f_1(x,y)=\pi(x^\tau [\begin{smallmatrix} 0 & -\alpha \\ 1 & 0\end{smallmatrix}]y)=
		x^\tau \pi([\begin{smallmatrix} 0 & -\alpha \\ 1 & 0\end{smallmatrix}]) y=0.
		\]	
		It follows that $B$ is a Lagrangian of $\pi f_1$.
		Since $\rrk_BB_B=1=\frac{1}{2}\rrk_BA$ and
		$BA=A=P$, we are done.
		
\medskip

	\noindent {\it Case II.   $\tau|_T$ is  
		the standard $S$-involution of $T$.}
		Since $\tau$ is unitary,
		$(\sigma,\veps)$ is necessarily orthogonal.
		Thus, $\sigma$ is orthogonal and $\veps=1$.

		By Proposition~\ref{PR:types-of-involutions-Az},
		$\dim_S\Sym_{-1}(T,\sigma)=1=\dim_S\Sym_{-1}(A,\sigma)$,
		so we must have
		$T=S+\Sym_{-1}(A,\sigma)=S\oplus \lambda S$
		with $\lambda:=[\begin{smallmatrix} 0 & \alpha \\ -1 & 0\end{smallmatrix}]$.
		Using Lemma~\ref{LM:dfn-of-pi-A-B}(ii), it is   routine to
		check that $\pi:A\to B$
		is given by $\pi[\begin{smallmatrix} x & y \\ z & w\end{smallmatrix}]
		=\frac{1}{2}(x+w)[\begin{smallmatrix} 1 & 0 \\ 0 & 1\end{smallmatrix}]+
		\frac{1}{2}(\alpha^{-1}y-z)[\begin{smallmatrix} 0 & \alpha \\ -1 & 0\end{smallmatrix}]$.

		Consider the $1$-hermitian hyperbolic form $f_1:A\times A\to A$
		given by $f(x,y)=x^\sigma [\begin{smallmatrix} 0 &  \alpha \\ 1 & 0\end{smallmatrix}] y$ (the space
		$[\begin{smallmatrix} S & S \\ 0 & 0\end{smallmatrix}]$
		is a Lagrangian).
		As in Case I, one readily checks that $\pi [\begin{smallmatrix} 0 &  \alpha \\ 1 & 0\end{smallmatrix}]=0$
		and hence $B$ is  the
		required Lagrangian of $\pi f_1$.
	\end{proof}
	
	The case where $\tau$ is unitary and $(\sigma,\veps)$
	is symplectic
	is addressed in the following proposition.
	Note that we allow  $R$ to be semilocal.

	\begin{prp}\label{PR:Eii:unitary-symplectic-case}
		With Notation~\ref{NT:proof-of-Es},
		suppose that    $R$ is connected semilocal  
		and 
		$(\sigma,\veps)$ is symplectic.
		Let $(P,f)\in\Herm[\veps]{A,\sigma}$. Then:
		\begin{enumerate}[label=(\roman*)]
			\item If $ \tau $ is unitary, $T$ is connected, 
			$[A]\neq 0$
		  	and $\pi f$ is hyperbolic,
			then $4\mid \rrk_A P$.
			\item If $[A]=0$, then $f$ is hyperbolic.
			\item If $ \tau $ is unitary, $T$ is connected, $[A]=0$ and $\pi f$ admits a Lagrangian
			$M$ with $MA=P$, then $4\mid \rrk_A P$.
		\end{enumerate}
	\end{prp}
	
	\begin{proof}
		Note that $S=R$, and  thus $T$ quadratic \'etale over $R$.	
		By Lemma~\ref{LM:invs-of-quad-et-algs} and the connectivity of $R$,
		the involution $\tau|_T:T\to T$ is either $\id_T$ or the standard
		$R$-involution of $T$. Thus, when $\tau:B\to B$ is unitary, 
		$\tau|_T$ is the standard $R$-involution
		of $T$. 		
	
		(i) Since $\tau$ is unitary, $\tau|_T$ is the standard involution of $T$.
		For the sake of contradiction, 		
		suppose that $\rrk_AP$ is not divisible by $4$.
		By Corollary~\ref{CR:constant-even-ranks}(ii),
		there exists $V\in\rproj{B}$ such that $\rrk_B P=2\rrk_BV$.
		Since $\rrk_B P_B=\iota\rrk_A P$, 
		we have $\rrk_AP=2n$, where $n:=\rrk_BV$.
		Thus,	   
		$\rrk_BV$ is odd. By Corollary~\ref{CR:index-divides-rrk}
		and Theorem~\ref{TH:period-divides-index}, $n  [B]=0$.
		On the other hand, $2[B]=2[A_T]=0$,
		because $A$ has an involution of the first kind, so
		$[B]=0$.
		We now apply Reduction~\ref{RD:common-reduction}
		to assume that $B=T$, $\deg A=2$, $\sigma$ is orthogonal and $\veps=-1$.

		By  Lemma~\ref{LM:invertible-symmetric-elements},
		there exists $c\in \Sym_{-1}(T,\tau)\cap \units{T}$.
		We apply $c$-conjugation, see \ref{subsec:conjugation}
		and Proposition~\ref{PR:simult-conj},
		to replace $\sigma$, $f$, $\veps$ by $\Int(c)\circ \sigma$, $c f$, $-\veps$.
		Now, $\veps=1$ and $\sigma$ is symplectic.
		Let $\lambda,\mu\in\units{A}$ be as in Lemma~\ref{LM:strcture-of-quat}(iii)
		(so $\lambda^\sigma=-\lambda$ and $\mu^\sigma=-\mu$)
		and note that $\Sym_1(A,\sigma)=R$.
		
		Since $  \rrk_AP=2n$,
		the $A$-module $P$ is free
		(Lemma~\ref{LM:rank-determines}).
		Thus,
		by Proposition~\ref{PR:diagonalizable-herm-forms}, we may
		assume that $f=\langle \alpha_1,\dots,\alpha_n\rangle_{(A,\sigma)}$ 
		with $\alpha_1,\dots,\alpha_n\in\Sym_1(A,\sigma)\cap\units{A}=\units{R}$. 
		Now,  it is routine  to check that, upon identifying
		$B_B^2$ with $A_B$ via $(b_1,b_2)\mapsto b_1+\mu b_2$,
		the form $\pi f$ is just
		$\langle \alpha_1, -\eta \alpha_1,\alpha_2,-\eta \alpha_2,\dots,\alpha_n,-\eta \alpha_n\rangle_{(T,\tau)}$,
		where $\mu^2=\eta$.
		Since $\pi f$ is hyperbolic, 
		$\eta^n\equiv \disc(\pi f)	\equiv\disc (n\cdot\langle 1,-1\rangle_{(T,\tau)})
		\equiv 1\bmod \Nr_{T/R}(\units{T})$
		(see~\ref{subsec:disc}).	
		Since $n$ is odd, this means
		that there exists $t\in \units{T}$ with $t^\sigma t=\eta=\mu^2$.
		As $\mu b=b^\sigma \mu$ for all $b\in B$, the element $e:=\frac{1}{2}(1+\mu^{-1}t)$
		is an idempotent. Furthermore, $e\notin\{0,1\}$, otherwise
		$\mu\in B$.
		Thus, $eA_A$ is a proper nonzero summand of $A_A$, so $\rrk_AeA=1$. 
		But this means that $[A]=[eAe]=0$ (Corollary~\ref{CR:degree-of-endo-ring}),
		a contradiction.

		(ii) By Reduction~\ref{RD:common-reduction},
		we may assume that $\sigma$ is orthogonal, hence $\veps=-1$.
		By Theorem~\ref{TH:invariant-primitive-idempotent},
		there exists $e\in A$
		such that $\rrk_AeA=1$ and $e^\sigma=e$.
		Applying $e$-transfer (see~\ref{subsec:conjugation}),
		we may replace $A$, $\sigma$, $f$
		with $eAe$, $\sigma|_{eAe}$, $f_e$
		and assume that $A=R$
		and $f:P\times P\to R$
		is an anti-symmetric unimodular bilinear form.
		Every such $f$ is hyperbolic,
		e.g., apply the argument in \cite[Lemma~7.7.2]{Scharlau_1985_quadratic_and_hermitian_forms}
		to a basis element of $P$.	
		
		(iii) Arguing as in (i), we may
		assume that $B=T$, $\deg A=2$, $\sigma$
		is symplectic and $\veps=1$.
		Let $\lambda,\mu\in\units{A}$ be as in Lemma~\ref{LM:strcture-of-quat}(iii).
		Then, writing $\tau_1:=\tau$, $\tau_2= \id_B$ and $\pi_1:=\pi$,
		we are in the situation of \ref{subsec:octagon}.
		Thus, by Remark~\ref{RM:finer-exactness}(i),
		there exists $(Q,g)\in \Herm[-1]{T,\id_T}$
		such that $\rho_2(Q,g)\cong (P,f)$.
		Since $g:Q\times Q\to T$ is an anti-symmetric unimodular bilinear form,
		$\rrk_BQ$ must be even,
		and since $\rrk_BQ$ is constant  
		($T$ is connected), we have
		$\iota\rrk_AP=\iota \rrk_AQA=2\rrk_BQ$. It follows
		that $4\mid \rrk_AP$. 
	\end{proof}
	
	\begin{prp}\label{PR:Eii-unitary-R-field}
		With Notation~\ref{NT:proof-of-Es},
		suppose that $R$ is a field and $(\tau,\veps)$ is unitary or symplectic.
		Let   $(P,f)\in \Herm[\veps]{A,\sigma}$
		be a   hermitian space such that
		$\rrk_AP$ is  even
		and  
		$\pi f$ is hyperbolic.
		If $ \tau $ is unitary and $(\sigma,\veps)$ is symplectic,
		we also assume that $4\mid \rrk_AP$.
		Let	  $M $ be a Lagrangian of $\pi f$
		such that $\rrk_BM=\frac{1}{2}\rrk_BP$.
		Then there exists $\vphi \in U^0(\pi f)$
		such that $\vphi M\cdot A=P$.
	\end{prp}
	
	\begin{proof}
		Thanks to Proposition~\ref{PR:rationality-for-Ei},
		when $R$ is infinite, it is enough to prove the proposition
		after base-changing to an algebraic closure of $R$, in which case $[A]=0$.
		On the other hand, if $R$ is finite, then $[A]=0$ by
		Wedderburn's theorem. We may therefore assume that $[A]=0$.
		
 		We claim that $\pi f$ admits a Lagrangian $M'$
 		such that $M'A=P$ and $\rrk_B M'=\frac{1}{2}\rrk_BP$.
 		To that end, we split into three cases.
 		
\smallskip 		
 		
 		{\it \noindent Case I. $S$ is not a field.}
 		Then $M'$ exists by Proposition~\ref{PR:Eii-unitary-A-split-S-split}.

\smallskip 	
 	
 		{\it \noindent Case II. $S$ is a field, $ \tau $ is unitary and $(\sigma,\veps)$
 		is symplectic.} 
 		Then $f$ is hyperbolic by Proposition~\ref{PR:Eii:unitary-symplectic-case}(ii) 
 		and $4\mid \rrk_BP$ by assumption, so $M'$ exists by Proposition~\ref{PR:Eii:rrk-divisible-by-four}.
 	
\smallskip 	
 		
 		{\it \noindent Case III. $S$ is a field, and $ \tau $ is not unitary or $(\sigma,\veps)$ is not symplectic.} 
		Using Proposition~\ref{PR:ansio-Witt-equivalent},		
		write $(P,f)=(P_1,f_1)\oplus (P_2,f_2)$
		with $f_1$ anisotropic and $f_2$ hyperbolic.
		Then $[\pi f_1]=[\pi f]=0$ in $W_\veps(B,\tau)$, so
		$[\pi f_1]$ is hyperbolic by Theorem~\ref{TH:trivial-in-Witt-ring}(ii).
		
		By virtue of Remark~\ref{RM:Es-for-field},
		any Lagrangian $M_1$ of $\pi f_1$
		satisfies $M_1A=P_1$.
		We claim that one can choose   $M_1$
		such that $\rrk_BM_1=\frac{1}{2}\rrk_BP$.
		Indeed, by Corollary~\ref{CR:constant-even-ranks}(ii),
		$\rrk_A P_2$ is even, so $\rrk_A P_1$ is also even.
		Thus, $\rrk_B P_1$ is even.
		Since $[B]=[A\otimes_ST]=[T]$,
		there exists $V\in\rproj{B}$ such that $\rrk_BV=\deg T=1$ (Proposition~\ref{PR:degree-of-endo-ring}(iii)).
		By Lemmas~\ref{LM:rank-determines-hyperbolic} and~\ref{LM:unimodular-implies-rrk-sigma-inv}, there is an 
		isometry $\Hyp[\veps]{V^n}\to \pi f_1 $,
		where $n=\frac{1}{2}\rrk_B T$. Take $M_1$ to be the image of $V^n$ in $P_1$.

		Next, by Proposition~\ref{PR:Eii-unitary-A-split-S-field},
		there exists a Lagrangian $M_2$ of $\pi f_2$
		with $M_2A=P_2$ and $\rrk_B M_2=\frac{1}{2}\rrk_B P_2$. Take $M'=M_1\oplus M_2$.
		
\smallskip

		Now, since $\rrk_BM'=\frac{1}{2}\rrk_BP=\rrk_B M$,
		Lemma~\ref{LM:Lag-transitive-action} and		
		Proposition~\ref{PR:U-zero-description}
		imply that 
		there exists $\vphi\in U^0(\pi f)$ such that $\vphi M=M'$,
		so $\vphi M\cdot A=P$.
	\end{proof}

	\begin{thm}\label{TH:Eii-unitary-symplectic}
		Theorem~\ref{TH:Eii-holds} holds
		when $R$ is connected and $(\tau,\veps)$ is symplectic  or unitary.
	\end{thm}
	
	\begin{proof}
		Recall that we are given
		$(P,f)\in\Herm[\veps]{A,\sigma}$
		such that $[\pi f]=0$. By Theorem~\ref{TH:trivial-in-Witt-ring}(ii),
		$\pi f$ is hyperbolic.

		(i) Suppose that that $\pi f$
		admits a Lagrangian $M$ such that $MA=P$.
		If $ \tau $ is unitary
		and $(\sigma,\veps)$ is symplectic,
		then  parts (i) and (iii) of Proposition~\ref{PR:Eii:unitary-symplectic-case}  
		imply that $4\mid \rrk_AP$. Moreover,
		part (i) of this proposition implies that $[A]=0$
		when $\tau $ is unitary,
		$(\sigma,\veps)$ is symplectic and $4\nmid \rrk_AP$. In
		this case, part (ii) of that proposition says that $f$ is hyperbolic.
		
		Conversely, suppose that $(\tau,\veps)$
		is symplectic, or $(\sigma,\veps)$ is not symplectic,
		or $4\mid\rrk_AP$.
		Let $M$ be a Lagrangian of $\pi  f$
		and let $\frakm_1,\dots,\frakm_t$ denote
		the maximal ideals of $R$.
		By Proposition~\ref{PR:Eii-unitary-R-field},
		for all $1\leq i\leq t$,
		there exists $\vphi_i\in U^0(\pi f(\frakm_i))$
		such that $\vphi_i (M(\frakm_i))\cdot A(\frakm_i)=P(\frakm_i)$.
		By Theorem~\ref{TH:U-zero-mapsto-onto-closed-fibers},
		there exists $\vphi \in U^0(\pi f)$
		such that $\vphi(\frakm_i)=\vphi_i$ for all $i$.
		Thus, $M':=\vphi M$ is a Lagrangian of $\pi f$
		such that $M'  A+P\frakm_i=P$ for all $i$.
		By  Nakayama's Lemma $\ann_R (P/M'A)$ is not
		contained in any maximal ideal of $R$, so it must be $R$ and $M' A=P$.
		
		(ii) This statement is vacuous under our assumptions.
		
		(iii) 
		By Proposition~\ref{PR:Eii-S-not-connected},
		Proposition~\ref{PR:Eii-T-not-connected} and (i), we only need to consider the case where 
		$T$ is connected,
		$ \tau $ is unitary, $(\sigma,\veps)$ is symplectic and $[A]=0$.
		In this case, 
		$f$ is hyperbolic by Proposition~\ref{PR:Eii:unitary-symplectic-case}(ii), so
		we may  take $f'$ to be the zero form 
		and let $M=0$.
	\end{proof}

\subsection{Case (2)}
\label{subsec:Eii-orthogonal}

	We now prove  
	Theorem~\ref{TH:Eii-holds} in Case \ref{item:Eii:orth},
	namely, when $R$ is connected and   $(\tau,\veps)$
	is orthogonal.
	The main difference with Case~\ref{item:Eii:sym-or-unit} 
	is the failure of Proposition~\ref{PR:Eii-unitary-R-field}.
	Thus,   the majority of the argument will be
	dedicated to effectively characterizing the
	Lagrangians $M$ of $\pi f$ for which Proposition~\ref{PR:Eii-unitary-R-field} fails.

\medskip

	Throughout this subsection, we assume, on top of Notation~\ref{NT:proof-of-Es},
	that $(\tau,\veps)$ is orthogonal, hence
	$\tau|_T=\id_T$ and $S=R$. This also means that $(\sigma,\veps)$
	is orthogonal (Lemma~\ref{LM:tau-orth-implies-sigma-orth}).
	
	Following Remark~\ref{RM:Weil-rest-of-U},
	given
	$(Q,g)\in\Herm[\veps]{B,\tau}$, 
	we write ${\uU}_T(g)$ for the group $T$-scheme
	of isometries of $g$, and $\uU(g)=\uU_R(g)$ for the $R$-scheme
	of isometries of $g$. The corresponding neutral
	components are denoted ${\uU}^0_T(g)$ and $\uU^0(g)=\uU^0_R(g)$.
	It was observed in 	Remark~\ref{RM:Weil-rest-of-U}
	that $\calR_{T/R}{\uU}_T(g)=\uU(g)$
	and $\calR_{T/R}\uU^0_T(g)=\uU^0(g)$,
	where $\calR_{T/R}$ is the Weil restriction.
	Combining this with Proposition~\ref{PR:U-zero-description}, 
	we see that $\uU^0(g)$ is the scheme-theoretic kernel
	of 
	\[
	\calR_{T/R}(\Nrd):\uU (g)=\calR_{T/R}\uU_T (g)\to \calR_{T/R}\umu_{2,T} .
	\]
	We abbreviate $\calR_{T/R}(\Nrd)$ to $\Nrd$.
	The norm map $\Nr_{T/R}:T\to R$ induces a morphism of affine group $R$-schemes,
	\[
	\Nr_{T/R}:\calR_{T/R}\umu_{2,T} \to \umu_{2,R},
	\]
	and its kernel is $\umu_{2,R}$, viewed as a subgroup $R$-scheme of  $\calR_{T/R}\umu_{2,T}$
	via the inclusion $R\to T$.
	We   write
	\[
	N:=\Nr_{T/R}\circ \Nrd:\uU(g)\to \umu_{2,R}.
	\]
	
	Given $(P,f)\in\Herm[\veps]{A,\sigma}$, Lemma~\ref{LM:B-endo-of-P}(ii) implies  that the   diagram
	\[
	\xymatrix{
	U(f) \ar@{^{(}->}[r] \ar[d]^{\Nrd} &
	U(\pi f) \ar[d]^{\Nrd}	\\
	\mu_2(R)\ar@{^{(}->}[r] &
	\mu_2(T)
	}
	\]	
	commutes.
	Thus, given $\vphi \in U(f)$, we may speak about the reduced
	norm of $\vphi$ without specifying if we view $\vphi$ as an
	isometry of $f$ or $\pi f$.

	Finally, recall from \ref{subsec:Lagrangians} that $\Lag(\pi f)$
	denotes the set of Lagrangians $M$ of $\pi f$ with $\rrk_BM=\frac{1}{2}\rrk_BP$,
	and  these are   all the Lagrangians of $\pi f$ because $\tau|_T=\id_T$.
	In particular, if $\pi f$ is hyperbolic, then $\iota \rrk_AP= \rrk_BP$  must
	be even. Recall also the sheaf $\uLag(\pi f)$ over $(\Aff/T)_{\fppf}$;
	we write $\calR_{T/R}\uLag(\pi f)$ for its Weil restriction,
	which is the sheaf on $(\Aff/R)_{\fppf}$ mapping an $R$-ring $S$ to $\Lag(\pi f_S)$.

	\begin{lem}\label{LM:Eii-orth-fT-hyperbolic}
		With Notation~\ref{NT:proof-of-Es},
		suppose that $(\tau,\veps)$ is orthogonal.
		Let  $(P,f)\in\Herm[\veps]{A,\sigma}$
		and assume that $\pi f$
		is hyperbolic. Then $(P_T,f_T)$
		is hyperbolic.
	\end{lem}
	
	\begin{proof}
		By Lemma~\ref{LM:splitting-quad-et-algs},
		we have $T_T\cong T\times T$.
		Let $e:=(1_T,0_T)\in T_T$.	
		By assumption, $\pi f_T$ is   hyperbolic,
		so by Proposition~\ref{PR:Eii:pi-f-T-not-connected},
		the
		$e$-transfer of $f_T$ (see~\ref{subsec:conjugation}) 
		is also hyperbolic. Thus,   $f_T$ is hyperbolic.
	\end{proof}

	\begin{prp}\label{PR:dfn-of-Psi-f}
		With Notation~\ref{NT:proof-of-Es},
		suppose that $(\tau,\veps)$ is orthogonal.
		Let  $(P,f)\in\Herm[\veps]{A,\tau}$ and assume that $\pi f$ is hyperbolic.
		Let  $\uU(\pi f)$ act  on $\umu_{2,R}$
		via $N$.
		Then there exists a unique $\uU(\pi f)$-equivariant
		natural transformation,
		\[
		\Psi_f:\calR_{T/R}\uLag(\pi f)\to \umu_{2,R},
		\]
		such that for any $R$-ring 
		$R_1$ and any $L_1\in \Lag(  f_{R_1})$, one has
		$\Psi_f(L_1)=1$ in $\mu_2(R_1)$.
		The map $\Psi_f$ has the following additional
		properties:
		\begin{enumerate}[label=(\roman*)]
			\item If $f$ is hyperbolic and $L\in\Lag(f)$, then
			$\Psi_f=\Nr_{T/R}\circ \calR_{T/R}\Phi_L^{(\pi f)}$ (notation as in Proposition~\ref{PR:partition-of-Lag}).
			\item If $(P',f')\in\Herm[\veps]{A,\sigma}$
			and $\pi f'$ is hyperbolic,
			then $\Psi_{f\oplus f'}(M\oplus M')=\Psi_f(M)\cdot \Psi_{f'}(M')$
			for all $M\in\Lag(\pi f)$, $M'\in \Lag(\pi f')$.
			\item Let $e\in B$ be a $\sigma$-invariant idempotent
			such that $\rrk_BeB$ is positive and constant
			on the fibers of $\Spec T\to \Spec R$.
			Then $\Psi_f(M)=\Psi_{f_e}(Me)$ for all
			$M\in\Lag(\pi f)$ (notation as in \ref{subsec:conjugation}).
		\end{enumerate}
	\end{prp}
	
	Note that $L_1$ is a Lagrangian of $\pi f_{R_1}$ 
	because we can find $L'_1\in \Lag(f_{R_1})$
	such that $L_1\oplus L'_1=P$ as $A$-modules (see~\ref{subsec:Witt-grp})
	and
	$\pi f(L_1,L_1)=\pi f(L'_1,L'_1)=0$.
	
	\begin{proof}
		Fix some $M_0\in\Lag(\pi f)$,
		write $\Phi_0=\calR_{T/R}\Phi_{M_0}^{(\pi f)}:\calR_{T/R}\uLag(\pi f)\to \calR_{T/R}\umu_2$
		(see   Proposition~\ref{PR:partition-of-Lag} for the definition of $\Phi_{M_0}$),
		and let $\Psi_0:=\Nr_{T/R}\circ \Phi_0$. It is clear
		that $\Psi_0:\calR_{T/R}\uLag(\pi f)\to\umu_{2,R}$ is $\uU(\pi f)$-equivariant.

		We   claim that for any $R$-ring   $R_1$
		and $V,W\in \Lag(f_{R_1})$, we have $\Psi_0(V)=\Psi_0(W)$
		in $\mu_2(R_1)$.
		Since $\umu_{2,R}$ is a sheaf on $(\Aff/R)_{\fpqc}$,
		it is enough to check that $\Psi_0(V)=\Psi_0(W)$ after base-changing along
		a faithfully flat ring homomorphism $R_1\to R_2$. 
		By Proposition~\ref{PR:transitive-action-of-uUf},  
		we can choose $R_2$ such that there exists $\vphi\in U(f_{R_2})$
		with $ V\otimes_{R_1}R_2=\vphi (W\otimes_{R_1}R_2)$. 
		Since $\Nrd(\vphi)\in \mu_2(R_2)$ and $\Psi_0$ is $\uU(\pi f)$-equivariant, we have		
		$\Psi_0(V\otimes_{R_1}R_2)=\Nr_{T/R}(\Nrd(\vphi))\cdot  \Psi_0(W\otimes_{R_1}R_2)=
		\Psi_0(W\otimes_{R_1}R_2)$ in $\mu_2(R_2)$, as required.
		
		Let $R_0:=T$.
		Then $f_{R_0}$ is hyperbolic
		by 	Lemma~\ref{LM:Eii-orth-fT-hyperbolic}.
		Fix some $L_0\in \Lag(  f_{R_0})$  and
		write $\theta:=\Psi_0(L_0)\in\mu_2(R_0)$.	
		We claim that $\theta$ is in fact in $\mu_2(R)$.
		To that end, let $i_1,i_2:R_0\to R_0\otimes R_0$ denote the  
		homomorphisms
		$r\mapsto r\otimes 1$, $r\mapsto 1\otimes r$.
		By the previous paragraph, we have
		$i_1\Psi_0(L_0)=\Psi_0(L_0\otimes_{i_1} (R_0\otimes R_0))=\Psi_0(L_0\otimes_{i_2} (R_0\otimes R_0))
		=i_2\Psi_0(L_0)$ in $\mu_2(R_0\otimes R_0)$.
		Since $\umu_{2,R}$ is a sheaf
		on $(\Aff/R)_{\fpqc}$, this means that $\theta\in \mu_2(R)$.

		Define $\Psi_f=\theta^{-1}\cdot \Psi_0$.
		Then $\Psi_f:\calR_{T/R}\uLag(\pi f)\to \umu_2$ is $\uU(\pi f)$-equivariant
		and $\Psi_f(L_0)=1$ in $\mu_2(R_0)$.
		Let $R_1$ be an  $R$-ring and let $L_1\in \Lag(f_{R_1})$.
		By what we have shown above, 
		$\Psi_0(L_1 \otimes_{R_1}(R_0\otimes R_1))=\Psi_0(L_0\otimes_{R_0}(R_0\otimes R_1))=\theta$
		in $\mu_2(R_0\otimes R_1)$.
		Since $R_1\to R_0\otimes R_1$ is faithfully flat, 
		this means that $\Psi_0(L_1)=\theta$ in $\mu_2(R_1)$, so $\Psi_f(L_1)=1$.
		Thus, $\Psi_f$ satisfies the condition in the proposition.

		If $\Psi':\calR_{T/R}\uLag(\pi f)\to \umu_{2,R}$ also satisfies
		the condition in the proposition, then $\Psi'(L_0)=1=\Psi(L_0)$.
		If $R_1$ is an $R$-ring and $M\in \Lag(\pi f_{R_1})$,
		then, by Proposition~\ref{PR:transitive-action-of-uUf},
		there exists a faithfully flat   $R_0\otimes R_1$-ring $R_2$
		and $\vphi\in U(\pi f_{R_2})$
		such that $\vphi(L_0\otimes_{R_0}R_2)=M\otimes_{R_1}R_2$.
		Thus, $\Psi'(M)=N(\vphi)\Psi'(L_0)=N(\vphi)\Psi_f(L_0)=\Psi_f(M)$
		in $\mu_2(R_2)$. Since $R_1 \to R_0\otimes R_1\to R_2$
		is faithfully flat, this means that $\Psi'(M)=\Psi_f(M)$ in $\mu_2(R_1)$,
		so $\Psi'=\Psi_f$.
		
\medskip		
		
		We finally verify the additional properties (i)--(iii).

		(i) Take $M_0=L$ and $L_0=L_{R_0}$ 
		in the construction of $\Psi_f$; one gets $\theta=1$.

		(ii) It is enough to prove the equality
		after base-changing to $R_0$. It is then a consequence of (i)
		(take $L=L_0$) and Proposition~\ref{PR:partition-of-Lag}(ii).
		
		(iii) Again, we may base change  to $R_0$ first. The claim
		then follows from  (i) and item \ref{item:e-transfer-and-Phi}
		in \ref{subsec:conjugation}.
	\end{proof}	
	
	It turns out that  
	$\Psi_f$ is often constant on $\Lag(\pi f)$.

	\begin{lem}\label{LM:Eii-Psi-f-constant}
		With Notation~\ref{NT:proof-of-Es},
		suppose that  $R$ is connected   semilocal and $(\tau,\veps)$ is orthogonal.
		Let $(P,f)\in\Herm[\veps]{A,\tau}$ and assume that $\pi f$ is hyperbolic.
		Then $\Psi_f:\Lag(\pi f)\to \mu_2(R)=\{\pm 1\}$
		is onto if and only if $T\cong R\times R$, $[A]=0$ and $P\neq 0$.
	\end{lem}
	
	\begin{proof}
		The lemma is clear if $P=0$, so assume $P\neq 0$.	
	
		Let $M,M'\in \Lag(\pi f)$.
		By Lemma~\ref{LM:Lag-transitive-action},
		there exists $\vphi\in U(\pi f)$
		such that $\vphi M=M'$, hence $\Nr_{T/R}(\Nrd(\vphi))\Psi_f(M)=\Psi_f(M')$.
		From this we see that the condition  that $\Psi_f:\Lag(\pi f)\to \mu_2(R)$
		is onto
		is equivalent to the existence of $\vphi\in U(\pi f)$
		with $\Nr_{T/R}\Nrd(\vphi)=-1$ in $\mu_2(R)$.
		
		Suppose that $[A]=0$ and $T=R\times R$, and
		let $e=(1_R,0_R)$ and $e'=(0_R,1_R)$. By Proposition~\ref{PR:Eii:pi-f-T-not-connected},
		we may identify $U(\pi f)$ with $U(f_e)\times U(f_{e'})$,
		and under this identification, $\Nrd:U(\pi f)\to \mu_2(T)$
		is just $\Nrd\times \Nrd:U(f_e)\times U(f_{e'})\to \mu_2(R)\times\mu_2(R)$.
		Since $[Be]=[Be']=[A]=0$ (Lemma~\ref{LM:idempotent-in-T}(iii)), 
		this map is onto by Theorem~\ref{TH:criterion-for-det-one}.
		One readily checks that $\Nr_{T/R}:\mu_2(T)\to\mu_2(R)$ is also onto,
		so we conclude that there exists $\vphi\in U(\pi f)$
		with $\Nr_{T/R}\Nrd(\vphi)=-1$.

		Conversely, suppose  that $\vphi\in U(\pi f)$ satisfies
		$\Nr_{T/R}\Nrd(\vphi)=-1$.
		If $T$ were connected, then we would have $\Nr_{T/R}(\mu_2(T))=\Nr_{T/R}(\{\pm1\})=1$,
		so we must have $T\cong R\times R$ (Lemma~\ref{LM:non-connected-S})
		and $\Nrd(\vphi)\in \{(1,-1),(-1,1)\}$.
		Let $e$ and $e'$ denote the nontrivial idempotents of $T$. 
		Appealing to Proposition~\ref{PR:Eii:pi-f-T-not-connected}
		as in the previous paragraph, we see that
		$\vphi|_{Pe}\in U(f_e)$ 
		and $\vphi|_{Pe'}\in U(f_{e'})$,
		and either $\Nrd(\vphi|_{Pe})=-1$
		or  $\Nrd(\vphi|_{Pe'})=-1$.
		Thus, by
		Theorem~\ref{TH:criterion-for-det-one}, $[eAe]=0$ or $[e'Ae']=0$.
		Since $[A]=[eAe]=[e'Ae']$
		(Lemma~\ref{LM:idempotent-in-T}(iii)),  $[A]=0$.
	\end{proof}
	
	\begin{prp}\label{PR:Eii:orth:R-field-A-split}
		With Notation~\ref{NT:proof-of-Es},
		suppose that  $R$ is a field, $T\cong R\times R$,
		$[A]=0$ and  $(\tau,\veps)$
		is orthogonal.
		Let $(P,f)\in\Herm[\veps]{A,\sigma}$
		be a hyperbolic hermitian space. Then  there exists  $M\in\Lag(\pi f)$
		with $MP=A$.
		Every such $M$ satisfies $\Psi_f(M)=(-1)^{\frac{1}{2}\rrk_AP}$.
	\end{prp}
	
	\begin{proof}
		By Reduction~\ref{RD:common-reduction}, Corollary~\ref{CR:degree-of-endo-ring}
		and Proposition~\ref{PR:dfn-of-Psi-f}(iii),
		we may assume that $B=T$, $\deg A=2$ and $\tau$ is orthogonal.
		As a result, $\veps=1$. 
		Recall that $\sigma$ is also orthogonal in this case
		(Lemma~\ref{LM:tau-orth-implies-sigma-orth}).

		Let $e$ denote a nontrivial idempotent of $T$.
		We identify $A$ with $\nMat{R}{2}$ in such a way that
		$e=[\begin{smallmatrix} 1 & 0 \\ 0 & 0\end{smallmatrix}]$.
		Thus, $B=T=R+Re$ consist of the diagonal matrices,
		and $\pi:A\to B$ is given by $\pi[\begin{smallmatrix} a & b \\ c & d\end{smallmatrix}]=
		[\begin{smallmatrix} a & 0 \\ 0 & d\end{smallmatrix}]$.
		Since $e^\sigma=e$, there exist $\alpha\in \units{R}$
		such that $\sigma$ is given by
		$\sigma:[\begin{smallmatrix} a & b \\ c & d\end{smallmatrix}]\mapsto
		[\begin{smallmatrix} a & \alpha c \\ \alpha^{-1}b & d\end{smallmatrix}]$.
		
		Let $f_1:A  \times A  \to A$ be the hyperbolic
		$1$-hermitian form given by
		$f_1(x,y)=x^\sigma [\begin{smallmatrix} 0 & \alpha  \\ 1 & 0\end{smallmatrix}] y$
		($[\begin{smallmatrix} R & R  \\ 0 & 0\end{smallmatrix}]$ is a Lagrangian).
		Since $\rrk_AA=2$ and $\rrk_AP$ is even (because $\pi f$ is hyperbolic),
		we have $(P,f)\cong \frac{\rrk_AP}{2}\cdot (A,f_1)$ (Lemma~\ref{LM:rank-determines-hyperbolic}).
		Thus, it is enough to prove the existence of $M$ when $(P,f)=(A,f_1)$.
		To that end, take $M:=B=[\begin{smallmatrix} R & 0  \\ 0 & R\end{smallmatrix}]$;
		it is a Lagrangian because $A=M\oplus M'$ and $f_1(M',M')=0$ for $M'=
		[\begin{smallmatrix} 0 & R  \\ R & 0\end{smallmatrix}]$.

		We proceed with proving the second statement of the proposition.
		Suppose that $M\in\Lag(\pi f)$ satisfies $MA=P$.
		Write $e'=1-e$.
		Using   Proposition~\ref{PR:Eii:pi-f-T-not-connected},
		we shall view $\pi f$ as $f_e\times f_{e'}$
		and identify
		$U(\pi f)$ and $\Lag(\pi f)$ with $U(f_e)\times U(f_{e'})$
		and $\Lag(f_e)\times \Lag(f_{e'})$, respectively.

		Since $M\in\Lag(\pi f)$, we have $Me\in\Lag(f_e)$, and so $MeA\in\Lag(f)$
		(see item~\ref{item:e-transfer-Lags} in \ref{subsec:conjugation}).
		Similarly, $Me'A\in\Lag (f)$.
		Since $MA=P$, we have $MeA+Me'A=P$ and $A$-length considerations
		force $P=MeA\oplus Me'A$.

		By   
		Lemma~\ref{LM:Lag-transitive-action},
		there exists $\vphi\in U(f)$ such that $\vphi (MeA)=Me'A$.
		Write $\vphi_e=\vphi|_{Pe}$.
		Then, viewing
		$(\vphi_e,1)$ as an element
		of $U(f_e)\times U(f_{e'})=U(\pi f)$
		and working    in $\Lag(\pi f)=\Lag(f_e)\times\Lag(f_{e'})$,
		we have 
		\begin{align*}
		(\vphi_e,1)\cdot M&=(\vphi_e,1)(Me,Me')
		=(\vphi_e (Me),Me')\\
		&=(\vphi(MeA)\cdot e,Me')=
		(Me'Ae ,Me'Ae')= Me'A.
		\end{align*}
		By Proposition~\ref{PR:dfn-of-Psi-f},
		we have $N(\vphi_e,1)\cdot \Psi_f(M)=\Psi_f(Me'A)=1$, because $Me'A\in\Lag(f)$.
		Furthermore, by Proposition~\ref{PR:computation-of-Phi},
		$MeA\oplus Me'A=P$
		implies that  $\Nrd(\vphi)=(-1)^{\frac{1}{2}\rrk_AP}$.
		Together, this gives 
		$
		\Psi_f(M)=N(\vphi_e,1)^{-1}=\Nrd(\vphi_e)\cdot \Nrd(1)=
		\Nrd(\vphi)= (-1)^{\frac{1}{2}\rrk_AP}$, as required.
	\end{proof}
	
	\begin{cor}\label{CR:Eii-orth-Psi-restriction}
		With Notation~\ref{NT:proof-of-Es},
		suppose that   $(\tau,\veps)$
		is orthogonal.
		Let  $(P,f)\in\Herm[\veps]{A,\sigma}$
		and let $M\in\Lag(\pi f)$.
		If $M A=P$, then $\Psi_f (M)=(-1)^{\frac{1}{2}\rrk_AP}$.
	\end{cor}
	
	\begin{proof}
		By Lemma~\ref{LM:mu-two-check},
		it is enough to prove the corollary
		after specializing to an algebraic closure 
		of $k(\frakp)$ for all $\frakp\in \Spec R$,
		so assume that $R$ is an algebraically closed field.
		Then
		$[A]=0$ and $T\cong R\times R$. We claim that $f$ is hyperbolic. Indeed,
		$f_T$ is hyperbolic by Lemma~\ref{LM:Eii-orth-fT-hyperbolic}
		and  $T\cong R\times R$, so $f$ is also hyperoblic.
		The corollary therefore
		follows from Proposition~\ref{PR:Eii:orth:R-field-A-split}.
	\end{proof}

	\begin{prp}\label{PR:Eii-orth-R-field-A-arbit}
		With Notation~\ref{NT:proof-of-Es},
		suppose that   $(\tau,\veps)$
		is orthogonal and $R$ is a field.
		Let $(P,f)\in\Herm[\veps]{A,\sigma}$ 
		and let $M\in \Lag(\pi f)$.
		Then   there exists $\vphi\in U^0(\pi f)$
		such that $\vphi M\cdot A=P$
		if and only if   $\Psi_f(M)=(-1)^{\frac{1}{2}\rrk_PA}$.
	\end{prp}

	\begin{proof}
		If $\vphi M\cdot A=P$ for $\vphi\in U^0(\pi f)$,
		then $\Psi_f(M)=N(\vphi)\Psi_f(M)=\Psi_f(\vphi M)=(-1)^{\frac{1}{2}\rrk_PA}$
		by Corollary~\ref{CR:Eii-orth-Psi-restriction}. We turn to prove the converse.
	
		Using Proposition~\ref{PR:ansio-Witt-equivalent},
		write $(P,f)=(P_1,f_1)\oplus (P_2,f_2)$
		with $f_1$ anisotropic and $f_2$ hyperbolic.
		Since $[\pi f_1]=[\pi f]=0$ in $W_\veps(B,\tau)$,
		the form $\pi f_1$ is hyperoblic by Theorem~\ref{TH:trivial-in-Witt-ring}(ii).
		Let $M_1\in\Lag(\pi f_1)$. Arguing as in Remark~\ref{RM:Es-for-field},
		we see that $M_1A=P_1$, and $\Psi_{f_1}(M_1)=(-1)^{\frac{1}{2}\rrk_AP_1}$
		by Corollary~\ref{CR:Eii-orth-Psi-restriction}.
		We now split into cases.

\medskip

		\noindent {\it Case I.  $[A]=0$ and $T\cong R\times R$.}
		By Proposition~\ref{PR:Eii:orth:R-field-A-split},
		there exists $M_2\in \Lag(\pi f_2)$ such that $M_2A=P_2$.
		Write $M'=M_1\oplus M_2$. Since $M'A=P$, we have $\Psi_f(M')=(-1)^{\frac{1}{2}\rrk_AP}$
		by Corollary~\ref{CR:Eii-orth-Psi-restriction}.
		By  
		Lemma~\ref{LM:Lag-transitive-action},
		there exists $\psi\in U(\pi f)$
		such that $\psi M=M'$. Since  $\Psi_f(M')=(-1)^{\frac{1}{2}\rrk_AP}=\Psi_f(M)$,
		this means that $\Nrd(\psi)\in\ker (\Nr_{T/R}:\mu_2(T)\to\mu_2(R))=\mu_2(R)$.
		
		If $\Nrd(\psi)=1$, take $\vphi$ to be $\psi$.
		If $\Nrd(\psi)=-1$, then $P\neq 0$. Since $[A]=0$, Theorem~\ref{TH:criterion-for-det-one}
		implies that there exists $\xi\in U(f)$ with $\Nrd(\xi)=-1$.
		Then $\xi$ is an $A$-linear isometry of $\pi f$,
		hence $\xi \psi M$ is a Lagrangian of $\pi f$
		satisfying $\xi \psi M\cdot A=\xi(\psi M\cdot A)=\xi P=P$, so
		take $\vphi=\xi \psi$.
		
\medskip

		\noindent {\it Case II.   $[A]=0$ and $T$ is a field.}
		Let $L_2\in\Lag(f_2)$.
		By definition, we have $\Psi_{f_2}(L_2)=1$,
		hence $\Psi_{f}(M_1\oplus L_2)=	\Psi_{f_1}(M_1)\cdot\Psi_{f_2}(L_2)=
		(-1)^{\frac{1}{2}\rrk_AP_1}$ (Proposition~\ref{PR:dfn-of-Psi-f}(ii)).
		On the other hand, 	
		by Lemma~\ref{LM:Eii-Psi-f-constant}, $\Psi_f(M_1\oplus L_2)=\Psi_f(M)=(-1)^{\frac{1}{2}\rrk_AP}$,
		so $\frac{1}{2}\rrk_AP_2=\frac{1}{2}(\rrk_AP-\rrk_AP_1)$
		must be even.
		Now, by Proposition~\ref{PR:Eii:rrk-divisible-by-four},
		there exists $M_2\in \Lag(\pi f_2)$
		such that $M_2A=P_2$.
		Proceed as in Case~I.

\medskip
		
		\noindent {\it Case III. $[A]\neq 0$.}
		By Wedderburn's Theorem, $R$ is infinite.
		Therefore, thanks to Proposition~\ref{PR:rationality-for-Ei}(ii),
		we are reduced into proving the proposition when
		$R$ is algebraically closed. This is covered
		by   Case I.
	\end{proof}
	
	From Proposition~\ref{PR:Eii-orth-R-field-A-arbit},
	we see that in order to apply the proof of Theorem~\ref{TH:Eii-unitary-symplectic}
	to our situation, we need to find a Lagrangian $M\in\Lag(\pi f)$
	with $\Psi_f(M)=(-1)^{\frac{1}{2}\rrk_AP}$.
	The following two propositions, which address
	the cases $[B]\neq 0$ and $[B]=0$ respectively, 
	characterize precisely when such $M$
	exists.
	
	\begin{prp}\label{PR:Eii:Psi-good-if-B-nonzero}
		With Notation~\ref{NT:proof-of-Es},
		suppose that $R$ is  semilocal, $T$ is connected and
		$(\tau,\veps)$ is orthogonal.
		Let $(P,f)\in \Herm[\veps]{A,\sigma}$
		and assume that $\pi f$ is hyperbolic.
		If $[B]\neq 0$,  
		then $\Psi_f(M)=  (-1)^{\frac{1}{2}\rrk_PA}$
		for all $M\in \Lag(\pi f)$.
	\end{prp}
	
	\begin{proof}
		For the sake of contradiction,
		suppose that there exists $M\in \Lag(\pi f)$
		with $\Psi_f(M)=  (-1)^{\frac{1}{2}\rrk_PA+1}$.
		By Lemma~\ref{LM:Eii-orth-fT-hyperbolic},
		there exists $L\in\Lag(f_T)$,
		and $\rrk_{A_T}L$ is constant
		because it equals $ \frac{1}{2}\rrk_{A_T}P_T$.
		
		Suppose that $\rrk_{A_T} L$ is odd. Then $(\rrk_{A_T} L)\cdot [A_T]=0$ 
		by
		Corollary~\ref{CR:index-divides-rrk} and Theorem~\ref{TH:period-divides-index},
		and $2[A_T]=0$ because $A_T$ has a  $T$-involution. Thus, $[B]=[A_T]=0$, a contradiction.
		
		Suppose that $\rrk_{A_T} L$ is even.
		Then $ \frac{1}{2}\rrk_AP$ is also even.
		Now, $\Psi_f(M_T)=\Psi_f(M)=-1$ while
		$\Psi_f(L)=1$.
		By Lemma~\ref{LM:Eii-Psi-f-constant}, this means that $[B]=[A_T]=0$,
		so again, we have reached a contradiction.
	\end{proof}

	\begin{lem}\label{LM:Eii-factoring-x-summand}
		With Notation~\ref{NT:proof-of-Es},
		suppose that $R$ is semilocal, $\deg B=1$,
		$\tau=\id_B$ and $\veps=1$.
		Let $(P,f)\in \Herm[\veps]{A,\sigma}$
		and assume  that $\pi f$ is hyperbolic
		and $\rrk_A P$ is   constant and greater than $2$.
		Then 
		there exists $x\in P$ such that $f(x,x)\in\units{A}$
		and $\pi f(x,x)=0$.
	\end{lem}
	
	\begin{proof}
		Define $\lambda,\mu$ as in Lemma~\ref{LM:strcture-of-quat}(i)
		(so $\lambda^\sigma=\lambda$ and $\mu^\sigma=-\mu$).
		
\medskip	
	
		\Step{1} We first prove the existence
		of $x$ when   $R$ is a field.		
		Write $(P,f)=(P_1,f_1)\oplus (P_2,f_2)$
		with $f_1$ anisotropic and $f_2$ hyperbolic
		(Proposition~\ref{PR:ansio-Witt-equivalent}).
		As in the proof of Proposition~\ref{PR:Eii-orth-R-field-A-arbit},
		$\pi f_1$ and $\pi f_2$ are hyperbolic, so both $\rrk_AP_1$ and $\rrk_AP_2$
		are even.
		By assumption,   $\rrk_AP_1>0 $ or $\rrk_AP_2\geq 4$.
		
		If $\rrk_AP_1>0$, then there exists nonzero $  x\in P_1$ such that $\pi f_1(x,x)=0$.
		Thus, $f_1(x,x)\in\Sym_1(A,\sigma)\cap\ker\pi=\mu\lambda R$.
		Since $f_1$ is anisotropic, $f_1(x,x)\neq 0$,
		so  $f (x,x)=f_1(x,x)\in \mu\lambda \units{ R}\subseteq\units{A}$.
		
		If $\rrk_AP_2\geq 4$, then  $f_2$ has
		an orthogonal summand isomorphic to  $\langle \mu\lambda,-\mu\lambda\rangle_{(A,\sigma)}$ 
		(Lemma~\ref{LM:rank-determines-hyperbolic}). 
		Now take $x$ to be the vector
		corresponding to $(1_A,0_A)\in A^2$ in $P$.

\medskip

		\Step{2} We continue to assume that $R$ is a field.
		Let $x,y\in P$ be two elements such that 
		$\pi f(x,x)=\pi f(y,y)=0$ and $\rrk_B xB=\rrk_B yB=1$.
		We claim that there exists
		$\vphi\in U^0(\pi f)$ such that $\vphi x=y$. 
		
		Suppose first that $T$ is a field.
		Our assumptions imply that $xB\cong yB $ as $B$-modules.
		Since $\pi f$ is unimodular, there exists $x'\in P$
		such that $\pi f(x,x')=1$. In particular, the restriction
		of $f$ to $Q=xB\oplus xB'$ is unimodular,
		so $P=Q\oplus Q^\perp$. Since $\rrk_BP=\iota\rrk_AP\geq 4$, we have
		$Q^\perp\neq 0$. By
		Theorem~\ref{TH:criterion-for-det-one},
		there exists $\psi_0\in U(f|_{Q^\perp\times Q^\perp})$
		with $\Nrd(\psi_0)=-1$. Let $\psi=\id_Q\oplus \psi_0\in U(\pi f)$.
		By Theorem~\ref{TH:Witt-extension},
		there exists $\vphi\in U(\pi f)$ with $\vphi x=y$.
		If $\Nrd(\vphi)=1$, we are done.
		If not, $\Nrd(\vphi)=-1$ (because $T$ is a field)
		and we can replace $\vphi$ with $\vphi\psi$
		to get $\Nrd(\vphi)=1$.
		
		When $T$ is not a field, we have $T=R\times R$
		and we can apply the argument of the previous paragraph separately over each
		factor of $T$.
	
\medskip			
		
		\Step{3} We now prove the proposition for all $R$.
		Since $\rrk_B P=\iota\rrk_AP\geq 4$ and $\pi f$ is hyperbolic,
		there exists $y\in P$ such that $\pi f(y,y)=0$
		and $yB$ is a summand of $P_B$ of reduced rank $1$. 		
		
		Let $\frakm_1,\dots,\frakm_t$ denote the maximal ideals of $R$.
		By Step 1, for all $1\leq i\leq t$, there exists
		$x_i\in P(\frakm_i)$ such that $\pi f(\frakm_i)(x_i,x_i)=0$
		and $f(\frakm_i)(x_i,x_i)\in\units{A(\frakm_i)}$.
		The latter condition implies that $\ann_{B(\frakm_i)} x_i=0$,
		so $\rrk_{B(\frakm_i)} x_iB(\frakm_i)=1$.
		Thus, by Step~2, there exists $\vphi_i\in U^0(\pi f(\frakm_i))$ 
		such that $\vphi_i y=x_i$.
		
		By Theorem~\ref{TH:U-zero-mapsto-onto-closed-fibers},
		there exists $\vphi\in U(\pi f)$
		such that $\vphi (\frakm_i)=\vphi_i$ for all $1\leq i\leq t$.
		Let $x=\vphi y$. Then $f(x,x)=f(y,y)=0$
		and $f(x,x)(\frakm_i)=f(\frakm_i)(x_i,x_i)\in\units{A(\frakm_i)}$
		for all $i$. By Lemma~\ref{LM:invertability-test},
		$f(x,x)\in\units{A}$, as required.		
	\end{proof}
	
	The next proposition makes use of the discriminant of hermitian forms over $(A,\sigma)$, see \ref{subsec:disc}.
	
	\begin{prp}\label{PR:Eii-orth-quat-necessary-cond}
		With Notation~\ref{NT:proof-of-Es},
		suppose that $R$ is   semilocal,  $T$ is connected, 
		$[B]=0$ and $(\tau,\veps)$ is orthogonal.
		Let $(P,f)\in \Herm[\veps]{A,\sigma}$
		and assume that $\pi f$ is hyperbolic.
		Then   $\Psi_f(M)\equiv (-1)^{\frac{1}{2}\rrk_AP}$ for
		some $M\in\Lag(\pi f)$ if and only if
		$\disc(f)=\disc(T/R)^{\frac{1}{2}\rrk_AP}$.
		When this fails, $[A]=0$, $\disc(f)=\disc(T/R)^{\frac{1}{2}\rrk_AP+1}$
		and $f$ is isotropic.
	\end{prp}
	
	\begin{proof}
		Recall that $\rrk_AP$ is even because $\pi f$ is hyperbolic.
		Furthermore, $\rrk_AP$ is constant because $T$ is connected.
		By Reduction~\ref{RD:common-reduction},
		we may assume that $\deg B=1$, $\deg A=2$, $\tau=\id_T$
		and $\veps=1$.
		Define $\lambda,\mu$ as in Lemma~\ref{LM:strcture-of-quat}(i)
		(so $\lambda^\sigma=\lambda$ and $\mu^\sigma=-\mu$).
		By Lemma~\ref{LM:Eii-Psi-f-constant}, $\Psi_f$
		is constant on $\Lag(\pi f)$ and we
		denote the value that it attains  
		by $\bar{\Psi}_f$.  
		The proposition is clear if $P=0$, so   assume $P\neq 0$.

		Suppose first that $\rrk_AP>2$. By Lemma~\ref{LM:Eii-factoring-x-summand},
		there exists $x\in P$ with $f(x,x)\in\units{A}$
		and
		$\pi f(x,x)=0$. 
		Write $P_1=xA$, $P_2=P_1^\perp$ and 
		let $f_i=f|_{P_i\times P_i}$ ($i=1,2$).
		Since $f(x,x)\in\units{A}$,
		we have $(P,f)=(P_1,f_1)\oplus(P_2, f_2)$.
		Moreover, $f(x,x)\in \ker \pi\cap \Sym_1(A,\sigma)\cap\units{A}=\mu\lambda \units{R}$,
		so $(P_1,f_1)\cong \langle \alpha\mu\lambda\rangle_{(A,\sigma)}$
		for some $\alpha\in\units{R}$. 
		Thus, $\disc(f_1)\equiv -\Nrd(\mu)\Nrd(\alpha\mu\lambda)\equiv  \alpha^2\lambda^2\equiv \disc(T/R)\bmod
		(\units{R})^2$ (Proposition~\ref{PR:disc-orth-basic-props}(v)).
		Since $\pi f(x,x)=\pi f(x\mu,x\mu)=0$ and
		$xA=xB\oplus x\mu B$, we have $xB\in \Lag(\pi f_1)$.
		Moreover, 
		$xB\cdot A=xA$ implies that  $\bar{\Psi}_{f_1}=-1$
		(Corollary~\ref{CR:Eii-orth-Psi-restriction}).
		Since $\pi f_1$ is hyperbolic, $[\pi f_2]=[\pi f]=0$ in $W_\veps(B,\tau)$, so  $\pi f_2$
		is hyperbolic by Theorem~\ref{TH:trivial-in-Witt-ring}(ii).
		Now, 
		$\disc(f)=\disc(f_1)\disc(f_2)=\disc(T/R)\disc(f_2)$
		and $\bar{\Psi}_{f}=\bar{\Psi}_{f_1}  \bar{\Psi}_{f_2}=-\bar{\Psi}_{f_2}$
		(Proposition~\ref{PR:dfn-of-Psi-f}(ii)),
		so the proposition holds for $(P,f)$
		if and only if it holds for $(P_2,f_2)$.
		Replacing $(P,f)$ with $(P_2,f_2)$
		and repeating this process, we eventually reduce to the case where $\rrk_AP=2$.
		
		Suppose henceforth that $\rrk_AP=2$. We may assume that $P=A_A$
		and $f=\langle a\rangle_{(A,\sigma)}$ for some $a\in \Sym_1(A,\sigma)\cap\units{A}$.
		Write $a=b_1+\mu b_2$ with
		$b_1,b_2\in B$  and let $\theta$
		denote  the standard $R$-involution of $T$ (so $\lambda^\theta=-\lambda$).
		Note that $a^\sigma=a$ implies $b_2^\theta=-b_2$
		and, by Proposition~\ref{PR:disc-orth-basic-props}(v),
		$\disc(f)
		\equiv -\Nrd(\mu)\Nrd(b_1+\mu b_2)\equiv	  \mu^2(b_1^\theta b_1-\mu^2b_2^\theta b_2)
		\equiv \mu^2(b_1^\theta b_1+\mu^2b_2^2)\bmod(\units{R})^2$.
		
		Straightforward computation shows
		that the Gram matrix of $\pi f$ relative to the $B$-basis $\{1,\mu\}$
		is 
		\[
		X=\SMatII{b_1}{-\mu^2b_2}{-\mu^2b_2}{-\mu^2b_1^\theta} \in \nGL{T}{2}.
		\]
		Thus, $\disc (\pi f)\equiv
		\mu^2b_1^\theta b_1+\mu^4b_2^2\equiv \disc (f)\bmod(\units{T})^2$.
		Since $\pi f$ is hyperbolic, $\disc(\pi f)=(\units{T})^2$,
		so
		there exists $d\in\units{T}$ such that
		\[d^2=\mu^2(b_1^\theta b_1 +\mu^2 b_2^2)\equiv \disc(f)\bmod (\units{R})^2 .\]

		Identifying $A_B$ with $B^2$ via the basis $\{1,\mu\}$,
		let  $M=[\begin{smallmatrix} -\mu^2 b_2+d \\ -b_1 \end{smallmatrix}]T+
		[\begin{smallmatrix} \mu^2b_1^\theta \\ -\mu^2 b_2-d \end{smallmatrix}]T$
		and $M'= [\begin{smallmatrix} -\mu^2 b_2-d \\ -b_1 \end{smallmatrix}]T+
		[\begin{smallmatrix} \mu^2b_1^\theta \\ -\mu^2 b_2+d \end{smallmatrix}]T$.
		Viewing $M$ a subset of $A$,
		we have
		\[M=(-\mu^2 b_2+d   -\mu b_1)T+(\mu^2b_1^\theta   -\mu^3 b_2-\mu d)T.\]
		We claim that $M+M'=B^2$. To see this, observe that
		$[\begin{smallmatrix} -2\mu^2b_2 \\ -2b_1 \end{smallmatrix}]=
		[\begin{smallmatrix} -\mu^2 b_2+d \\ -b_1 \end{smallmatrix}]+
		[\begin{smallmatrix} -\mu^2 b_2-d \\ -b_1 \end{smallmatrix}]$
		and $[\begin{smallmatrix} -2\mu^2b_1^\theta \\ 2\mu^2b_2 \end{smallmatrix}]
		=-[\begin{smallmatrix} \mu^2b_1^\theta \\ -\mu^2 b_2-d \end{smallmatrix}]-
		[\begin{smallmatrix} \mu^2b_1^\theta \\ -\mu^2 b_2+d \end{smallmatrix}]$
		live in $M+M'$, and they generate $B^2_B$ because they
		are the columns of the matrix $[\begin{smallmatrix}
		0 & 2 \\ -2 & 0 \end{smallmatrix}]X\in \nGL{T}{2}$.
		Furthermore, we have $M\cap M'=0$
		because 
		\[\left[\begin{smallmatrix}
		b_1 & -\mu^2 b_2+d\\
		-\mu^2 b_2-d & -\mu^2b_1^\theta
		\end{smallmatrix}\right]M=0
		\qquad
		\text{and}
		\qquad
		\left[\begin{smallmatrix}
		b_1 & -\mu^2 b_2-d \\
		-\mu^2 b_2+d & -\mu^2b_1^\theta
		\end{smallmatrix}\right]M'=0,
		\]
		while
		$[\begin{smallmatrix}
		b_1 & -\mu^2 b_2+d\\
		-\mu^2 b_2-d & -\mu^2b_1^\theta
		\end{smallmatrix}]+[\begin{smallmatrix}
		b_1 & -\mu^2 b_2-d \\
		-\mu^2 b_2+d & -\mu^2b_1^\theta
		\end{smallmatrix}]=2X\in \nGL{T}{2}$.
		It is routine to check that $\pi f(M,M)=\pi f(M',M')=0$,
		so we conclude that $M\in \Lag (\pi f)$.
		
		We claim that 
		$d\in \units{R} $ or $d\in\lambda  \units{R} $.
		Indeed, write $d=\alpha+\beta\lambda$ with $\alpha,\beta\in R$.
		Then $  (\alpha^2+\lambda^2\beta^2)+2\alpha\beta\lambda=d^2\in \units{R}$,
		so  $\alpha\beta=0$ and $\alpha R+\beta R=R$
		(because $d^2=\alpha^2+\lambda^2\beta^2\in \alpha R+\beta R$).
		Multiplying the latter equation by $\alpha$, we get $\alpha^2R=\alpha R$,
		so there exists $c\in R$ with $c\alpha^2=\alpha$.
		The element $\alpha c$ is an idempotent and $R$ is connected,
		hence $\alpha c\in\{0,1\}$. If $\alpha c=0$, then $\alpha =c\alpha^2=0$
		and  $d\in \lambda\units{R}$.
		On the other hand, if $\alpha c=1$, then $\beta= 0$, because
		$\alpha\beta=0$, 
		and $d\in \units{R}$.

		Now,
		if $d = \alpha\lambda $ for some $\alpha\in\units{R}$,
		then    $\disc(f)= \lambda^2(\units{R})^2=\disc(T/R)$,
		and
		\begin{align*}
		2\alpha\lambda\mu&=
		\mu^3 b_2 +\alpha\lambda \mu-\mu^2b_1^\theta+
		\mu^2b_1^\theta   -\mu^3 b_2-\alpha\mu \lambda\\
		&=
		(-\mu^2 b_2+d   -\mu b_1)\mu+
		(\mu^2b_1^\theta   -\mu^3 b_2-\mu d)\in MA 
		\end{align*}
		(note that $b_2\mu=\mu b_2^\theta =-\mu b_2$).
		Thus,
		$MA=A$ and $\bar{\Psi}_f=-1$ by Corollary~\ref{CR:Eii-orth-Psi-restriction}.
		
		On the other hand, if 
		$d\in\units{R}$,
		then		
		$\disc(f)=  (\units{R})^2 =\disc(T/R)^2$ and
 		\begin{align*}
		&(-\mu^2 b_2+d   -\mu b_1)\mu=
		-(\mu^2b_1^\theta   -\mu^3 b_2-\mu d)\in M,\\
		& (\mu^2b_1^\theta   -\mu^3 b_2-\mu d)\mu=-(-\mu^2 b_2+d   -\mu b_1)\mu^2\in M.
		\end{align*}
		Thus, $M$ is  an $A$-module,
		and
		it follows that $\pi(f(M,M)\mu)=\pi f(M,M\mu)=\pi f(M,M)=0$,
		hence $f(M,M)=0$.
		Similarly, $M'$ is also an $A$-module with $f(M',M')=0$,
		so $f$ is hyperbolic and $M$ is a Lagrangian of $f$.
		In particular, $f$ is isotropic.
		Now, by the characterizing property
		of $\Psi_f$ in Proposition~\ref{PR:dfn-of-Psi-f},
		$\bar{\Psi}_f=\Psi_f(M)=1$, and by Corollary~\ref{CR:Eii-orth-Psi-restriction}, 
		there is no $M'\in \Lag(\pi f)$ with $M'A=P$. 
		Moreover, $\rrk_AM=\frac{1}{2}\rrk_AA=1$,
		so $[A]=[\End_A(M)]=[R]=0$ by Proposition~\ref{PR:degree-of-endo-ring}(i).
		
		The proposition follows  because only one of the previous
		cases can hold. Indeed, we cannot have $\bar{\Psi}_f=1$
		and $\bar{\Psi}_f=-1$ simultaneously.
	\end{proof}

	\begin{thm}\label{TH:Eii-holds-orth}
		Theorem~\ref{TH:Eii-holds}
		holds when $(\tau,\veps)$ is orthogonal.
	\end{thm}
	
	\begin{proof}
		Recall that we are given $(P,f)\in \Herm[\veps]{A,\sigma}$
		such that $[\pi f]=0$ in $W_\veps(B,\tau)$.
		By Theorem~\ref{TH:trivial-in-Witt-ring}(ii),
		$\pi f$ is hyperoblic. As explained
		in the introduction to this subsection,
		this means that $\rrk_AP$ is even.
		
		(i) This part is vacuous under our assumptions.
		
		(ii) The pair $(\sigma,\veps)$ is orthogonal by Lemma~\ref{LM:tau-orth-implies-sigma-orth}.
				
		Suppose that there exists $M\in \Lag (\pi f)$
		with $MA=P$. Then $\Psi_f(M)=(-1)^{\frac{1}{2}\rrk_AP}$
		by Corollary~\ref{CR:Eii-orth-Psi-restriction}. 
		Now, by Proposition~\ref{PR:Eii-orth-quat-necessary-cond},
		either $[B]\neq 0$,
		or $\disc(f)=\disc(T/R)^{\frac{1}{2}\rrk_AP}$, as required.
		
		Conversely, suppose that $[B]\neq 0$,
		or $\disc(f)=\disc(T/R)^{\frac{1}{2}\rrk_AP}$.
		If there exists $M\in \Lag(\pi f)$
		with $\Psi_f(M)=(-1)^{\frac{1}{2}\rrk_AP}$,
		then we can argue as in the second paragraph of the proof of Theorem~\ref{TH:Eii-unitary-symplectic}(i),
		replacing Proposition~\ref{PR:Eii-unitary-R-field}
		with Proposition~\ref{PR:Eii-orth-R-field-A-arbit},
		to show that there exists $M'\in\Lag (\pi f)$
		with $M'A=P$.
		The existence of $M$ follows from Proposition~\ref{PR:Eii:Psi-good-if-B-nonzero}
		if $[B]\neq 0$ and from Proposition~\ref{PR:Eii-orth-quat-necessary-cond}
		if $[B]=0$.
		
		Proposition~\ref{PR:Eii-orth-quat-necessary-cond} also implies
		that $[A]=0$, $\disc(f)=\disc(T/R)^{\frac{1}{2}\rrk_AP+1}$ and $f$ is
		isotropic
		if $[B]=0$ and  $\disc(f)\neq \disc(T/R)^{\frac{1}{2}\rrk_AP }$.

		(iii) By Propositions~\ref{PR:Eii-S-not-connected} 
		and~\ref{PR:Eii-T-not-connected},
		we may assume that $T$ is connected.
		By (ii), we may further assume
		that $[B]=0$ and $\disc(f)\neq \disc(T/R)^{\frac{1}{2}\rrk_AP}$,
		in which case $[A]=0$ and $\Psi_f(M)=(-1)^{\frac{1}{2}\rrk_AP+1}$
		for all $M\in \Lag(\pi f)$.
		
		Fix some $M\in \Lag(\pi f)$.
		Since $[A]=[R]$,
		there exists $V\in \rproj{A}$
		such that $\rrk_AV=\deg R=1$
		(Proposition~\ref{PR:degree-of-endo-ring}(iii)).
		Then $M\oplus V$ is a Lagrangian
		of $\pi (f\oplus \Hyp[\veps]{V})$
		which satisfies $\Psi_{f\oplus \Hyp[\veps]{V}}(M\oplus V)= (-1)^{\frac{1}{2}\rrk_AP+1}
		\Psi_{ \Hyp[\veps]{V}}(V)=(-1)^{\frac{1}{2}\rrk_AP+1}
		=(-1)^{\frac{1}{2}\rrk_A(P\oplus V\oplus V^*)}$ (Proposition~\ref{PR:dfn-of-Psi-f}(ii),
		Lemma~\ref{LM:unimodular-implies-rrk-sigma-inv}).
		Replacing $(P,f)$ and $M$
		with $(P\oplus V\oplus V^*,f\oplus \Hyp[\veps]{V})$ and $M\oplus V$,
		we may assume that $\Psi_f(M)=(-1)^{\frac{1}{2}\rrk_A P}$
		and proceed as in the proof of the ``if'' part of (ii).
	\end{proof}

\section{Proof of Theorem~\ref{TH:exactness}}
\label{sec:completion-of-proof}

	We can now prove Theorem~\ref{TH:exactness}.
	We use the notation of \ref{subsec:octagon}.

	\begin{proof}[Proof of Theorem~\ref{TH:exactness}]
		Assume $R$ is semilocal, and recall
		from \ref{subsec:octagon} that $(A,\sigma)$ is an Azumaya $R$-algebra
		with involution, $\lambda,\mu\in\units{A}$ satisfy
		$\lambda^2\in S:=\Cent(A)$, $\lambda\mu=-\mu\lambda$, $\lambda^\sigma=-\lambda$,
		$\mu^\sigma=-\mu$, and we have $B=\Cent_A(\lambda)$, $T=S[\lambda]$, $\tau_1=\sigma|_B$, 
		$\tau_2=\Int(\mu^{-1})\circ \sigma|_B$. To avoid later ambiguity, we henceforth
		write $\sigma_1$ in place
		of $\sigma$.

		By Lemma~\ref{LM:octagon-basic-properties}, $R,S,T,B,A$
		satisfy the assumptions of Notation~\ref{NT:proof-of-Es}.
		Furthermore, $\pi_1$ coincides with $\pi=\pi_{A,B}$ of Lemma~\ref{LM:dfn-of-pi-A-B}.
		We shall consider
		four possibilities for the involution $\sigma$    in Notation~\ref{NT:proof-of-Es}, namely,
		$\sigma_1$ (i.e.\ $\sigma$ from \ref{subsec:octagon}),
		$\Int(\lambda^{-1})\circ \sigma_1$,
		$\Int(\mu^{-1})\circ \sigma_1$
		and
		$\Int((\lambda\mu)^{-1})\circ \sigma_1$.
		The involution $\tau=\sigma|_B$ from Notation~\ref{NT:proof-of-Es} is $\tau_1$ in the first two cases
		and $\tau_2$ in the last two cases, because $\lambda$ commutes with elements from $B$.
		Recall that $\rho:(B,\tau)\to (A,\sigma)$ is the inclusion map.

		\smallskip

		According to Theorem~\ref{TH:exactness-equiv-conds},
		we need prove conditions \ref{item:Eii}, \ref{item:Ei}, \ref{item:Eii-tag}, \ref{item:Ei-tag}.
		
		\smallskip
		
		{\it \noindent Proof of \ref{item:Eii}.}
		Let $(P,f)\in\Herm[\veps]{A,\sigma_1}$ and assume $[\pi_1 f ]=0$.
		Then existence of the required Lagrangian of $\pi_1 f$ follows
		by applying Theorem~\ref{TH:Eii-holds}(iii) to $(P,f)$ with $\sigma_1$ in place of $\sigma$.
		
		\smallskip
		
		{\it \noindent Proof of \ref{item:Ei}.}
		Let $(Q,g)\in\Herm[\veps]{B,\tau_1}$ and assume $[\rho_1g]=0$
		in $W_{-\veps}(A,\sigma)$. Put
		$\sigma=\Int(\lambda^{-1})\circ \sigma_1$ and $\tau=\tau_1$.
		Then $\rho_1 g=\lambda \rho g$, where the right hand side
		is  $\lambda$-conjugation (see~\ref{subsec:conjugation}) 
		of the base-change of $g$ along $\rho$ (see~\ref{subsec:herm-base-change}).
		Thus, $[\rho  g]=0$ in $W_\veps(A,\sigma)$, so
		by Theorem~\ref{TH:Ei-holds}(ii), there exists $(Q',g')\in \Herm[\veps]{B,\tau_1}$
		with $[g]=[g']$ and a Lagrangian $L$ of $\rho g'$ such that $L\oplus Q'=Q'A$.
		Since $L$ is also a Lagrangian of $\rho_1 g'=\lambda \rho g'$, we have
		established \ref{item:Ei} for $(Q,g)$.
		
		\smallskip
		
		{\it \noindent Proof of \ref{item:Eii-tag}.}
		Let $(P,f)\in\Herm[-\veps]{A,\sigma_1}$ and assume $[\pi_2 f ]=0$.
		Put  $\sigma =\Int(\mu^{-1})\circ \sigma_1$, $\tau=\tau_2$ 
		and let $f_2=\mu^{-1}f\in \Herm[\veps]{A,\sigma }$.
		Then $\pi_2 f=\pi f_2$.
		By applying Theorem~\ref{TH:Eii-holds}(iii) to $(P,f_2)$,
		we see that there exists $(P',f'_2)\in \Herm[\veps]{A,\sigma }$
		with $[f_2]=[f'_2]$ and a Lagrangian $M$ of $\pi_2f'_2$
		such that $MA=P'$. Put $f'=\mu f'_2$. Then $(P',f')\in \Herm[-\veps]{A,\sigma_1}$,
		$[f']=[f]$ and $M$ is a Lagrangian of $\pi_2 f'=\pi f'_2$ with $MA=P'$.
		
		\smallskip
		
		{\it \noindent Proof of \ref{item:Ei-tag}.}
		Let $(Q,g)\in\Herm[\veps]{B,\tau_2}$ and assume $[\rho_2g]=0$.
		Put $\sigma =\Int((\lambda\mu)^{-1})\circ \sigma_1$ 
		and $\tau=\tau_2$.
		Then  $\rho_2 g=(\lambda \mu) \rho g$,
		and the proof proceeds as in the case of \ref{item:Ei}.
		
		\smallskip
		
		\noindent This completes the proof of Theorem~\ref{TH:exactness}.			
	\end{proof}

	In fact, Theorems~\ref{TH:Ei-holds} and~\ref{TH:Eii-holds}
	allow us to describe the image 	of the functors $\pi_1,\pi_2,\rho_1,\rho_2$
	when $T$ is connected. This description is given in the following theorem,
	which can  be regarded as a refinement of Theorem~\ref{TH:exactness};
	it
	will be needed
	for some of the applications.

	\begin{thm}\label{TH:finer-exactness}
		With   notation as in \ref{subsec:octagon},
		suppose that $R$ is semilocal and $T$ is connected.
		\begin{enumerate}[label=(\roman*)]
			\item Let $(P,f)\in\Herm[\veps]{A,\sigma}$.
			Then there exists $(Q,g)\in \Herm[-\veps]{B,\tau_2}$
			with $\rho_2 g\cong f$ if and only if 
			$[\pi_1 f]=0$ and one of the following hold:
			\begin{enumerate}[label=(\arabic*)]
			\item $(\sigma,\veps)$ is not symplectic;
			\item $4\mid \rrk_AP$.
			\end{enumerate}
			If $[\pi_1 f]=0$ and conditions (1)--(2) fail,
			then $[A]=0$ and $f$ is hyperbolic.

			\item Let $(Q,g)\in \Herm[\veps]{B,\tau_1}$.
			Then there exists $(P,f)\in \Herm[\veps]{A,\sigma}$
			with $\pi_1 f\cong g$ if and only if   $[\rho_1 g]=0$
			and one of the the following hold:
			\begin{enumerate}[label=(\arabic*)]
			\item $(\sigma,-\veps)$ is not orthogonal;
			\item $[A]=0$;
			\item $[B]\neq 0$;
			\item $(\sigma,-\veps)$ is   orthogonal, $[B]=0$,
			$\rrk_BQ$ is even and $[D(g)]=\frac{\rrk_QB}{2}[A]$
			(see \ref{subsec:disc}; $\tau_1$ is unitary).
			\end{enumerate}
			If  $[\rho_1 g]=0$ and conditions (1)--(4) fail, then
			$\rrk_BQ$ is even, $[D(g)]=(\frac{\rrk_QB}{2}+1)[A]$
			and $g$ is isotropic.	
			
			\item Let $(P,f)\in\Herm[-\veps]{A,\sigma}$.
			Then there exists $(Q,g)\in \Herm[ \veps]{B,\tau_1}$
			with $\rho_1 g\cong f$ if and only if   $[\pi_2 f]=0$
			and one of the following hold:
			\begin{enumerate}[label=(\arabic*)]
			\item $(\tau_2,\veps)$ is not orthogonal;
			\item $[B]\neq 0$;
			\item $(\tau_2,\veps)$ is orthogonal,  
			$\rrk_AP$ is even, and $\disc(f)=\disc(T/R)^{\frac{1}{2}\rrk_AP}$
			(see \ref{subsec:disc};  $(\sigma,-\veps)$ is orthogonal in this case).
			\end{enumerate}
			If $[\pi_2 f]=0$ and conditions (1)--(3)   fail, then $[A]=0$,
			$\rrk_AP$ is even, $\disc(f)=\disc(T/R)^{\frac{1}{2}\rrk_AP+1}$
			and $f$ is isotropic.
			
			\item  Let $(Q,g)\in \Herm[\veps]{B,\tau_2}$.
			Then there exists $(P,f)\in \Herm[\veps]{A,\sigma}$
			with  $\pi_2 f\cong g$
			if and only if $[\rho_2 g]=0$.
		\end{enumerate}
	\end{thm}
	
	\begin{proof}
		We will use the   notation and observations from  the first two paragraphs of
		the proof of Theorem~\ref{TH:exactness}. In particular, we write $\sigma_1$
		for $\sigma$ of \ref{subsec:octagon}.

		(i) By Remark~\ref{RM:finer-exactness}\ref{item:RM:finer:rho-two-image},
		$(Q,g)$ exists if and only if $\pi_1 f$ admits a Lagrangian
		$M$ with $MA=P$.
		Put $\sigma=\sigma_1$ and $\tau=\tau_1$,
		and observe that $\tau:B\to B$ is unitary
		because $\tau_1|_T\neq \id_T$
		and $T$ is connected (see Proposition~\ref{PR:types-of-involutions-Az}). 
		Applying Theorem~\ref{TH:Eii-holds}(i) to $f$ with $\sigma_1$ in place of $\sigma$
		now gives the required statement.
		
		(ii) By Remark~\ref{RM:finer-exactness}\ref{item:RM:finer:pi-one-image},
		$(P,f)$ exists if and only if $\rho_1 g$ admits   a Lagrangian
		$L$ with $Q\oplus L=QA$.
		Put   $\sigma =\Int(\lambda^{-1})\circ \sigma_1$
		and $\tau=\tau_1$. 
		Then $\rho_1 g=\lambda \rho g$
		(notation as in \ref{subsec:herm-base-change}, \ref{subsec:conjugation}),
		so $\rho g\in \Herm[\veps]{A,\sigma}$ has the same Lagrangians as $\rho_1g$.
		The statement therefore follows by applying
		Theorem~\ref{TH:Ei-holds}(i) to $(Q,g)$; note that $(\sigma,\veps)$
		and $(\sigma_1,-\veps)$ have the same type by Corollary~\ref{CR:type-conjugation}(i).
		
		(iii) By Remark~\ref{RM:finer-exactness}\ref{item:RM:finer:rho-one-image},
		$(Q,g)$ exists if and only if $\pi_2 f$ admits a Lagrangian
		$M$ with $MA=P$.
		Put $\sigma =\Int(\mu^{-1})\circ \sigma_1$,
		$\tau=\tau_2$ and let $f_2=\mu^{-1}f\in \Herm[\veps]{A,\sigma_2}$.
		Then $\pi_2 f=\pi f_2$ and both forms have the same Lagrangians.
		In addition, $f$ is isotropic if and only if $f_2$ is isotropic,
		and $\disc(f)=\disc(f_2)$ (Proposition~\ref{PR:disc-orth-basic-props}(ii)).
		The statement therefore follows by applying parts (i) and (ii) of Theorem~\ref{TH:Eii-holds} to 
		$(P,f_2)$.
		Note that $(\sigma,\veps)$
		and $(\sigma_1,-\veps)$ have the same type by Corollary~\ref{CR:type-conjugation}(i),
		and
		$(\tau,\veps)$ cannot be unitary if $(\sigma,\veps)$ is symplectic,
		because $\tau_2|_T=\id_T$ when $\sigma_1|_S=\id_S$.
	
		(iv) By Remark~\ref{RM:finer-exactness}\ref{item:RM:finer:pi-two-image},
		$(P,f)$ exists if and only if $\rho_2 g$ admits   a Lagrangian
		$L$ with $Q\oplus L=QA$.
		Put $\sigma=\Int((\lambda\mu)^{-1})\circ \sigma_1$
		and $\tau=\tau_2$. 
		Then, as in (ii), $\rho_2 g=(\lambda \mu) \rho g$,
		and
		the statement follows by applying
		Theorem~\ref{TH:Ei-holds}(i) to $(Q,g)$.
		Note that $(\tau,\veps)$
		is not unitary when $(\sigma,\veps)$ is orthogonal,
		again because $\tau_2|_T=\id_T$ when $\sigma|_S=\id_S$.
	\end{proof}
	
	\begin{cor}\label{CR:exactness-anisotropic}
		With the notation of \ref{subsec:octagon},
		suppose that $R$ is a semilocal and $T$ is connected.
		Let  $(P,f)\in\Herm[\veps]{A,\sigma}$ be anisotropic and assume
		that $[\pi_1 f]=0$ in $W_\veps(B,\tau)$.
		Then there exists $(Q,g)\in \Herm[-\veps]{B,\tau_2}$
		with $\rho_2 g\cong f$.
		
		Similar statements hold for the image
		of $\pi_1$, $\rho_1$, $\pi_2$.
	\end{cor}
	
	\begin{cor}\label{CR:exactness-field-of-frac}
		With the notation of \ref{subsec:octagon},
		suppose that $R$ is a regular semilocal domain with fraction field
		$K$ and $T$ is connected.
		Let  $(P,f)\in\Herm[\veps]{A,\sigma}$ and assume
		that $[\pi_1 f]=0$ in $W_\veps(B,\tau)$.
		Then there exists $(Q,g)\in \Herm[-\veps]{B,\tau_2}$
		with $\rho_2 g\cong f$
		if and only if there exists
		$(Q',g')\in \Herm[-\veps]{B_K,\tau_{2,K}}$
		with $\rho_2 g'\cong f_K$.
		
		Similar statements hold for the image
		of $\pi_1$, $\rho_1$, $\pi_2$.
	\end{cor}
	
	\begin{proof}
		Since $R$ is a regular domain and $T$
		is connected and finite \'etale over $R$, the ring $T$ is also a regular domain
		\cite[\href{https://stacks.math.columbia.edu/tag/03PC}{Tag 03PC}]{DeJong_2018_stacks_project}.
		In particular, $T_K$ is a field.
		By the Auslander--Goldman theorem \cite[Theorem~7.2]{Auslander_1960_Brauer_Group},
		the natural maps $\Br R\to \Br K$ and $\Br T\to \Br T_K$ are injective. Thus, 
		$[A]=0$ if and only if $[A_K]=0$ and $[B]=0$ if and only if $[B_K]=0$.
		Furthermore, since $R$ is integrally closed, the map $\units{R}/(\units{R})^2\to \units{K}/(\units{K})^2$
		is injective. Finally, since $T$ is connected, the
		type of  $(\tau_2, \veps)$ is the same
		as the type of $(\tau_{2,K},  \veps)$
		and the type of $(\sigma,\pm\veps)$ is the same
		as the type of $(\sigma_K,\pm \veps)$.
		The corollary now follows readily from Theorem~\ref{TH:finer-exactness}.
	\end{proof}

\section{Applications}
\label{sec:applications}

\subsection{Quadratic \'Etale and Quaternion Azumaya Algebras}
\label{subsec:applications:low-degree}

	When   $R$
	is a field, Grenier-Boley and Mahmoudi \cite{Grenier_2005_octagon_of_Witt_grps} noted that 
	the octagon \eqref{EQ:octagon} recovers two exact sequences of 
	Lewis  
	\cite{Lewis_1982_improved_exact_sequences} involving
	Witt groups
	associated to separable quadratic field extensions and quaternion division algebras.
	We generalize these sequences to quadratic \'etale algebras and quaternion (i.e.\ degree-$2$) Azumaya
	algebras  over   semilocal rings.
	
\medskip

	Before we begin, recall from Lemma~\ref{LM:quad-etale-over-semilocal}
	that when $R$ is semilocal (and $2\in\units{R}$),
	every quadratic \'etale $R$-algebra is of the form $R[\lambda\where \lambda^2=\alpha]$
	for some $\alpha\in\units{R}$, and in this case, the standard
	$R$-involution sends $\lambda $ to $-\lambda$.
	Quaternion $R$-algebras admit a similar description, 
	which is well-known when $R$ is a field.
	
	\begin{lem}\label{LM:quat-alg-semilocal}
		Let $A$ be a quaterion Azumaya algebra over a semilocal ring
		$R$.	
		Then there exist $\lambda,\mu\in A$
		such that $\lambda^2,\mu^2\in\units{R}$, $\lambda\mu=-\mu\lambda$
		and $\{1,\lambda,\mu,\mu\lambda\}$
		is an $R$ basis of $A$.
		Furthermore, $A$ admits a unique symplectic involution, $\sigma$,
		which satisfies $\lambda^\sigma=-\lambda$
		and $\mu^\sigma=-\mu$.
	\end{lem}
	
	\begin{proof}
		It is enough to consider the case where $R$ is connected. Otherwise,
		write $R$ as a product of connected rings and work over each factor separately.	
	
		The map $\sigma:a \mapsto \Trd_{A/R}(a)-a$
		is a symplectic involution of $A$, see \cite[Theorem~4.1]{Saltman_1978_Azumaya_algebras_w_involution}.
		It is unique by  \cite[Proposition~I.1.3.4]{Knus_1991_quadratic_hermitian_forms}.
		By Lemma~\ref{LM:invertible-symmetric-elements},
		there exists $\lambda\in \Sym_{-1}(A,\sigma)\cap \units{A}$.
		Then $-\lambda=\lambda^\sigma=\Trd (\lambda)-\lambda$,
		hence  $\Trd(\lambda)=0$.
		Thus, $\lambda^2=\Trd(\lambda)\lambda-\Nrd(\lambda)=-\Nrd(\lambda)\in\units{R}$.
		Since $R\cap \lambda R\subseteq \Sym_1(A,\sigma)\cap \Sym_{-1}(A,\sigma)=0$,
		it follows that $T:=R[\lambda]$ is a quadratic \'etale $R$-algebra
		with $R$-basis $\{1,\lambda\}$.
		By Corollary~\ref{CR:rank-over-max-etale},
		$\rrk_T A_A=2$, so
		we are in the setting of Notation~\ref{NT:proof-of-Es}
		and the lemma follows from Lemma~\ref{LM:strcture-of-quat}.
		(Of course, there are more direct proofs.)
	\end{proof}

	\begin{cor}\label{CR:exactness-for-quaternions}
		Let $R$ be a semilocal ring,
		let $A$ be a quaternion Azumaya $R$-algebra
		and let $\lambda,\mu,\sigma$ be as in Lemma~\ref{LM:quat-alg-semilocal}.
		Write $B=R[\lambda]$ and $\tau=\sigma|_B$.
		Then the sequence 
		\begin{align*}
		0 \to
		W_1(A,\sigma) & \xrightarrow{\pi_1} 
		W_{1}(B,\tau)\xrightarrow{\rho_1}
		W_{-1}(A,\sigma)\xrightarrow{\pi_2}
		W_{1}(B,\id_B)\xrightarrow{\rho_2}
		W_{-1}(A,\sigma) \\
		& \xrightarrow{\pi_1}  
		W_{-1}(B,\tau) \xrightarrow{\rho_1}
		W_1(A,\sigma)
		\to 0 
		\end{align*}
		in which the maps are defined as in \ref{subsec:octagon}
		(with $\tau_1=\tau$ and $\tau_2=\id_B$)
		is exact.
	\end{cor}	
	
	\begin{proof}
		By Proposition~\ref{PR:Eii:unitary-symplectic-case}(ii), $W_{-1}(B,\id_B)=0$.
		The corollary therefore follows from  Theorem~\ref{TH:exactness}.
	\end{proof}
	
	\begin{cor}\label{CR:exactness-for-quad-etale}
		Let $R$ be a semilocal ring and let $T$ be a quadratic
		\'etale $R$-algebra with standard involution $\theta$.
		Let $\rho: R\to T$ denote the inclusion map,
		viewed as a morphism from $(R,\id_R)$ to $(T,\theta)$
		or $(T,\id_T)$,
		and let $\lambda\in T$ be an element
		such that $\lambda^2\in \units{R}$ and $T=R\oplus\lambda R$
		($\lambda$ always exists by Lemma~\ref{LM:quad-etale-over-semilocal}).
		Then the sequence
		\begin{align*}
		0 \to 
		W_{1}(T,\theta)\xrightarrow{\Tr }
		W_{1}(R,\id_R)\xrightarrow{\lambda\rho}
		W_{1}(T,\id_T)\xrightarrow{\Tr }
		W_{1}(R,\id_R)\xrightarrow{\lambda\rho}
		W_{-1}(T,\theta) 
		\to 0 
		\end{align*}
		with maps given by 
		$\Tr (  g)=   \Tr_{T/R}\circ g $,
		$\lambda \rho( f)=\lambda \cdot \rho f$ (notation as in \ref{subsec:herm-base-change},
		\ref{subsec:conjugation})
		is exact.
	\end{cor}	
	
	Baeza \cite[Korollar~2.9]{Baeza_1974_torsion_of_Witt_group_semilocal}
	and Mandelberg \cite[Proposition~2.1]{Mandelberg_1975_classification_of_quad_forms_semilocal_ring}
	established the exactness at the left-to-middle
	and middle terms, respectively.
	Baeza \cite[Theorem~V.5.8]{Baeza_1978_quadratic_forms_semilocal_rings} later  proved
	the exactness at these terms without assuming that $2\in\units{R}$ 
	and   gave   an alternative ending to the exact sequence.
	
	\begin{proof}	
		Let $A=\nMat{R}{2}$ and let $\sigma:A\to A$ be the symplectic
		involution $[\begin{smallmatrix} a & b \\ c & d\end{smallmatrix}]\mapsto 
		[\begin{smallmatrix} d & -b \\ -c & a\end{smallmatrix}]$.
		Write $\alpha:=\lambda^2$.
		We embed $(T,\theta)$ in $(A,\sigma)$ by identifying
		$\lambda$ with $[\begin{smallmatrix} 0 & \alpha \\ 1 & 0\end{smallmatrix}]$.
		Let $\mu:=[\begin{smallmatrix} 1 & 0 \\ 0 & -1\end{smallmatrix}]$. 
		Then $A,\sigma,\lambda,\mu, B:=T$ and $\tau:=\theta$
		satisfy  the assumptions of Corollary~\ref{CR:exactness-for-quaternions}.		
		By Proposition~\ref{PR:Eii:unitary-symplectic-case}(ii), $W_{1}(A,\sigma)=0$,
		so the exact sequence of Corollary~\ref{CR:exactness-for-quaternions} reduces
		to:
		\begin{align} \label{EQ:five-terms-raw}
		0 \to 
		W_{1}(T,\theta)\xrightarrow{\rho_1}
		W_{-1}(A,\sigma)\xrightarrow{\pi_2}
		W_{1}(T,\id_T)\xrightarrow{\rho_2}
		W_{-1}(A,\sigma)\xrightarrow{\pi_1}
		W_{-1}(T,\theta) 
		\to 0  
		\end{align}
		
		Let $u=2[\begin{smallmatrix} 0 & 1 \\ -1 & 0\end{smallmatrix}]$,
		$e= [\begin{smallmatrix} 1 & 0 \\ 0 & 0\end{smallmatrix}]$,
		and let $\trans$ denote the   transpose involution on $A$.
		Then $u$-conjugation induces an isomorphism 
		$W_{-1}(A,\sigma)\to W_1(\nMat{R}{2},\trans)$ 
		and $e$-transfer induces an isomorphism 
		$W_1(\nMat{R}{2},\trans)\to W_1(R,\id_R)$; see~\ref{subsec:conjugation}.
		We claim that under the resulting isomorphism
		$W_{-1}(A,\sigma)\to W_1(R,\id_R)$,
		the maps $\rho_1,\pi_2,\rho_2,\pi_1$ in  
		\eqref{EQ:five-terms-raw} become
		$\Tr,(-\lambda)\rho,\Tr,\lambda \rho$, respectively.
		This will imply that the sequence in the corollary is exact (the sign
		change in the second map does not matter).
		
		To see that $\rho_1$ and $\rho_2$ correspond to $\Tr$,
		note that for every $Q\in\rproj{B}$, the map $x\mapsto xe:Q\to QAe$ is a natural isomorphism of $R$-modules.
		Indeed, it is routine to check this for $Q=B_B$ and the general case follows
		from the naturality and the fact that every $Q\in\rproj{B}$ is a summand of $B^n_B$
		for some $n$.
		One readily checks that $e^\tau u (\lambda x)e=
		e^\tau u(\lambda\mu x)e=e\Tr_{T/R}(x)$ for all $x\in T$.
		Using this, it is routine to check that, upon identifying
		$eAe$ with $R$, the isomorphism
		$x\mapsto xe:Q\to QAe$ is an isometry
		from $\Tr g$ to $(u(\rho_1g))_e$, resp.\ $(u(\rho_2g))_e$, which is what we want.
		
		We now show that $\pi_1$ and $\pi_2$ correspond to $\lambda\rho$
		and $(-\lambda)\rho$, respectively. 
		Given $V\in \rproj{R}$, we view $V^2$ as a right $A$-module
		by considering pairs in $V^2$ as ${1\times 2}$ matrices and letting
		$A=\nMat{R}{2}$ act by matrix multiplication.
		If $(V,f)\in\Herm[1]{R,\id_R}$, let $f':V^2\times V^2\to A$
		be given by $f'((x,y),(z,w))=[\begin{smallmatrix}f(x,z) & f(x,w) \\ f(y,z)& f(y,w)\end{smallmatrix}]$.
		Then $(V^2,f')$ is a $1$-hermitian space over $(A,\trans)$
		and $f'_e\cong f$.
		It is enough to show that $\pi_1(u^{-1}f')\cong\lambda\rho f$
		and $\pi_2(u^{-1}f')\cong (-\lambda )\rho f$.
		The $A$-module $V^2$ inherits a $T$-module
		structure, and the map $x\otimes (a+\lambda b)\mapsto (2xa,2x\alpha b):V_T\to V^2$
		($x\in V$, $a,b\in R$)
		is an isomorphism of $T$-modules. (As in the previous paragraph,
		it is enough to check this for $V=R_R$.)
		Note also that $\pi_1[\begin{smallmatrix}a & b \\ c& d\end{smallmatrix}]
		=\frac{1}{2}(a+d)+\frac{1}{2}(\alpha^{-1} b+c)\lambda$
		and $\pi_2[\begin{smallmatrix}a & b \\ c& d\end{smallmatrix}]
		=\frac{1}{2}(a-d)+\frac{1}{2}(\alpha^{-1} b-c)\lambda$ (e.g.\
		use Lemma~\ref{LM:dfn-of-pi-A-B}(ii)).
		It is now routine to check that the isomorphism
		$V_T\to V^2$  
		is   an isometry
		from $\lambda \rho f$ to $\pi_1 (u^{-1}(f'))$, resp.\ from $(-\lambda)\rho f$ to $\pi_2(u^{-1} (f'))$,
		which is  what we want.
	\end{proof}

	\begin{cor}\label{CR:injectivity-for-quat-algs}
		Let $A$ be a quaternion Azumaya algebra over a semilocal ring $R$
		and let $\sigma:A\to A$ be the unique symplectic involution of $A$.
		Then the map
		\[ [f]\mapsto [\Trd_{A/R}\circ f]:W_1(A,\sigma)\to W_1(R,\id_R)\]
		is injective.
	\end{cor}
	
	\begin{proof}
		Let $\lambda,\mu,\sigma$ be as in Lemma~\ref{LM:quat-alg-semilocal}
		and let $B,\tau$ be as in Corollary~\ref{CR:exactness-for-quaternions}.
		Corollaries~\ref{CR:exactness-for-quaternions} and~\ref{CR:exactness-for-quad-etale}
		imply that the maps $\pi_1:W_1(A,\sigma)\to W_1(B,\tau)$
		and $\Tr_{B/R}:W_1(B,\tau)\to W_1(R,\id_R)$
		are injective. Their composition is  
		the map in the corollary.
	\end{proof}
	
	We now generalize a theorem  of Jacobson \cite{Jacobson_1940_note_on_hermitian_forms}
	(see \cite[Theorems~10.1.1, 10.1.7]{Scharlau_1985_quadratic_and_hermitian_forms}
	for a modern restatement) from fields to semilocal rings. Here we  
	need Corollary~\ref{CR:exactness-anisotropic}.
	
	\begin{thm}\label{TH:Jacobson-semilocal}
		Let $A$ be a quadratic \'etale   (resp.\ quaternion Azumaya) algebra over a semilocal ring $R$
		and let $\sigma:A\to A$ be the standard involution (resp.\ unique symplectic involution) of $A$.
		Write $\Tr=\Tr_{A/R}$ (resp.\ $\Tr=\Trd_{A/R}$)
		and let $(P,f),(P',f')\in\Herm[1]{A,\sigma}$. Then:
		\begin{enumerate}[label=(\roman*)]
			\item  $(P,f)$
			is isotropic if and only if $(P,\Tr\circ f)\in\Herm[1]{R,\id_R}$ is isotropic.
			\item $(P,f)\cong (P',f')$
			if and only if  $(P,\Tr\circ f)\cong (P',\Tr\circ f')$ in $\Herm[1]{R,\id_R}$.
		\end{enumerate}
	\end{thm}
	
	\begin{proof}		
		(i) By writing $R$ as  a product of
		connected rings and working over each factor separately,
		we may assume that $R$ is connected.
		
		We begin with the case where $A$ is quadratic \'etale over $R$ and $\Tr=\Tr_{A/R}$.
		If $A$ is not connected, then $A=R\times R$ and $\sigma$ is the exchange involution
		$(x,y)\mapsto (y,x)$ (Lemma~\ref{LM:non-connected-S}). In this case, $(P,f)$ is always
		hyperbolic (Example~\ref{EX:exchange-involution}) and thus so is $(P,\Tr    f)$.
		A hyperbolic space is isotropic if and only if its underlying
		module is nonzero, so the equivalence holds.
		
		Suppose now that $A$ is connected.
		It is clear that if $(P,f)$ is isotropic, then so is $(P,\Tr  f)$.
		Conversely, assume that $(P,\Tr f)$ is isotropic.
		By Proposition~\ref{PR:ansio-Witt-equivalent},
		there is anisotropic  $(Q,g)\in\Herm[1]{R,\id_R}$ and
		a nonzero $U\in\rproj{R}$
		such that $ \Tr f \cong g\oplus \Hyp[1]{U}$. By Corollary~\ref{CR:exactness-anisotropic}
		and the isomorphism between \eqref{EQ:five-terms-raw} and the exact
		sequence in the Corollary~\ref{CR:exactness-for-quad-etale},
		there is $(P'',f'')\in \Herm[1]{A,\sigma}$ such that
		$\Tr f''\cong g$. Corollary~\ref{CR:exactness-for-quad-etale} also tells us that
		$\Tr:W_1(A,\sigma)\to W_1(R,\id_R)$ is injective, so  $[f'']=[f]$. 
		Since $ \Tr f \cong \Tr f''\oplus \Hyp[1]{U}$,
		we have $\rrk_A P'' < \rrk_AP$. Thus,
		Theorem~\ref{TH:trivial-in-Witt-ring}(i) implies that there
		is a nonzero  $V\in \rproj{A}$ such that
		$f\cong f''\oplus \Hyp[1]{V}$, 
		so $f$ is isotropic.
		
		We now consider with the case where $A$ is quaternion Azumaya.
		Define $B$, $\tau_1$ and $\pi_1$ as in Corollary~\ref{CR:exactness-for-quaternions}.
		Since we proved part (i) for 	quadratic \'etale algebras, 
		and since $\Trd_{A/R}=\Tr_{B/R}\circ \pi_1$, it is enough
		to show that $(P,f)$ is isotropic if and only if $(P,\pi_1 f)$
		is isotropic. The ``only if'' part is clear so we turn to the ``if'' part.
		Suppose that $(P,\pi_1f) $ is isotropic.
		Using Proposition~\ref{PR:ansio-Witt-equivalent}, choose anisotropic
		$(Q,g)\in\Herm[\veps]{B,\tau_1}$ and nonzero $U\in\rproj{B}$
		such that $\pi_1 f\cong g\oplus\Hyp[1]{U}$.
		If $B$ is connected, then arguing as in the previous paragraph shows that $f$ is isotropic. 
		If $B$ is not connected, then $B=R\times R$ by Lemma~\ref{LM:non-connected-S}.
		Let $e=(1,0)\in B$. By Lemma~\ref{LM:idempotent-in-T}(iii), we have $[A]=[eB]=[R]=0$
		in $\Br R$, so by Lemma~\ref{PR:Eii:unitary-symplectic-case}(ii), $f$
		is hyperbolic and therefore isotropic ($P\neq 0$ because $\pi_1f$ is isotropic).
		 
		(ii) The map $\Tr:W_1(A,\sigma)\to W_1(R,\id_R)$
		is injective  
		by Corollary~\ref{CR:exactness-for-quad-etale} in
		the quadratic \'etale case and by
		Corollary~\ref{CR:injectivity-for-quat-algs} in the quaternion Azumaya case.
		The statement now follows from Theorem~\ref{TH:trivial-in-Witt-ring}(iii).
	\end{proof}
	
	\begin{remark}\label{RM:why-is-Jacobson-difficult}
		When $R$ is a field,   Theorem~\ref{TH:Jacobson-semilocal}(i)
		is straightforward. Indeed, if there is some nonzero $x\in P$ with
		$\Tr f(x,x)=0$, then $f(x,x)=0$ because $f(x,x)\in\Sym_1(A,\sigma)=R$.
		Since $R$ is a field, $xA$ is an $A$-summand of $P$, so the form 
		$f$ is isotropic. 
		However, this argument does not work when $R$ is a semilocal ring, because $xA$ may not be an
		$A$-summand of $P$ even when $xR$ is an $R$-summand of $P$.
		For example, take $R=\Z_{(5)}$, $A=\Z_{(5)}[i]$ where $i=\sqrt{-1}$, let $\sigma:A\to A$
		act by complex conjugation 
		and consider $f=\Trings{1,-1}_{(A,\sigma)}$. Put
		$x=(1+2i,1+2i)\in A^2$. Then $\Tr f(x,x)=0$
		and $xR$ is an $R$-summand of $A^2_A$, but $A/xA$ has nonzero $5$-torsion elements,
		which means that $xA$ cannot be an $A$-summand of $A^2_A$.
	\end{remark}

\subsection{The Grothendieck--Serre Conjecture}

	Let $R$ be a regular local ring with fraction field $K$.
	Recall from the introduction that the
	Grothendieck--Serre conjecture
	asserts that   for every reductive (connected) group
	$R$-scheme $\bfG$, the restriction map $\HH^1_{\et}(R,\bfG)\to \HH^1_{\et}(K,\bfG)$
	has trivial kernel.
	We now use the corollaries
	of 	\ref{subsec:applications:low-degree}
	and
	results of Balmer--Walter
	\cite{Balmer_2002_Gersten_Witt_complex}
	and Balmer--Preeti \cite{Balmer_2005_shifted_Witt_groups_semilocal}  to establish
	the conjecture
	for some outer forms of $\uGL_n$ and $\uSp_{2n}$
	when   $\dim R\leq 4$. (Note that  in contrast
	to many sources discussing the conjecture, $R$ is not assumed to contain a field.)
	
\medskip

	In order to translate   statements about Witt groups to cases
	of the Grothendieck--Serre conjecture,
	we use the
	following proposition.
	The  case $A=R$ 
	is contained in \cite[Proposition~1.2]{Colliot_1979_quadratic_forms_over_semilocal_rings}.

    \begin{prp}
    	\label{PR:GS-equiv-conds}
    	Let $R$ be a  semilocal regular domain with fraction field
    	$K$, let $(A,\sigma)$ be an Azumaya $R$-algebra
    	with  involution and let $\veps\in \Cent(A)$ be an element
    	such that $\veps^\sigma\veps=1$.
        Then the following conditions are equivalent:
        \begin{enumerate}[label=(\alph*)]
            \item  The restriction map $W_\veps(A,\sigma)\to W_\veps(A_K,\sigma_K)$ is injective.
            \item Every two hermitian spaces $(P,f),(P',f')\in\Herm[\veps]{A,\sigma}$ 
            such that $f_K\cong f'_K$ are isomorphic.
            \item[(c)] 
            The restriction map $\HH^1_{\et}(R,\uU(f))\to \HH^1_{\et}(K,\uU(f))$ has trivial kernel
            for all $(P,f)\in\Herm[\veps]{A,\sigma}$,
        	\item[(d)] 
            The restriction map $\HH^1_{\et}(R,\uU^0(f))\to \HH^1_{\et}(K,\uU^0(f))$ has trivial kernel
            for all  $(P,f)\in\Herm[\veps]{A,\sigma}$.
        \end{enumerate}
    \end{prp}

    \begin{proof}
		(a)$\implies$(b): 
		By (a),   $[f]=[f']$ in $W_\veps(A,\sigma)$.
		Since $P_K\cong P'_K$, we have $\rrk_AP\cong \rrk_AP'$,
		so    $f\cong f'$ by Theorem~\ref{TH:trivial-in-Witt-ring}(iii).

        (b)$\implies$(a): Let
        $(P,f)\in\Herm[\veps]{A,\sigma}$
        and suppose that $[f_K]=0$.
        Then, by Theorem~\ref{TH:trivial-in-Witt-ring}(ii), $f_K$ is hyperbolic, say
        $(P_K,f_K)\cong (Q\oplus Q^*,\Hyp[\veps]{Q})$ with $Q\in\rproj{A_K}$.
        Since $R$ is regular, $\ind A=\ind A_K$
        \cite[Proposition~6.1]{Antieau_2014_topological_period_index}.
        Thus, by Theorem~\ref{TH:index-description} and Corollary~\ref{CR:index-divides-rrk},
        there is $L\in \rproj{A}$ with $\rrk_A L=\rrk_{A_K} Q$.
        By Lemma~\ref{LM:rank-determines}, $Q\cong L_K$.
        This means that  $f_K \cong \Hyp[\veps]{Q}\cong (\Hyp[\veps]{L})_K$,
        so by (b), $f\cong \Hyp[\veps]{L}$ and $[f]=0$.

        (b)$\iff$(c): It is well-known  
        that   $\HH^1_{\et}(R,\uU(f))$ is in correspondence
        with isomorphism classes of  hermitian spaces $(P',f')\in \Herm[\veps]{A,\sigma}$
        that become isomorphic to   $(P,f)$ after base-changing   to some faithfully flat finitely presented
        (i.e.\ fppf) 
        $R$-algebra,  see \cite[Proposition~5.1]{Bayer_2017_rational_isomorphism_forms}, for instance.
        By \cite[Proposition~A.1]{Bayer_2016_patching_weak_approximation}, $(P',f')$ is such a space
        if and only if  $\rrk_AP=\rrk_AP'$, or equivalently,
        $\rrk_{A_K}P_K=\rrk_{A_K}P'_K$. The equivalence now follows
        from the fact that the correspondence is compatible with base change.
        
        (c)$\iff$(d): 
       	If $(\sigma,\veps)$ is symplectic or unitary,
       	then $\uU^0(f)=\uU(f)$ (Proposition~\ref{PR:U-zero-description})  and the statement is trivial.
       	
       	Assume  $(\sigma,\veps)$ is orthogonal.  
		Then 
        $
        1\to \uU^0(f)\to \uU(f)\xrightarrow{\Nrd} \umu_2 \to 1
        $
        is a short exact sequence of sheaves on $(\Aff/R)_{\et}$.  
        (To see that the last map is surjective, pass
        to the stalks and apply Theorem~\ref{TH:criterion-for-det-one}.)
        This induces a commutative  diagram of pointed sets,
        \[
        \xymatrix{
        U(f) \ar[r]^{\Nrd} \ar[d] & 
        \{\pm 1\} \ar[r]\ar@{=}[d] & 
        \HH^1_{\et}(R,\uU^0(f)) \ar[r]\ar[d]^\alpha &
        \HH^1_{\et}(R,\uU(f))  \ar[d]^\beta \ar[r] &
        \HH^1_{\et}(R,\umu_2) \ar[d]^\gamma\\
        U(f_K) \ar[r]^{\Nrd}  & 
        \{\pm 1\} \ar[r]  & 
        \HH^1_{\et}(K,\uU^0(f)) \ar[r]  &
        \HH^1_{\et}(K,\uU(f)) \ar[r] &
        \HH^1_{\et}(K,\umu_2)   
        }
        \]
        in which the rows
        are  cohomology exact sequences
        and the vertical arrows are restriction maps.
		We need to prove that $\alpha$ has trivial kernel if and only if $\beta$
		has trivial kernel.        
        
        The map $\gamma$ has trivial kernel by \cite[Proposition~2.2]{Colliot_1978_cohomology_of_mult_grps},
        for instance.
        Thus, the Four Lemma implies that $\beta$
        has trivial kernel whenever   $\alpha$ has trivial kernel.
        
        Next,
		since $R$ is regular, $[A]=0$  if and only if $[A_K]=0$ \cite[Theorem~7.2]{Auslander_1960_Brauer_Group}.
		Thus, by Theorem~\ref{TH:criterion-for-det-one}, $\Nrd:U(f)\to \{\pm 1\}$ and
		$\Nrd:U(f_K)\to  \{\pm 1\}$ have the same image.
		Using this, an easy diagram chase shows that if $\beta$ has trivial kernel, then so does $\alpha$.       
    \end{proof}
    
    Conditions (a)--(d) of Proposition~\ref{PR:GS-equiv-conds} are conjectured
    to hold under the assumptions of the proposition. 
    This was affirmed by Gille \cite[Theorem~7.7]{Gille_2013_coherent_herm_Witt_grps}
    (see also section 3.3 of that paper) when $R$ is regular local and contains a field.
    Furthermore, we have:

    \begin{thm}[Balmer, Preeti, Walter] \label{TH:BPW}
		Let $R$ be a  semilocal regular domain with $\dim R\leq 4$ 
		and let  
		$K$ denote the fraction field of $R$.
		Then the restriction map
		$W_1(R,\id_R)\to W_1(K,\id_K)$ is injective. 
	\end{thm}

	\begin{proof}
		Balmer and Walter  \cite[Corollary~10.4]{Balmer_2002_Gersten_Witt_complex} 
		proved the theorem when $R$ is local,
		and
		Balmer and Preeti  
		\cite[p.~3]{Balmer_2005_shifted_Witt_groups_semilocal} 
		showed  that it is enough to require that $R$ is semilocal. 
	\end{proof}
    
    We   establish the following additional cases.
    
    \begin{thm}\label{TH:GS-new}
    	Let $R$ be a regular semilocal domain with $\dim R\leq 4$,
    	let $K$ denote the fraction field of $R$
    	and let $(A,\sigma)$ be an Azumaya $R$-algebra with involution.
    	Assume that
    	\begin{enumerate}[label=(\arabic*)]
    		\item $\ind A=1$ and $\sigma$ is unitary, or
    		\item $\ind A\leq 2$ and $\sigma$ is symplectic.
    	\end{enumerate}
    	Then restriction map $W_1(A,\sigma)\to W_1(A_K,\sigma_K)$ is injective.
    \end{thm}
    
	\begin{proof}
		When $\ind A=1$ and $\sigma$ is symplectic, we
		have $W_1(A,\sigma)=0$ (Proposition~\ref{PR:Eii:unitary-symplectic-case}(ii)).
		We may therefore assume that
		$\ind A=2$ when $\sigma$ is symplectic.	
	
		By Theorem~\ref{TH:invariant-primitive-idempotent},
		there exists a $\sigma$-invariant idempotent $e\in A$
		such that $\deg eAe=\ind A$.
		Applying $e$-transfer (see~\ref{subsec:conjugation}),
		we may assume that $\deg A=\ind A$, namely,
		that $A$ is quadratic \'etale or quaternion Azumaya over $R$.

		Consider the commutative square
		\[
		\xymatrix{
			W_1(A,\sigma) \ar[r] \ar[d] &
			W_1(R,\id_R) \ar[d] \\
			W_1(A_K,\sigma_K) \ar[r] &
			W_1(K,\id_K)	
		}
		\]
		in which the vertical arrows are restriction maps
		and the horizonal arrows are given by $[f]\mapsto[\Tr_{A/R}\circ f]$
		if $A$ is quadratic \'etale, or $[f]\mapsto [\Trd_{A/R}\circ f]$
		if $A$ is quaternion Azumaya.
		The right vertical arrow is injective 
		by Theorem~\ref{TH:BPW} and the top horizontal arrow 
		is injective by Corollary~\ref{CR:exactness-for-quad-etale}
		when $A$ is quadratic \'etale,
		and by Corollary~\ref{CR:injectivity-for-quat-algs} when
		$A$ is quaternion Azumaya. Thus,
		the left vertical arrow is also injective.
	\end{proof}
	
	As a corollary, we verify some cases of the Grothendieck--Serre conjecture.
	
	\begin{cor}\label{CR:GS-new}
		With notation and assumptions as in Theorem~\ref{TH:GS-new},
		the  restriction map 
		$\HH^1_{\et}(R,\uU (A,\sigma))\to
		\HH^1_{\et}(K,\uU (A,\sigma))$ has trivial kernel.
	\end{cor}

	\begin{proof}
		We have $\uU(A,\sigma)\cong \uU(f)$, where $f:A\times A\to A$
		is the $1$-hermitian form $f(x,y)=x^\sigma y$, so 
		the corollary    follows from Proposition~\ref{PR:GS-equiv-conds}
		and Theorem~\ref{TH:GS-new}.
	\end{proof}    
    
\subsection{Purity}

	Let $R$ be a regular domain with fraction field $K$
	and let $\bfG$ be a reductive (connected)
	group $R$-scheme. 
	Recall from the introduction that we say that purity
	holds for $\bfG$   if 
	\[
	\im \left(\HH^1_{\et}(R,\bfG)\to \HH^1_{\et}(K,\bfG)\right)=
	{\textstyle\bigcap_{\frakp\in R^{(1)}}} \left(\HH^1_{\et}(R_\frakp,\bfG)\to \HH^1_{\et}(K,\bfG)\right),
	\]
	where $R^{(1)}$ denotes the set of height-$1$ primes in $\Spec R$.
	The local purity conjecture asserts that purity
	holds for $\bfG$ whenever $R$ is regular semilocal.	
	
	The following proposition allows us to prove purity for some
	group schemes by establishing certain results about hermitian forms.

	\begin{prp}\label{PR:purity-equiv-conds}
		Let $R,A,\sigma,\veps$ and $K$ be as in Proposition~\ref{PR:GS-equiv-conds}.
		Suppose that:
		\begin{enumerate}[label=(\arabic*)]
			\item every anisotropic $\veps$-hermitian space
		over $(A,\sigma)$ remains
		anisotropic after base changing along $R\to K$, and
			\item $\im\big(W_\veps(A,\sigma)\to W_\veps(A_K,\sigma_K)\big)=\bigcap_{\frakp\in R^{(1)}}
			\im\big(W_\veps(A_\frakp,\sigma_\frakp)\to W_\veps(A_K,\sigma_K)\big)$.
		\end{enumerate}
		Then the equivalent conditions of Proposition~\ref{PR:GS-equiv-conds} hold,
		and purity holds for $\uU(f)$ for every $(P,f)\in  \Herm[\veps]{A,\sigma}$.	
	\end{prp}
	
	Both (1) and (2) are conjectured to hold when $R$ is a regular semilocal domain.
	
	
	\begin{proof}
		We first prove that
		condition (a) of Proposition~\ref{PR:GS-equiv-conds} holds.
		Let $w$ be a Witt class in $\ker(W_\veps(A,\sigma)\to W_\veps(A_K,\sigma_K))$.
		By Proposition~\ref{PR:ansio-Witt-equivalent},
		$w$ is represented by an anisotropic $(P,f)\in\Herm[\veps]{A,\sigma}$.
		Since $[f_K]=0$, the form $f_K$ is hyperoblic (Theorem~\ref{TH:trivial-in-Witt-ring}(ii)).
		If $P\neq 0$, then  $f_K$ is isotropic, and by   assumption (1), so is $f$, contradicting
		our choice of $f$. Thus, $P=0$ and $w=0$.

		Next, let $(P,f)\in \Herm[\veps]{A,\sigma}$. 
		Recall from the proof of Proposition~\ref{PR:GS-equiv-conds} that $\HH^1_{\et}(R,\uU(f))$
		classifies isomorphism classes of hermitian spaces $(P',f')\in\Herm[\veps]{A,\sigma}$
		with $\rrk_AP=\rrk_AP'$.
		Thus,   purity for
		$\uU(f)$  is the equivalent to saying that if $(P_0,f_0)\in \Herm[\veps]{A_K,\sigma_K}$
		is such that for every $\frakp\in R^{(1)}$,
		there is $(P',f')\in \Herm[\veps]{A_\frakp,\sigma_\frakp}$
		with $f'_K\cong f_0$,
		then there is $(P'',f'')\in \Herm[\veps]{A,\sigma}$
		such that $  f''_K\cong f_0$.
		
		Given   $(P_0,f_0)\in \Herm[\veps]{A_K,\sigma_K}$,
		assumption (2) implies that
		there is $(\tilde{P},\tilde{f})\in \Herm[\veps]{A,\sigma}$
		such that $[f_0]=[\tilde{f}_K]$. By Proposition~\ref{PR:ansio-Witt-equivalent},
		we may take $\tilde{f}$ to be anisotropic. By (1), $\tilde{f}_K$ is also anisotropic.
		By Corollary~\ref{CR:constant-even-ranks}(i),
		$\rrk_{A_K} \tilde{P}_K$ and $\rrk_{A_K} P_0$
		are constant. If
		$\rrk_{A_K}\tilde{P}_K>\rrk_{A_K} P_0$,
		then by Theorem~\ref{TH:trivial-in-Witt-ring}(i),
		there is a nonzero $V\in \rproj{A_K}$ such
		that $\tilde{f}_K\cong f_0 \oplus \Hyp[\veps]{V}$, contradicting
		the fact that $\tilde{f}_K$ is anisotropic.
		Thus, $\rrk_{A_K}\tilde{P}_K\leq \rrk_{A_K} P_0$.
		Applying Theorem~\ref{TH:trivial-in-Witt-ring}(i) again,
		we get $V\in \rproj{A_K}$ such
		that $\tilde{f}_K  \oplus \Hyp[\veps]{V} \cong f_0$.
		As in the proof of Proposition~\ref{PR:GS-equiv-conds},
		there is $L\in\rproj{A}$ with $V\cong L_K$,
		so $f_0\cong (\tilde{f}\oplus \Hyp[\veps]{L})_K$.
	\end{proof}
	
	Condition (1) of Proposition~\ref{PR:purity-equiv-conds} is known as  \emph{purity for $W_\veps(A,\sigma)$}.
	Provided $R$ 
	is regular local and contains a field, 
	it was established by Ojanguren and Panin \cite{Ojanguren_1999_purity_for_Witt_group}
	when $A=R$ and $\veps=1$, and by Gille \cite[Theorem~7.7]{Gille_2013_coherent_herm_Witt_grps}
	for general  $A,\sigma,\veps$.

	Condition (2) was proved by Panin and Pimenov \cite[Theorem~1.1]{Panin_2010_rational_isotropy_implies_isotropy}
	for $(A,\sigma)=(R,\id_R)$ 
	when $R$ is a regular semilocal domain  containing an infinite field $k$, and 
	Scully \cite[Theorem 5.1]{Scully_Artin_Springer_over_semilocal} 
	eliminated the assumption that $k$ is infinite.

	We use these results together with Theorem~\ref{TH:Jacobson-semilocal}
	to prove purity
	for some outer forms of $\uGL_n$ and $\uSp_{2n}$.

	\begin{thm}\label{TH:purity-new}
		Let $R$ be a regular local ring
		containing a field  and let $K$ denote the fraction field of $R$.
		Let $(A,\sigma)$ be a quadratic \'etale $R$-algebra with its
		standard  involution or a quaternion Azumaya $R$-algebra
		with its symplectic involution.
		Let $(P_0,f_0)\in\Herm[1]{A_K,\sigma_K}$ be a hermitian
		space such that for every $\frakp\in R^{(1)}$,
		there exists $(P^{(\frakp)},f^{(\frakp)})\in \Herm[1]{A_{ \frakp},\sigma_{ \frakp}}$
		such that $(P_0,f_0)\cong (P^{(\frakp)}_K,f^{(\frakp)}_K)$.
		Then there exists $(P ,f )\in \Herm[1]{A,\sigma}$
		such that $(P_0,f_0)\cong (P_K,f_K)$.
		Equivalently,  for every $(P,f)\in \Herm[1]{A,\sigma}$,
		purity holds for $\uU(f)$.
	\end{thm}
	
	\begin{proof}
		We need to prove conditions (1) and (2) of Proposition~\ref{PR:purity-equiv-conds}.
		We noted above that (1) holds in our situation,  see Gille 
		\cite[Theorem~7.7]{Gille_2013_coherent_herm_Witt_grps}, so it remains to prove (2).
		Suppose that $(P,f)\in \Herm[1]{A,\sigma}$ is anisotropic 
		and let $\Tr$ be as in Theorem~\ref{TH:Jacobson-semilocal}.
		By part (i) of that theorem, $(P, \Tr\circ f)$ is an anisotropic
		$1$-hermitian space over $(R,\id_R)$,
		and by   \cite[Theorem 5.1]{Scully_Artin_Springer_over_semilocal}, so
		is $(P_K,\Tr\circ f_K)$.
		Applying Theorem~\ref{TH:Jacobson-semilocal}(i) again, shows that $(P_K,f_K)$
		is anisotropic, which is what we want.
	\end{proof}
	
\subsection{The Kernel of The Restriction Map}

	In our final application we characterize 
	the kernel of the restriction
	map $W_1(R,\id_R)\to W_1(S,\id_S)$
	when $R$ is a $2$-dimensional regular domain (not necessarily semilocal)
	and $S$ is a quadratic \'etale $R$-algebra.
	When $R$ is a field, 
	this is a celebrated theorem of Pfister, see \cite[Theorem~I.5.2]{Scharlau_1985_quadratic_and_hermitian_forms},
	for instance.

	The proof makes use of Colloit-Th\'el\`ene and Sansuc's purity theorem in dimension $2$
	\cite[Corollary~6.14]{Colliot_1979_quadratic_fiberations}
	and a theorem of Pardon \cite[Theorem~5]{Pardon_1982_Gersten_conjecture} 
	asserting that $W_1(R,\id_R)\to W_1(K,\id_K)$ is injective when $R$ is regular of dimension
	$2$ with fraction field $K$.

	\begin{thm}\label{TH:kernel-of-rest-quads}
		Let $R$ be a   regular domain of dimension $\leq 2$ and   let $S$
		be a quadratic \'etale $R$-algebra with standard involution $\theta$.
		Then the sequence
		\[
		W_1(S,\theta)\xrightarrow{[g] \mapsto [\Tr_{S/R}\circ g]} W_1(R,\id_R)\xrightarrow{[f]\mapsto [f_S]}
		W_1(S,\id_S)
		\]
		is exact in the middle.
	\end{thm}
	
	\begin{proof}
		When $S$ is not connected,   $S=R\times R$ (Lemma~\ref{LM:non-connected-S}) 
		and the theorem is straightforward.
		Assume that $S$ is a domain henceforth. We abbreviate $\Tr_{S/R}$ to $\Tr$
		and let $K$ denote the fraction field of $R$.
		
		The sequence is a chain complex in the middle by virtue of Proposition~\ref{PR:octagon-is-a-complex}
		and the proof of Corollary~\ref{CR:exactness-for-quad-etale}; this can also be checked directly.
		
		Let $(P,f)\in \Herm[1]{R,\id_R}$
		and assume that $[f_S]=0$ in $W_1(S,\id_S)$.
		Then $[f_{S\otimes K}]=0$ in $W_1(S_K,\id_{S_K})$.
		By virtue of Corollary~\ref{CR:exactness-for-quad-etale},
		there exists $(Q_0,g_0)\in \Herm[1]{S_K,\theta_K}$
		such that $[\Tr g_0]= [f_K]$.
		Adding a hyperbolic space to $(Q_0,g_0)$,
		we may assume that $\dim_K Q_0 > \dim_K P_K$.
		
		Write $f'_0=\Tr g_0$. Then there exists a $K$-vector space $V$
		such that $f'_0\cong f_K\oplus \Hyp[1]{V}$.
		Choose $U\in\rproj{R}$ with $U_K\cong V$
		and let $f'=f\oplus \Hyp[1]{U}$.
		Then $f'_K\cong f'_0=\Tr g_0$.
		
		Let $\frakp\in R^{(1)}$. By Corollary~\ref{CR:exactness-field-of-frac} and the proof
		of Corollary~\ref{CR:exactness-for-quad-etale},
		there exists
		$(Q^{(\frakp)},g^{(\frakp)})\in \Herm[1]{S_\frakp,\theta_\frakp}$
		such that $\Tr g^{(\frakp)}\cong f'_\frakp$.
		In particular, $\Tr g^{(\frakp)}_K\cong f'_K=\Tr g_0$.
		Since $\Tr:W_1(S_K,\theta_K)\to W_1(K,\id_{K})$
		is injective
		(Corollary~\ref{CR:exactness-for-quad-etale}),
		$[g^{(\frakp)}_K]=[g_0]$, and since $\dim_K Q^{(\frakp)}_K=\dim_K Q_0$,
		this means that $g^{(\frakp)}_K\cong g_0$.
		
		Fix some $(W,h)\in\Herm[1]{S,\theta}$ with $\rank_S W=\dim_{S_K} Q_0$.
		Recall from the proof of (b)$\iff$(c) in Proposition~\ref{PR:GS-equiv-conds},
		that $\HH^1_{\et}(R,\uU(h))$ classifies isomorphism
		classes of unimodular $1$-hermitian forms $(W',h')$ over $(S,\theta)$ with $\rank_S W=\rank_S W'$.
		Furthermore, $\uU(h)\to \Spec R$ is reductive (see~\ref{subsec:isometry-group}).
		Thus, by   Colloit-Th\'el\`ene and Sansuc's theorem on purity   in dimension $2$ 
		\cite[Corollary~6.14]{Colliot_1979_quadratic_fiberations},
		there exists $(Q,g)\in \Herm[1]{S,\theta}$
		such that $g_K\cong g_0$.
		
		Note that $[\Tr g_K]=[\Tr g_0]=[f_K]$.
		By \cite[Theorem~5]{Pardon_1982_Gersten_conjecture} or \cite[Corollary~10.2]{Balmer_2002_Gersten_Witt_complex}
		(here we need $\dim R\leq 3$),
		the map $W_1(R,\id_R)\to W_1(K,\id_K)$ is injective,
		so $[\Tr g]=[f]$.
	\end{proof}
	
	\begin{remark}
		Theorem~\ref{TH:kernel-of-rest-quads} also holds
		if $R\to S$ is replaced with a quadratic \'etale
		covering of regular integral schemes $Y\to X$; the proof is exactly the same.
		For the definition of the Witt group in this more general setting, consult
		\cite[\S1.2.1]{Balmer_2005_Witt_groups} and 
		\cite[\S1.5, \S1.6]{Gille_2009_hermitian_GW_complex_II}.
	\end{remark}

\bibliographystyle{plain}
\bibliography{MyBib_18_05}

\end{document}